\documentclass[10pt,a4article]{amsart}
\usepackage{amsthm,amsfonts,amsmath,amssymb,latexsym,amscd}
\usepackage{epsfig,graphics,color}
\usepackage[matrix,arrow,curve]{xy}
\usepackage{hyperref}
\usepackage{mathrsfs}
\usepackage{verbatim}
\usepackage{bm}
\usepackage{enumerate}
\usepackage{tikz-cd}
\usepackage{mathtools}
\title[String topology operations under Chen's iterated integrals]
{String topology operations under Chen's iterated integrals and homotopy transfer}
\author{Kai Cieliebak}
\address{Universit\"at Augsburg, Universit\"atsstrasse 14, 86159 Augsburg, Germany}
\email{kai.cieliebak@math.uni-augsburg.de}
\author{Evgeny Volkov}
\address{Universitat Polit\`ecnica de Catalunya,
Av.~Dr.~Mara\~n\'on 44--50, 08028 Barcelona, Spain}
\email{evgeny.volkov@upc.edu} 
\theoremstyle{plain}
\newtheorem{theorem}{Theorem}[section]

\newtheorem{corollary}[theorem]{Corollary}
\newtheorem{cor}[theorem]{Corollary}
\newtheorem{proposition}[theorem]{Proposition}
\newtheorem{prop}[theorem]{Proposition}
\newtheorem{lemma}[theorem]{Lemma}
\newtheorem{lem}[theorem]{Lemma}

\theoremstyle{remark}

\newtheorem{remark}[theorem]{Remark}

\newtheorem{example}[theorem]{Example}

\newtheorem{definition}{Definition}

\newcommand{\Id}{{{\mathchoice {\rm 1\mskip-4mu l} {\rm 1\mskip-4mu l}
{\rm 1\mskip-4.5mu l} {\rm 1\mskip-5mu l}}}}
\newcommand{\id}{{\rm id}}

\newcommand{\ol}{\overline}

\newcommand{\p}{\partial}

\newcommand{\om}{\omega}
\newcommand{\Om}{\Omega}
\newcommand{\eps}{\varepsilon}
\newcommand{\into}{\hookrightarrow}
\newcommand{\la}{\langle}
\newcommand{\ra}{\rangle}
\newcommand{\wt}{\widetilde}
\newcommand{\wh}{\widehat}
\newcommand{\onto}{\twoheadrightarrow}

\renewcommand{\AA}{\mathcal{A}}

\newcommand{\Z}{{\mathbb{Z}}}
\newcommand{\R}{{\mathbb{R}}}
\newcommand{\C}{{\mathbb{C}}}
\newcommand{\Q}{{\mathbb{Q}}}

\newcommand{\m}{{\bf m}}

\newcommand{\F}{{\bf F}}
\newcommand{\G}{{\bf G}}

\newcommand{\fs}{{\mathfrak s}}

\newcommand{\im}{{\rm im}\,}        

\newcommand{\vol}{{\rm vol}}

\newcommand{\conn}{{\rm conn}}

\newcommand{\inn}{{\rm int}}

\newcommand{\can}{{\rm can}}

\newcommand{\ev}{{\rm ev}}

\newcommand{\pt}{{\rm pt}}

\newcommand{\IBL}{{\rm IBL}}
\newcommand{\dIBL}{{\rm dIBL}}

\newcommand{\cross}{{\rm cross}}

\newcommand{\Flag}{{\rm Flag}}

\newcommand{\ver}{{\rm vert}}
\newcommand{\Ver}{{\rm Vert}}
\newcommand{\Edge}{{\rm Edge}}
\newcommand{\Leaf}{{\rm Leaf}}
\newcommand{\Hom}{{\rm Hom}}

\newcommand{\main}{{\rm main}}
\newcommand{\hidden}{{\rm hidden}}
\newcommand{\gen}{{\rm gen}}

\newcommand{\EE}{\mathcal{E}}
\newcommand{\BB}{\mathcal{B}}

\newcommand{\MM}{\mathcal{M}}
\newcommand{\CC}{\mathcal{C}}
\newcommand{\LL}{\mathcal{L}}

\newcommand{\HH}{\mathcal{H}}
\newcommand{\XX}{\mathcal{X}}
\renewcommand{\AA}{\mathcal{A}}
\newcommand{\UU}{\mathcal{U}}

\newcommand{\ZZ}{\mathcal{Z}}
\newcommand{\YY}{\mathcal{Y}}

\newcommand{\RR}{\mathcal{R}}

\newcommand{\WW}{\mathcal{W}}
\newcommand{\bs}{{\bf s}}
\newcommand{\cyc}{{\rm cyc}}
\newcommand{\g}{{\mathfrak g}}         

\newcommand{\fp}{{\mathfrak p}}

\newcommand{\ff}{{\mathfrak f}}
\newcommand{\fm}{{\mathfrak m}}
\newcommand{\fn}{{\mathfrak n}}
\newcommand{\fF}{{\mathfrak F}}

\newcommand{\fg}{{\mathfrak g}}
\newcommand{\fG}{{\mathfrak G}}

\newcommand{\bH}{{\bf H}}
\newcommand{\W}{\mathcal{W}}

\newcommand{\Bl}{{\rm Bl}_+}
\newcommand{\Alg}{A}
\newcommand{\KS}{{\rm KS}}

\begin{document}
\maketitle
 
\begin{abstract}
We develop the analytical foundations for integrals over configuration
spaces used to relate chain-level $S^1$-equivariant string topology to
perturbative Chern-Simons theory. As an application, we prove that the
composition of Chen's iterated integral with homotopy transfer
intertwines the involutive Lie bialgebra structures on homology.
\end{abstract}

\tableofcontents 


\section{Introduction}\label{sec:intro}

Let $M$ be a closed, connected, oriented $n$-dimensional manifold and
$\Lambda=C^\infty(S^1,M)$ its free loop space. 
In their seminal paper~\cite{Chas-Sullivan99} and its
sequel~\cite{Sullivan-open-closed}, M.~Chas and D.~Sullivan 
introduced operations on the homology of $\Lambda$ which go under the
name {\em string topology}.
In the equivariant setting, they consist of the {\em string bracket}
$\mu^{S^1}$ and the {\em string cobracket} $\lambda^{S^1}$ on the
$S^1$-equivariant homology\footnote{
In this article all (co)homology is with $\R$-coefficients.}
$H_*^{S^1}(\Lambda,\Lambda_0)$ relative to
the constant loops $\Lambda_0\subset\Lambda$. They define the
structure of an {\em involutive Lie bialgebra}. 

For applications in symplectic topology (see
e.g.~\cite{Fukaya-Lag,Cieliebak-Latschev}), it became important to
establish the underlying structure on the chain level. This structure
was algebraically described in~\cite{Cieliebak-Fukaya-Latschev} under
the name {\em $\IBL_\infty$-structure}. Moreover, an approach was
outlined for realizing this structure in the simply connected case.  
It is based on {\em Chen's iterated integrals}~\cite{Chen73,Chen77},
which give rise to a homomorphism 
\begin{equation}\label{eq:chen-intro}
   \bar J_{\lambda*}: H_*^{S^1}(\Lambda,q_0)\longrightarrow \ol{HC}^*_\lambda(\Om).
\end{equation}
Here $q_0$ is a basepoint in $M$, viewed as a constant loop, and 
$\ol{HC}^*_\lambda(\Om)$ is the reduced cyclic cohomology of the de Rham
complex $\Om=\Om^*(M)$, viewed as a differential graded
algebra; see~\S\ref{sec:chen}. Based on Jones' theorem~\cite{Jones},
it is proved in~\cite{Cieliebak-Volkov-cyc} that $\bar J_{\lambda*}$
is an isomorphism if $M$ is simply connected. Since the string bracket
and cobracket canonically lift to operations on $H_*^{S^1}(\Lambda,q_0)$
(denoted by the same letters, see~\S\ref{sec:stringtop}), this transfers
the problem to the right hand side of~\eqref{eq:chen-intro}. 
We can transform the problem further using the isomorphism
\begin{equation}\label{eq:G-intro}
  \G_\lambda^*:HC_\lambda^*(\Om)\stackrel{\cong}\longrightarrow
  H(B^{\cyc *}\HH).
\end{equation}
Here $B^{\cyc *}\HH$ is the {\em dual cyclic bar complex} of a harmonic
subspace $\HH\subset\Om$, and the map $\G_\lambda^*$ is defined via
homotopy transfer of $A_\infty$-algebras (\cite{Kontsevich-Soibelman},
see~\S\ref{sec:ainf}). 
In~\cite{Cieliebak-Volkov}, we construct an $\IBL_\infty$-structure on
$B^{\cyc *}\HH$ and prove that it is independent of all choices up to
$\IBL_\infty$-homotopy equivalence. To show that this is the desired
chain-level structure underlying equivariant string topology, it
remains to relate its induced operations on homology to the string
bracket and cobracket. This is the content of the present paper.

In order to formulate the main result, recall that the
$\IBL_\infty$-structure in~\cite{Cieliebak-Volkov} consists of
operations $\fp_{k,\ell,g}^\fm$ corresponding to compact, oriented,
connected surfaces of genus $g\geq 0$ with $k\geq 1$ incoming and
$\ell\geq 1$ outgoing boundary components (twisted by a Maurer-Cartan
element $\fm$, see~\S\ref{sec:alg}). The operation $\fp_{1,1,0}^\fm$
is the dual of the Hochschild differential on $B^{\cyc *}\HH$ (see
equation~\eqref{eq:p110m} below), and the operations
$\fp_{2,1,0}^\fm=\fp_{2,1,0}$ and $\fp_{1,2,0}^\fm$ descend to
operations on homology $H(B^{\cyc *}\HH)=H(B^{\cyc *}\HH,\fp_{1,1,0}^\fm)$
(denoted by the same letters) defining an involutive Lie bialgebra
structure. 
See also~\cite{Chen-Esh-Gan}.
Consider the composition of degree zero maps
\begin{equation*}
\xymatrix{
  H_*^{S^1}(\Lambda,q_0) \ar[r]^{\bar J_{\lambda *}} 
  & \ol{HC}_\lambda^*(\Om) \ar[r]^{\iota_*} 
  & HC_\lambda^*(\Om) \ar[r]^{\G_\lambda^*\ \ \ \ \ }_{\cong\ \ \ \ \ } 
  & H(B^{\text{\rm cyc}*}\HH,\fp_{1,1,0}^{\m}),
}
\end{equation*} 
where $\iota_*$ is the canonical map from reduced to non-reduced
cyclic cohomology. 
The following is our main theorem. 

\begin{theorem}\label{thm:intro}
In the setup above, the composed degree zero map
\begin{equation*}
  \F := \G_\lambda^*\circ\iota_*\circ \bar J_{\lambda *}:H_*^{S^1}(\Lambda,q_0)\longrightarrow
H(B^{\text{\rm cyc}*}\HH,\fp_{1,1,0}^{\m})
\end{equation*}
intertwines the string bracket $\mu^{S^1}$ with $\fp_{2,1,0}$
and the string cobracket $\lambda^{S^1}$ with $2\,\fp_{1,2,0}^\m$.
\end{theorem}

In view of the isomorphism $\G_\lambda^*$, we define the reduced homology
$$  
  \ol{H}(B^{\text{\rm cyc}*}\HH,\fp_{1,1,0}^\m) := \ol{HC}_\lambda^*(\Om).
$$
Then in the simply connected case we obtain the following corollary.

\begin{cor}\label{cor:intro}
In Theorem~\ref{thm:intro}, assume in addition that
$M$ is simply connected. Then the operations $\fp_{2,1,0}$ and
$2\,\fp_{1,2,0}^\m$ descend to operations $\ol\fp_{2,1,0}$ and $2\,\ol\fp_{1,2,0}^\m$
on $\ol{H}(B^{\text{\rm cyc}*}\HH,\fp_{1,1,0}^\m)$ which correspond to
$\mu^{S^1}$ and $\lambda^{S^1}$ under the isomorphism
\begin{equation*}
  \bar J_{\lambda *}:H_*^{S^1}(\Lambda,q_0)\stackrel{\cong}\longrightarrow
  \ol{H}(B^{\text{\rm cyc}*}\HH,\fp_{1,1,0}^{\m}).
\end{equation*}
\end{cor}

\begin{remark}
A result very similar to Theorem~\ref{thm:intro} has been
obtained previously 
and independently by Naef and Willwacher in~\cite{Naef-Willwacher}. 
Their approach builds for simply connected $M$ on a finite
dimensional Poincar\'e duality model for $\Om^*(M)$ provided by
Lambrechts and Stanley~\cite{Lambrechts-Stanley}, and for non-simply
connected $M$ on a dgca model for the configuration space of points on
$M$ constructed by Campos and Willwacher~\cite{Campos-Willwacher}. 
By contrast, our approach is more analytic, working directly with
configuration space integrals and variants of Chen's iterated
integrals in the smooth setting. 
\end{remark} 

\begin{remark}
Fukaya~\cite{Fukaya-Lag} has suggested a different approach to
chain-level string topology via de Rham chains. This approach has the
advantage of not requiring simple connectivity; it has been successfully
implemented by Irie~\cite{Irie} in the non-equivariant case.
\end{remark}

\begin{remark}
The string bracket and cobracket actually lift to 
$H^{S^1}_*(\Lambda) / \chi(M) H_*^{S^1}(q_0)$, where $\chi(M)$ is the
Euler characteristic of $M$ (see~\S\ref{sec:stringtop}), while the the
operations $\fp_{2,1,0}^\fm=\fp_{2,1,0}$ and $\fp_{1,2,0}^\fm$ exist on 
non-reduced homology $H(B^{\cyc *}\HH,\fp_{1,1,0}^\fm)$. However,
according to~\cite[Theorem 3.5]{Cieliebak-Volkov-cyc}, the Chen map
$\bar J_{\lambda *}$ in~\eqref{eq:chen-intro} only becomes an isomorphism 
for simply connected $M$ if we take homology relative to $q_0$ on the
left hand side and reduced cohomology on the right hand side.
\end{remark}

{\bf Structure of the paper. }
The proof of Theorem~\ref{thm:intro} requires a combination of
algebraic, topological, and analytic techniques. Accordingly, this
paper consists of four parts, of which the first three are largely
independent. 

{\em I. Algebra (\S\S\ref{sec:alg}--\ref{sec:ainf}). }
Here we recall the necessary definitions and facts about cochain
complexes and DGAs with pairings, dIBL-algebras, and $A_\infty$-algebras.

{\em II. Topology (\S\S\ref{sec:stringtop}--\ref{sec:chen}). }
This part concerns string topology and Chen's iterated integrals. 
In~\S\ref{ss:stringtop} we determine the possible domains of
definition for non-equivariant and equivariant string topology
operations. In~\S\ref{ss:transvers} we present a chain level
definition of the loop coproduct which is suitable for combining it
with Chen's iterated integrals and homotopy transfer.

{\em III. Analysis (\S\S\ref{sec:fibre}--\ref{sec:analysis}). }
This part is the heart of the paper. Its results may be of independent
interest; for example, they also serve as the analytic underpinnings
for~\cite{Cieliebak-Volkov}. 
We begin in~\S\ref{sec:fibre} with a formula relating
pullback and fibre integration of integrable differential forms.
This is followed by a discussion of propagators in~\S\ref{sec:prop}.  
In~\S\ref{sec:blow-up} we recall oriented real blow-ups and proper
transforms and develop an abstract setting for Stokes' theorem. 
In~\S\ref{sec:general graphs} we apply this to generalized
configuration spaces associated to graphs. Building on results of
Paw\l{}ucki on semi-analytic sets~\cite{Pawlucki}, we prove Stokes' theorem for such
spaces. We also establish a general vanishing result for integrals
over hidden faces, generalizing earlier such results due to
Kontsevich~\cite{Kontsevich-feynman} and others (see
e.g.~\cite{Bott-Taubes,Cattaneo-Mnev}). 
In~\S\ref{sec:analysis} we apply the results of~\S\ref{sec:general
  graphs} to the configuration spaces that are relevant for the proof
of the main theorem.

{\em IV. Proof (\S\S\ref{sec:graphs}--\ref{sec:rel-stringtop}). }
Here we combine the results of the first three parts to prove
Theorem~\ref{thm:intro} and Corollary~\ref{cor:intro}.
We begin in~\S\ref{sec:graphs} with a detailed discussion of ribbon
graphs, labellings and their extensions, and operations on graphs.
In~\S\ref{sec:intconfig} we define integrals over configuration
spaces associated to graphs and establish their main properties.
In~\S\ref{sec:ops} we use configuration space integrals to define the
chain map and chain homotopies entering the proof and derive their
main properties.  
The actual proof is contained in~\S\ref{sec:rel-stringtop}. 

{\bf Acknowledgements. }
This article is part of the second author's habilitation thesis. 
We thank T.\,Ekholm, M.\,Hutchings and J.\,Latschev for valuable feedback.

\section{Cochain complexes and differential graded algebras with
pairings}\label{sec:alg}

In this section we collect some basic notions and facts about cochain
complexes and differential graded algebras with
pairings. See~\cite{Cieliebak-Fukaya-Latschev,Cieliebak-Hajek-Volkov}
for more background.

\subsection{Graded vector spaces}\label{ss:gradedvect}

Let $A=\bigoplus_{i\in \Z}A^i$ be a $\Z$-graded  
$\R$-vector space. For $m\in\Z$ let $A[m]$ be the degree shift of $A$ by $m$, i.e. 
$A[m]^i:=A^{i+m}$. Most often we will need the degree shift by $1$, that is $A[1]$.
For $x\in A$ of homogeneous degree, i.e. $x\in A^k$ for some $k\in\Z$,
the degree of $x$ as an element of $A$ will be denoted by $\deg x$ and
the degree of $x$ as an element of $A[1]$ will be denoted by $|x|$, so that
$$
  |x|=\deg x-1.
$$
We define the {\em graded dual} of $A$ by
$$
  A^*:=\bigoplus_{i\in \Z}\Hom(A^i,\R),
$$
and grade it by giving $\phi\in\Hom(A^i,\R)$ degree $i$.\footnote{ 
Another frequently used convention gives $\phi\in\Hom(A^i,\R)$ degree $-i$.} 

{\bf Permutation actions. } 
Consider an integer $n$ and an (ordered) partition
$\bs=(s_1,\dots,s_\ell)$ of $s=s_1+\dots+s_\ell$ with $s_b\geq 1$ for
all $b$.
(In later sections, $n$ will be the dimension of a manifold $M$ and 
$s_b$ will be the number of leaves on the $b$-th boundary component of
a ribbon surface, see~\S\ref{ss:basiccomb}.) 
We abbreviate
$$
\Alg(\bs):=\bigotimes_{b=1}^\ell\Alg[1]^{\otimes s_b}[3-n]\,.
$$
Note that abstractly 
$$
  \Alg(\bs)\cong \Alg[1]^{\otimes s}[(3-n)\ell],
$$
but it is important to keep in mind the additional structure induced by $\bs$.
Any decomposable tensor
$$
  \alpha=\alpha^1\otimes\dots\otimes\alpha^{\ell}\in \Alg(\bs)
$$ 
with
$$
  \alpha^b=\alpha^b_1\otimes\dots\otimes\alpha^b_{s_b}\in \Alg[1]^{\otimes s_b}
$$
can also be written in the form
$$
  \alpha=\alpha_1\otimes\dots\otimes\alpha_s\in \Alg[1]^{\otimes s}.
$$
We introduce the operation
\begin{equation}\label{eq:defP}
  P:\Alg^{\otimes s} \to \Alg^{\otimes s},\qquad 
  P(\alpha) = (-1)^{P(\alpha)}\alpha,\quad\text{with}\quad P(\alpha): = \sum_{j=1}^s(s-j)\deg\alpha_j
\end{equation}
and the operation
\begin{equation}\label{eq:defPb}
  P_b:\bigotimes_{b=1}^\ell\Alg[1]^{\otimes s_b}\to \bigotimes_{b=1}^\ell\Alg[1]^{\otimes s_b},\qquad 
  P_b(\alpha) = (-1)^{(3-n)P_b(\alpha)}\alpha
\end{equation}
(the subscript ``$b$'' stands for ``boundary'') with the sign exponent
$$
  P_b(\alpha) := \sum_{b=1}^\ell(\ell-b)|\alpha^b|,
$$
where $|\alpha^b|$ is the total shifted degree of
$\alpha^b\in\Alg[1]^{\otimes s_b}$. 

We denote by $S_s$ the group of permutations of $\{1,\dots,s\}$.
Let $S(\bs)\subset S_s$ be the subset of permutations $\eta$   
consisting of cyclic permutations within each consecutive group of
$s_b$ elements and a permutation $\eta_b\in S_\ell$ of the $\ell$ groups.
(This will later correspond to relabellings of a graph, see~\S\ref{ss:basiccomb}.)
A permutation $\eta\in S(\bs)$ gives rise to a new partition
$\bs\eta_b$ defined by
$$
  (\bs\eta_b)_j:=(\bs)_{\eta_b(j)}.
$$
Let us discuss how $\eta\in S(\bs)$ acts on the tensor product            
$\Alg(\bs)$.
We introduce the naive action permuting factors without signs by
$$
  \eta(\alpha) := \alpha_{\eta(1)}\otimes\cdots\otimes\alpha_{\eta(s)}.
$$
We define the analytic and algebraic actions
$$
  \Alg[1]^{\otimes s}[(3-n)\ell]\cong\Alg(\bs)\longrightarrow
  \Alg(\bs\eta_b)\cong\Alg[1]^{\otimes s}[(3-n)\ell]
$$
as follows:
$$
   \eta_{an}(\alpha):=(-1)^{\eta+(n-1)\eta_b+\eta_{an}(\alpha)}\eta(\alpha),\qquad
   \eta_{alg}(\alpha):=(-1)^{\eta_{alg}(\alpha)}\eta(\alpha).
$$
Here $(-1)^\eta$ is the sign of the permutation $\eta$, 
$\eta_{an}(\alpha)$ is the sign exponent for permuting the elements
$\alpha_j$ with their degrees in $\Alg$, and 
$\eta_{alg}(\alpha)$ is the
sign exponent for cyclically permuting the 
tensor factors of $\alpha^b$ for each $b$
with their shifted degrees $|\alpha_j|=\deg(\alpha_j)-1$, and
then permuting the elements $\alpha_b$
themselves according to $\eta_b$ with
degrees additionally shifted by $3-n\equiv n-1$.  
The two actions are related by the commuting diagram
\begin{equation}\label{eq:alg-ana-action2}
\xymatrix{
  \Alg^{\otimes s} \ar[d]_{P_b\circ P} \ar[r]^{\eta_{an}} &\Alg^{\otimes s} \ar[d]^{P_b\circ P} \\
  \Alg[1]^{\otimes s}[(3-n)\ell] \ar[r]^{\eta_{alg}} &\Alg[1]^{\otimes s}[(3-n)\ell].
}
\end{equation}

\begin{remark}
The above actions of $S(\bs)$ on $\Alg(\bs)$
are based on precomposition, therefore, these are right actions (even though we
write them on the left). We extend these actions to the dual 
$(\Alg(\bs))^*$ by taking conjugate maps.
Since conjugation is a contravariant functor,
the resulting actions on $(\Alg(\bs))^*$ are left actions.
\end{remark}

{\bf Cyclic operations. }
The above terminology allows us to introduce the following maps on tensor powers of $A[1]$. 
Let $\tau_{2\to 1}\in \Z_k$
denote the cyclic permutation $(1,2,\dots,k)\mapsto (k,1,\dots,k-1)$
and $\tau_{1\to 2}:=\tau_{2\to 1}^{-1}$.
Set $t_{an}:=(\tau_{2\to1})_{an}$ and
$t_{alg}:=(\tau_{2\to1})_{alg}$
its associated analytic and
algebraic actions on $A[1]^{\otimes k}$. Explicitly,
\begin{equation}\label{eq:def-talg}
\begin{aligned}
  t_{an}(x_1\otimes\cdots\otimes x_k) &:=
  (-1)^{k-1+x_k(x_1+\cdots+x_{k-1})} x_k\otimes x_1\otimes\cdots\otimes x_{k-1},\cr
  t_{alg}(x_1\otimes\cdots\otimes x_k) &:=
  (-1)^{|x_k|(|x_1|+\cdots+|x_{k-1}|)} x_k\otimes x_1\otimes\cdots\otimes x_{k-1}\,.
\end{aligned}
\end{equation}
Here $k-1$ in the sign exponent is the sign of the permutation $\tau_{2\to 1}$.
We set
\begin{equation}\label{eq:def-Nalg}
  N_{an}:=1+t_{an}+\dots+t_{an}^{k-1},
  \qquad N_{alg}:=1+t_{alg}+\dots+t_{alg}^{k-1}.
\end{equation}
We will be particularly interested in the case that $A=\Om^*(M)$ is
the de Rham algebra of an $n$-manifold $M$. We have the canonical cross product embedding
\begin{equation}\label{eq:cross-emb}
  \times:\Om^*(M)^{\otimes k}\into \Om^*(M^k).
\end{equation}
Define the operation $t_{an}$ (denoted by the same letter) on
$\Om^*(M^k)$ by the pullback
$$
  t_{an}:=(-1)^{k-1}(M^{\tau_{1\to 2}})^*,
$$
where for $\eta\in S_k$ we define $M^\eta:M^k\to M^k$ by
$(x_1,\dots,x_k)\mapsto (x_{\eta(1)},\dots,x_{\eta(k)})$.  
It is straightforward (see also Lemma~\ref{lem:acta}) 
to compute that 
$$
t_{an}\circ\times=\times\circ t_{an}.
$$
In other words, the operation $t_{an}$ on 
$\Om^*(M)^{\otimes k}$ extends to its completion
$\Om^*(M^k)$. This justifies the name ``analytic''.
Similarly for the other analytic operations. For example, $N_{an}$
can be extended to an operation on $\Om^*(M^k)$ by formula~\eqref{eq:def-Nalg}.
Moreover, since $\Z_k$ is generated by $\tau_{1\to 2}$, we can write any element
$\sigma\in\Z_k$ as $\sigma=\tau_{1\to 2}^m$ for some $m=0,\dots,k-1$
and define on $\Om^*(M^k)$ the operation
\begin{equation}\label{eq:def-sigma-an}
  \sigma_{an}:=(t_{an})^m.`
\end{equation}

{\bf The dual cyclic bar complex. }
We define the {\em bar complex} of $A$ and its graded dual by
$$
  BA := \bigoplus_{k=1}^{\infty}A[1]^{\otimes k},\qquad 
  B^*A := \prod_{k=1}^{\infty}(A[1]^{\otimes k})^*.
$$
Note that the direct sum becomes a direct product in the dual.
Similarly, we define the {\em cyclic bar complex} 
$$
  B^{\text{\rm cyc}}A := \bigoplus_{k=1}^{\infty}A[1]^{\otimes k}/\im(1-t_{alg})
$$
and its graded dual, the {\em dual cyclic bar complex} 
$$
  B^{\text{\rm cyc}*}A = \prod_{k=1}^{\infty}\Bigl(A[1]^{\otimes k}/\im(1-t_{alg})\Bigr)^*.
$$
Note that $t_{alg}$ generates the algebraic action of $\Z_k$ on 
$A[1]^{\otimes k}$, and dually on $(A[1]^{\otimes k})^*$, 
and we can identify 
$(A[1]^{\otimes k}/\im(1-t_{alg}))^*$ with the subcomplex of
functionals on $A[1]^{\otimes k}$ fixed under the action of $\Z_k$.


\subsection{Cochain complexes with pairing}\label{ss:coch}

A {\em cochain complex} $(A,d)$ is a graded vector space $A$ with a
differential $d$ of degree $1$. 
A {\em pairing of degree $n\in \Z$} on $(A,d)$ is a
bilinear form $(\cdot,\cdot)\colon A\times A\to \R$ which for all
homogeneous $x, y\in A$ satisfies the degree condition 
\begin{equation*}
	(x,y)\neq 0\quad\Longrightarrow\quad\deg x+\deg y=n,
\end{equation*}
graded symmetry
\begin{align}\label{Eq:Symmetry}
   ( x, y ) = (-1)^{\deg x \deg y} 
   ( y, x ),
\end{align}
and compatibility with the differential
\begin{equation}\label{eq:DiffCyclic}
   ( d x, y ) = (-1)^{1+\deg x} 
   ( x, d y ). 
\end{equation}
We write $x\perp y$ if $( x, y ) = 0$ and say that $x,y$ are \emph{orthogonal}.
The subcomplex of elements of $A$ orthogonal to a given subcomplex
$B\subset A$ will be denoted by  
\[
	B^{\perp} := \{ x\in A\mid x\perp B\}.
\]
We call a pairing $(\cdot,\cdot)\colon A\times
A\to\R$ \emph{nondegenerate} if the induced map
\begin{equation*}
	A\longrightarrow Hom(A,\R),\qquad
	x\longmapsto ( x,\cdot )
\end{equation*}
is injective, and {\em perfect} if it is an isomorphism. Observe that
a nonnegatively graded cochain complex with a perfect pairing is
finite dimensional.  
Following the terminology in~\cite{Cieliebak-Fukaya-Latschev}, a
cochain complex with a perfect pairing is called {\em cyclic}. If a
perfect pairing restricts to a subcomplex as a perfect pairing, then
the subcomplex is called {\em cyclic}.

{\bf Propagators and symmetric projections. }
Next we recall some notions and facts from~\cite{Cieliebak-Hajek-Volkov}.
Let $(A, d,(\cdot,\cdot))$ be a cochain complex with pairing.
A {\em projection} on $A$ is a degree $0$ chain map $\pi\colon A\to A$
which satisfies $\pi\circ \pi = \pi$.  
We say that a degree $-1$ map $P\colon A\to A$ is a 
\emph{homotopy operator} if the map 
$-(d\circ P + P \circ d)\colon A\to A$ is a projection.
Every homotopy operator $P$ determines a projection 
\[
	\pi_P := \Id+ d\circ P+P\circ d\colon A\longrightarrow A
\]
which is a {\em quasi-isomorphism}, i.e., the induced map on cohomology $H(\pi_P)\colon H( A)\to H( A)$ is an isomorphism.
Given a projection $\pi\colon A\to A$ which is a quasi-isomorphism, we say that 
a degree $-1$ map $P\colon A\to A$ is a \emph{homotopy operator with respect to $\pi$} if it is a homotopy operator and $\pi_P=\pi$, so that
\begin{equation}\label{eq:P2} 
	 d\circ P+P\circ d=\pi-\Id. 
\end{equation}
We say that a homotopy operator $P\colon  A\to A$ is a
\emph{propagator} if it satisfies the symmetry property 
\begin{equation}\label{eq:symmpropprelim}
	 (Px,y) = (-1)^{\deg x}(x,Py). 
\end{equation}
The associated projection 
$\pi_P\colon A\to A$ is then \emph{symmetric}:
\begin{equation*}
	 (\pi_P x,y) = ( x,\pi_P y).
\end{equation*}

Given a subcomplex $B\subset  A$, we say that a projection $\pi\colon
A\to A$ is onto $B$ if $\im\pi=B$, and we identify
$\pi$ with the induced surjection $\pi\colon A\onto B$ in this case. 
A homotopy operator~$P\colon A\to A$ with respect to a projection
$\pi\colon A\onto B$ exists if and only if~$\pi$ is a
quasi-isomorphism.  
In the case with pairing we have 

\begin{lemma}[{\cite[Lemma~2.2]{Cieliebak-Hajek-Volkov}}]\label{lem:retraction}
Let $( A, d,(\cdot,\cdot))$ be a cyclic cochain complex and
$B\subset  A$ a subcomplex. 
A propagator $P$ with respect to a projection $\pi\colon A\onto B$
exists if and only if~$\pi$ is symmetric and a quasi-isomorphism. 
\end{lemma}

If $(A,d,(\cdot,\cdot))$ is a cyclic cochain complex, then any
quasi-isomorphic cyclic subcomplex $B\subset A$ admits a unique
symmetric projection $\pi_B\colon A \onto B$ by sending the orthogonal
complement $B^\perp$ to $0$.
Lemma~\ref{lem:retraction} now implies

\begin{cor}\label{cor:retraction}
Let $( A, d,(\cdot,\cdot))$ be a cyclic cochain complex and
$B\subset  A$ a quasi-isomorphic cyclic subcomplex. 
Then there exists a propagator 
$P\colon A\to A$ such that $\im\pi_P=B$.
\qed
\end{cor}

{\bf Harmonic subspaces. }
Let $(A,d,(\cdot,\cdot))$ be a cochain complex with pairing. In view
of~\eqref{eq:DiffCyclic}, the pairing descends to its homology $H(A)$.  
In this subsection we consider the case that $H(A)$ is finite dimensional
and the induced pairing on $H(A)$ is nondegenerate, so $H(A)$ becomes
a cyclic cochain complex with trivial differential.
Our main example for this will be the de Rham complex 
$(\Om^*(M), d,(\cdot,\cdot))$ of a closed oriented manifold with the 
intersection pairing~\eqref{eq:defineintpair}.

By a {\em harmonic subspace} $\HH$ we mean any complement of $\im d$ in
$\ker d$, so that
$$
   \ker d=\HH\oplus\im d.
$$

\begin{lemma}\label{lem:Horthog}
Let $(A,d,(\cdot,\cdot))$ be a cochain complex with pairing such that
$H(A)$ is finite dimensional and the induced pairing on $H(A)$ is nondegenerate.
Let $\HH\subset\ker d$ be a harmonic subspace. Then we get a direct sum decomposition
\begin{equation}\label{eq:Horthog}
  A=\HH\oplus\HH^\perp\quad\text{with}\quad \HH^\perp\cap\ker d=\im d.
\end{equation}
The projection $\pi_\HH:A\to A$ onto $\HH$ along $\HH^\perp$ is
symmetric, and it is a quasi-isomorphism as a map $A\to\HH$. 
\end{lemma}

\begin{proof}
Nondegeneracy of the pairing on $H(A)$ and the fact that every
cohomology class has a unique harmonic representative implies that 
the restriction of the pairing to $\HH$ is nondegenerate.
Pick a basis $\{e_a\}$ of $\HH$ and define its dual basis $\{e^b\}$ of
$\HH$ by $(e_a,e^b)=\delta_a^b$. Then
$$
\pi_\HH:A\to A,\qquad \pi_\HH(x) :=
\sum_a(e_a,x)e^a
$$
is a projection with image $\HH$ and kernel $\HH^\perp$, which
shows the first equation in~\eqref{eq:Horthog}. 
The inclusion $\HH^\perp\cap\ker d\supset\im d$ is obvious. 
For the converse inclusion, consider $x\in\HH^\perp\cap\ker d$.
In view of~\eqref{eq:DiffCyclic} this implies $x\perp \HH\oplus\im
d=\ker d$, and therefore $0=(x,y)=([x],[y])$ for all $y\in\ker d$,
where $[x],[y]$ denote the cohomology classes. By nondegeneracy of the 
pairing on homology this implies $[x]=0$, hence $x\in\im d$. 
This proves the second equation in~\eqref{eq:Horthog}, and the
asserted properties of $\pi_\HH$ follow directly from this. 
\end{proof}

We combine Corollary~\ref{cor:retraction} 
and Lemma~\ref{lem:Horthog} to get the following statement.

\begin{cor}\label{cor:harm-prop}
Let $(A,d,(\cdot,\cdot))$ be a cochain complex with pairing such that
$H(A)$ is finite dimensional and the induced pairing on $H(A)$ is nondegenerate.
Let $\HH\subset\ker d$ be a harmonic subspace.
Then there exists a propagator 
$P\colon A\to A$ such that its projection  
$\pi_P$ is exactly the orthogonal projection onto $\HH$.
\qed
\end{cor}

\subsection{$\IBL_\infty$ and ${\rm dIBL}$ structures}\label{ss:dIBL} 

Here we recall from~\cite{Cieliebak-Fukaya-Latschev} the notions of 
$\IBL_\infty$ structures and their Maurer-Cartan elements, and spell
them out in the special case of $\dIBL$ structures relevant to this paper. 

{\bf $\IBL_\infty$-structures. }
Consider a $\Z$-graded $\R$-vector space $C$. For $k\geq 1$ we
define the $k$-fold symmetric product
$$
    E_kC := C[1]^{\otimes k}/S_k 
$$
as the quotient by the algebraic action of the symmetric group $S_k$. 
According to~\cite{Cieliebak-Fukaya-Latschev}, an {\em $\IBL_\infty$
structure of degree $d\in\Z$} on $C$ is a collection of linear maps
$$
    \fp_{k,\ell,g} : E_kC \to E_{\ell}C,\qquad k\ge 1,\ \ell \ge 1,\
    g\geq 0
$$ 
of degree 
$$
    |\fp_{k,\ell,g}| = -2d(k+g-1)-1
$$ 
satisfying a series of quadratic relations which can formally be
written as $\hat\fp\circ\hat\fp=0$ for the generating series
$\hat\fp$ of the operations $\fp_{k,\ell,g}$. We will spell out these
relations below in the special case of a $\dIBL$ structure relevant
for this paper. 

In~\cite{Cieliebak-Fukaya-Latschev}, the notions of
$\IBL_\infty$-morphisms and their homotopies are defined, and it is
proved that quasi-isomorphisms are homotopy equivalences. Morphisms
from the ground field give rise to the notion of a {\em Maurer-Cartan
  element}. This is a collection of elements 
$$
    \m_{\ell,g}\in \wh E_\ell C,\qquad \ell\geq 1,\ g\geq 0
$$
in a suitable completion $\wh E_\ell C$ of $E_\ell C$ of degree
$$
   |\m_{\ell,g}| = -2d(g-1)
$$ 
satisfying a series of quadratic relations which can formally be
written as $\hat\fp(e^{\m}) = 0$.
A Maurer-Cartan element $\m=\{\m_{\ell,g}\}$ gives rise to a {\em
  twisted $\IBL_\infty$-structure} $\fp^\m=\{\fp_{k,\ell,g}^\m\}$. 
Again, we will spell this out in the relevant special case to which we
now turn. 

{\bf $\dIBL$ structures. }
Let 
$$
  \tau_{alg}:C[1]^{\otimes 2}\to C[1]^{\otimes 2},\qquad 
  N_{alg}:C[1]^{\otimes 3}\to C[1]^{\otimes 3}
$$
be the algebraic action of the flip of two elements and
the symmetrization operator defined in~\S\ref{ss:gradedvect}.
A {\em $\dIBL$-structure of degree $d$} is an $\IBL_\infty$-structure
consisting of only the three operations
\begin{align*}
  \fp_{1,1,0}:C[1]\to C[1],\qquad
  \fp_{2,1,0}:C[1]^{\otimes 2}\to C[1],\qquad
  \fp_{1,2,0}:C[1]\to C[1]^{\otimes 2}
\end{align*}
(with all operations zero) satisfying the symmetry properties 
\begin{align*}
  \fp_{2,1,0}\circ \tau_{alg}=\fp_{2,1,0},\qquad
  \tau_{alg}\circ\fp_{1,2,0}=\fp_{1,2,0}.
\end{align*}
Their degrees are
$$
  |\fp_{1,1,0}|=-1,\qquad |\fp_{2,1,0}|=-2d-1,\qquad |\fp_{1,2,0}|=-1
$$
and the quadratic relations $\hat\fp\circ\hat\fp=0$ spell out as follows:
\begin{align*}
&\fp_{1,1,0}\circ\fp_{1,1,0}=0,\cr
& \fp_{1,1,0}\circ \fp_{2,1,0}+
\fp_{2,1,0}\circ(\fp_{1,1,0}\otimes \id+\id\otimes \fp_{1,1,0})=0,\cr
&(\fp_{1,1,0}\otimes \id+\id\otimes \fp_{1,1,0})\circ \fp_{1,2,0}+
\fp_{1,2,0}\circ\fp_{1,1,0}=0, \cr
&\fp_{2,1,0}\circ(\fp_{2,1,0}\otimes \id)\circ N_{alg}=0,\cr
&N_{alg}\circ(\fp_{1,2,0}\otimes \id)
\circ\fp_{1,2,0}=0,\cr
&\fp_{1,2,0}\circ\fp_{2,1,0}
+(\fp_{2,10}\otimes \id)\circ(\id\otimes\fp_{1,2,0})
+(\id\otimes\fp_{2,10})\circ(\fp_{1,2,0}\otimes \id)=0,\cr
&\fp_{2,1,0}\circ\fp_{1,2,0}=0.
\end{align*}
The first relation says that $\fp_{1,1,0}$ is a differential, the next
two relations show that $\fp_{2,1,0}$ and $\fp_{1,2,0}$ descend to
homology $H(C[1],\fp_{1,1,0})$, and the last four relations mean that
they define on homology the structure of an involutive Lie bialgebra,
see~\cite{Cieliebak-Fukaya-Latschev}. 

{\bf Maurer-Cartan elements in $\dIBL$-algebras. }
Consider now a Maurer-Cartan element $\m=\{\m_{\ell,g}\}$ in a
$\dIBL$-algebra $(C,\fp_{1,1,0},\fp_{2,1,0},\fp_{1,2,0})$. For the
purposes of this paper, we will only need the following two facts.
First, according to~\cite[equation (9.5)]{Cieliebak-Fukaya-Latschev},
if the only nontrivial term in $\m$ is $\m_{1,0}$ then the
Maurer-Cartan equation $\hat\fp(e^{\m}) = 0$ spells out as the two
equations  
\begin{equation}\label{eq:discMC}
\fp_{1,1,0}\m_{1,0}+\frac{1}{2}\fp_{2,1,0}
(\m_{1,0}\otimes\m_{1,0})=0,
\end{equation}
\begin{equation}\label{eq:p120MC}
\fp_{1,2,0}\m_{1,0}=0.
\end{equation}
Second, according to~\cite[equation (9.5)]{Cieliebak-Fukaya-Latschev},
the first three operations of the twisted $\IBL_\infty$ structure
associated to an arbitrary Maurer-Cartan element $\m$ are given by
\begin{equation}\label{eq:twistedoper}
\begin{aligned}
  \fp_{1,1,0}^\m =& \fp_{1,1,0}+\fp_{2,1,0}(\m_{1,0},\cdot),\qquad
  \fp_{2,1,0}^\m = \fp_{2,1,0}, \cr
  \quad &\fp_{1,2,0}^\m = \fp_{1,2,0} + \wh\fp_{2,1,0}^\conn(\m_{2,0},\cdot).
\end{aligned}
\end{equation}
To spell out the expression $\wh\fp_{2,1,0}^\conn(\m_{2,0},\cdot)$, we abbreviate
\begin{equation}\label{eq:p21012}
  p_{210}^{12} := \fp_{2,1,0}\otimes\id: C[1]^{\otimes 3}\to\R.
\end{equation}
According to~\cite[equation~(2.1)]{Cieliebak-Fukaya-Latschev} we then have
\begin{equation}\label{eq:210120hat}
  \wh\fp_{2,1,0} = p_{210}^{12} \circ
  (\id+\tau_{23}+\tau_{23}\circ\tau_{12}): C[1]^{\otimes 3}\to C[1]
\end{equation}
with the algebraic action of the transpositions $\tau_{12},\tau_{23}$
on $C[1]^{\otimes 3}$. Using this, we compute for any $\psi\in C[1]$:
\begin{equation}\label{eq:hat-unwrap}
\begin{aligned}
  \wh\fp_{2,1,0}^\conn(\m_{2,0}\otimes \psi) 
  &\stackrel{(1)}{=} p_{210}^{12}\circ(\tau_{23} +
  \tau_{23}\circ\tau_{12}) (\m_{2,0}\otimes \psi) \cr
  &\stackrel{(2)}{=} 2p_{210}^{12}\circ\tau_{23}(\m_{2,0}\otimes \psi).
\end{aligned}
\end{equation}
Here equality~(1) follows from~\eqref{eq:210120hat}, noting that the
term with $\id$ drops out because it gives rise to a disconnected
surface (see~\cite[\S2]{Cieliebak-Fukaya-Latschev}), and equality~(2)
follows from the symmetry $\m_{2,0}\circ\tau_{12}=\m_{2,0}$.

{\bf The ${\rm dIBL}$ structure associated to a cyclic cochain
  complex. }
Let now $(A=\bigoplus_{i=0}^nA^i,d,(\cdot,\cdot))$ be a cyclic cochain
complex of degree $n$. 
In the following the Einstein summation convention will be understood. 
Let $\{e_a\}$ be a basis of $A$ and $\{e^a\}$ the dual basis with
respect to the {\em cyclic pairing}
\begin{equation}\label{eq:defcycpair}
  \la x,y\ra: =( -1)^{\deg x}(x,y),
\end{equation} 
i.e.
$$
     \la e_a,e^b\ra = \delta_a^b.
$$
We define the coproduct
\begin{equation*}
   c_{120}:BA\longrightarrow BA\otimes BA
\end{equation*}
on elementary tensors $x=x_1\otimes\cdots\otimes x_k$ of homogeneous degree by 
\begin{equation}\label{eq:defc120}
   \begin{aligned}
   c_{120}(x):=\sum_{k_1=0}^k(-1)^{|e^a|\,|x^{(1)}|+|e_a|+(n-1)|e_ax^{(1)}|}(e_a\otimes
   x_1\otimes\cdots \otimes x_{k_1})\otimes\cr 
   (e^a\otimes x_{k_1+1}\otimes\cdots\otimes x_k),
   \end{aligned}
\end{equation}
where we have abbreviated $x^{(1)}:=x_1\otimes\cdots\otimes x_{k_1}$.
Similarly, we define the product 
\begin{equation*}
   c_{210}:BA\otimes BA\longrightarrow BA
\end{equation*}
on elementary tensors $x:=x_1\otimes\cdots\otimes x_{k_1}$ and
$y:=y_1\otimes\cdots\otimes y_{k_2}$ by
\begin{equation}\label{eq:defc210}
   c_{210}(x\otimes y):=\frac{1}{2}(-1)^{|e^a|\,|x|+|e_a|+ (n-1)|x|}e_a\otimes
   x_1\otimes\cdots\otimes x_{k_1}\otimes e^a\otimes y_1\otimes\cdots\otimes
   y_{k_2}. 
\end{equation}
It is straightforward but tedious to verify that the maps
$c_{120}$ and $c_{210}$ do not depend on the chosen basis $\{e_a\}$.
Using the cyclization operator $N_{alg}$ from~\eqref{eq:def-Nalg}, we define
\begin{equation}\label{eq:dIBLcanon}
  \fp_{1,1,0}:=d^*,\qquad
  \fp_{2,1,0}:=(c_{120}\circ N_{alg})^*,\qquad
  \fp_{1,2,0}:=(c_{210}\circ N_{alg}^{\otimes 2})^*.
\end{equation}

\begin{proposition}[{\cite{Cieliebak-Fukaya-Latschev}}]\label{prop:canondIBL}
The triple 
$(\fp_{1,1,0},\,\fp_{2,1,0},\,\fp_{1,2,0})$ defined in~\eqref{eq:dIBLcanon}
constitutes a ${\rm dIBL}$ structure  of degree $n-3$ on $B^{\text{\rm cyc}*}A[2-n]$.
\end{proposition}

Following~\cite{Hajek-thesis}, we denote this dIBL-algebra by 
$$
  \dIBL(A) := (B^{\text{\rm cyc}*}A[2-n],\,\fp_{1,1,0},\,\fp_{2,1,0},\,\fp_{1,2,0}).
$$

\subsection{Differential graded algebras with pairings}\label{ss:dga1} 

In this subsection we recall the Hochschild and cyclic complexes of a
differential graded algebra (DGA), and the canonical Maurer-Cartan
element of a cyclic DGA. 

{\bf DGAs and their Hochschild and cyclic complexes. }
A {\em differential graded algebra (DGA)} $(A,d,\wedge)$ is a 
nonnegatively graded cochain complex $(A,d)$ equipped with an
associative product $\wedge\colon A\times A \to A$ of degree  
$0$ satisfying the Leibniz rule
$$
	d(x\wedge y) = d x \wedge y + 
	(-1)^{\deg x} x \wedge d y.
$$
Let us recall the definition of its Hochschild and cyclic complexes,
see~\cite{Loday} and~\cite[Example 2.8]{Cieliebak-Volkov-cyc}.   
We define 
$$
  C(A) := \bigoplus_{l=0}^\infty C_l(A),\qquad C_l(A):=A\otimes A[1]^{\otimes l}.
$$
Note that this differs from the bar complex $BA$ by a missing degree
shift in the first tensor component, so the identity defines a canonical map of degree $0$
\begin{equation}\label{eq:can-deg-0}
  \iota: C(A)[1]\longrightarrow BA. 
\end{equation}
It is customary to denote decomposable elements of $C_l(A)$ by
$[a_0\mid a_1\mid \cdots\mid a_l]$ with vertical bars in place of
tensor product signs when dealing with the Hochschild complex.
The image $\iota [a_0\mid a_1\mid \cdots\mid a_l]$
in $BA$ will usually be renumbered from 
$1$ to $l+1$. If the degree shift and renumbering are clear from
context we will omit $\iota$ in the notation.

We extend the differential $d$ from $A$ to $C(A)$ as a derivation with
respect to the {\it nonshifted} degrees.
The classical operations on $C(A)$ are defined on decomposable elements
of homogeneous degree by the formulas
\begin{equation}\label{eq:di}
\begin{aligned}
   d_i[a_0\mid\dots\mid a_l] &:= [a_0\mid\dots\mid
   a_i\wedge a_{i+1}\mid\dots\mid a_l]\quad\text{for }0\leq i\leq l-1,\cr
   d_l[a_0\mid\dots\mid a_l] &:= (-1)^{\deg a_l(\deg a_0+\dots+\deg a_{l-1})}
   [a_l\wedge a_0\mid a_1\mid\dots\mid a_{l-1}],\cr 
   t_l[a_0\mid \dots\mid a_l] &:= (-1)^{\deg a_l(\deg a_0+\dots+\deg a_{l-1})}
   [a_l\mid a_0\mid\dots\mid a_{l-1}]. 
\end{aligned}
\end{equation}
These operations give rise to operations $b$, $t$ and $N$ defined 
on homogeneous elements $c\in C_l(A)$ as follows
(cf.~\cite[equation~(4)]{Cieliebak-Volkov-cyc}): 
\begin{equation}\label{eq:cyc-signs}
\begin{aligned}
   b(c) &:=(-1)^{\deg c-l+1}\sum_{i=0}^l(-1)^id_i(c)\in C_{l-1}(A), \cr 
   t &:= (-1)^lt_l:C_l(A)\to C_l(A), \cr
   N &:= 1+t+\cdots+t^l:C_l(A)\to C_l(A).
\end{aligned}
\end{equation}
Note that in the notation of \S\ref{ss:gradedvect} we have $t=t_{an}$
and $N=N_{an}$,
since the sign of the cyclic rotation of $l+1$ elements is $(-1)^l$.
The above operations give rise to the {\em Hochschild complex}
$$
   \bigl(C(A),d+b\bigr)
$$ 
and the {\em Connes cyclic complex} 
\begin{equation}\label{eq:Connes-complex}
   \Bigl(C^{\lambda}(A):=\iota(C(A)[1])/\im(1-t),d+b\Bigr).
\end{equation}
Note that the canonical projection
\begin{equation}\label{eq:connes-deg-0}
C(A)[1]\rightarrow C^{\lambda}(A)
\end{equation}
has degree $0$.
If $A$ is unital we denote by $C^{\lambda}_{(1)}(A)\subset
C^{\lambda}(A)$ the subspace generated 
by elements containing $1$ at some position. Since the differential $d+b$ preserves 
$C^{\lambda}_{(1)}(A)$, this gives rise to the {\em reduced Connes cyclic complex} 
$$
   \Bigl(\overline{C^{\lambda}}(A):=C^{\lambda}(A)/C^{\lambda}_{(1)}(A), d+b\Bigr).
$$
The corresponding homologies are 
\begin{itemize}
\item $HH_*(A):=H(C(A),d+b)$ --- Hochschild homology,
\item $HC^\lambda_*(A):=H(C^{\lambda}(A),d+b)$ --- cyclic homology, 
\item $\overline{HC}^{\lambda}_*(A):=H(\overline{C^{\lambda}}(A),d+b)$
  --- reduced cyclic homology. 
\end{itemize}
The homologies of the corresponding dual complexes are
\begin{itemize}
\item $HH^*(A)$ --- Hochschild cohomology, 
\item  $HC_{\lambda}^*(A)$ --- cyclic cohomology
\item $\overline{HC}_{\lambda}^*(A)$ --- reduced cyclic cohomology.
\end{itemize}

{\bf DGAs with pairing. }
A {\em pairing} on a DGA $(A,d,\wedge)$ is a pairing 
$(\cdot,\cdot)\colon A \times A\to\R$ on the cochain complex $(A,d)$ 
satisfying in addition
\begin{equation}\label{Eq:ProdCyclic}
	( x \wedge y,z) = (x, y\wedge z).
\end{equation}
Using \eqref{Eq:Symmetry}, condition \eqref{Eq:ProdCyclic} is
equivalent to the \emph{cyclicity condition}
\begin{equation}\label{Eq:ProdCyclicOld}
	(x\wedge y, z)  = (-1)^{\deg z(\deg x+\deg y)} (z\wedge x, y).
\end{equation}
The last condition is in turn equivalent to the {\em triple intersection product}
\begin{equation}\label{eq:deftriple}
   \m^\can(x_0,x_1,x_2):= 
   (-1)^{\deg x_1+n}(x_0\wedge x_1,x_2)
\end{equation}
defining an element in $B_3^{\text{\rm cyc}*}A$. 

Our main example of a DGA with pairing is the de Rham algebra
$(\Om^*(M),d,\wedge)$ of a closed oriented manifold $M$ with the wedge
product $\wedge$ and the {\em intersection pairing}
\begin{equation}\label{eq:defineintpair}
  (\alpha,\beta):=\int_M\alpha\wedge\beta.
\end{equation}
Note that the intersection pairing 
$(\cdot,\cdot)$ is nondegenerate but not perfect. 

{\bf Cyclic DGAs and their Maurer-Cartan elements. }
A {\em cyclic DGA} is a DGA with a perfect pairing.  
In this context we have two important examples of Maurer-Cartan
elements. 

\begin{example}\label{ex:canon}
Let $(A,d,\wedge,(\cdot,\cdot))$ be a cyclic DGA. Recall
from~\S\ref{ss:dIBL} the canonical $\dIBL$ structure $\dIBL(A)$
associated to the cyclic cochain complex $(A,d,(\cdot,\cdot))$.
According to~\cite[Proposition 12.5]{Cieliebak-Fukaya-Latschev}, the
element $\m^\can\in B_3^{{\cyc}*}A$ from~\eqref{eq:deftriple} gives
rise to a Maurer-Cartan element $\m^\can$ in $\dIBL(A)$ (denoted by
the same letter by a slight abuse of notation) by setting
$\m^\can_{1,0}:=\m^\can$ and $\m^\can_{\ell,g}:=0$ for $(\ell,g)\neq(1,0)$. We
will call it {\em the canonical Maurer-Cartan element}
associated to a cyclic $DGA$.
\end{example}

The following example goes back to~\cite[Remark~12.11]{Cieliebak-Fukaya-Latschev}. 
Although it will not be used in this paper, we include it for
completeness because it motivates the analytic Maurer-Cartan element
in~\S\ref{ss:MCan}.

\begin{example}\label{ex:pushfalg}
Let $(A,d,\wedge,(\cdot,\cdot))$ be a cyclic DGA and $\HH\subset \AA$
a harmonic subspace. Let $P$ be a propagator corresponding to $\HH$
(which exists according to Corollary~\ref{cor:harm-prop}). 
Then the canonical Maurer-Cartan element $\m^\can$ on $\dIBL(A)$
induces via pushforward under homotopy transfer a Maurer-Cartan
element $\m^\HH$ on $\dIBL(\HH)$. According
to~\cite[Corollary~5.7]{Cieliebak-Volkov}, the value of
$\m^\HH_{\ell,g}$ on a tensor product $\alpha=\alpha^1_1\cdots
\alpha^1_{s_1}\otimes\dots\otimes \alpha^\ell_1\cdots \alpha^\ell_{s_\ell}$
of $\alpha^b_j\in \HH$ is given by  
\begin{equation}\label{eq:pushfalg-lg}
   \m^\HH_{\ell,g}(\alpha) = 
   \sum_{\Gamma\in \RR_{\ell,g}}\m^\HH_\Gamma(\alpha)
\end{equation} 
with
\begin{align}\label{eq:pushfalg}
    \m^\HH_\Gamma(\alpha)
    := \frac 1 {\ell!}\sum_{\frak i \in
      \frak I}\pm\hspace{-10pt} \prod_{t\in
    \Edge(\Gamma)} \hspace{-10pt}P^{\frak i(1,t),\frak i(2,t)}\hspace{-10pt}
    \prod_{v \in \Ver(\Gamma)}\hspace{-10pt} \m^{\can}(e_{\frak i(v,1)}, e_{\frak i(v,2)}, e_{\frak i(v,3)}).
\end{align} 
Here $\RR_{\ell,g}$ denotes the set of isomorphism classes of
connected trivalent ribbon graphs $\Gamma$ of genus $g$ with $\ell$
boundary components (see~\S\ref{sec:graphs}). Given such a graph
$\Gamma$, we denote by $\frak I$ the set of maps from the interior
flags of $\Gamma$ to the index set $I$ for a chosen basis
$\{e_a\}_{a\in I}$ of $A$, and by $P^{ab}=\la Pe^a,e^b\ra$ the
components of $P$ with respect to the dual basis $\{e^a\}$.
\end{example}

\section{$A_\infty$-algebras}\label{sec:ainf}

In this section we recall some background on $A_\infty$-algebras, see
e.g.~\cite{Keller}, and generalize some notions of
Section~\S\ref{ss:gradedvect} to this setting.  

{\bf Notation.}
Throughout this paper we abuse notation by writing the sign
$(-1)^{\deg a}$ associated to an element as $(-1)^a$, the sign of
a permutation $\sigma$ as $(-1)^\sigma$, etc. This becomes a bit awkward
when we combine sign exponents to expressions like $a+\eta$, but it
should always be clear what is meant. 
\smallskip

Let $\AA=\bigoplus_{i\geq 0}\AA^i$ be a nonnegatively graded $\R$-vector space.
A sequence $\fm=\{\fm_k\}_{k\ge 1}$ of operations 
$$
   \fm_k:\AA[1]^{\otimes k}\longrightarrow \AA[1],\quad k\ge 1
$$
of degree $1$ is called an {\em $A_\infty$-structure} on $\AA$ if for each
integer $r\geq 1$ the following holds:
$$
\sum_{k+l=r+1,\atop k,l\ge 1}\sum_{c=1}^k
\fm_k\circ (\id^{\otimes (c-1)}\otimes \fm_l\otimes\id^{\otimes (k-c)})=0.
$$
Here we adopt the convention to write operations to the left of elements and stipulate 
that moving an element $x$ past an operation of degree $q$ produces a sign $(-1)^{q|x|}$.
The pair $(\AA,\fm)$ is then called an {\em $A_\infty$-algebra}. 
Note that $\fm_1$ is a differential on $\AA$.

Let $(\AA,\fm)$ and $(\BB,\fn)$ be two $A_\infty$-algebras. A collection
$\ff=\{\ff_k\}_{k\ge 1}$ of linear maps
$$
  \ff_k:\AA[1]^{\otimes k}\longrightarrow \BB[1]
$$
of degree $0$ is called an {\em $A_\infty$-morphism} if the following
identity holds for all $k\ge 1$: 
$$
\sum_{i+j<k}\ff_{i+j+1}\circ (id^{\otimes i}\otimes \fm_{k-(i+j)}\otimes id^{\otimes j})=
\sum_{i_1+\cdots+i_r=k}\fn_r\circ (\ff_{i_1}\otimes\cdots\otimes\ff_{i_r}).
$$
The map $\ff_1:\AA[1]\rightarrow \BB[1]$ is called the {\em linear
  part} of $\ff$. Note that $\ff_1$ is a chain map 
with respect to the differentials $\fm_1$ on $A[1]$ and $\fn_1$ on $B[1]$. 
A morphism $\ff$ is called a {\em quasi-isomorphism} if its linear
part induces an isomorphism on homology, and in this case the two
$A_\infty$- algebras are called {\em quasi-isomorphic}. It follows
from~\cite[Section~10.4.3]{Loday-Vallette} that being quasi-isomorphic
is an equivalence relation. For the notions of homotopy of
$A_\infty$-morphisms and homotopy equivalence see~\cite[Section~3.7]{Keller}.
The following classical result can be 
found in~\cite{Lefevre-Hasegawa}. See also~\cite[Theorem 6]{Vallette}. 

\begin{theorem}\label{thm:quismhequiv}
Two $A_\infty$-algebras are quasi-isomorphic if and 
only if they are homotopy equivalent.
\end{theorem}


\subsection{Hochschild and cyclic homology of $A_\infty$-algebras}\label{ss:hochcycainf}

Let $(\AA,\fm)$ be an $A_\infty$-algebra. We define its bar complex
as in~\S\ref{ss:gradedvect} as the graded vector space
$$
  B\AA:=\bigoplus_{l\ge 1}B_l\AA,\qquad B_l\AA:=\AA[1]^{\otimes l}.
$$
Recall the maps $t_{alg}$ and $N_{alg}$ from~\S\ref{ss:gradedvect},
which in the context of $A_\infty$-algebras we denote by 
$$
  t_\AA:=t_{alg},\qquad N_\AA:=N_{alg}.
$$
We define the differential $b_\AA$ on $B_l\AA$
by the formula
$$
  b_\AA := \sum_{k=1}^l(\fm_k\otimes id^{\otimes (l-k)})\circ
  (1+t_\AA+\cdots+t_\AA^{k-1}) + \sum_{k=1}^l\sum_{c=1}^{l-k}id^{\otimes c}\otimes
  \fm_k\otimes id^{\otimes (l-k-c)}.
$$
This differential descends to the cyclic bar complex
$$
   B^{\text{\rm cyc}}\AA = B\AA/\im(1-t_\AA).
$$
It agrees with~\cite[Remark~2.4]{Mescher} after suitably adjusting signs.
The corresponding discussion with our sign conventions can be found 
in~\cite[Section~3.3]{Hajek-thesis}.
The complexes $(B\AA,b_\AA)$ and $(B^{\text{\rm cyc}}\AA,b_\AA)$ are called the
{\em Hochschild and cyclic complex}, respectively, and the
corresponding homologies 
$$
    HH(\AA) := H(B\AA,b_\AA),\qquad HC^\lambda(\AA) := H(B^{\text{\rm cyc}}\AA,b_\AA)
$$
are called the {\em Hochschild and cyclic homology} of the $A_\infty$-algebra $\AA$.

\begin{remark}
We use the notation $B\AA$ and $B^{\text{\rm cyc}}\AA$ in order to distinguish it
from the Hochschild and cyclic complexes $C(A)$ and $C^\lambda(A)$ of a DGA.
\end{remark}

\subsection{From DGAs to $A_\infty$-algebras}\label{ss:dgaainf} 

Next we explain the transition between DGAs and
$A_\infty$-algebras. This will be used frequently in the sequel to
match the de Rham complex (which is a DGA) with the results
in~\cite{Cieliebak-Fukaya-Latschev} (which are formulated for
$A_\infty$-algebras).  

Let $(A,\wedge,d)$ be a DGA. We turn it into an $A_\infty$-algebra as follows. 
Set $\AA:=A$ as a graded vector space and equip it with the operations
\begin{equation}\label{eq:dgatoainf}
   \fm_1:=d,\qquad \fm_2(x,y):=(-1)^xx\wedge y,\qquad
   \fm_k:=0 \text{ for }k\ge 3.
\end{equation}

In this situation we have two Hochschild complexes, $(B\AA,b_\AA)$
from~\S\ref{ss:hochcycainf} and $(C(A),d+b)$ from~\S\ref{ss:dga1}. 
The underlying vector spaces are the same by construction, but their
gradings differ by $1$. 
Our next goal is to relate the differentials.      
Recall the sign operator $P$ from~\eqref{eq:defP}. 
We introduce another sign operator
$$
  Q:\bigoplus_{l\ge 1} A^{\otimes l}\longrightarrow \bigoplus_{l\ge 1} A^{\otimes l}
$$
acting on decomposable elements 
$x=x_1\otimes\cdots\otimes x_l$ of homogeneous degree by
\begin{equation}\label{eq:defQ}
  Q(x_1\otimes\cdots\otimes
  x_l):=(-1)^{Q(x)}x_1\otimes\cdots\otimes x_l,\qquad
  Q(x)=x(l+1)+\frac{l(l+1)}{2}.
\end{equation}
The following lemma is a straightforward but tedious sign computation,
see~\cite{Volkov-thesis}. 

\begin{lemma}\label{lem:QP}
In the setup above and with $t$ defined in~\eqref{eq:cyc-signs} we have
$$
Q\circ t=t\circ Q,\qquad P\circ t_\AA=t\circ P
$$
and
$$
Q\circ P\circ b_\AA=(d+b)\circ Q\circ P.
$$
\end{lemma}


\subsection{Functoriality of Hochschild and cyclic homology}\label{ss:functoriality}

It is a standard fact
that the Hochschild complex is functorial, see
e.g.~\cite[Theorem~2.8]{Mescher}. That is,  
a morphism $\ff:\AA\rightarrow \BB$ of $A_\infty$-algebras induces 
a chain map 
\begin{equation}\label{eq:fF}
  \fF = \bigoplus_{1\le l\le N}\fF_{N;l}:(B\AA,b_\AA)\longrightarrow (B\BB,b_\BB)
\end{equation}
defined by
\begin{equation}\label{eq:fFlevelwise}
\fF_{N;l}:=\sum_{i_1+\cdots+i_l=N}
(\ff_{i_1}\otimes\cdots\otimes\ff_{i_l})\circ 
(1+t_\AA+\cdots+t_\AA^{i_1-1}):B_N\AA\to B_l\BB.
\end{equation}
The cyclic complex is also functorial. Indeed, observe that
$$
(1+t_\AA+\cdots+t_\AA^{i_1-1})\circ(1-t_\AA)=1-t_\AA^{i_1},\quad 
(f_{i_1}\otimes\cdots\otimes f_{i_l})\circ t_\AA^{i_1}=
t_\BB\circ (f_{i_2}\otimes\cdots\otimes f_{i_l}\otimes f_{i_1}),
$$
and therefore
$$
\fF_{N;l}\circ (1-t_\AA)=\sum_{i_1+\cdots+i_l=N}(\ff_{i_1}\otimes\cdots\otimes\ff_{i_l})\circ (1-t_\AA^{i_1})
=(1-t_\BB)\sum_{i_1+\cdots+i_l=N}\ff_{i_1}\otimes\cdots\otimes\ff_{i_l}.
$$
Hence $\fF_{N;l}$ maps $\im(1-t_\AA)$ to 
$\im(1-t_\BB)$, so the map $\fF$ descends to a chain map between the cyclic complexes.

\begin{lemma}\label{lem:iso1}
Let $\ff:(\AA,\fm)\rightarrow (\BB,\fn)$ be a morphism of
$A_\infty$-algebras whose linear part induces an isomorphism on
homology 
$$
\ff_{1*}:H_*(\AA,\fm_1)\stackrel{\cong}{\longrightarrow}H_*(\BB,\fn_1).
$$
Then $\fF$ induces isomorphisms on Hochschild and cyclic homology
$$
\fF_{*}: HH_*(\AA)\stackrel{\cong}\longrightarrow HH_*(\BB),\qquad
\fF_{*}: HC_*^{\lambda}(\AA)\stackrel{\cong}\longrightarrow HC_*^{\lambda}(\BB).
$$
\end{lemma}


\begin{proof}
We give the argument for Hochschild homology (see
also~\cite[Theorem~2.10]{Mescher}); the argument for cyclic homology 
is analogous. 
It uses standard techniques from spectral sequences, see e.g.~\cite{Weibel}.  
Consider the word-length filtration on $B\AA$. The corresponding spectral sequence is 
exhaustive and bounded from below. Therefore, by the Classical Convergence Theorem~\cite[Theorem~5.5.1]{Weibel}, it converges to $HH_*(\AA)$. Its first page is
$$
   E_1^{(p,q)}(\AA) = H^p(\AA^{\otimes q},\fm_1),
$$
where $\fm_1$ is extended to $\AA^{\otimes q}$ as a derivation. An
analogous discussion applies to $\BB$. By the K\"unneth formula,
$\ff_1$ induces an isomorphism between the first pages 
$$
   E_1^{(p,q)}(\AA)\stackrel{\cong}\longrightarrow E_1^{(p,q)}(\BB).
$$
Therefore, by~\cite[Theorem~5.2.12]{Weibel}, the map $\fF$ induces an isomorphism 
between the Hochschild homologies.  
\end{proof}

\subsection{Homotopy transfer}\label{ss:homottrans}

In this subsection we recall the homotopy transfer theorem by
Kontsevich and Soibelman.

\begin{prop}[Homotopy transfer for
$A_\infty$-algebras~\cite{Kontsevich-Soibelman}]\label{prop:KS}  
Let $(\AA,\{\fm_k\}_{k\ge 1})$ be an $A_\infty$-algebra and write $d=\fm_1$.  
Suppose we are given a subcomplex $\BB\subset(\AA,d)$ with inclusion
$\iota:\BB\into \AA$, a projection $\pi:\AA\to \BB$ left inverse to
$\iota$, and a homotopy operator 
$P$ with respect to $\pi$ (see~\S\ref{ss:coch}). 
Then there exists a canonical $A_\infty$-structure
$\{\fm^\KS_k\}_{k\ge 1}$ on $\BB$ with $\fm^\KS_1=d$,
together with an $A_\infty$-homotopy equivalence
$$
  \fg=\{\fg_k\}_{k\ge 1}:(\BB,\{\fm^\KS_k\}_{k\ge 1})
  \stackrel{\sim}\longrightarrow (\AA,\{\fm_k\}_{k\ge 1}) 
$$ 
with $\fg_1=\iota$.
\end{prop}

\begin{proof}
This is proved in~\cite{Kontsevich-Soibelman},
see also~\cite[Proposition~6.1]{Cieliebak-Hajek-Volkov}.
For later purposes, we outline the construction using rooted trees
(see~\S\ref{sec:graphs} for background on graphs and trees).
Here a {\em rooted tree} $T$ is a planar tree with a distinguished vertex 
of valency $1$ (the {\em root or root vertex}) 
and all the other vertices of valency at least $3$.
We number its leaves counterclockwise from the root as $1,\dots,k$
and we orient its leaves and edges towards the root. 
We associate to $T$
three operations
\begin{equation}\label{eq:mKS}
  \fm^\KS_T:\BB^{\otimes k}\rightarrow \BB,\qquad 
  \wh\fm^\KS_T:\AA^{\otimes k}\rightarrow \AA,\qquad 
  \fg_T:\BB^{\otimes k}\rightarrow \AA 
\end{equation}
by assigning the inputs to the leaves, moving them through the tree
using its orientation, and reading off the output at the root, where
we apply
\begin{itemize}
\item the operation $\fm_l$ at each vertex with $l$ inputs (i.e., of
  valency $l+1$); 
\item the homotopy operator $P$ along each non-root edge;
\item along the root edge the projection $\pi$ for $\fm^\KS_T$, the
  identity map $\id$ for $\wh\fm^\KS_T$, and the homotopy operator $P$
  for $\fg_T$.   
\end{itemize}
(The operation $\wh\fm^\KS_T$
is not needed in the present proof and only included for later reference.)
Denoting by $RT_k$ the set of isomorphism classes of 
rooted trees with $k$ leaves, the desired $A_\infty$-structure
$\fm^\BB$ and $A_\infty$-morphism $\g$ are given by
\begin{equation}\label{eq:def-htstr}
  \fm^\KS_1 := \m_1,\qquad
  \fm^\KS_k := \sum_{T\in RT_k}\fm^\KS_T \quad\text{for }k\ge 2,
\end{equation}
\begin{equation}\label{eq:def-htrmap}
  \fg_1:=\iota,\qquad \fg_k := \sum_{T\in RT_k}\fg_T \quad\text{for }k\ge 2.
\end{equation}
\end{proof}

\subsection{Cyclic $A_\infty$-algebras and Maurer-Cartan elements}
\label{ss:cyclicainf}

A {\em cyclic $A_\infty$-algebra} is an $A_\infty$-algebra $(\AA,\fm)$ 
endowed with a pairing $(\cdot,\cdot)$ of degree $n$ such that
$(\AA,\fm_1,(\cdot,\cdot))$ is a cyclic cochain complex in the sense
of~\S\ref{ss:coch} and for all $x_0,x_1,\dots x_k$, $k\ge 2$, we have
\begin{equation}\label{eq:cyc}
  \la\fm_k(x_1,\dots,x_k),x_0\ra = (-1)^{|x_0|(|x_1|+\cdots+|x_k|)}
  \la\fm_k(x_0,\dots,x_{k-1}),x_k\ra. 
\end{equation}
Here $|x_i|$ are the shifted degrees in $\AA[1]$ and
$\la\cdot,\cdot\ra$ is the cyclic pairing defined in~\eqref{eq:defcycpair}. 
Note that by cyclicity of $(\AA,\fm_1,(\cdot,\cdot))$ this relation
also holds for $k=1$. The definition of a cyclic $A_\infty$-algebra
is consistent with that of a cyclic DGA, i.e., 
a DGA $A$ is cyclic if and only if the corresponding $A_\infty$-algebra
$\AA$ is cyclic. 

Recall from Example~\ref{ex:canon} the canonical MC element $\m^\can$ for 
$\dIBL(A)$ if $A$ is a cyclic DGA. This has the following
generalization to the $A_\infty$ case. 
Let $(\AA,d,(\cdot,\cdot))$ be a cyclic cochain complex of degree $n$.
Following~\cite{Cieliebak-Fukaya-Latschev}, for a degree $1$ linear map
$$
   \fm_k:\AA[1]^{\otimes k}\rightarrow \AA[1], \quad k\ge 2
$$  
we define $\fm_k^+\in {\rm Hom}(\AA[1]^{\otimes(k+1)},\R)$ by
\begin{equation}\label{eq:ainfmck}
   \fm_k^+(x_0,x_1,\dots,x_k) := (-1)^n \la\fm_k(x_0,\dots,x_{k-1}),x_k\ra.
\end{equation}
Then $\fm_k$ satisfies \eqref{eq:cyc} if and only if $\fm_k^+\in
B^{cyc*}_{k+1}\AA$. If this holds for all $k$ we set
$$
  \fm^+ := (\fm_k^+)_{k\ge 2} \in B^{\cyc *}\AA. 
$$ 
Recall the canonical $\dIBL$ structure $\dIBL(\AA)$ associated to
$\AA$ as a cyclic cochain complex. The following statement is~\cite[
Proposition~12.3 and equation~(12.6)]{Cieliebak-Fukaya-Latschev}.

\begin{proposition}[{\cite{Cieliebak-Fukaya-Latschev}}]\label{prop:ainfmc} 
Let $(\AA,d,(\cdot,\cdot))$ be a cyclic cochain complex, and let
$\{\fm_k\}_{k\ge 2}$ satisfy \eqref{eq:cyc}. Then $\{\fm_k\}_{k\ge 1}$ 
with $\fm_1:=d$ constitutes an $A_\infty$-structure if and only if
$\fm^+$ satisfies equation~\eqref{eq:discMC}. 
Moreover, if this is the case, then the twisted differential 
$$
  \fp_{1,1,0}^{\fm^+} = d^*+\fp_{2,1,0}(\fm^+,\cdot)
$$ 
is the dual map to the Hochschild differential $b_\AA$.
\end{proposition}

Equation~\eqref{eq:p120MC} does not in general hold. 
By~\cite[Proposition 12.5]{Cieliebak-Fukaya-Latschev}, it holds for the
triple intersection product $\m^\can$ of a cyclic DGA from~\eqref{eq:deftriple}.

Let $\AA$ be either a cyclic $DGA$ (the {\em algebraic case}), or the
de Rham algebra of a closed oriented manifold. Let $\HH\subset \AA$ be
a harmonic subspace and $P$ an associated propagator (which exists by
Corollary~\ref{cor:harm-prop} in the algebraic case, and by
Proposition~\ref{prop:existsprop} in the de Rham case).

Then on the one hand, Proposition~\ref{prop:KS} gives us the
pushforward $A_\infty$-structure $\fm^{\KS}$ on $\HH$. Let 
$\fm^{\KS+}=\{\fm^{\KS+}_k\}_{k\ge 2}$ be defined by~\eqref{eq:ainfmck}. 
On the other hand, let $\m^\HH$ be the Maurer-Cartan element on $\dIBL(\HH)$
defined in Example~\ref{ex:pushfalg} in the algebraic case, and
in~\S\ref{ss:MCan} in the de Rham case (always using the same propagator $P$). 

\begin{proposition}\label{prop:consist}
In the setup above, we have
\begin{equation}\label{eq:consist}
  \fm^{\KS+}=\m^\HH_{1,0}.
\end{equation}
In particular, $\fm_k^{\KS+}\in B^{cyc*}_{k+1}\HH$, so the
$A_\infty$-structure $\fm^{\KS}$ on $\HH$ is cyclic. 
\end{proposition}

\begin{proof}
We will sketch the proof and refer to 
~\cite{Volkov-thesis} for details, in particular regarding signs.
Consider a labelled trivalent planar tree $\Gamma$ with $s$ leaves, where a
labelling is a numbering of the leaves compatible with their cyclic
order (see Definition~\ref{def:labelling}). Converting the leaf
number $s$ of $\Gamma$ into an edge and declaring the univalent vertex
at its end the root, we get a rooted tree that we call $T(\Gamma)$.
This gives us a canonical bijection
\begin{equation}\label{eq:bijection-trees}
  \RR_s \stackrel{\cong}\longrightarrow RT_{s-1}^3,\qquad \Gamma\mapsto T(\Gamma)
\end{equation}
between isomorphism classes of labelled trivalent trees with $s$ leaves and
isomorphism classes of rooted trivalent trees with $s-1$ leaves.

Fix $\Gamma$ as above and $\alpha=\alpha_1\otimes\dots\otimes\alpha_s\in B_s\AA$
with $s\ge 3$. On the one hand, applying~\eqref{eq:ainfmck} to the
operation $\wh\fm_{T(\Gamma)}^{\KS}$ defined in~\eqref{eq:mKS} we get
\begin{equation}\label{eq:def-MC-KS+}
  \wh\fm_{T(\Gamma)}^{\KS+}(\alpha) :=
  (-1)^n\la\wh\fm_{T(\Gamma)}^\KS(\alpha_1 \otimes\dots\otimes\alpha_{s-1}),\alpha_s\ra.
\end{equation} 
On the other hand, we have $\m^\HH_\Gamma(\alpha)$ defined in~\eqref{eq:pushfalg}
in the algebraic case, and in~\eqref{eq:mGamma} in the de Rham case
(applied to inputs from $\AA$ rather than $\HH$). We claim that 
\begin{equation}\label{eq:maineq-MC+}
  \wh\fm_{T(\Gamma)}^{\KS+}(\alpha) = \m^\HH_\Gamma(\alpha).
\end{equation}
To see this in the algebraic case, we orient the edges of $T(\Gamma)$
towards the root and order the leaves counterclockwise starting from
the root. We pick an ordering of the vertices and a basis $(e_i)$ of
$A$ and consider the coordinate expression~\eqref{eq:pushfalg}. 
Now~\eqref{eq:maineq-MC+} follows by performing the sum
in~\eqref{eq:pushfalg} iteratively starting at the leaves, using
the identities
$$
	 P e^i = \sum_j P^{ij} e_j\quad\text{and}\quad x = \sum_i \la
         x, e^i \ra e_i = \sum_i \la e_i, x\ra e^i.
$$
In the de Rham case, the proof is structurally the same.
Technically, one has to convert the iterated integral
on the left hand side of~\eqref{eq:maineq-MC+} 
into an integral over the product of many copies of $M$ on the right
hand side of~\eqref{eq:maineq-MC+} by Fubini's theorem. 

Now we restrict our attention to the case $\alpha\in\HH^{\otimes
  s}$. Splitting $\wh\fm_{T(\Gamma)}^{\KS}(\alpha_1
\otimes\dots\otimes\alpha_{s-1})$ into its components in $\HH$ and in
$\HH^\perp$, we see that  
$$
  \wh\fm_{T(\Gamma)}^{\KS+}(\alpha) = (-1)^n\la\fm_{T(\Gamma)}^{\KS}(\alpha_1
  \otimes\dots\otimes\alpha_{s-1}),\alpha_s\ra
  = \fm_{T(\Gamma)}^{\KS+}(\alpha).
$$
Therefore, equation~\eqref{eq:maineq-MC+} becomes
$$
  \fm_{T(\Gamma)}^{\KS+}(\alpha) = \m^\HH_\Gamma(\alpha).
$$
We sum both sides of the last equation over all $\Gamma\in \RR_s$. The
right hand side gives us $\m^\HH_{1,0}(\alpha)$ according  
to~\eqref{eq:pushfalg-lg} in the algebraic case, respectively
to~\eqref{eq:mlg} in the de Rham case. The left hand side gives us
$\fm^{\KS+}(\alpha)$ according to~\eqref{eq:def-htstr}
and~\eqref{eq:ainfmck}, using the bijection~\eqref{eq:bijection-trees}.
\end{proof}

In context of Proposition~\ref{prop:consist}, set 
$$
  \m:=\fm^{\KS+}=\m^\HH_{1,0}.
$$
Then we have the following implication:
\begin{equation}\label{eq:p110m}
  b_\HH^* = \fp_{1,1,0}^{\fm} =\fp_{2,1,0}(\fm,\cdot).
\end{equation}

{\bf Application to Hochschild and cyclic homology. }
In the context of Proposition~\ref{prop:KS}, let $\AA:=\Om^*(M)$ be
the de Rham algebra of a closed oriented manifold $M$ and
$\BB:=\HH\subset \Om^*(M)$ a harmonic subspace, both viewed as a  
$A_\infty$-algebras via~\eqref{eq:dgatoainf}.
Let $\fg:\HH\to\Om^*(M)$ be the $A_\infty$-homotopy equivalence
provided by Proposition~\ref{prop:KS}. 
According to~\S\ref{ss:functoriality}, the morphism $\fg$ induces a
chain map between the bar complexes 
$$
   \fG=\bigoplus_{1\le l\le N}\fG_{N;l}:
   B\HH\longrightarrow B\Om^*(M)
$$
defined as in~\eqref{eq:fF}. By Lemma~\ref{lem:iso1},
the map $\fG$ induces isomorphisms on Hochschild and cyclic homology
(of $A_\infty$-algebras).
Together with Lemma~\ref{lem:QP}, this implies that the degree $0$ chain map
\begin{equation}\label{eq:bG}
   \G:=(-1)^{n+1}QP\fG: B\HH\longrightarrow C(\Om^*(M))[1]
\end{equation}
induces an isomorphism between the Hochschild homology of
$\HH$ (viewed as an $A_\infty$-algebra) and the degree shifted
Hochschild homology of $\Om^*(M)$ (now viewed as a DGA).
Here the degree shift by $1$ results from
equation~\eqref{eq:can-deg-0} relating the Hochschild complex 
and the bar complex. 
The map $\G$ induces a degree $0$ chain map between the cyclic
  complexes 
\begin{equation}\label{eq:bG-lambda}
   \G_\lambda : B^\cyc\HH \longrightarrow C^\lambda(\Om^*(M)).
\end{equation}
Note that according to the definition~\eqref{eq:Connes-complex} of the
Connes cyclic complex there is no degree shift on the right hand side. 
Lemmas~\ref{lem:iso1} and~\ref{lem:QP} together with 
the universal coefficient theorem and equation~\eqref{eq:p110m} 
imply that $\G_\lambda$ induces a degree $0$ isomorphism on cyclic cohomology
\begin{equation}\label{eq:G-iso}
  \G_\lambda^*:HC_\lambda^*(\Om^*(M))\stackrel{\cong}\longrightarrow
  H(B^{\cyc *}\HH,\fp_{1,1,0}^\m).
\end{equation}


\section{Equivariant and non-equivariant string topology}\label{sec:stringtop}

In this section we recall the algebraic structures on equivariant and
non-equivariant loop space homology and discuss their relation.  
All homology groups will be with $\R$-coefficients.

\subsection{Loop space homology}\label{ss:loop}

In this subsection, $X$ denotes a connected topological space and
$\Lambda=C^0(S^1,X)$ its free loop space with the obvious action of
the circle $S^1=\R/\Z$.  
Let $\Lambda_0\subset \Lambda$ be the subspace of constant loops. We
fix a base point $q_0\in X$. The natural inclusion of constant loops 
$$
\iota:X\cong \Lambda_0\into\Lambda
$$
induces the inclusion of pairs
$$
\iota_0:(\Lambda,q_0)\longrightarrow (\Lambda,\Lambda_0).
$$
Let $ES^1\to BS^1$ be the universal $S^1$-bundle. The quotient of
$\Lambda\times ES^1$ by the (free) diagonal circle action defines a circle bundle
$$
  \pi:\Lambda\times ES^1 \to \Lambda\times_{S^1}ES^1 =: \Lambda_{S^1}
$$
whose base is the Borel space $\Lambda_{S^1}$ of $\Lambda$.
The $S^1$-equivariant homology of $\Lambda$ is defined as 
$$
  H_*^{S^1}(\Lambda):=H_*(\Lambda_{S^1}),
$$
and similarly for the pairs of spaces $(\Lambda,\Lambda_0)$ and
$(\Lambda,q_0)$. The Gysin of the preceding circle bundle is
$$
   \cdots H_i(\Lambda\times ES^1) \stackrel{\pi_*}\longrightarrow
   H_i(\Lambda_{S^1}) \stackrel{\cap e}\longrightarrow
   H_{i-2}(\Lambda_{S^1})\stackrel{\pi^*}\longrightarrow
   H_{i-1}(\Lambda\times ES^1)\cdots
$$
where $e$ is the Euler class of the bundle. Following Chas and
Sullivan~\cite{Chas-Sullivan99} we write this as 
\begin{equation}\label{eq:Gysin}
   \cdots H_i\Lambda \stackrel{\EE}\longrightarrow
   H_i^{S^1}\Lambda \stackrel{\cap e}\longrightarrow
   H_{i-2}^{S^1}\Lambda \stackrel{\MM}\longrightarrow
   H_{i-1}\Lambda\cdots
\end{equation}
with the ``mark'' and ``erase'' maps $\MM=\pi^*$ and $\EE=\pi_*$. Here  
$\EE\MM=0$ and $\MM\EE=\Delta:H_i\Lambda\to H_{i+1}\Lambda$ is the BV operator. 

Let us introduce chain level versions of $\EE$, $\Delta$, $\MM$
that we denote by the same letters.
The chain level version of $\EE$ is clear --- postcomposition with
$\pi$. For $\Delta$, let $f:B\rightarrow \Lambda$ be a singular
simplex. The application of $\Delta$ to $f$ is defined by
(a triangulation of) the map 
$$
  \Delta f:S^1\times B\rightarrow \Lambda,\qquad (\Delta f)_{(s,p)}(t):=f_p(s+t).
$$
For $\MM$, let $g:B\to\Lambda_{S^1}$ be a singular simplex. 
Pick a lift $\wt g:B\to \Lambda\times ES^1$ of $g$ (which exists
because $\pi:\Lambda\times ES^1\to \Lambda_{S^1}$ is a circle bundle
and $B$ is contractible) and consider the diagram
\begin{equation}\label{eq:lift}
\xymatrix{
  & \Lambda\times ES^1 \ar[d]^{\pi} \ar[r]^-{pr_1} & \Lambda \\
  B \ar[ur]^{\wt g} \ar[r]^-{g} & \Lambda_{S^1}, \\
}
\end{equation}
where $pr_1$ is the projection onto the first component. 
Now the application of map $\MM$ to $g$ is defined as
\begin{equation}\label{eq:chainmark}
  \MM g:=\Delta(pr_1\circ \wt g).
\end{equation}
The based loop space $\Om$ is the fibre of the path fibration
$p:\Lambda\to X$, $p(\gamma)=\gamma(0)$, and the inclusion of constant
loops $\iota:\Lambda_0\into\Lambda$ is a section of this fibration 
(i.e.~$p\iota=\id_X$).
Hence, the long exact sequence of the pair $(\Lambda,\Lambda_0)$ decomposes
into split short exact sequences
$$
   0\longrightarrow H_*\Lambda_0 \stackrel{\iota_*}\longrightarrow H_*\Lambda
   \longrightarrow H_*(\Lambda,\Lambda_0) \longrightarrow 0.
$$
Since the subspace 
$\Lambda_0\subset\Lambda$ is $S^1$-invariant, we also get a
long exact sequence of this pair in equivariant homology
$$
   \cdots H^{S^1}_*\Lambda_0 \stackrel{\iota^{S^1}_*}\longrightarrow H^{S^1}_*\Lambda
   \longrightarrow H^{S^1}_*(\Lambda,\Lambda_0) \longrightarrow H^{S^1}_{*-1}\Lambda_0 \cdots
$$
where
$$
   H^{S^1}_*\Lambda_0 = H_*(\Lambda_0\times BS^1) 
   \cong H_*\Lambda_0\otimes H_*(BS^1) \cong
   H_*(X)[u],\qquad |u|=2.  
$$
\begin{remark}
Since the projection $p:\Lambda\to X$ is not $S^1$-equivariant, the
map $\iota^{S^1}_*$ in the long exact sequence above need not be injective.
For example, Frauenfelder and Pajitnov~\cite[Proposition 6.1]{Frauenfelder-Pajitnov}
have shown that for a $\Q$-inessential closed oriented manifold $M$, the map 
$\iota^{S^1}_*$ sends the
fundamental class $[M]$ to zero. Their argument uses Goodwillie's theorem.
The same argument shows that for $M$ simply connected (which is a
special case of being $\Q$-inessential), $\iota^{S^1}_*$ sends everything to
zero except multiples $[q_0]u^k$ of the point class $[q_0]\in H_0(M)$. 
As a consequence, the map
$$
   H^{S^1}_*\Lambda / \iota^{S^1}_*H_*^{S^1}\Lambda_0 \into
   H^{S^1}_*(\Lambda,\Lambda_0) 
$$
is injective but not necessarily surjective. By contrast, we have an isomorphism
$$
   H^{S^1}_*\Lambda / \iota^{S^1}_*H_*^{S^1}(q_0) \cong
   H^{S^1}_*(\Lambda,q_0). 
$$
\end{remark}
   
The circle bundle 
$\Lambda\times ES^1\to\Lambda_{S^1}$ restricts over
the subset $X\cong \Lambda_0\subset\Lambda$ to the bundle $X\times ES^1\to X\times BS^1$.
So the Gysin sequence of the restriction to $X\subset\Lambda$,
$$
   \cdots H_iX \stackrel{\EE}\longrightarrow
   H_i^{S^1}X \stackrel{\cap e}\longrightarrow
   H_{i-2}^{S^1}X \stackrel{\MM}\longrightarrow
   H_{i-1}X\cdots
$$
is given by tensoring with $H_*X$ the Gysin sequence of the circle
bundle $ES^1\to BS^1$,
$$
   \cdots H_i(ES^1) \stackrel{\EE}\longrightarrow
   H_i(BS^1) \stackrel{\cap e}\longrightarrow
   H_{i-2}(BS^1) \stackrel{\MM}\longrightarrow
   H_{i-1}(ES^1)\cdots
$$
Since the mark map has degree $+1$ and $H_*(ES^1)$ vanishes in all
positive degrees, the mark map vanishes in the last sequence and thus
also in the previous one,
\begin{equation}\label{eq:mark0}
   \MM=0:H_*^{S^1}X\to H_{*+1}X. 
\end{equation}
So the Gysin sequence over the constant loops becomes the short exact sequence
$$
   0 \longrightarrow H_*X \longrightarrow H_*X[u] \stackrel{\cdot
     u^{-1}}\longrightarrow H_{*-2}X[u] \longrightarrow 0.
$$

\begin{proposition}\label{prop:const2}
Let $q_0\in X\cong\Lambda_0\subset \Lambda$ be as above
and assume in addition that $X$ is simply connected.
Then the map 
\begin{equation}\label{eq:inj}
  \iota_{0*}^{S^1}:H_*^{S^1}(\Lambda,q_0)\longrightarrow H^{S_1}_*(\Lambda,\Lambda_0)
\end{equation}
induced by $\iota_0$ on relative equivariant homology is injective.
\end{proposition}

The proof is given in Appendix~\ref{app:const2}. 
The smooth version of this result is Proposition~\ref{prop:const} below.

\subsection{String topology operations}\label{ss:stringtop}

In this subsection we recall the string topology operations on
non-equivariant and equivariant loop space homology.
The discussion mostly follows~\cite{Chas-Sullivan99}. 

Throughout the rest of this section, $M$ denotes an oriented connected
manifold of dimension $n$ and 
$$
  \Lambda:=C^\infty(S^1,M)
$$
its space of smooth loops. Since the inclusion of smooth loops into
continuous loops is an $S^1$-equivariant homotopy equivalence, this
causes no ambiguity when it comes to homology and all results of the
previous subsection carry over to the smooth setting. 
As before, we fix a base point $q_0\in M$ and denote the natural inclusions
$$
  \iota:M\cong \Lambda_0\into\Lambda,\qquad 
  \iota_0:(\Lambda,q_0)\longrightarrow (\Lambda,\Lambda_0).
$$
The following result follows from~\cite[Lemma~4.7]{Cieliebak-Volkov-cyc}
by elementary algebraic topology. 

\begin{lemma}\label{lem:loophomfindim}
If the manifold $M$ is closed and simply connected, then the graded 
$\R$-vector spaces $H_*\Lambda$, $H_*^{S^1}\Lambda$, $H^*_{S^1}\Lambda$,
$H_*^{S^1}(\Lambda,q_0)$ and $H^*_{S^1}(\Lambda,q_0)$ are finite
dimensional in each degree.
\end{lemma}

{\bf Products. }
The loop homology $H_*\Lambda$ carries the degree $-n$ {\em loop
  product} $\mu$, which together with the BV operator gives
$H_{*+n}\Lambda$ the structure of a BV-algebra~\cite{Chas-Sullivan99}.
It induces the degree $2-n$ {\em string bracket} $\mu^{S^1}$ on
$H^{S^1}_*\Lambda$ defined by 
$$
   \mu^{S^1} := \EE\mu(\MM\otimes\MM).
$$
The relations of a BV-algebra for $(\mu,\Delta)$ imply that
$\mu^{S^1}$ is a Lie bracket. 
Since by~\eqref{eq:mark0} the mark map vanishes on constant loops, the image of the map
$\iota^{S^1}_*:H^{S^1}_*M\to H^{S^1}_*\Lambda$ lies in the center of $\mu^{S^1}$.
Hence, the string bracket descends to the quotient space
$
   H^{S^1}_*\Lambda / H_*^{S^1}\Lambda_0
$ 
as well as quotients by any subspaces of $H_*^{S^1}\Lambda_0$.
Of these, we will be interested in the loop homology relative to a point and
the {\em reduced loop homology},
$$
   H^{S^1}(\Lambda,q_0) = H^{S^1}_*\Lambda / H^{S^1}(q_0) \quad\text{and}\quad
   \ol{H}^{S^1}(\Lambda) := H^{S^1}_*\Lambda / \chi H^{S^1}(q_0),
$$
where  $\chi=\chi(M)$ is the Euler characteristic of $M$.

{\bf Coproducts. }
The loop homology $H_*(\Lambda,\Lambda_0)$ relative to the constant
loops carries the degree $1-n$  {\em Goresky--Hingston coproduct}
$\ol\lambda$, which was introduced in~\cite{Sullivan-open-closed}
and studied further in~\cite{Goresky-Hingston}. 
It is shown in~\cite{Cieliebak-Hingston-Oancea-III} that
$\ol\lambda$ extends to a coproduct $\lambda$ on reduced
loop homology $\ol{H}_*\Lambda:=H_*\Lambda/\chi H_*(q_0)$. 
We will refer to $\lambda$ as the {\em loop coproduct}.
The operations $\mu,\Delta$ also descend to reduced loop homology,
where together with $\lambda$ they define the structure of a 
BV unital infinitesimal bialgebra~\cite{Latschev-Oancea}.
The coproduct $\lambda$ induces the degree $2-n$ {\em string
  cobracket} $\lambda^{S^1}$ on $\ol{H}^{S^1}_*\Lambda$ via
$$
   \lambda^{S^1} := (\EE\otimes \EE)\lambda \MM.
$$
More precisely, consider the composition
$$
   H^{S^1}_*\Lambda \stackrel{\MM}\longrightarrow
   H_*\Lambda \longrightarrow
   \ol{H}_*\Lambda \stackrel{\lambda}\longrightarrow
   \ol{H}_*\Lambda\otimes\ol{H}_*\Lambda \stackrel{\EE\otimes \EE}\longrightarrow   
   \ol{H}^{S^1}_*\Lambda\otimes\ol{H}^{S^1}_*\Lambda.
$$
Since $\MM$ vanishes on constant loops, we can pass to
$\ol{H}^{S^1}_*\Lambda$ in the first term to obtain the desired map 
$\lambda^{S^1}$.
The relations satisfied by $(\mu,\lambda,\Delta)$ imply that
$(\mu^{S^1},\lambda^{S^1})$ define on $\ol{H}^{S^1}_*\Lambda$ the
structure of an involutive Lie bialgebra, see~\cite{Latschev-Oancea}. 
Since both $\mu^{S^1}$ and $\lambda^{S^1}$ vanish on constant loops,
this structure descends to further quotients such as 
$H^{S^1}_*(\Lambda,q_0)$ and $H^{S^1}_*\Lambda / H_*^{S^1}\Lambda_0$.
Note that the induced cobracket $\ol\lambda^{S^1}$ on
$H^{S^1}_*\Lambda / H_*^{S^1}\Lambda_0$
is related to the Goresky--Hingston
coproduct by the same relation as above, 
\begin{equation}\label{eq:GH-relation}
   \ol\lambda^{S^1} = (\EE\otimes \EE)\ol\lambda \MM.
\end{equation}

\begin{remark}
For non-simply connected $M$, the extension $\lambda$ of the coproduct
$\ol\lambda$ is not unique but depends on the choice of a generic
vector field on $M$. For example, for $M=S^1$ there are two choices
corresponding to the two classes of nowhere vanishing vector
fields~\cite{Cieliebak-Hingston-Oancea-III}, 
and this ambiguity persists for the cobracket $\lambda^{S^1}$. 
Since the ambiguity lies in the constant loops, which are anihilated by
Chen's iterated integrals, everything in the sequel (in particular the
discussion in~\ref{sec:rel-stringtop}) will be true for {\em any}
choice of $\lambda$ and $\lambda^{S^1}$. 
\end{remark}

\subsection{Definition of the loop product}\label{ss:defchainprod}

In this and the following subsection we provide chain-level
definitions of the loop product and the Goresky--Hingston coproduct. While for the product 
this is just the one given in~\cite{Chas-Sullivan99}, for the
coproduct we adapt the definition to the purposes of this paper.

We retain the setting of the previous subsection: $M$ is an oriented
$n$-dimensional manifold and $\Lambda=C^\infty(S^1,M)$, where $S^1=\R/\Z$. 
Moreover, we will
use some basic definitions and properties of manifolds with corners
and real oriented blow-ups from~\S\ref{sec:blow-up}. In particular, 
a {\em nice} submanifold of a manifold with corners $B$ is a
submanifold with corners $C$ which is closed as a subset such that
$C\cap \p_kB=\p_k C$ and $C$ is transverse to $\p_kB$ for all $k$.  
Recall that $S^1=\R/\Z$. A map
$$
  f:B\longrightarrow \Lambda
$$
is called {\em smooth (resp.~analytic)} if the associated map $B\times
S^1\to M$ is smooth (resp.~analytic). 
To $f$ we associate its time zero evaluation map
$$
  \ev^0f:B\to M,\qquad p\mapsto f_p(0). 
$$
Consider now two compact manifolds with corners $B_j$ and smooth maps
$$
  f_j:B_j\to \Lambda,\qquad j=1,2
$$ 
such that their time zero evaluation maps
are transverse to each other. In other words, the product map
$$
  \ev^0f_1\times\ev^0f_2:B_1\times B_2\to M\times M
$$
is transverse to the diagonal $\Delta_2\subset M\times M$. We define the
domain of the loop product as the fibre product 
\begin{equation}\label{eq:def-D-mu}
  D_{\mu(f_1,f_2)}
  := (\ev^0f_1\times\ev^0f_2)^{-1}(\Delta_2)
  = B_1\times_{\Delta_2} B_2.
\end{equation}
The above transversality implies that $D_{\mu(f_1,f_2)}\subset 
B_1\times B_2$ is a nice submanifold and we define
$\mu(f_1,f_2):D_{\mu(f_1,f_2)}\to\Lambda$ by\footnote{
The map $\mu(f_1,f_2)$ actually lands in the space
$\Lambda^{ps}\subset\Lambda$ of piecewise smooth loops (the same
occurs for the coproduct); since
we are only interested in statements on homology and
$H_*\Lambda^{ps}=H_*\Lambda$, we will ignore this distinction in the notation.
}
\begin{equation}
  \mu(f_1,f_2)_{(p_1,p_2)}(t)  :=
  \begin{cases}
  f_{1,p_1}(2t), & t\in [0,1/2],  \\
  f_{2,p_2}(2t-1), & t\in [1/2,1].
\end{cases}
\end{equation}
Now two homology classes $c_1,c_2\in H_*\Lambda$ can be represented by
smooth cycles $\sum_ia_if_1^i$ and $\sum_jb_jf_2^j$, $a_i,b_j\in\R$, such
that the time zero evaluations of $f_1^i$ and $f_2^j$ are transverse
for all $i,j$. One easily sees that
$\mu(c_1,c_2):=\sum_{i,j}a_ib_j[\mu(f_1^i,f_2^j)]$ is independent of
these representations and defines an operation on homology
$H_*\Lambda$, which is the loop product and denoted 
by the same letter $\mu$.

\subsection{Definition of the loop coproduct}\label{ss:transvers}

Let $B$ be a compact manifold with corners and
$$
  f:B\longrightarrow \Lambda
$$
a smooth map. The domain of definition of the loop coproduct is defined as
\begin{equation}\label{eq:defDf}
  D_f:=Closure(\overset{\circ}{D_f})\subset B\times [0,1],
\end{equation}
where
$$
\overset{\circ}{D_f}:=\{(p,t)\in B\times (0,1)\mid f_p(0)=f_p(t)\}.
$$
Consider the evaluation map 
$$
  e_f:B\times [0,1]\longrightarrow M\times M,\qquad (p,t)\mapsto (f_p(0),f_p(t))
$$
and the time zero derivative
$$
  v_f:B\to TM,\qquad p\mapsto f_p'(0).
$$

\begin{definition}\label{def:nondeg}
A smooth map
$f:B\to \Lambda$ is called {\em nondegenerate} if the following
conditions are satisfied (where transversality is meant stratawise on $B$):
\begin{enumerate} 
\item the restriction $e_f|_{B\times (0,1)}$ is transverse 
to the diagonal $\Delta\subset M\times M$; 
\item the map $v_f$ is transverse to the zero section $M\subset TM$. 
\end{enumerate}
\end{definition}

\begin{lemma}\label{lem:trans}
The set of nondegenerate maps $f:B\to\Lambda$ is open in
$C^\infty(B\times S^1,M)$ with respect to the $C^2$-topology,
and the set of real analytic nondegenerate maps is dense with respect
to the $C^\infty$-topology.  
\end{lemma}

\begin{proof}
Openness of condition (ii) in the $C^2$-topology is clear; openness of
nondegeneracy follows because (ii) implies (i) along $B\times\{0,1\}$
(see the proof of Lemma~\ref{lem:coprod-trans} below) 
and (ii) is open away from a neighbourhood of $B\times\{0,1\}$.
The Thom transversality theorem (see~\cite[Theorem
  2.3.2]{Cieliebak-Eliashberg-Mishachev24})  yields density of 
smooth nondegenerate maps in the $C^\infty$-topology, and
by~\cite[Theorem 5.53]{Cieliebak-Eliashberg12} smooth nondegenerate
maps can be $C^\infty$-approximated by real analytic ones.
\end{proof}

Recall from~\S\ref{sec:blow-up} the notions of a nice submanifold and
the oriented real blow-up. Nondegeneracy of $f$ has the following consequences.

\begin{lemma}\label{lem:coprod-trans}
For $f:B\to \Lambda$ nondegenerate the following holds. 

(a) The set 
\begin{equation*}
  Z_f:=\{p\in B\mid f_p'(0)=0\}
\end{equation*}
is a nice codimension $n$ submanifold of $B$.

(b) The set $D_f$ is a nice codimension $n$ submanifold of $B\times [0,1]$
with boundary
\begin{equation}\label{eq:pDf}
  \p D_f=\p_0D_{f}\cup \p_1D_{f}\cup \p_BD_f,
\end{equation}
where 
$$
  \p_0D_{f}=Z_f\times\{0\},\quad \p_1D_{f}=Z_f\times\{1\},\quad
  \p_BD_f=D_f\cap\p B.
$$
(c) The map $e_f$ lifts to a smooth maps between the oriented real blow-ups 
$$
\wt e_f:\Bl(B\times [0,1],D_f)\longrightarrow \wt M^2.
$$
\end{lemma}

\begin{proof}
Part (a) follows immediately from condition (ii), and parts (b) and
(c) are clear outside the set $D_f\times\{0,1\}\subset B\times[0,1]$.
Consider therefore a point $(p,0)\in D_f$ (the case $(p,1)\in D_f$ is analogous).
We first claim that $p\in Z_f$. 

To see this, we pick coordinates $x\in\R^d$ near $p\in B$ and $y\in\R^n$
near $f_p(0)\in M$. We write 
$$
  f(x,t) = a(x) + tb(x,t)
$$
with smooth functions $a,b$, so that
$$
  \overset{\circ}{D_f} = \{(x,t)\mid t>0,\ b(x,t)=0\}.   
$$
Since $(p,0)\in D_f$, by definition of $D_f$ there exists a sequence
$(x_k,t_k)\to (p,0)$ with $t_k>0$ and $b(x_k,t_k)=0$. Then
$$
  f(p,t_k)-f(p,0) =  t_kb(p,t_k) = t_k\bigl(b(p,t_k)-b(x_k,t_k)\bigr).
$$
Since both terms in the last bracket converge to $b(p,0)$ as
$k\to\infty$, we get
$$
  f_p'(0) = \lim_{k\to\infty}\frac{f(p,t_k)-f(p,0)}{t_k} = 0
$$
and thus $p\in Z_f$. This proves the claim. 
Next, note that 
$$
  v_f(x) = b(x,0). 
$$
Since $v_f(p)=0$, condition (ii) implies that $D_xb(p,0)$ is surjective.
Let us choose the coordinates $x=(x_1,x_2)\in\R^d=\R^{d-n}\times\R^n$
near $p$ such that $Z_f$ corresponds to $\R^{d-n}\times\{p_2\}$. Since
$b\equiv 0$ on $Z_f\times\{0\}$, we conclude that $D_{x_1}b(p,0)=0$
and $D_{x_2}b(p,0)$ is an isomorphism. By the implicit function theorem,
on a neighbourhood $U=V\times[0,\eps)$ of $(p,0)$ we have
$$
  b(x,t)=0 \Longleftrightarrow x_2=g(x_1,t)
$$
for a smooth function $g$. Therefore, the zero set
$$
  A:=\{(x,t)\in U \mid b(x,t)=0\}
$$
is a nice submanifold of $U$. Its description as a graph shows that
$A$ is the closure in $U$ of 
$A\cap(V\times(0,\eps))=\overset{\circ}{D_f}\cap U$, and therefore
$$
  A=D_f\cap U.
$$
Applying the same argument at $t=1$, this shows that $D_f$ is a nice
submanifold. The proof also gives the description of its boundary, so
part (b) is proved. 

For part (c), we retain the notation from above, so the evaluation map
near $(p,0)$ writes
$$
  e_f(x,t) = \bigl(a(x),a(x)+tb(x,t)\bigr). 
$$
As normal direction to $D_f$ at $(x,0)$ near $(p,0)$ we can use
$\{0\}\times\R^n$ in the splitting above. Then the evaluation map
extends to the boundary of the blow-ups by
$$
  \wt e_f\bigl((x,0),[\xi_2]\bigr) = \bigl(a(x),[D_{x_2}b(x,0)\cdot\xi_2]\bigr),
$$
where $0\neq\xi_2\in \{0\}\times\R^n$ and the brackets $[\ ]$ denote
the class in the oriented projectivization (see~\S\ref{ss:blow-up}).
One readily verifies that this induces a smooth map $\wt e_f$. 
\end{proof}

For a loop $\gamma\in\Lambda$ and $t\in[0,1]$ such that
$\gamma(0)=\gamma(t)$ we obtain two loops
$\gamma|_{[0,t]},\gamma|_{[t,1]}\in\Lambda$ defined by 
$$
  \gamma|_{[0,t]}(s):=\gamma(ts)\quad\text{and}\quad
  \gamma|_{[t,1]}(s):=\gamma\bigl(t+(1-t)s\bigr),\qquad s\in S^1.
$$
Let now $f:B\to\Lambda$ be nondegenerate. We define
\begin{equation}\label{eq:chain-coproduct}
\ol\lambda f:D_f\to\Lambda\times\Lambda,\qquad
(p,t)\mapsto \Bigl(f(p)|_{[0,t]},f(p)|_{[t,1]}\Bigr).
\end{equation}
Recall the description of $\p D_f$ from~\eqref{eq:pDf} and note that
$$
  \ol\lambda f(\p_0D_f)\subset\Lambda_0\times\Lambda,\qquad
  \ol\lambda f(\p_1D_f)\subset\Lambda\times\Lambda_0.
$$
Thus $\ol\lambda$ defines a chain map
$$
   \ol\lambda:NC_*(\Lambda)\to C_*(\Lambda\times\Lambda,
   \Lambda_0\times\Lambda\cup\Lambda\times\Lambda_0),
$$
where $NC_*(\Lambda)\subset C_*(\Lambda)$ denotes the subspace of
nondegenerate chains (linear combinations of nondegenerate maps from simplices).
Via approximation by nondegenerate chains (see Lemma~\ref{lem:trans})
and the K\"unneth formula, it induces on homology a map (denoted by
the same letter)
$$
  \ol\lambda:H_*\Lambda\to H_*(\Lambda\times\Lambda,
   \Lambda_0\times\Lambda\cup\Lambda\times\Lambda_0) \cong
   H_*(\Lambda,\Lambda_0)^{\otimes 2}.
$$
Since the map $\ol\lambda$ vanishes on constant loops in homology, it
descends to a coproduct on $H_*(\Lambda,\Lambda_0)$ which is the
Goresky--Hingston coproduct.
   
\begin{remark}
To obtain the induced map $\ol\lambda$ on homology, we need to
approximate a singular cycle in $\Lambda$ by a nondegenerate {\em
  cycle}. This is possible using Lemma~\ref{lem:trans}, which also
holds in relative form, and induction over the strata of the standard
simplex. Alternatively, we
can use Thom's representability theorem~\cite{Thom54} to represent 
a class $c\in H_*\Lambda$ by a map $f:B\to\Lambda$ from a {\em closed}
oriented manifold $B$ (recall that we are using $\R$-coefficients).
\end{remark}

\section{Chen's iterated integrals}\label{sec:chen}

Let $M$ be a closed connected oriented manifold. In this section we
define various versions of Chen's iterated integrals, relating chain
complexes built out of the de Rham complex $\Om^*(M)$ to the singular
chain complex of the free loop space $\Lambda=C^\infty(S^1,M)$. 
The discussion follows~\cite{Cieliebak-Volkov-cyc}. 
We borrow the notation about loops spaces from~\S\ref{sec:stringtop},
and about Hochschild and cyclic complexes from~\S\ref{ss:dga1}.
Throughout this section $B,B_j$ denote compact connected oriented
manifolds with corners. 

\subsection{Chen's iterated integrals in the Hochschild setting}\label{ss:chenHoch} 

Consider the $k$-dimensional standard simplex
$$
   \Delta^k = \{(t_1,\dots,t_k)\in \R^k\mid 0\le t_1\le t_2\le\dots\le
   t_k\le 1\},\qquad \Delta^0=\{0\} 
$$
with its face maps (parametrizing the boundary faces)
$\delta_j:\Delta^{k-1}\longrightarrow \p\Delta^k$ defined by
\begin{gather*}
   \delta_0(t_1,\dots,t_{k-1})=(0,t_1,\dots,t_{k-1}),\qquad
   \delta_k(t_1,\dots,t_{k-1})=(t_1,\dots,t_{k-1},1), \cr
   \delta_j(t_1,\dots,t_{k-1})=(t_1,\dots,t_j,t_j,\dots,t_{k-1}),\qquad j=1,\dots,k-1. 
\end{gather*}
For later use we denote
\begin{equation}\label{eq:bdrycyc}
  \p_j\Delta^k := \im \delta_{j-1},\qquad j=1,\dots,k+1.
\end{equation}
We give $\Delta^k$ the induced orientation from $\R^k$ and
$\p\Delta^k$ the boundary orientation. Then $\delta_0$ is orientation
reversing, 
and in general $\delta_j$ changes orientation by $(-1)^{j+1}$. 



Consider a smooth map
$$
  f:B\longrightarrow \Lambda,
$$
where smoothness is understood in terms of the corresponding map
$B\times S^1\rightarrow M$.
We denote the value of $f$ at $p\in B$ by $f_p$, and the evaluation of
$f_p$ at time $t\in S^1$ by $f_p(t)=f(p,t)$. For $k\ge 0$ we define the evaluation map
$$
   ev_f:B\times \Delta^k\longrightarrow M^{k+1}
$$
by
$$
   ev_f(p,t_1,\dots,t_k) := \bigl(f_p(0),f_p(t_1),\dots,f_p(t_k)\bigr). 
$$
We define the {\em Chen pairing} 
$$
   \la\cdot,\cdot\ra:\Om^{*+k}(M^{k+1})\times C_*(\Lambda)\longrightarrow\R
$$
on $\om\in \Om^{i+k}(M^{k+1})$ and a smooth map $f:B\to \Lambda$ with
$\dim B=i$ by
\begin{equation}\label{eq:chenpair}
    \la\om,f\ra := \int_{B\times\Delta^k}ev_f^*\om.
\end{equation}
It gives rise to degree preserving linear maps called {\em Chen's iterated integrals}
$$ 
  I:\bigoplus_{k\geq 0}\Om^{*+k}(M^{k+1})\to C^*(\Lambda),\qquad (I\om)(f):=\la\om,f\ra,
$$
and dually 
$$
   J:C_*(\Lambda)\to \bigoplus_{k\geq 0}\bigl(\Om^{*+k}(M^{k+1})\bigr)^\vee,\qquad (Jf)(\om):=\la\om,f\ra.
$$
Composition with the canonical cross product maps
$$
   \times : C_k(\Om^*(M)) = \Om^*(M)\otimes\Om^{*+1}(M)^{\otimes k}\to \Om^{*+k}(M^{k+1})
$$
resp.~their duals allows us to view $I$ and $J$ as maps from the
Hochschild complex resp.~to its dual (denoted by the same letters)
$$ 
   I:C_*(\Om^*(M))\to C^*(\Lambda),\qquad 
   J:C_*(\Lambda)\to C_*\bigl(\Om^*(M)\bigr)^{\vee}.
$$
Moreover, equation~(22) of~\cite{Cieliebak-Volkov-cyc} yields the following
compatibility of the operators $I$ and $J$ with the BV operators (the
BV operator $\Delta^*$ was called $P$ there):  
\begin{equation}\label{eq:BV-compat}
  I\circ B=-\Delta^*\circ I,\qquad B^*\circ J=-J\circ \Delta.
\end{equation}
The main takeaway from Chen's iterated integrals is the following result,
see Proposition~3.1 and Theorem~3.3 in~\cite{Cieliebak-Volkov-cyc} and the references therein. 

\begin{theorem}\label{theorem:chen}
Chen's iterated integrals $I$ and $J$ define chain maps
$$
   I:\Bigl(C_*(\Om^*(M)),d+b\Bigr)\longrightarrow \Bigl(C^*(\Lambda),d\Bigr)
$$
and 
$$
   J:\Bigl(C_*(\Lambda),\p\Bigr)\longrightarrow\Bigr((C_*\bigl(\Om^*(M)\bigr)^{\vee},d^*+b^*\Bigr).
$$
If $M$ is simply connected, then $I$ and $J$ induce isomorphisms on 
the respective homologies. 
\hfill$\square$
\end{theorem}

\subsection{Chen's iterated integrals in the cyclic setting}\label{ss:chen-cyc} 

We begin by introducing the $k$-dimensional cyclic simplex
$$
   \Delta^k_{\rm cyc}:=\{(t_1,\dots,t_k)\in (S^1)^k\mid t_1\le t_2\le\dots\le
   t_k\le t_1\},\qquad k\geq 1, 
$$
where ``$\le$'' denotes the cyclic order.
Observe that there is a natural map 
$$
\Delta^k\rightarrow \Delta^k_{\rm cyc}\,.
$$
Given a smooth map $f:B\longrightarrow \Lambda$,
we define the cyclic evaluation map
\begin{equation}\label{ev-hat0}
 \widehat{ev_f}:B\times\Delta^k_{\rm cyc}\longrightarrow M^k, \qquad
 \widehat{ev_f}(p,t_1,\dots,t_k) := \bigl(f_p(t_1),\dots,f_p(t_k)\bigr).  
\end{equation}
We define the {\em cyclic Chen pairing}
$$
   \la\cdot,\cdot\ra_\cyc:\Om^{*+k}(M^{k})\times C_*(\Lambda)\longrightarrow\R
$$
by
\begin{equation}\label{eq:cycchenpair}
   \la\om,f\ra_\cyc :=(-1)^{\deg\om-k+1}\int_{B\times\Delta^{k}_{\rm cyc}}\wh{ev_f}^*\om=
   (-1)^{\dim B+1}\int_{B\times\Delta^k}\wh{ev_f}^*(N_{an}\om).
\end{equation}
Here the second equality is derived in the proof
of~\cite[Lemma~3.2]{Cieliebak-Volkov-cyc},
where $N_{an}=t_{an}^1+\cdots+t_{an}^k$ is defined in~\eqref{eq:def-Nalg}. 
The sign is chosen to match the convention in~\cite{Cieliebak-Volkov-cyc}. 
The {\em Connes (or cyclic) version of Chen's iterated integral} is
the degree preserving map
$$
   I_\lambda:\bigoplus_{k\geq 1}\Om^{*+k}(M^k)\longrightarrow C^*(\Lambda), \qquad
   I_\lambda(\om)(f) := \la\om,f\ra_\cyc\,,
$$
and dually
$$
   J_\lambda:  C_*(\Lambda) \longrightarrow \bigoplus_{k\geq
     1}\bigl(\Om^{*+k}(X^{n})\bigr)^\vee, \qquad 
   J_\lambda(f)(\om) := \la\om,f\ra_\cyc\,.
$$
The cyclic maps $I_\lambda$ and $J_\lambda$ are chain maps as well.

\begin{remark}\label{rem:cyc-noncyc}
According to~\cite[Lemma~3.2(b)]{Cieliebak-Volkov-cyc} the map
$I_\lambda$ descends to the Connes cyclic complex, and according 
to~\cite[Lemma~4.9(b)]{Cieliebak-Volkov-cyc} the map $J_\lambda$
lands in the dual Connes cyclic complex.
We denote the resulting maps by the same letters,
\begin{equation}\label{eq:I_lambda}
  I_\lambda:C_*^\lambda(\Om^*(M))\to C^*(\Lambda),\qquad
  J_\lambda:C_*(\Lambda)\to C^*_\lambda(\Om^*(M)).   
\end{equation}
\end{remark}
The explicit expression
$$
  J_\lambda=B^*\circ J
$$
given in~\cite[Lemma~4.9(a)]{Cieliebak-Volkov-cyc} 
and the second equation in~\eqref{eq:BV-compat} yield
\begin{equation}\label{eq:lm4.9}
J_\lambda=-J\circ \Delta.
\end{equation}

The main takeaway from Chen's iterated integrals in the cyclic setting is the following
result, see Lemma~3.2,  Theorem~3.5, Corollary~3.6 and Corollary~4.11 in~\cite{Cieliebak-Volkov-cyc} and the references therein.

\begin{theorem}[\cite{Cieliebak-Volkov-cyc}]\label{thm:cyc}
The cyclic Chen iterated integrals give rise to chain maps 
$$
    \bar I_\lambda: \Bigl(C_*^\lambda(\Om^*(M)),d+b\Bigr)\longrightarrow
   \Bigl(C^*(\Lambda\times_{S^1}ES^1),d\Bigr)
$$
and
$$
   \bar J_\lambda: \Bigl(C_*(\Lambda\times_{S^1}ES^1),d\Bigr)\longrightarrow
      \Bigl(C^*_\lambda(\Om^*(M)),d+b\Bigr).   
$$
If $M$ is simply connected and $q_0\in M$ a basepoint, 
then the induced map on homology
$$
 \bar I_{\lambda *}:HC_*^\lambda(\Om^*(M))\longrightarrow H_{S^1}^*(\Lambda)
$$
restricts to an isomorphism
(denoted by the same letter)
$$
  \bar I_{\lambda *}:\overline{HC}_*^\lambda(\Om^*(M))
  \stackrel{\cong}{\longrightarrow}H_{S^1}^*(\Lambda,q_0),  
$$
and dually the map 
$$
 \bar J_{\lambda *}:H^{S^1}_*(\Lambda)\longrightarrow HC^*_\lambda(\Om^*(M))
$$
gives rise to an isomorphism 
(denoted by the same letter)
$$
  \bar J_{\lambda *}:H^{S^1}_*(\Lambda,q_0)
  \stackrel{\cong}{\longrightarrow}\overline{HC}^*_\lambda(\Om^*(M)). 
$$
\end{theorem}


We recall the definition of $\bar I_\lambda$ given
in~\cite[Section~3.2]{Cieliebak-Volkov-cyc}.
For a smooth simplex $g:B\to \Lambda\times_{S^1}ES^1$, pick a lift
$\wt g: B\to \Lambda\times ES^1$ as in~\eqref{eq:lift} and define
$$
(\bar I_\Lambda\om)(g):=(I_\lambda\om)(pr_1\circ \wt g).
$$
The result does not depend on the lift due to~\cite[Lemma~3.2
  (d)]{Cieliebak-Volkov-cyc}. 
Since $I_\lambda$ and $J_\lambda$ are an adjoint pair, the discussion
before~\cite[Lemma~4.9]{Cieliebak-Volkov-cyc} gives us the following 
definition of $\bar J_\lambda$:
\begin{equation*}
  \bar J_\lambda(g):=J_\lambda(pr_1\circ \wt g).
\end{equation*}
We substitute in this equation the expression for $J_\lambda$ given
in~\eqref{eq:lm4.9} and use~\eqref{eq:chainmark} to get
$$
  \bar J_\lambda=-J\circ\MM.
$$
Passing to homology, this gives
$$
  \bar J_{\lambda *}=-J_*\MM.
$$
In particular, the right hand side lands in the Connes cyclic homology. 
For our purposes the last relation can be taken as a definition of 
$\bar J_{\lambda *}$. We precompose the last equation with $\EE$, use
the relation $\Delta=\MM\EE$ and equation~\eqref{eq:lm4.9} on homology 
to get 
\begin{equation}\label{eq:cycaux}
\bar J_{\lambda *}=-J_*\MM\quad \text{and}\quad J_{\lambda *}=\bar J_{\lambda *} \EE.
\end{equation}

\begin{remark}\label{rem:homfindim}
If $M$ is simply connected, then the graded $\R$-vector spaces
$HH_*(\Om^*(M))$, 
$HC_*^\lambda(\Om^*(M))$,
$HC^*_\lambda(\Om^*(M))$,
$\ol{HC}_*^\lambda(\Om^*(M))$ and
$\ol{HC}^*_\lambda(\Om^*(M))$
are finite dimensional in each degree. 
Indeed, in view of Lemma~\ref{lem:loophomfindim} we can conclude as
follows. For $HH_*(\Om^*(M))$ this follows from Theorem~\ref{theorem:chen}.
For $\ol{HC}_*^\lambda(\Om^*(M))$ and $\ol{HC}^*_\lambda(\Om^*(M))$
this follows from Theorem~\ref{thm:cyc}.
For $HC_*^\lambda(\Om^*(M))$ and $HC^*_\lambda(\Om^*(M))$ this follows
from the relation between the reduced and non-reduced Connes cyclic (co)homology.
\end{remark}

Theorem~\ref{thm:cyc} leads to an alternative proof of
Proposition~\ref{prop:const2} for a smooth manifold, which we restate
in the following proposition. 

\begin{proposition}\label{prop:const}
If $M$ is simply connected, then the map 
$$ 
\iota_{0*}^{S^1}:H^{S^1}_*(\Lambda,q_0)\longrightarrow 
H^{S^1}_*(\Lambda,\Lambda_0)
$$
from~\eqref{eq:inj} (induced by the obvious inclusion of pairs) is
injective and, moreover, has a canonical section
\begin{equation}\label{eq:proj} 
  P^{S^1}:H^{S^1}_*(\Lambda,\Lambda_0) \longrightarrow
  H^{S^1}_*(\Lambda,q_0)\,,\qquad P^{S^1}\circ\iota_{0*}^{S^1}=\id. 
\end{equation} 
\end{proposition}

\begin{proof}
Observe that the cyclic pairing $\la\cdot,\cdot\ra_\cyc$ vanishes on
chains with values in the constant loops. Therefore, the isomorphism
$\bar J_{\lambda *}$ above factors through $\iota_{0*}^{S^1}$:
\begin{equation}
H^{S^1}_*(\Lambda,q_0)\stackrel{\iota_{0*}^{S^1}}{\longrightarrow} 
H^{S^1}_*(\Lambda,\Lambda_0)\stackrel{\wh J_{\lambda *} }{\longrightarrow}
\overline{HC}^*_\lambda(\Om^*(M))\,,\qquad  
\wh J_{\lambda *}\circ \iota_{0*}^{S^1}=\bar J_{\lambda *}
\end{equation}
Since $\bar J_{\lambda *}$ is an isomorphism by Theorem~\ref{thm:cyc},
the map $\iota_{0*}^{S^1}$ is a monomorphism and the section is given
by the formula 
\begin{equation}
  P_M:=(\bar J_{\lambda *})^{-1}\circ \wh J_{\lambda *}.
\end{equation}
\end{proof}

\subsection{Chen's iterated integrals for chains in the square of the loop space}
\label{ss:chenprod}

The goal of this section is to introduce cyclic Chen's integrals for
chains in the space $\Lambda\times\Lambda$. For motivation and
preparation, consider first two smooth maps $f_j:B_j\rightarrow \Lambda$, $j=1,2$.
This allows us to form the product map
$f_1\times f_2:B_1\times B_2\rightarrow \Lambda\times \Lambda$. 
For $k_1,k_2\ge 1$ we define the cyclic evaluation map for the product as 
\begin{equation}
\begin{aligned}\label{eq:ev-hat}
 &\wh{ev_{f_1\times f_2}}:B_1\times B_2\times\Delta^{k_1}_{\rm
    cyc}\times \Delta^{k_2}_{\rm cyc}\longrightarrow M^{k_1}\times
  M^{k_2}, \cr 
 &\wh{ev_{f_1\times f_2}}
 (p_1,p_2,t_1,\dots,t_{k_1},\wh t_1,\dots,\wh t_{k_2}):=\cr
 &\bigl(f_{1,p_1}(t_1),\dots,f_{1,p_1}(t_{k_1}),
 f_{2,p_2}(\wh t_1),\dots,f_{2,p_2}(\wh t_{k_2})
 \bigr).  
\end{aligned}
\end{equation}
Let us define the flip map 
$$
\sigma:B_1\times B_2\times \Delta^{k_1}\times\Delta^{k_2}\longrightarrow 
B_1\times \Delta^{k_1}\times B_2\times\Delta^{k_2}
$$
swapping the factors $B_2$ and $\Delta^{k_1}$.
Note that $\sigma$ changes orientation by $(-1)^{k_1\dim B_2}$, and 
the cyclic evaluation maps of $f_1$, $f_2$ and $f_1\times f_2$
(defined on noncyclic simplices) satisfy the relation
\begin{equation}\label{eq:sigmacommute}
  (\wh{ev_{f_1}}\times 
  \wh{ev_{f_2}})\circ\sigma=\wh{ev_{f_1\times f_2}}.
\end{equation}
Now for any pair of forms $\om_j\in \Om^{*+k_j}(M^{k_j})$, $j=1,2$, we
abbreviate $s_j:=\deg\om_j-k_j+1$ and compute
\begin{equation}
\begin{aligned}\label{eq:prodsimpl}
&\la\om_1,f_1\ra_\cyc\la\om_2,f_2\ra_\cyc 
\cr
&\stackrel{(1)}{=}(-1)^{s_1+s_2}
\int_{B_1\times \Delta^{k_1}}\wh{ev_{f_1}}^*
(N_{an}\om_1)
\int_{B_2\times \Delta^{k_2}}\wh{ev_{f_2}}^*
(N_{an}\om_2)\cr
&\stackrel{(2)}{=}
(-1)^{s_1+s_2} 
\int_{B_1\times \Delta^{k_1}\times B_2\times \Delta^{k_2}}\wh{ev_{f_1}}^*(N_{an}\om_1)
\times\wh{ev_{f_2}}^*(N_{an}\om_2)\cr
&\stackrel{(3)}{=}
(-1)^{s_1+s_2+k_1(\deg\om_2-k_2)}
\int_{B_1\times B_2\times \Delta^{k_1}\times \Delta^{k_2}}\wh{ev_{f_1\times f_2}}^*
(N_{an}\om_1\times N_{an}\om_2),
\end{aligned}
\end{equation} 
Here equality~(1) follows from equation~\eqref{eq:cycchenpair};
equality~(2) from Fubini's theorem;
and equality~(3) by invariance of integration under $\sigma$ and
relation~\eqref{eq:sigmacommute}, where for the sign exponent we use
that $\dim B_2+k_2=\deg \om_2$ unless the second integral vanishes.

Consider now a smooth map $f=(f^1,f^2):B\rightarrow \Lambda\times \Lambda$.
Definition~\eqref{eq:ev-hat} suggests to define for any $k_1,k_2\ge 1$
the evaluation map
\begin{equation}
\begin{aligned}\label{eq:ev-hat1}
 &\wh{ev_{f}}:B\times\Delta^{k_1}_{\rm cyc}\times \Delta^{k_2}_{\rm cyc}\longrightarrow M^{k_1}\times M^{k_2}, \cr
 &\wh{ev_{f}}(p,t_1,\dots,t_{k_1},\wh t_1,\dots,\wh t_{k_2}) := 
 \bigl(f_p^1(t_1),\dots,f_p^1(t_{k_1}),
 f_p^2(\wh t_1),\dots,f_p^2(\wh t_{k_2})\bigr).  
\end{aligned}
\end{equation}
The right hand side of~\eqref{eq:prodsimpl} suggests to define
for any pair of forms $\om_j\in \Om^{*+k_j}(M^{k_j})$, $j=1,2$, the
pairing 
\begin{equation}\label{eq:prodpair}
\la(\om_1,\om_2),f\ra_\cyc:=
(-1)^s
\int_{B\times \Delta^{k_1}\times \Delta^{k_2}}\wh{ev_f}^*
(N_{an}\om_1\times N_{an}\om_2)
\end{equation}
with
$$
 s:= (\deg\om_1-k_1)+(\deg\om_2-k_2)+k_1(\deg\om_2-k_2). 
$$
We define the shifted degree of an element $\om\in\Om^m(M^k)$ as $|\om|:=m-k$. 
The cyclic pairing allows us to define the degree preserving map
$$
   I_\lambda^2:
  \bigl(\bigoplus_{k\geq 1}\Om^{*+k}(M^k)\bigr)^{\otimes 2}
   \to C^*(\Lambda\times\Lambda), \qquad
   I_\lambda^2(\om_1\otimes\om_2)(f) := 
   \la(\om_1,\om_2),f\ra_\cyc\,,
$$
and dually
$$
   J_\lambda^2:C_*(\Lambda\times\Lambda) \to \Bigl(\bigl(\bigoplus_{k\geq
     1}\Om^{*+k}(X^{n})\bigr)^{\otimes 2}\Bigr)^\vee, \qquad 
   J_\lambda^2(f)(\om_1\otimes\om_2) := 
   \la(\om_1,\om_2),f\ra_\cyc\,.
$$
Again, the maps $I_\lambda^2$ and $J_\lambda^2$ are chain maps.
In fact, the chain map properties of the various Chen integral maps
are all proved by analogous arguments: the chain map property of
$I_\lambda$ follows from Stokes' theorem on $B\times \Delta^k$, the
one for $I_\lambda\otimes I_\lambda$ from Stokes' theorem on 
$B_1\times \Delta^{k_1}\times B_2\times \Delta^{k_2}$, and the one for
$I_\lambda^2$ (and dually for $J_\lambda^2$) from Stokes' theorem on
$B\times \Delta^{k_1}\times\Delta^{k_2}$. See~\cite{Volkov-thesis} for
details.

\begin{lemma}\label{lem:homolprod}
Let us identify $H_*(\Lambda\times\Lambda)$ with $H_*(\Lambda)\otimes
H_*(\Lambda)$ by means of the K\"unneth isomorphism.
Then we have the following equalities on homology:
\begin{equation*}\label{eq;homolprod}
 I_{\lambda*}^2=I_{\lambda*}\otimes I_{\lambda*},\qquad 
 J_{\lambda*}^2=J_{\lambda*}\otimes J_{\lambda*}.
\end{equation*}
\end{lemma}

\begin{proof}
The computation~\eqref{eq:prodsimpl} for arbitrary $\om_1,\om_2$ shows that
the evaluations of $I_\lambda^2$ and $I_\lambda\otimes I_\lambda$
on a product simplex $f=f_1\times f_2$ coincide. Since $I_\lambda^2$
and $I_\lambda\otimes I_\lambda$ are chain maps, this yields the first
equality on homology. The second one follows by duality. 
\end{proof}

\section{Fibre integration}\label{sec:fibre}



In this section $B,B_i,E,E_i,F$ denote compact oriented manifolds, 
possibly with corners. All integrals are understood in the Lebesgue sense. 

\begin{definition}\label{def:L1}
A {\em measurable differential form} on $B$ is a measurable section
$\om$ of the bundle of exterior forms $\Lambda^*TB\to B$.
We call $\om$ {\em integrable} if for every smooth test form
$\gamma\in \Om^*(B)$ the integral $\int_B\om\wedge \gamma$ exists.
The space of integrable forms on $B$ will be denoted by $\Om_{int}^*(B)$.
Any integrable form can be uniquely represented as a sum over
$p=0,\dots,\dim B$ of integrable forms of degree $p$.
Similarly, any integrable form
on $B_1\times B_2$ can be represented as a sum over $p=0,\dots,\dim B_1$
and $q=0,\dots,\dim B_2$ of integrable forms of bidegree $(p,q)$.
\end{definition}

\begin{remark}
(a) Since $B$ is compact, integrability of $\om$ is equivalent to the
following: in every coordinate patch, the coefficients of $\om$ are integrable functions. 

(b) If $\om$ has top degree, then we can write it as $\phi\,\vol_B$ for
a smooth volume form $\vol_B$ and a measurable function $\phi$, and
$\om$ is integrable if and only if $\phi$ is integrable. 
 
(c) In the subsequent discussion, we will always speak of explicit
forms and not their equivalence classes modulo subsets of measure
zero. The reason is that modifying a form on a subset of measure
zero can change the integral of its pullback under a smooth map.
\end{remark}

Let us fix a (compact oriented) manifold $F$ of dimension $d=\dim F$
and a positive smooth volume form $\vol_F$ on $F$. 
We also also fix a positive integer $n$. 
By an {\em $F$-bundle} we will mean a smooth fibre bundle $p:E\to B$ with
fibre $F$ over a (compact oriented) $n$-dimensional base manifold $B$. The
total space $E$ inherits an orientation according to the convention
``fibre first, base second''. 
Our goal is to define the pushforward $p_*\om$ of integrable forms
$\om$ on $E$.

We will reduce this to the following special case. 
By a {\em ball} we will mean a manifold diffeomorphic to the closed
unit ball in $\R^n$. For $0\leq k\leq n$, a {\em basis of $k$-forms on
$B$} is a collection of $\alpha_I\in\Om^k(B)$ for multi-indices $I$ of
length $|I|=k$ which gives a basis of $\Lambda^kT_qB$ at each $q\in B$.
For the trivial $F$-bundle $p:F\times B\to B$ over a ball $B$, we can then
write each integrable $(k+d)$-form $\om$ on $F\times B$ uniquely as
\begin{equation}\label{eq:omega-normal}
  \om = \sum\nolimits_If_I\,\vol_F\wedge p^*\alpha_I+\om_{rest}
\end{equation}
with integrable functions $f_I$ on $F\times B$ and an integrable
$(k+d)$-form $\om_{rest}$ on $F\times B$ with the following property:
for each point $(\xi,q)\in F\times B$ and a basis 
$X_1,\dots,X_d$ of $T_\xi F$ we have 
$$
  \om_{rest}(X_1,\dots,X_{\dim F},\cdot)=0.
$$
More generally, we will associate to every pair $(p,\om)$ consisting of an $F$-bundle
$p:E\longrightarrow B$ and an integrable $(k+d)$-form $\om$ on $E$ a
measure zero subset $Z_{min}(\om)$ of $B$ and an integrable $k$-form
$p_*\om$ on $B$ 
satisfying the following axioms.

(VAN) We have
$$
  p_*\om|_{Z_{min}(\om)}=0.
$$
(NAT) For a bundle isomorphism  
$$
\xymatrix{
  E_1 \ar[d]^{p_1} \ar[r]^{\psi}_{\cong} & E_2 \ar[d]^{p_2} \\
  B \ar[r]^{\id} & B\,
}  
$$
and an integrable $(k+d)$-form $\om$ on $E_2$ we have
$$
  p_{1*}(\psi^*\om) = \pm p_{2*}\om,\qquad Z_{min}(\psi^*\om)=Z_{min}(\om),
$$
with the plus sign if $\psi$ is orientation preserving and the minus
sign otherwise. 

(WEDGE) For each smooth form $\alpha$ on $B$ we have
$$
  p_*(\om\wedge p^*\alpha)=p_*(\om)\wedge\alpha,\qquad
  Z_{min}(\om\wedge p^*\alpha)=Z_{min}(\om)\cap\{\alpha\ne 0\}.
$$
(SUB) For each open subset $U$ of $B$ we have
$$
  p_*(\om)|_U=p_*(\om|_{p^{-1}(U)}),\qquad
  U\cap Z_{min}(\om)=Z_{min}(\om|_{p^{-1}(U)}). 
$$
(PROD1) For the product bundle $E=F\times B$ over a ball $B$ and $\om$
  given by~\eqref{eq:omega-normal} we have
$$
  Z_{min}(\om)=\bigcup\nolimits_IZ_{min}(f_I\,vol_F),\qquad
  p_*\om=\sum\nolimits_Ip_*(f_I\,vol_F\wedge p^*\alpha). 
$$
(PROD2) For a product bundle $E=F\times B$ and
$$
  \om= f\,\vol_F
$$
with an integrable function $f$ on $F\times B$ we have
$$
  Z_{min}(\om) = \{q\in B\mid \int_{q\in F}f(\xi,q)vol_F(\xi) 
  \text{ does not exist}\},
$$
$$
  p_*\om|_{Z_{min}(\om)}=0,\qquad (p_*\om)(q) = \int_{\xi\in
    F}f(\xi,q)vol_F(\xi) \quad\text{for } q\in B\setminus Z_{min}(\om).
$$
Note that Fubini's theorem implies that in (PROD2) the set
$Z_{min}(\om)$ does indeed have measure zero and the function $p_*\om$
is integrable on $B$. 

\begin{lemma}\label{lem:push-forward-int}
There exists a unique assignment 
$(p,\om)\mapsto (p_*\om,\, Z_{min}(\om))$
that satisfies the above axioms. Moreover, for each 
smooth form $\beta$ on $B$ and integrable form $\om$ on $E$ we have
\begin{equation}\label{eq:fibre-Fubini}
   \int_E\om\wedge p^*\beta = \int_B p_*\om\wedge\beta.
\end{equation}
\end{lemma}

\begin{proof}
{\bf Step 1. }
Consider first the product bundle $E=F\times B$ over a ball $B$. Pick
a basis of $k$-forms $\alpha_I$ on $B$ and write $\om$ uniquely
as~\eqref{eq:omega-normal}. Note that the $\alpha_I$ vanish
nowhere. In view of axioms (VAN), (WEDGE), (PROD1) and (PROD2) we must define
\begin{align*}
  Z_{min}(\om) &:= \bigcup\nolimits_I\{q\in B\mid \int_{q\in F}f_I(\xi,q)vol_F(\xi) 
  \text{ does not exist}\}, \cr
  p_*\om|_{Z_{min}(\om)} &:=0, \cr
  (p_*\om)_q &:= \sum\nolimits_I\Bigl(\int_{\xi\in
    F}f_I(\xi,q)vol_F(\xi)\Bigr) \wedge p^*\alpha_I\quad\text{for } q\in B\setminus Z_{min}(\om).
\end{align*}
Consider now a different basis of $k$-forms $\beta_J$ on $B$. Then we
can write uniquely $\alpha_I=\sum_Jh_{IJ}\beta_J$ with functions
$h_{IJ}$ on $B$ and 
\begin{equation*}
  \om = \sum\nolimits_Jg_J\,\vol_F\wedge p^*\beta_J+\om_{rest},\qquad
  g_J(\xi,q)=\sum\nolimits_Ih_{IJ}(q)f_I(\xi,q). 
\end{equation*}
Considering the complement of $Z_{min}(\om)$ and using that the matrix
$h_{IJ}(q)$ is invertible for each $q\in B$,
this shows that the definition of $Z_{min}(\om)$ does not depend on the
basis of $k$-forms.
Now linearity of the integral implies that the
definition of $p_*\om$ also does not depend on the basis of $k$-forms.
The definitions are also clearly independent of the choice of volume
form $\vol_F$, which we can also allow to be negative and vary
smoothly over $B$. 

Consider next a bundle isomorphism $\psi:F\times B\to F\times B$
covering the identity map on $B$. Then 
\begin{equation*}
  \psi^*\om = \sum\nolimits_If_I\circ\psi\,\psi^*\vol_F\wedge p^*\alpha_I+\psi^*\om_{rest}
\end{equation*}
is the presentation~\eqref{eq:omega-normal} of $\psi^*\om$ with
respect to the volume form $\psi^*\vol_F$. Invariance of integration
under $\Psi|_{F\times \{q\}}$ for each $q\in B$ implies that 
$$
  p_*(\psi^*\om) = \pm p_*\om,\qquad Z_{min}(\psi^*\om)=Z_{min}(\om),
$$
with the plus sign if $\psi$ is orientation preserving and the minus
sign otherwise. 
This gives us the assignment $(p,\om)\mapsto (p_*\om,\, Z_{min}(\om))$
satisfying all the above axioms for any trivial bundle $p:E\to B$ over
a ball: apply the definition for a product bundle in an orientation
preserving trivialization, which by the preceding discussion does not
depend on the choice of orientation preserving trivialization.

{\bf Step 2. }
Let now $p:E\to B$ be any $F$-bundle and $\om$ an integrable form on $E$.
Pick a finite cover of $B$ by balls $U_i$ over which the bundle is trivial.
Independence of the trivialization above implies that 
$$
Z_{min}(\om|_{p^{-1}(U_i)})\cap U_j =
Z_{min}(\om|_{p^{-1}(U_j)})\cap U_i.
$$
In other words, the sets $Z_{min}(\om|_{p^{-1}(U_i)})$ glue nicely to
$$
  Z_{min}(\om) := \bigcup\nolimits_iZ_{min}(\om|_{p^{-1}(U_i)}).
$$
Another implication of the independence of trivialization is that for
all pairs $i,j$ we have 
$$
  p_*(\om|_{p^{-1}(U_i)})|_{U_i\cap U_j} =
  p_*(\om|_{p^{-1}(U_j)})|_{U_i\cap U_j}.
$$
This allows us to define $p_*\om$ unambiguously by 
$$
  (p_*\om)|_{U_j} := p_*(\om|_{p^{-1}(U_j)}).
$$
By construction, this assignment satisfies all the axioms. 
Equation~\eqref{eq:fibre-Fubini} with $\beta=1$ follows from Fubini's
theorem. Applying this partial case with $\om\wedge p^*\beta$ in place
of $\om$ and using axiom (WEDGE) gives us
equation~\eqref{eq:fibre-Fubini} with general $\beta$.
\end{proof}

\begin{remark}\label{rem:nonlin}
The assignment $p_*:\om\mapsto p_*\om$ is in general nonlinear.
This is so because we can have integrable forms $\om_1$ and $\om_2$
with $Z_{min}(\om_1)=Z_{min}(\om_2)\ne\emptyset$,
but $Z_{min}(\om_1+\om_2)=\emptyset$ due to cancellations.
Note, however, that linearity holds up to a set of measure zero.
\end{remark}

Consider now an $F$-bundle $p:E\to B$ and a smooth map
$g:B_1\longrightarrow B$. This induces the pullback bundle described
by the diagram 
$$
\xymatrix{
  g^*E \ar[d]^{p_1} \ar[r]^{\wt g} & E \ar[d]^{p} \\
  B_1 \ar[r]^{g} & B\,.
}  
$$
Consider an integrable form $\om$ on $E$. According to
Lemma~\ref{lem:push-forward-int}, 
fibre integration of $\om$ defines an integrable
form $p_*\om$ on $B$. By contrast, the pullback $\wt g^*\om$ exists as a
measurable form, but it need not be integrable in
general. Nonetheless, we have the following result.

\begin{lemma}\label{lem:pullback}
In the setting above, 
assume that the form $\om$ on $E$ and its pullback $\wt g^*\om$ on
$g^*E$ are both integrable. Then we have the following commutativity
relation in which all appearing forms are integrable:
\begin{equation}\label{eq:pullpush}
  g^*p_*\om=p_{1*}\wt g^*\om.
\end{equation}
\end{lemma}


\begin{proof}
{\bf Step 1. }
Using local triviality of fibre bundles and partitions of unity, we
reduce the question to the product case as follows. Consider a ball
$U\subset B$ over which we have a trivialization
$\Phi:p^{-1}(U)\stackrel{\cong}\longrightarrow F\times U$.
By definition of the pullback bundle, this induces a trivialization
$\wt\Phi:p_1^{-1}(g^{-1}(U))\stackrel{\cong}\longrightarrow F\times
g^{-1}(U)$ such that the following diagram commutes:
\begin{center}
\begin{tikzcd}
F\times g^{-1}(U) \arrow[rr, "\id\times g"]     &  & 
F\times U           \\
p_1^{-1}(g^{-1}(U))\, \arrow[rr, "\wt g"] \arrow[u, "\wt \Phi"] &  & p^{-1}(U)\,.
\arrow[u, "\Phi"]
\end{tikzcd}
\end{center}
By compactness of $B$, we can pick a finite collection
$\{U_j\}_{j=1}^k$ of balls $U_j\subset B$ over which the bundle $E$ is
trivial such that the interiors $\inn\,U_j$ cover $B$. Let
$\{\rho_j\}_{j=1}^k$ be a partition of unity subordinate to this
cover. Restrict our attention to one ball $U_j$, replacing $\om$ with
$\rho_j\om$, and restricting the domain to a ball in $g^{-1}(U_j)$,
we reduce the situation to the product case.

{\bf Step 2. }
By Step 1, we can assume that we are in the product setting
$$
\xymatrix{
  F\times B_1 \ar[d]^{p_1} \ar[r]^{\wt g=\id\times g} & F\times B \ar[d]^{p} \\
  B_1 \ar[r]^{g} & B
}  
$$
with balls $B_1,B$. Dropping in the representation~\eqref{eq:omega-normal}
the term $\om_{rest}$ (which doesn't contribute to either side
of~\eqref{eq:pullpush}) and restricting to one summand, we may assume
that 
$$
  \om = f\, vol_F\wedge p^*\alpha
$$
with an integrable function $f:F\times B\to\R$ and a nowhere
vanishing smooth form $\alpha$ on $B$. Then
$$
  \wt g^*\om = (\id\times g)^*\om = f\circ (\id\times g)\,
  vol_F\times g^*\alpha,
$$
and thus
\begin{equation}\label{eq:pushf-coord}
  p_*\om=\left(\int_Ff\, vol_F\right)\,\alpha,\qquad 
  p_{1*}\wt g^*\om=\left(\int_Ff\circ (\id\times g)\,vol_F\right)\,g^*\alpha.
\end{equation}
Note that for any $\xi\in F$ and $p_1\in B_1$, setting 
$p:=g(p_1)\in B$ we have  
$$
  f(\xi,p)=f\circ(\id\times g)(\xi,p_1),
$$
and therefore
\begin{equation}\label{eq:key-int-eq}
\int_Ff\circ (\id\times g) (\xi,p_1)\, vol_F(\xi)=
\int_Ff(\xi,p)\, vol_F(\xi)\,.
\end{equation}
Thus the two integrals exist only simultaneously, i.e., the sets 
\begin{align*}
  Z_{min}(\om) &= \{p\in B\mid \int_Ff(\xi,p)\, vol_F(\xi)\text{
    does not exist}\},\cr
  Z_{min}(\wt g^*\om) &= \{p_1\in B_1\mid \int_F f\circ(\id\times
  g)(\xi,p_1)\, vol_F(\xi) \text{ does not exist}\}
\end{align*}
are related by
$$
  Z_{min}(\wt g^*\om)=g^{-1}(Z_{min}(\om)).
$$
By the integrability assumption on $\om$ and $\wt g^*\om$, the sets 
$Z_{min}(\om)$ and $Z_{min}(\wt g^*\om)$ both have measure zero. Thus
we can restrict $g$ to a map  
$$
  B_1\setminus Z_{min}(\wt g^*\om)\to B\setminus Z_{min}(\om)
$$ 
between full measure subsets over which the integrals on both
sides exist. Applying $g^*$ to the first equation in~\eqref{eq:pushf-coord}
and then using the second equation yields the desired commutativity relation
$$
  g^*p_*\om
  = \left(\int_Ff\circ (\id\times g)\, vol_F\right)\,g^*\alpha
  = p_{1*}\wt g^*\om.
$$
\end{proof}

Here is a useful application of this lemma.  

\begin{cor}\label{cor:pullback-aut}
Let $p:E\longrightarrow B$ be an $F$-bundle and $\wh g:E\to E$ a
diffeomorphism covering a diffeomorphism $g:B\to B$, 
that is $g\circ p=p \circ\wh g$. Then 
$$
  g^*p_* = \pm p_*\wh{g}^*
$$
on integrable forms, with the plus sign if $\wh g$ preserves 
the orientation of the fibres and the minus sign otherwise.
\end{cor}

\begin{proof}
Consider the commuting diagram
$$
\xymatrix{
  E \ar[d]^{p} \ar[r]^{\wh g}_\cong & E \ar[d]^{p} & \ar[l]_{\wt
    g}^\cong g^*E \ar[d]^{p_1} \\
  B \ar[r]^{g}_\cong & B & \ar[l]_g^\cong B\,.
}  
$$
Thus $\psi:=\wt g^{-1}\circ\wh g:E\to g^*E$ is a bundle isomorphism
covering the identity, so by axiom (NAT) we get 
$$
  p_{1*} = \pm p_*\psi^*.
$$
Since $\wh g$ is a diffeomorphism, it pulls back integrable forms to
integrable ones. Therefore, Lemma~\ref{lem:pullback} and the previous
displayed equation yield the desired equality on integrable forms:
$$
  g^*p_* = p_{1*}\wt g^* = \pm p_*\psi^*\wt g^* = p_*\wh{g}^*.
$$
\end{proof}

\begin{remark}\label{rem:linear}
Recall from Remark~\ref{rem:nonlin} that the map $p_*$ is not
linear. Assume that we are in the setting of Lemma~\ref{lem:pullback}.
Integrating both sides of~\eqref{eq:pullpush} over $B_1$ yields
$$
  \int_{B_1}g^*p_*\om = \int_{B_1}p_{1*}\wt g^*\om = \int_{g^*E}\wt g^*\om,
$$
where the second equality follows from ~\eqref{eq:fibre-Fubini} with $\beta=1$.
Note that the last expression is linear in $\om$.
Since we will always apply Lemma~\ref{lem:pullback} in this integrated
form, the nonlinearity of $p_*$ will not matter in practice.
\end{remark}

\begin{remark}\label{rem:fibre-smooth}
All the results in this section obviously carry over to the case
where the manifolds are noncompact but the form $\om$ has compact support. 
Specializing to smooth forms, the construction in
Lemma~\ref{lem:push-forward-int} associates to each $F$-bundle 
$p:E\to B$ a pushforward map on compactly supported forms
\begin{equation}\label{eq:pushfdef}
  p_*:\Om_c^{k+\dim F}(E)\to \Om_c^k(B)
\end{equation}
satisfying the axioms and all the preceding results. In this setting
the map $p_*$ is linear and $Z_{min}(\om)=\emptyset$ for all
$\om\in\Om_c^{k+\dim F}(E)$. 
\end{remark}


\section{Propagators}\label{sec:prop}

In this section we construct propagators in the sense
of~\S\ref{ss:coch} for the de Rham complex of a closed oriented
manifold $M$. These propagators will 
crucially enter the configuration space integrals in later sections. 
An extended version of the discussion in this section can be found
in~\cite{Cieliebak-Volkov}. 

\subsection{Poincar\'e duality}\label{ss:poincdual}

We first recall the basic properties of cup and cap products, following the 
conventions in Hatcher~\cite{Hatcher}. Let $X$ be a closed oriented
manifold (which will later be $M\times M$). We denote homology classes on $X$ by
$a,b,c$ and cohomology classes by $\alpha,\beta,\gamma$. We write the
pairing between cohomology and homology as $\int_a\beta$. Then the cup
and cap product are related by
$$
   \int_c\alpha\cup\beta = \int_{c\cap\alpha}\beta.
$$
The cap product with the fundamental class defines the Poincar\'e
duality isomorphism
$$
   PD: H^k(X)\to H_{\dim X-k}(X),\qquad PD(\alpha):=[X]\cap \alpha. 
$$ 
So we have
$$
   \int_X\alpha\cup\beta = \int_{PD(\alpha)}\beta.
$$
We will denote the inverse map to Poincar\'e duality also by $PD$. 
The intersection product of two homology classes is then given by
\begin{equation}\label{eq:PD-intersection}
   a\cap b = \int_a PD(b) = \int_X PD(a)\cup PD(b). 
\end{equation}

\subsection{Harmonic projections}\label{ss:harmproj}

From now on $M$ denotes a closed oriented manifold of dimension
$n$. Recall the de Rham complex
$$
   \bigl(\Om = \Om^*(M),d,\wedge\bigr)
$$
with the intersection pairing $(\alpha,\beta) = \int_M\alpha\wedge\beta$
defined in~\eqref{eq:defineintpair}.
Let us fix a complementary subspace $\HH$ to $\im d$ in $\ker d$,
i.e., such that 
$$
   \ker d = \im d\oplus\HH. 
$$
The space $\HH$ is a harmonic subspace in the sense of~\S\ref{ss:coch}
and we will refer to its elements as {\em harmonic forms}, although
they need not be harmonic with respect to any metric. The de Rham
cohomology $H^*(M)$ is finite dimensional and the induced pairing
on $H^*(M)$ is nondegenerate.
Therefore, by Lemma~\ref{lem:Horthog},
$$
   \HH^\perp = \{\alpha\in\Om\mid (\alpha,\beta)=0 \text{ for all }\beta\in\HH\}
$$
is a complement to $\HH$ in $\Om$ and we have the orthogonal projection
$$
   \Pi:\Om = \HH\oplus\HH^\perp \to \HH. 
$$
We pick a basis $h_i$ of $\HH$ and define its dual basis $h^i$ by 
$$
   (h_i,h^j) = \delta_i^j.
$$
In terms of these bases (see the proof of Lemma~\ref{lem:Horthog}) we have
\begin{equation}\label{eq:projdef}
   \Pi = \sum_i (h_i,\cdot)h^i,
\end{equation}
or more explicitly (using $\deg\alpha=\deg h^i$), 
\begin{equation}\label{eq:int-kernel}
   (\Pi\alpha)(y)
   = \sum_i \Bigl(\int_{M_x}h_i(x)\wedge\alpha(x)\Bigr)h^i(y)
   = \int_{M_x}\Pi(x,y)\wedge\alpha(x)
\end{equation}
with the smooth integral kernel $\Pi\in\Om^n(M_x\times M_y)$ (which we
denote by the same letter by a slight abuse of language) given by 
\begin{equation}\label{eq:Pi-kernel-h}
   \Pi(x,y) = \sum_i(-1)^{h^i}h_i(x)\wedge h^i(y).
\end{equation}
Here and in the sequel we sometimes denote by $M_x$ the factor of $M$
corresponding to the variable $x$. The integration over $M_x$ in~\eqref{eq:int-kernel} is
viewed as the fibre integral with respect to the projection $M_x\times
M_y\to M_y$ onto the second factor. This projection is chosen so that
the convention ``fibre first, base second'' gives the canonical
orientation of $M_x\times M_y$. 
The following lemma is proved 
in~\cite{Cieliebak-Volkov}.

\begin{lem}\label{lem:Poincare}
The integral kernel~\eqref{eq:Pi-kernel-h} is closed and represents
the Poincar\'e dual to the diagonal $\Delta_2=\{x=y\}\subset M\times M$.
Moreover, it has the symmetry
\begin{equation}\label{eq:Pi-symmetry}
   \Pi(x,y)=(-1)^n\Pi(y,x). 
\end{equation}
\end{lem}

{\bf Switching to the algebraic convention. }
In order to be consistent with~\cite{Cieliebak-Fukaya-Latschev},
we now replace the pairing~\eqref{eq:defineintpair} by 
the cyclic one, see~\eqref{eq:defcycpair}. Explicitly, 
\begin{equation}\label{eq:pairing}
   \la\alpha,\beta\ra := (-1)^\alpha\int_M\alpha\wedge\beta.
\end{equation}
This leaves the harmonic subspace $\HH$ and the orthogonal
projection $\Pi:\Om\to\HH$ unchanged.
For a basis $e_a$ of $\HH$ we now define its dual basis $e^a$ by 
$$
   \la e_a,e^b\ra = \delta_a^b.
$$
Then bases $h_a,h^a$ as above determine bases $e_a,e^a$ by
$$
   e_a=h_a,\qquad e^a=(-1)^{e_a}h^a.
$$
The kernel of $\Pi$ writes in the new bases
\begin{equation*}
   \Pi(x,y) = \sum_i(-1)^{e^a+e_a}e_a(x)\wedge e^a(y) = (-1)^n\sum_ie_a(x)\wedge e^a(y),
\end{equation*}
which in view of the symmetry of $\Pi$ becomes
\begin{equation}\label{eq:Pi-kernel}
   \Pi(x,y) = \sum_ie_a(y)\wedge e^a(x).
\end{equation}

\subsection{Oriented real blow-up and propagators}\label{ss:blowupprop}

We denote by
$$
   \wt M^2 := \Bl(M^2,\Delta_2)
$$
the {\em oriented real blow-up} of the diagonal $\Delta_2$ in $M^2=M\times M$.
This is the compact oriented manifold with boundary obtained by
replacing the diagonal by its unit sphere normal bundle $N_{\Delta_2}$.
Thus the boundary $\p\wt M^2$ is canonically diffeomorphic to
$N_{\Delta_2}$. Note, however, that the orientation of $\p\wt M^2$ as
boundary of $\wt M^2$ is {\em opposite} to the orientation of
$N_{\Delta_2}$ as boundary of the unit disk normal bundle, oriented by the
usual convention ``fibre first, base second''.  

The oriented real blow-up comes with a smooth blow-down map
$$
    \wt M^2\stackrel{\pi}\longrightarrow M^2
$$
which restricts to a diffeomorphism $\wt M^2\setminus\p\wt M^2\to
M^2\setminus\Delta_2$ on the interior and to the bundle projection
$N_{\Delta_2}\to\Delta_2$ on the boundary. The projections $p_i:M\times M\to
M$ onto the two factors induce smooth fibre bundles
$$
   p_i:\wt M^2\to M,\qquad i=1,2
$$
with fibre the oriented real blow-up of $M$ at a point. The map
$\tau(x,y)=(y,x)$ canonically lifts to an involution
$$
   \tau:\wt M^2\to\wt M^2. 
$$
 
We denote the pullback of $\Pi\in\Om^n(M^2)$ from~\eqref{eq:Pi-kernel}
under the blow-down map by $\wt\Pi$. We will view $\wt M^2$ as the
fibre bundle 
$$
   p_2:F\to \wt M^2\to M
$$
via projection onto the second factor and denote by $\int_F$ the
corresponding fibre integration (see~\S\ref{sec:fibre}). We orient the sphere $\p F$ as the
boundary of $F$, which is {\em opposite} to its orientation as the
boundary of a unit normal disk. 

The following two lemmas are proved in~\cite{Cieliebak-Volkov}.

\begin{lemma}[\cite{Cieliebak-Volkov}]\label{lem:Green}
The form $\wt\Pi$ is exact. Moreover,
there exists a (non-unique) smooth $(n-1)$-form $\wt G$ on $\wt M^2$ such that
\begin{equation}\label{eq:propag}
   d\wt G = (-1)^n\wt\Pi.
\end{equation}
Any such $\wt G$ satisfies
\begin{equation}\label{eq:fibre-int-G}
   \int_{\p F}\wt G=(-1)^n,
\end{equation}
and it can be chosen to also satisfy
\begin{equation}\label{eq:Green-symmetry}
   \tau^*\wt G = (-1)^n \wt G.
\end{equation}
\end{lemma}

By a slight abuse of language we will call $\wt G$ as in
Lemma~\ref{lem:Green} a {\em propagator}. It gives rise to a linear map
$$
   P:\Om^*(M)\longrightarrow \Om^{*-1}(M)
$$
by the formula
\begin{equation}\label{eq:HG}
   P\alpha(y) := \int_{x\in M}G(x,y)\alpha(x) = \int_F\wt G\wedge p_1^*\alpha,
\end{equation}
where the right hand side is the fibre integral with respect to the
projection $p_2:\wt M^2\to M$ onto the second factor as in Lemma~\ref{lem:Green}. 
 
\begin{lemma}[\cite{Cieliebak-Volkov}]
The map $P$ defines a chain homotopy between $\Id$ and $\Pi$,
\begin{equation}\label{eq:homo}
   d\circ P + P\circ d = \Pi - \Id.
\end{equation}
\end{lemma}

\begin{remark}\label{rem:symmGG}
The symmetry~\eqref{eq:Green-symmetry} for the integral kernel $\wt G$
implies that the homotopy operator $P$ is symmetric, i.e., it is a
propagator in the terminology of \S\ref{ss:coch} (see
equation~\eqref{eq:symmpropprelim}). 
\end{remark}

We will denote the pushforward of $\wt G$ to $M\times M$ (which is singular along the diagonal) again by $G$.  
Combining~\eqref{eq:propag}, ~\eqref{eq:Pi-kernel}
and~\eqref{eq:Pi-symmetry} we then have
$$
   dG(x,y) = (-1)^n\Pi(x,y) = \Pi(y,x) = \sum_ae_a(x)\wedge e^a(y),
$$
hence
\begin{equation}\label{eq:dG}
   d\wt G = \sum_a\pi_1^*e_a\wedge \pi_2^*e^a. 
\end{equation}
In the sequel we will often abbreviate the above sum as $e_a\times
e^a$, using the cross product notation and the Einstein summation
convention, or even as $e_ae^a$ to save space. Let 
$$
\iota:\HH\otimes\HH^*\to\HH\otimes\HH\to
\Om^*(M^2)
$$
denote the composition of the following two maps: the first one is the
identification of $\HH^*$ with $\HH$ by means of the pairing
$\la\cdot,x\ra\mapsto x$, and the second one is the cross product.
Observe that $e_a\times e^a\in \Om^*(M^2)$
is the image under $\iota$ of the identity
$$
\id=\sum_a e_a\otimes \la\cdot,e^a\ra \in  
\HH\otimes\HH^*\cong \Hom(\HH,\HH).
$$
In particular the sum $e_a\times e^a$ depends only on $\HH$ and not on
the choice of the basis. The discussion in this section can be
summarized in the following ``de Rham analog''
of Corollary~\ref{cor:harm-prop}.

\begin{proposition}\label{prop:existsprop}
Let $M$ be a closed oriented manifold and $\bigl(\Om =
\Om^*(M),d,\,\wedge\bigr)$ its de Rham algebra equipped with 
the algebraic pairing $\la\cdot,\cdot\ra$, see~\eqref{eq:pairing}. Fix any 
complement $\HH$ of $\im d$ in $\ker d$ and define the projection $\Pi$ onto 
$\HH$ by~\eqref{eq:projdef}. Let $e_a$ be a basis 
of $\HH$ and $e^a$ the dual basis with respect to $\la\cdot,\cdot\ra$. Let 
$\wt G$ be any symmetric primitive of $\pi^*(\sum_a e_a\times e^a)$
on $\wt M^2$. Then the integral operator $P$ defined by~\eqref{eq:HG}
is a propagator with respect to $\Pi$ in the sense of~\S\ref{ss:coch}.  
\qed
\end{proposition}

\section{Blow-ups and Stokes' theorem}\label{sec:blow-up}

In this section we recall background on manifolds with
corners and blow-ups and develop an abstract setting for Stokes' theorem. 

\subsection{Manifolds with corners}\label{ss:corners}

Throughout this section, let $X$ be a {\em manifold with
corners}. Recall that this is is defined like a manifold,
with open subsets of $\R^n$ replaced by open subsets of $[0,\infty)^n$;
see e.g.~\cite{Joyce-corners}. 
For $k\geq 0$ we denote by $\p_kX$ its codimension $k$ stratum (where
exactly $k$ of the local coordinates are zero), and by $\p_{\ge k}X$
the (closed) union of strata of codimension at least $k$. 
The interior $\p_0X$ will also sometimes be denoted by $X_0$.

{\bf Nice submanifolds. }
By a ``closed submanifold'' we will mean a submanifold which is closed
as a subset (such as $\R\subset\C$). 
We say that a closed submanifold 
$C$ (possibly with boundary and corners) of $X$ is {\em nice}
if $C\cap \p_kX=\p_k C$ and $C$ is transverse to $\p_kX$ for all
$k$. In particular, a nice submanifold of a manifold without boundary
has no boundary. The natural inclusion of a nice submanifold will be
called a {\em nice embedding}. 

{\bf Transverse collections. }
Next, we will define transversality for a finite
collection of nice submanifolds $C_a$ of $X$. 
Assume first that $X$ and the $C_a$ have no boundary. Then $C_1$ and
$C_2$ intersect {\em transversely} if $T_pC_1+T_pC_2=T_pX$ for
all $p\in C_1\cap C_2$.
Assume now inductively that transversality has been
defined for any collection of $m-1$ submanifolds, for some $m\ge 3$.
Then we call a collection of $m$ submanifolds {\em transverse}
if any subcollection of $m-1$ submanifolds is
transverse and its intersection is transverse
to the remaining submanifold.

Let now $X$ be a manifold with corners. 
By a {\em boundary component} of $X$ we will mean the closure of a
codimension $1$ stratum (which is again a manifold with corners). Then
$$
  \p X := \p_{\geq 1}X = \bigcup_{b\in \BB}\p^bX
$$
where $\p^bX$, $b\in \BB$, are the boundary components. 
The following remark will be used repeatedly in the sequel. 

\begin{remark}\label{rem:extension}
By definition of a manifold with corners, near each $p\in \p_k X$, $k\geq 1$, we can
slightly extend $X$ and the boundary components $\p^b X$ that meet at $p$ 
to manifolds without boundary such that the extended $\p^b X$
intersect in the extended $X$ at $p$ like $k$ coordinate hyperplanes
intersect at $0$ in $\R^d$. It is easy to see that a closed smooth submanifold $C\subset X$ is nice if and only if near each point $p\in \p_kX$ it admits an extension such that the extended $C$ together with the $k$ extended boundary components forms a transverse family near $p$. 
\end{remark}

Since the notion of transversality is a local one, 
it is enough to define it near each $p\in X$. If $p\in \p_kX$
for some $k\geq 1$, then we say that a finite collection
$\{C_a\}_{a\in \AA}$ of nice submanifolds of $X$ 
intersects {\em transversely} at $p$ if
the combined collection of the extended $C_a$'s and the extended  
$\p^bX$'s in the extended $X$ from Remark~\ref{rem:extension}
intersects transversely near $p$ in the above sense {\it without boundary}.

In the following, by a {\em manifold with corners with a transverse collection} 
$$
  (X,C) = (X, \{C_a\}_{a\in \AA})
$$
we will mean a manifold with corners $X$ with a transverse collection of
nice submanifolds $C_a\subset X$ indexed by a finite set $\AA$.

Let $(X,C)$ be manifold with corners with a transverse collection, and 
$$
  g:Y\longrightarrow X
$$
be a smooth map from a manifold with corners. We say that the map $g$
is {\em transverse} to the collection $C=\{C_a\}$ if its graph
$$
   gr(g)=\{(y,g(y))\}\subset Y\times X
$$
together with the $Y\times C_a$ forms a transverse collection.
   
\begin{remark}\label{rem:seminice}
For technical reasons it will be convenient to formally add
to a transverse collection $\{C_a\}_{a\in \AA}$ the
collection of boundary components, i.e., to consider the extended collection 
$\{C_a\}_{a\in \AA}\amalg \{\p^b X\}_{b\in \BB}$. We reindex the 
extended collection using the set 
$$
  \AA_\p:=\AA\amalg\BB
$$ 
as $\{C_a\}_{a\in \AA_\p}$ by setting $C_a:=\p^a X$ for $a\in \BB$. 
This way, many statements about nice submanifolds of a
manifold with corners will formally follow from the corresponding
statements for manifolds without boundary. 
\end{remark}

\begin{remark}\label{rem:straighten}
For a manifold with corners with a transverse collection $(X,C)$, near
every point $p$ of $X$ there is a chart straightening all  
$C_a$'s and boundary components passing through $p$. 
\end{remark}

\subsection{Blow-ups}\label{ss:blow-up}

In this subsection we recall the basic facts about oriented real
blow-ups. We work in the category of oriented manifolds with corners.  

Let $(X, C=\{C_a\}_{a\in \AA})$ be manifold with corners with a
transverse collection as in the preceding subsection.
We stratify the union $\bigcup_{a\in \AA_\p}C_a$ as 
$$
  \bigcup_{a\in \AA_\p}C_a=\coprod_{\varnothing\ne J\subset \AA_\p}
  X_J,
$$
where for a nonempty subset $J\subset \AA_\p$ we set
\begin{equation}\label{eq:basicstrat0}
  X_J:=\bigcap_{a\in J}C_a\setminus\bigcup_{a\in\AA_\p\setminus J}C_a.
\end{equation}  
Recall that $\AA_\p=\AA\amalg\BB$.  
For $a\in \AA$ we denote by $N_{C_a}:=TX|_{C_a}/TC_a$ the normal bundle
to $C_a$ and introduce its oriented projectivization
$$
   P^+N_{C_a}:=(N_{C_a}\setminus C_a)/\sim,
$$
where $v_1\sim v_2$ if and only if $v_1=t v_2$ for some $t>0$.
This gives a sphere bundle
$$
  P^+N_{C_a}\longrightarrow C_a.
$$
For $a\in \BB$ we define $P^+N_{C_a}\stackrel{\cong}\to C_a$ to be the
trivial bundle whose fibre over each point consists of the inward
pointing normal vector. 
We define the bundle 
$$
  \pi_J:P^+N_J\longrightarrow X_J
$$
(with fibre a point or a product of spheres) as the pullback of the product
bundle 
$$
\prod_{a\in J}P^+N_{C_a}\longrightarrow 
\prod_{a\in J}C_a
$$
under the natural inclusion $X_J\into \prod_{a\in J}C_a$. 
We define (as a set)
$$
   \Bl(X,C) := (X\setminus \bigcup_{a\in \AA_\p}C_a)\amalg
   \coprod_{\varnothing\neq J\subset \AA_\p}P^+N_J.
$$
By a slight abuse of notation, we will often write
$$
  C = \bigcup_{a\in \AA_\p}C_a.
$$
The natural projection 
$$
  \pi: \Bl(X,C)\longrightarrow X
$$
is given by the identity on $X\setminus C$ and by $\pi_J$ on
$P^+N_J$. When there is no risk of confusion we will identify
$X\setminus C$ with its preimage under $\pi$.

\begin{lemma}\label{lem:blow-up}
The set $\Bl(X,C)$ carries the natural structure of a manifold with
corners such that $P^+N_J$ becomes part of the codimension $|J|$ boundary.
\end{lemma}

\begin{proof}
Let us first describe the blow-up $\wt\R^d:=\Bl(\R^d,0)$ of $\R^d$ at the
origin. It is defined semi-algebraically using the incidence relation
$$
  \wt\R^d=\{(l,x)\in P^+\R^d\times \R^d\mid x\in l\}.
$$
Consider the homeomorphism
$$
   \Phi:[0,\infty)\times S^{d-1}\stackrel{\cong}\longrightarrow
     \wt\R^d,\qquad (r,v)\mapsto ([v],rv),
$$
where $S^{d-1}\subset\R^d$ is the unit sphere and $[v]\in P^+\R^d$ is
the ray defined by $v$. Note that its inverse if given by
$\Phi^{-1}(l,x)=(|x|,v)$ for the unique representative $v\in S^{d-1}$
of $l$. We make $\wt\R^d$ a manifold with boundary by declaring $\Phi$
to be a global chart. 

For a linear subspace $E\subset \R^d$, the blow-up $\Bl(\R^d,E)$ is
diffeomorphic to the product $\wt\R^{d-\dim E}\times E$ using the
orthogonal splitting $\R^d=E^\perp\oplus E$. For a transverse
collection $C$ of linear subspaces in $\R^d$, the blow-up is
therefore diffeomorphic to the product of several $\wt\R^{d_a}$ and a
linear space. In view of Remark~\ref{rem:straighten}, this gives us
local manifold-with-corner charts for $\Bl(X,C)$. 
The statement about $P^+N_J$ follows immediately from this
description. 
\end{proof}

We call the manifold with corners 
$\Bl(X,C)$ together with the map $\pi$ the 
{\em (oriented real) blow-up of $X$ along $C$}.
The map $\pi$ is called the {\em blow-down map}.

\begin{remark}\label{rem:bdrytaut}
Blowing up a manifold with corners along its boundary strata does not 
change the manifold.
The reason to do this is the following. 
Once we include boundary strata in the blow-up locus, we can extend
the well-known theorem that blowing up a manifold {\em without
  boundary}  
along a transverse collection is a manifold with corners to
the case of the ambient manifold having boundary and corners.
\end{remark}

\begin{remark}\label{rem:blowdownker}
For $p\in P^+N_J$ the kernel $\ker d_p\pi$ of the
blow-down map equals the tangent space to the fibre of the  
product sphere bundle $\pi_J:P^+N_J\longrightarrow X_J$.
\end{remark}


\subsection{Proper transforms}\label{ss:blow-up-proptrans}

Let $(X,C)$ be a manifold with corners with a transverse collection and 
$$
  \pi:\Bl(X,C)\rightarrow X
$$ 
the corresponding blow-up. The {\em proper transform} of a subset
$Z\subset X$ is the closure of $Z\setminus C$ in $\Bl(X,C)$,
$$
  PT(Z) := Closure({\pi}^{-1}(Z\setminus C))\subset \Bl(X,C).
$$
\begin{definition}\label{def:approx-seq}
By definition of the proper transform, for any 
$q\in PT(Z)$ there exists a sequence $(x_n)\subset X\setminus C$
with $\pi^{-1}(x_n)\to q$. We call such a sequence {\em approximating}.  
\end{definition}

\begin{remark}\label{rem:closedness}
In our applications the subset $Z$ will usually be closed in $X$. In
this case we can identify the part of $PT(Z)$ which lies in the
complement of the exceptional divisor with $Z\setminus C$ by means of $\pi$.
\end{remark}

\begin{lemma}\label{lem:PTtransverse}
In the situation above, assume that $Z\subset X$ is a nice submanifold
such that $\{Z\}\cup C$ is a transverse collection. Then
$(Z,Z\cap C)$ is a manifold with corners with a transverse collection,
and the natural inclusion
$$
  \iota:(Z,Z\cap C)\longrightarrow (X,C)
$$
lifts to the blow-ups 
$$
  \wt\iota:\Bl(Z,Z\cap C)\longrightarrow \Bl(X,C)
$$
as a nice embedding with image
$$
\wt\iota(\Bl(Z,Z\cap C))=PT(Z).
$$
\end{lemma}
 
\begin{proof}
By definition $Z$ is a manifold with corners, and the transversality
hypothesis implies that $Z\cap C$ is a transverse collection in $Z$.
To define the map $\wt\iota$, recall that $C=\{C_a\}_{a\in\AA}$. 
Away from the preimage of $C$ the map $\wt\iota$ is defined
as $\iota$. Consider now a nonempty subset $J\subset\AA$ and a point
$(w,z)\in P^+N_J\subset\Bl(Z,Z\cap C)$, where $z\in Z_J=Z\cap X_J$ and
$w=(w_a)_{a\in J}$. For each $a\in J$ the tangent map $\iota_*$
induces an isomorphism $N(Z\cap C_a,Z)\to N(C_a,X)$
between the normal bundles, which descends to a map
$P_a^+\iota$ between the projectivizations. We define 
$$
  \wt\iota(w,z) := \Bigl(\bigl(P_a^+\iota(w_{a})\bigr)_{a\in J}\,,\,\iota(z)\Bigr)
  \in P^+N_J\subset\Bl(X,C).
$$
It is straightforward to check that $\wt\iota$ is a nice embedding.
Since the map $\iota$ induces a homeomorphism between $Z\setminus
Z\cap C$ and $\iota(Z)\setminus C$, its lift $\wt\iota$ induces a
homeomorphism between their closures in the respective blow-ups.
\end{proof}

Sometimes by abuse of language we will identify $\Bl(Z,Z\cap C)$ with 
$PT(Z)$ via $\wt\iota$.

\subsection{A general setup for Stokes' theorem}\label{ss:Stokes-gen}

In this subsection we introduce a general setup for Stokes' theorem. 

{\bf Manifolds with quasi-regular boundary. }
We begin with some definitions from Pawlucki's
article~\cite{Pawlucki}.

\begin{definition}\label{def:gen-q-reg}
Let $L$ be a topological space and $\p^{q-reg} L\subset L$ a closed subset. 
We say that $(L,\p^{q-reg} L)$ (or simply $L$) is {\em a manifold with 
quasi-regular boundary} if the following holds: the difference 
$L\setminus \p^{q-reg} L$ is an oriented manifold; for each point 
$p\in \p^{q-reg} L$
there exists an open neighbourhood $U\subset L$ of $p$ such that 
$U\setminus \p^{q-reg} L$ consists of $m\ge 1$ connected components
$\{U^j_0\}_{j=1}^m$ and each
$U^j:=U_0^j\amalg (\p^{q-reg} L\cap U)$ is a $C^1$-manifold with
boundary, with interior $U_0^j$ and boundary $(\p^{q-reg} L\cap U)$. 
The $U^j$ are called the {\em local regular components} at $p$. The
multiplicity $m$ can depend on $p$ but must be locally
constant. The open subset of $\p^{q-reg} L$ defined by the equation 
$m=1$ is denoted by $\p^{reg} L$ and called the {\em regular boundary}.
If $m=1$ constantly, then we get the well-known notion of an oriented
manifold with boundary. 
\end{definition}

\begin{remark}\label{rem:reg}
(a) The orientability hypothesis in Definition~\ref{def:gen-q-reg} is
included because it is needed for our applications to Stokes' theorem.
Let us emphasize that in a manifold with quasi-regular boundary the
local regular components are only required to have regularity $C^1$
(whereas otherwise we usually assume all objects to be of class $C^\infty)$.

(b) An example of a manifold with quasi-regular boundary is a graph,
where the multiplicity at a vertex is its degree. In the examples
relevant in this paper, arising from compactified configuration
spaces, the multiplicity will actually be one. We include it in the
definition so that we can directly appeal to the results in~\cite{Pawlucki}.
\end{remark}

An {\em odd $k$-form} on a manifold
$L$ is a $k$-form $\alpha$ on its orientation double cover $\wt L$ with
$\tau^*\alpha=-\alpha$ for the canonical involution $\tau:\wt L\to\wt L$. 
If $k=\dim L$ and $\alpha$ has compact support, then it has a
well-defined integral $\int_L\alpha$.  
In this terminology, each local regular component $U^j$ at $p\in
\p^{q-reg}L$ induces an odd $0$-form $\eps^j$ on $\p^{q-reg}L$ near $p$
whose value is $1$ on the boundary orientation. The sums
$\eps^1+\dots+\eps^m$ at all $p$ fit together to a $\Z$-valued odd
$0$-form $\eps$ on $\p^{q-reg}L$.

\begin{definition}\label{def:emb-q-reg}
Let $(L,\p^{q-reg} L)$ be a manifold with quasi-regular boundary and $(N,\p N)$
be a manifold with corners. We say that a map 
$\iota:(L,\p^{q-reg} L)\rightarrow (N,\p N)$ is an {\em embedding 
(of a manifold with quasi-regular boundary into a manifold with corners)}
if $\iota$ is injective, it restricts to an embedding between the interiors
$L\setminus \p^{q-reg} L\rightarrow N\setminus\p N$, and it restricts to an embedding 
of a manifold with boundary into a manifold with corners on each local
regular component of $L$. 
\end{definition}

\begin{remark}\label{rem:emb-imm}
Since immersions are locally embeddings and injectivity of $\iota$ is
required anyway, we can replace ``embedding'' by ``immersion'' in the
last condition of Definition~\ref{def:emb-q-reg}.
\end{remark}

{\bf Pairs. }
Now we introduce our general setup for Stokes' theorem. 

\begin{definition}\label{def:pair}
A {\em pair} $(\YY,\XX)$ consists of a (not necessarily compact)
manifold-with-corners $\YY$ and a closed subset $\XX\subset\YY$
with a decomposition
\begin{equation}\label{eq:basicstrat}
\XX=\XX_0\amalg\p X,
\end{equation}
where $\XX_0$ is an oriented $d$-dimensional submanifold
of $\YY_0$ whose closure equals $\XX$.
\end{definition}

\begin{remark}\label{rem:stratinduced}
More accurately, the datum of a pair should be $(\YY,\XX_0)$, which
determines $\XX=Closure(\XX_0)$. The notation $(\YY,\XX)$ is still
unambiguous, because in all our applications except the one in
Theorem~\ref{thm:Stokes-Pawlucki} (where we spell out $\XX_0$), the
stratification~\eqref{eq:basicstrat} on $\XX$ is induced by the 
stratification on $\YY$ as
$$
  \p\XX:=\p\YY\cap \XX,\qquad \XX_0:=\YY_0\cap \XX.
$$
\end{remark}

\begin{definition}\label{def:bdry}
Let $(\YY,\XX)$ be a pair. 
Consider the collection of all open subsets $\XX_1$ of $\p\XX$
such that the natural inclusion $\XX_0\amalg\XX_1\into \YY$
is an embedding of a manifold with quasi-regular boundary.
{\em The quasi-regular boundary 
$\p^{q-reg}\XX$} of $\XX$ is the subset of $\p\XX$ maximal with respect
to this property. The {\em regular boundary} $\p^{reg}\XX$ of $\XX$ is
defined by requiring that the multiplicity $m$ be equal to $1$.  
We denote
\begin{equation}\label{eq:Xhat}
  \wh\XX := \XX_0\amalg \p^{q-reg}\XX.
\end{equation}
This is by definition a quasi-regular submanifold of $\YY$ and as such 
its boundary $\p^{q-reg}\XX$ carries an odd $0$-form $\eps$.
When we want to emphasize its dependence on $\XX$ we will write $\eps_\XX$.
\end{definition}

\begin{remark}\label{rem:diffeo-pairs}
Given two pairs $(\YY_1,\XX_1)$ and $(\YY_2,\XX_2)$, any
diffeomorphism $\YY_1\to\YY_2$ which restricts to a homeomorphism
$\XX_1\to\XX_2$ induces a diffeomorphism $\wh\XX_1\to\wh\XX_2$.
\end{remark}

Recall the codimension $k$ boundary $\p_k\YY$ of $\YY$ and the
``codimension at least $k$'' part $\p_{\ge k}\YY$ of the boundary of
$\YY$.

\begin{definition}\label{def:hidden}
Assume that $\XX$ carries the induced stratification from
Remark~\ref{rem:stratinduced}. 
Then we define the {\em primary} and {\em hidden} parts of 
$\p^{q-reg}\XX$ as
$$
\p^{\main}\XX:=\p^{q-reg}\XX\cap\p_1\YY,
\qquad 
\p^{\hidden}\XX:=\p^{q-reg}\XX\cap\p_{\ge 2}\YY.
$$
\end{definition}

{\bf Stokes' theorem. }
Let $(\YY,\XX)$ be a pair and 
$\XX_1\subset \p^{q-reg}\XX$ be any open subset. The central object 
for Stokes' theorem is the following equation for any 
$\beta\in \Om^*(\YY)$:
\begin{equation}\label{eq:Stokes}
  \int_{\XX_0} d\beta=\int_{\XX_1}\eps\beta.
\end{equation}
Here $\int_{\XX_1}\eps\beta$ is understood as the integral of an odd
$(d-1)$-form.

\begin{definition}\label{def:Stokes}
We say that {\em Stokes' theorem holds} (or simply {\em Stokes holds})
for a pair $(\YY,\XX)$ if for every $d$-form $\alpha\in \Om^d(\YY)$ such that
$supp\,\alpha\cap\XX$ is compact the integral $\int_{\XX_0} \alpha$
exists, and for every $(d-1)$-form $\beta\in \Om^{d-1}(\YY)$ 
such that $supp\,\beta\cap\XX$ is compact equation~\eqref{eq:Stokes} holds with $\XX_1=\p^{q-reg}\XX$.
\end{definition}
 
\begin{remark}\label{rem:integrability}
A Riemannian metric on $\YY$ induces a $d$-form $\mu$ on $\XX_0$
defined by $\mu(v_1,\dots,v_d)=1$ on a positive orthonormal basis of a
tangent space, so that integration of $\mu$ defines the
$d$-dimensional Lebesgue measure on $\XX_0$ (cf.~\cite{Pawlucki}). For
$\alpha\in\Om^d(\YY)$ we can uniquely write $\alpha=f\mu$ for a smooth function
$f:\AA_0\to\R$ satisfying $|f|\leq|\alpha|$. Therefore, the existence 
of the integral $\int_{\XX_0} \alpha$ for every $\alpha\in \Om^d(\YY)$
such that $supp\,\alpha\cap\XX$ is compact is equivalent to
$\int_{K\cap\XX_0}\mu<\infty$ for every compact subset $K\subset\YY$. 
Note that this condition does not depend on the choice of Riemannian
metric on $\YY$. 
An analogous discussion applies to the integral $\int_{\XX_1}\eps\beta$.
\end{remark}

The reason to introduce an open subset $\XX_1$ of $\p^{q-reg}\XX$
in Definition~\ref{def:Stokes} instead of writing directly
$\p^{q-reg}\XX$ in equation~\eqref{eq:Stokes} is the following lemma.

\begin{lemma}\label{lem:Stokesimplies}
Let $(\YY,\XX)$ be a pair and
$\XX_1\subset \p^{q-reg}\XX$ be an open set.
Assume that equation~\eqref{eq:Stokes} holds for every $(d-1)$-form $\beta\in \Om^{d-1}(\YY)$ such that $supp\,\beta\cap\XX$ is compact.
Then 
\begin{equation}\label{eq:modmeszero}
\XX_1=\p^{q-reg}\XX
\end{equation}
up to subsets of $\{\eps=0\}$ and sets of measure $0$ in $\p^{q-reg}\XX$. 
\end{lemma}

\begin{proof}
Assume by contradiction that the set 
$Z:=\{\eps\ne 0 \}\setminus\XX_1$ has positive measure in
$\p^{q-reg}\XX$. Pick an open neighbourhood $U\subset\XX$ of a 
quasi-regular boundary point as in Definition~\ref{def:gen-q-reg} such
that $U\cap \p^{q-reg}\XX\subset\{\eps\ne 0\}$ and $U\cap Z\subset
\p^{q-reg}\XX$ has positive measure.  
Pick a compactly supported $(d-1)$-form $\beta$ on $U\cap\p^{q-reg}\XX$
such that $\int_{U\cap Z}\eps\beta\neq 0$. Extend $\beta$ 
to a $(d-1)$-form on $\YY$ with compact support in $U$ which we still
denote $\beta$.
Let $\XX^j_U$, $j=1,\dots,m$ be the local regular components in
$U$. Then the usual Stokes' theorem yields
$\int_{\p\XX^j_U}\eps^j\beta = \int_{\XX^j_U}d\beta$ for all $j$,
which together with equation~\eqref{eq:Stokes} yields 
$$
   \int_{U\cap\p^{q-reg}\XX}\eps\beta =
   \sum_j\int_{\p\XX^j_U}\eps^j\beta = \sum_j\int_{\XX^j_U}d\beta = 
   \int_{U\cap\XX_0}d\beta = \int_{U\cap\XX_1}\eps\beta.
$$
Thus $0 = \int_{U\cap\p^{q-reg}\XX\setminus\XX_1}\eps\beta =
\int_{U\cap Z}\eps\beta$, contradicting the choice of $\beta$.
\end{proof}

This lemma will be used as follows. By definition,  
the quasi-regular boundary $\p^{q-reg}\XX$ is the set of {\it all}
points with a certain property. While it is often difficult to
identify exactly all quasi-regular points of $\p\XX$, sometimes there
is an open subset $\XX_1$ of $\p^{q-reg}\XX$ for which we can prove
equation~\eqref{eq:Stokes}. Then equation~\eqref{eq:modmeszero} allows 
us to identify $\p^{q-reg}\XX$ up to a subset of $\{\eps=0\}$ and a
set of measure zero. For
an implementation of this idea see the proofs of
Lemma~\ref{lem:StokesProd} and Lemma~\ref{lem:StokesFibre}. 

\begin{remark}\label{rem:local}
(a) Stokes' theorem is local. More precisely, if Stokes' theorem holds
for the pair $(\YY,\XX)$, then for any open subset $U$ of $\YY$ it also 
holds for $(\YY\cap U,\XX\cap U)$. Conversely, given a pair $(\YY,\XX)$ 
and an open cover $\{U_i\}_{i\in I}$ of $\YY$ the following is true:
if Stokes' theorem holds for all $(\YY\cap U_i,\XX\cap U_i)$,
then it holds for $(\YY,\XX)$. 
The proof is straightforward using a partition of unity.\\
(b) Suppose that $\ZZ$ is a manifold with corners, $\YY\subset\ZZ$ a
nice submanifold, and $(\YY,\XX)$ a pair. Then $(\ZZ,\XX)$ is also a
pair, and Stokes' theorem holds for $(\YY,\XX)$ if and only if it
holds for $(\ZZ,\XX)$. 
\end{remark}

\section{Stokes' theorem for configuration spaces}\label{sec:general graphs}

In this section we prove a version of Stokes' theorem for
configuration spaces associated to graphs, building on results of
Paw\l{}ucki~\cite{Pawlucki} in the semi-analytic setting. Moreover, we
derive a vanishing result for integrals over hidden faces.

\subsection{Graphs and their configurations spaces}\label{ss:graphs}

We begin by briefly describing a suitable class of graphs.
See~\S\ref{sec:graphs} for a more extensive discussion of graphs.

Let $\Gamma$ be a finite (not necessarily connected)
graph. Informally, this is a finite set of vertices connected by
edges. We also allow free edges adjacent to only one vertex which we
call {\em leaves}. A {\em flag} is then a pair of a vertex with an
adjacent edge or leaf. 

More formally, we define $\Gamma$ as a finite set $\Flag$ of flags 
together with a collection $\Edge$ of disjoint two-element subsets
called edges and a decomposition of into disjoint subsets called
vertices. Flags not belonging to any edge are called leaves. 

We assume that some (possibly empty) subset of the vertices are
designated as {\em special}, and we 
denote by $\Ver$ the set of nonspecial vertices. Flags adjacent to
special vertices will also be called special. 
We denote by $d_j$ ($j\in\Ver$) the valencies of the nonspecial
vertices, by $d$ the total number of special flags, 
by $e$ the number of edges, and by $s$ the number of leaves. 

We order the nonspecial and special vertices, as well as the flags
around each vertex. This gives us the so-called {\em vertex order} on
the set of flags. On the other hand, we also order and orient the
edges and order the leaves. This gives us the so-called {\em edge
  order} on the flags.
In the remainder of this section, by a {\em graph} we will mean a
graph with special vertices and chosen orderings as above (these
choices correspond to extended labellings in the terminology
of~\S\ref{ss:basiccomb}).

{\bf Spaces associated to a graph. }
Consider now an $n$-dimensional manifold oriented $M$ without boundary. 
We associate to each flag a variable with values in $M$. Using the
vertex and edge orders we can write this space in two ways with a
canonical reordering diffeomorphism between them,
\begin{equation}\label{eq:reoder}
R_\Gamma:Y_\Gamma:=\Bigl(\prod_{j\in\Ver}M^{d_j}\Bigr)\times
M^d\stackrel{\cong}{\longrightarrow} X_\Gamma:=(M^2)^e\times M^s .
\end{equation}
We will identify the two spaces $X_\Gamma$ and $Y_\Gamma$
in~\eqref{eq:reoder} via the map $R_\Gamma$.

Each factor $M^{d_j}$ has the slim diagonal $M_j$ naturally
diffeomorphic to $M$. We define the {\em vertex diagonal} by 
\begin{equation}\label{eq:vertexdiag}
\Delta_{\ver}:=\prod_{j\in\Ver} M_j\stackrel{\iota_\ver}{\into}\prod_{j\in\Ver} M^{d_j}.
\end{equation}
It corresponds to setting the flags at each {\it nonspecial} vertex
equal. Let
$$
\Delta_2:=\{x=y\}\subset M\times M=M^2
$$
denote the diagonal in $M^2$. 
For each $l\in \Edge$ we define the double diagonal
corresponding to this edge by
$$
 \Delta_2^l:=(M^2\times\cdots\times M^2\times 
 \Delta_2\times
 M^2\times\cdots\times M^2)\times M^s\subset X_\Gamma,
$$
where $\Delta_2$ comes at the position corresponding to the edge $l$. The collection of all double diagonals forms a transverse family in $X_\Gamma$ and we define
the (fat) {\em edge diagonal}
$$
   \Delta_2^\Gamma := \bigcup_{l\in\Edge(\Gamma)}\Delta_2^l\subset X_\Gamma.
$$
Let $\wt M^2$ denote oriented real the blow-up of $M^2$ along the
diagonal $\Delta_2$. Then the oriented real blow-up of
$X_\Gamma$ along $\Delta_2^\Gamma$ is given by
$$
  \wt X_\Gamma := \Bl(X_\Gamma,\Delta_2^\Gamma) = (\wt M^2)^e\times M^s.
$$

{\bf Basic pairs. }
Let $\WW$ be a manifold with corners and $(M^d\times \WW,\ZZ)$ be a
pair as in Definition~\ref{def:pair}.  
Assume that the stratification of $\ZZ$ is induced by that of 
$M^d\times \WW$:
$$
\ZZ_0:=\ZZ\cap (M^d\times\WW_0),\quad \p\ZZ:=\ZZ\cap(M^d\times\p\WW).
$$
See Remark~\ref{rem:stratinduced}. We construct another pair,
incorporating the information carried by the graph $\Gamma$. 
Recall the blow-up $\wt X_\Gamma$ of $X_\Gamma$ along the edge
diagonal $\Delta_2^\Gamma$, and consider the
product of the natural blow-down map with the identity 
\begin{equation}\label{eq:pi}
\pi:\wt X_\Gamma\times \WW\longrightarrow X_\Gamma\times\WW.
\end{equation}

\begin{definition}\label{def:basic-pair}
The {\em basic pair} associated to the above setup is 
\begin{equation}\label{eq:graphtopair}
  (\YY,\XX) := (\YY_\Gamma,\XX_\Gamma) := \Bigl(\wt X_\Gamma\times \WW,PT(\Delta_{\ver}\times \ZZ)\Bigr),
\end{equation}
where the stratification of $\XX$ is induced from that of $\YY$ as
$$
\XX_0:=\XX\cap \YY_0,\quad \p\XX:=\XX\cap\p \YY.
$$
\end{definition}

Note that the natural blow-down map identifies the interior $\XX_0$ with
$\Delta_{\ver}\times\ZZ_0\setminus \Delta_2^\Gamma\times \WW_0$, which 
is an oriented open submanifold of $X_\Gamma\times \WW_0$. Thus
$(\YY,\XX)$ is indeed a pair. 

\begin{example}\label{ex:Bott-Taubes}
Let $\Gamma_{k,t}$ be the full graph on $k+t$ vertices,
where the first $k$ vertices are special.
The valency at each vertex equals $k+t-1$, we have 
$k$ special vertices and thus $d=k(k+t-1)$ special flags. Write
$$
  M^d=(M^{k+t-1})^k,
$$
so that each $M^{k+t-1}$ corresponds to its own special vertex. 
Let $\iota:N\into M$ be an embedding of a closed submanifold. Consider the map
$$
  ev_\iota:N^k\longrightarrow (M^{k+t-1})^k,\qquad (q_1,\dots,q_k)\mapsto 
  (x_1,\dots,x_k),
$$
where
$$
  x_j:=(\iota(q_j),\dots,\iota(q_j))\in M^{k+t-1}
$$
corresponds to the special vertex number $j$. Set 
$$
  \WW:=N^k\quad \text{and}\quad \ZZ:=gr(ev_\iota)\subset M^d\times N^k.
$$
Then the corresponding compactification $\XX$ recovers the space
$C_{k,t}(M;N)$ of Bott and Taubes~\cite{Bott-Taubes}.\footnote{
More precisely, this corresponds to the space
$C_t(M;(N_1,n_1),\dots,(N_m,n_m))$ in the Appendix of~\cite{Bott-Taubes}
with $m=1$, $N_1=N$ and $n_1=k$; the general case can be treated similarly.
}
According to~\cite[Proposition A.3]{Bott-Taubes}, $C_{k,t}(M;N)$ is
actually a manifold with corners.
Bott and Taubes use the following variant of this construction for
knot-theoretic purposes. Recall from~\S\ref{ss:chen-cyc} the cyclic simplex
$$
  \Delta^k_{\rm cyc} = \{(t_1,\dots,t_k)\in (S^1)^k\mid t_1\le t_2\le\dots\le t_k\le t_1\}.
$$
Let $\gamma:S^1\longrightarrow \R^3$ be a knot and
consider the evaluation map
$$
  ev_{\gamma}:\Delta^k_{\rm cyc}\longrightarrow
  ((\R^3)^{k+t-1})^k,\quad (t_1,\dots,t_k)\mapsto (x_1,\dots,x_k),
$$
where 
$$
  x_j:=(\gamma(t_j),\dots,\gamma(t_j))\in (\R^3)^{k+t-1}
$$
corresponds to the special vertex number $j$. The desired
compactification is then the compactification $\XX$ above in whose
construction $\Delta^k_{\rm cyc}$ plays the role of $N^k$ and
$ev_\gamma$ plays the role of $ev_\iota$.
The resulting space
$$
  C_{k,t}(\gamma):=C_{k,t}(\R^3;S^1)
$$
is a compactification of the configuration spaces of $k+t$ distinct
points in $\R^3$ of which the first $k$ lie on the knot and are
cyclicly ordered.
Bott and Taubes view $C_{k,t}(\gamma)$ as the fibers of a bundle
over the space of knots and use fibre integration of suitable
differential forms
to produce knot invariants.
\end{example}

\begin{example}\label{ex:string}
Let $M$ be a closed oriented manifold and $\Lambda M$ its free loop
space. Let $\Gamma$ be a graph with one special 
$d$-valent vertex. Consider a smooth map $f:B\longrightarrow \Lambda
M$ from a compact manifold with corners $B$ and its evaluation map
$$
ev_f:B\times \Delta^{d-1}\longrightarrow M^d,\quad
(p,t)\mapsto f_p(0),f_p(t_1),\dots,f_p(t_{d-1}).
$$
Here 
$$
\Delta^{d-1}:=\{t=(t_1,\dots,t_{d-1})\mid 
0\le t_1\le\dots\le t_{d-1}\le 1\}
$$
is the standard $(d-1)$ simplex.
Set $\WW:=B\times \Delta^{d-1}$ and $\ZZ:=gr(ev_f)$. 
The resulting compactification $\XX$,
which is in general not a manifold with corners,
will play a crucial role in~\S\ref{sec:analysis}.
\end{example}

\subsection{Hidden faces}\label{ss:hidden}

Let $J\subset \Edge$ be some collection of edges of $\Gamma$, and 
$\Gamma_J$ be the subgraph of $\Gamma$ formed by the edges in $J$ and
their adjacent vertices. Let $\Ver_J$ be the set of nonspecial
vertices of $\Gamma_J$.
Consider the subset\footnote{This definition differs slightly from
that in~\eqref{eq:basicstrat0} where we used the extended transverse collection.} 
$$
X_J:=\bigcap_{l\in J}\Delta_{2}^l\setminus (\bigcup_{l\notin J}\Delta_{2}^l)
$$
of $X_\Gamma$ and set 
\begin{equation}\label{eq:VJ}
V_J:=(X_J\times \WW)\cap (\Delta_{\ver}\times \ZZ)\quad \text{and}\quad 
\p_J\XX:=\XX\cap \pi^{-1}(V_J).
\end{equation}
Suppose that the intersection 
\begin{equation}\label{eq:q-reg-J}
\p_J^{q-reg}\XX:=\p_J\XX\cap \p^{q-reg}\XX
\end{equation}
is nonempty and consider
a point $q$ in it. By definition of $\p^{q-reg}\XX$, there exists near
$q$ a manifold-with-corners chart for $\YY$ in which $\p^{q-reg}\XX$
is a linear subspace. Then the linear inequalities for $\YY$
corresponding to indices in $J$ must be equalities on this subspace,
so an open neighbourhood of $q$ in $\p^{q-reg}\XX$ is contained in
$\pi^{-1}(X_J\times \WW)$. Therefore, $\p_J^{q-reg}\XX$ is open in 
$\p^{q-reg}\XX$, in particular $\p_J^{q-reg}\XX$ is a smooth
submanifold of $\YY$. Observe that 
\begin{equation}\label{eq:dec-J}
  \p^{q-reg}\XX=\coprod_{J\subset \Edge}\p_J^{q-reg}\XX
\end{equation}
is a finite disjoint union over the subsets $J$ of $\Edge$ for
which $\p_J^{q-reg}\XX$ is nonempty. Since the union is disjoint and
its members are open,
each $\p_J^{q-reg}\XX$ is also closed in $\p^{q-reg}\XX$.

\begin{lemma}\label{lem:flip}
In the setting above, assume that $\Gamma_J$ has a $2$-valent
nonspecial vertex $B$ with adjacent oriented edges $(A,B)$ and $(B,C)$
(see Figure~\ref{fig:hidden}, where $A$ and $C$ can be equal but are
different from $B$). Let 
\begin{equation*}
\tau:\wt X_\Gamma\times \WW\longrightarrow \wt X_\Gamma\times \WW
\end{equation*}
denote the involution swapping the two $\wt M^2$ factors of $\wt
X_\Gamma$ that correspond to the oriented edges $(A,B)$ and $(B,C)$. 
Then the map $\tau$ preserves the boundary locus $\p_J\XX$ and its
quasi-regular part $\p_J^{q-reg}\XX$. Moreover, the restriction  
of $\tau$ to $\p_J^{q-reg}\XX$ is orientation preserving for even $n$
and orientation reversing for odd $n$. 
\end{lemma}

\begin{proof}
Denote by $\Flag$ the set of flags of $\Gamma$, and by
$\Flag_J$ the subset of all flags in $\Gamma_J$ together with all
flags in $\Gamma$ adjacent to nonspecial vertices in $\Gamma_J$.

Consider a point $\xi=(p,q,w)\in V_J$, 
where $p\in \Delta_{\ver}\cap X_J$, $q\in M^d$, and $w\in \WW$ 
(with $(q,w)\in \ZZ$). Let $U$ be a small open neighbourhood of $\xi$ in 
$X_\Gamma\times\WW$ disjoint from all double diagonals not in $J$.   
Write $(p,q)=r=(r_a)_{a\in\Flag}$ and note that the 
$r_a$ for $a\in\Flag_J$ are all equal to the same
point $r^*\in M$.
Pick local coordinates near $r^*$ and near all the $r_b$ with
$b\notin\Flag_J$ to identify points of $X_\Gamma$ near $r$ with 
points in $(\R^n)^\Flag$. By shrinking $U$ if necessary, we may
assume that this defines coordinates on the projection of $U$
to $X_\Gamma$. 

{\bf The map $F$.}
Let $x,y_-,y_+,z\in\R^n$ denote the variables corresponding to the
flags $(A,B)$, $(B,A)$, $(B,C)$ and $(C,B)$, respectively.
See Figure~\ref{fig:hidden}.
\begin{figure}
\begin{center}
\includegraphics[width=\textwidth]{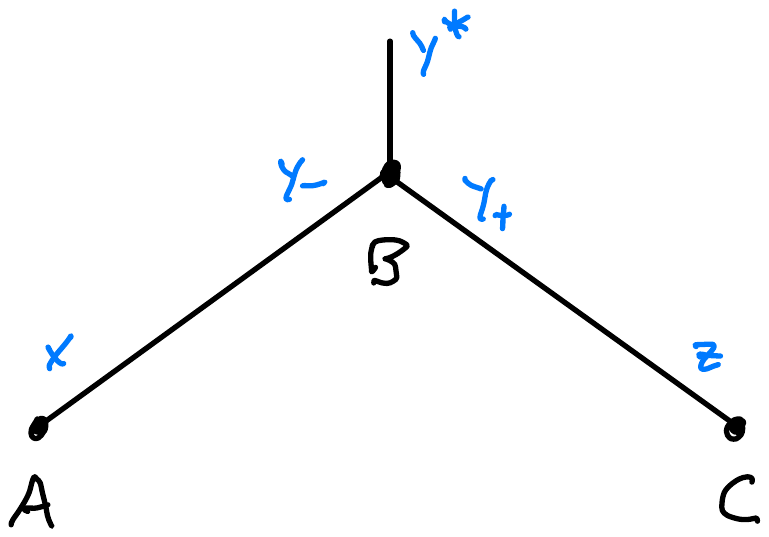}
\vspace{-4cm}
\caption{The involution on hidden faces}
\label{fig:hidden} 
\end{center}
\end{figure}
Let $y^*\in\R^n$ denote the variable corresponding to any other flag
adjacent to $B$.
We can assume that at the point $r=(p,q)$ we have 
$x=y_-=y_+=y^*=z=0$.
In these coordinates the involution $\tau$ is given by
$$
  (x,y_-,y_+,z)\mapsto (y_+,z,x,y_-)
$$
and the identity on all other variables. Note that $\tau$ does {\em not}
preserve the vertex diagonal $\Delta_{\ver}$ (corresponding to
$y_-=y_+=y^*$) away from the edge diagonal (corresponding to $x=y_-$ and
$y_+=z$). To remedy this, we consider the involution
$$
F:U\longrightarrow U
$$
defined by 
\begin{equation}\label{eq:defF}
(y_-,\,y_+,y^*)\mapsto (x+z-y_+,\,x+z-y_-,\,x+z-y^*) 
\end{equation}
in the $y$ variables, and the identity in all the other variables and 
in the $\WW$ factor. 
Note that $(\Delta_{\ver}\times\ZZ)\cap U$ is defined by the equations
\begin{equation*}
y_-=y_+=y^*.
\end{equation*}
On the subspace defined by these equations, the relation~\eqref{eq:defF}
reduces to 
\begin{equation}\label{eq:defKonts}
y\mapsto x+z-y 
\end{equation}
with $y:=y_-=y_+=y^*$
and we see that the following holds:
\begin{equation}\label{eq:pres-Delta}
\text{The map $F$ preserves $(\Delta_{\ver}\times\ZZ)\cap U$.}
\end{equation}
The double diagonals corresponding to $(A,B)$ and $(B,C)$ intersect
$U$ along $\{x=y_-\}$ and $\{y_+=z\}$, respectively, and $F$ acts
on $y_--x$ and $z-y_+$ as follows:
\begin{equation}\label{eq:swap}
\begin{cases}
y_--x\mapsto (x+z-y_+)-x=z-y_+,\cr  
z-y_+\mapsto z-(x+z-y_-)=y_--x.
\end{cases}
\end{equation}
Furthermore, the restriction of $\wh F$ to the intersection of any
other double diagonal with $U$ is the identity. 

{\bf The map $\wt F$.}
Property~\eqref{eq:swap} implies that $F$ extends to the blow-up of
$U$ along the product of the double diagonals in $J$ with 
$\WW$. 
This last blow-up is naturally identified with the preimage of $U$
under the map $\pi$ from~\eqref{eq:pi} and we get an involution
$$
\wt F:\pi^{-1}(U)\longrightarrow \pi^{-1}(U)
$$
covering $F$. The involution $\wt F$ allows us to make all the desired
conclusions.
By property~\eqref{eq:swap}, the restriction of $\wt F$ to
$\pi^{-1}(X_J\times \WW)$ coincides with the restriction of $\tau$ to
this set.
By property~\eqref{eq:pres-Delta}, the map $\wt F$ preserves the set  
$$
\XX_{U}:=\XX\cap \pi^{-1}(U),
$$ 
We conclude that the map $\wt F$ preserves $\p_J\XX\cap
\pi^{-1}(U)$ and acts on it as the restriction of $\tau$.
We need to improve this to preservation of the quasi-regular 
part $\p_J^{q-reg}\XX$ of $\p_J\XX$. Restricting the
decomposition~\eqref{eq:dec-J} to $\pi^{-1}(U)$ gives us that 
$\p_J^{q-reg}\XX\cap \pi^{-1}(U)$ is an open and closed subset of
the quasi-regular boundary 
$\p^{q-reg}(\XX_{U})$ of $\XX_{U}$ (as a subspace of $\pi^{-1}(U)$). 
The manifold $\p^{q-reg}(\XX_{U})$ is preserved under the 
ambient self-diffeomorphism $\wt F$ (because $\wt F$ preserves the
interior of $\XX_{U}$), and thus by $\tau$.
Since $\p_J^{q-reg}\XX$ can be covered by open neighborhoods of the form
$\pi^{-1}(U)$, this discussion implies that the map $\tau$ preserves 
$\p_J\XX$ as well as $\p_J^{q-reg}\XX=\p_J\XX\cap \p^{q-reg}\XX$.
The statement about orientations follows from
equation~\eqref{eq:defKonts} for the map $F|_{(\Delta_{\ver}\times \ZZ)\cap U}$. 
\end{proof}

\begin{remark}\label{rem:transfer}
(a) Let $\Gamma_J'$ be the graph obtained from $\Gamma$ removing all edges
in $\Edge\setminus J$. Then in the notation of the proof above,
the pair $(\pi^{-1}(U),\XX_U)$ can be seen
as a pair associated as in~\eqref{eq:graphtopair} to the 
graph $\Gamma_J'$ with $\R^n$ playing the role of $M$.
 
(b) The hypothesis of Lemma~\ref{lem:flip} implies that $J$ contains
at least two edges, so $\p_J^{q-reg}\XX$ is a hidden face in the sense
of Definition~\ref{def:hidden}. See~\cite[\S A.6]{Cieliebak-Volkov}
for the simplest example of a graph exhibiting a hidden face. 
\end{remark}

In the setup above, consider now an $(n-1)$-form $\eta$ on the blow-up 
$\wt M^2$ and a form $\alpha$ on $M^s\times\WW$.
Their cross product gives a form on $\wt X_\Gamma\times\WW$ defined by
\begin{equation}\label{eq:forms-int}
  \eta^e(\alpha):=\eta^{\times e}\times \alpha\,.
\end{equation}

\begin{cor}[Vanishing of integrals over hidden faces]\label{cor:cancelbasic}
Let $J$ a subset of the set of edges of the graph 
$\Gamma$ such that the graph $\Gamma_J$ has a $2$-valent nonspecial
vertex. Then for any $\eta\in \Om^{n-1}(\wt M^2)$ and
compactly supported $\alpha\in \Om^*(M^s\times\WW)$ we have
\begin{equation}\label{eq:cancellation}
\int_{\p_J^{q-reg}\XX_\Gamma}\eta^e(\alpha)=0.
\end{equation}
\end{cor}

\begin{proof}
We use the involution $\tau$ from Lemma~\ref{lem:flip}, which preserves 
$\p_J^{q-reg}\XX_\Gamma$ and acts on its orientation as $(-1)^n$.
We claim that 
$$
\tau^*\eta^e(\alpha)=(-1)^{n-1}\eta^e(\alpha).
$$ 
For this, consider the corresponding map $\tau$ on 
$\wt M^2\times \wt M^2$ (ignoring the other factors in $\wt X_\Gamma$
and $\WW$) and the projections $p_1,\,p_2$ onto the two factors. Then
$$
\tau^*(p_1^*\eta\wedge p_2^*\eta) = p_2^*\eta\wedge p_1^*\eta = (-1)^{n-1}p_1^*\eta\wedge p_2^*\eta,
$$ 
where the last equality holds because $\eta$ has degree $n-1$,
and the claim follows.

We conclude by invariance of integration. Namely, the map $\tau$
acts on the integrand as $(-1)^{n-1}$ and multiplies the orientation 
of the manifold we integrate over by $(-1)^n$. Therefore, the 
integral~\eqref{eq:cancellation} vanishes. 
\end{proof}

\begin{remark}\label{rem:vanish-extend}
Corollary~\ref{cor:cancelbasic} holds true for a larger class of
forms, namely for all compactly supported forms $\om$ on $\wt
X_\Gamma\times\WW$ satisfying for each involution $\tau$ as in
Lemma~\ref{lem:flip} (associated to a $2$-valent nonspecial
vertex $B$) the property
\begin{equation*}
  \tau^*\om=(-1)^{n-1}\om.
\end{equation*}
\end{remark}

\subsection{Stokes' theorem in the semi-analytic setting}
\label{ss:Stokes-ana}

Now we describe the main result of Paw\l{}ucki~\cite{Pawlucki} and apply
it to our situation. 

{\bf Analytic and semi-analytic sets. }
Let us first describe a special case of the setting  
of~\cite{Pawlucki}. Let $E$ be a finite dimensional real vector space. 
A subset $A$ of $E$ is called {\em semi-analytic} if for each point of
$E$ there exists an open neighbourhood $U\subset E$ and two
collections of real analytic functions
$\{f_i^j,g_i^j\}_{i=1,\dots,p}^{j=1,\dots,q}$ on $U$ such that 
$$
A\cap U=\bigcup_{i=1}^p\bigcap_{j=1}^q\{f_i^j=0,\,g_i^j>0\}.
$$
A semi-analytic subset $A$ of $E$ is called {\em analytic} if 
the functions $g_i^j$ can be taken constant equal to $1$.

We call manifold with corners $X$ {\em analytic} if the transition
maps between manifold-with-corner charts are restrictions of real analytic maps.
We can then extend the boundary components at a point $p\in\p_kX$ by the
corresponding coordinate hyperplanes in a real analytic chart. 
We call a subset $A\subset X$ {\em analytic} if in real analytic
charts it corresponds to the intersection of a quadrant with a real
analytic subset defined in a neighbourhood in $\R^n$. 
This ensures that the intersection $A\cap X_0$ is a semi-analytic set.
We call an analytic subset $C\subset X$ a {\em nice analytic submanifold}
if $C\cap X_0$ is a submanifold, and near each point $p\in \p_kX$ it admits an analytic extension
such that the extended $C$ together with the $k$ extended boundary
components forms a transverse family near $p$. 
A map between manifolds with corners is called {\em analytic} if it
locally extends to an analytic map between open subsets of linear spaces.

{\bf Stokes' theorem. }
%
%
The following theorem corresponds to the main result in Paw\l{}ucki~\cite{Pawlucki}.

\begin{theorem}\label{thm:Stokes-Pawlucki}
Let $\AA_0\subset E$ be a semi-analytic subset which is also an
oriented $d$-dimensional submanifold. Then
Stokes' theorem holds for the pair $(E,\,Closure(\AA_0))$. 
\end{theorem}

\begin{proof}
Note first that $(E,\AA)$ is a pair as in Definition~\ref{def:pair}, 
so it remains to verify the conditions in Definition~\ref{def:Stokes}.
Integrability over $\AA_0$ of every $d$-form $\alpha\in \Om^d(E)$ with
$supp\,\alpha\cap\AA$ compact follows from
Remark~\ref{rem:integrability} and~\cite[Lemma 3.5]{Pawlucki} . 
Equation~\eqref{eq:Stokes} is the content of~\cite[Theorem 3.7]{Pawlucki}. 
\end{proof}

We will apply this result to the following setting. Let $X$ be a
subset of a Euclidean space $F$ cut out by finitely many linear
inhomogeneous inequalities, 
$$
  X=\{h_i\ge 0\text{ for }i=1,\dots,K\}\subset F\cong \R^{\dim X}.
$$
We assume that the corresponding affine subspaces
$\{h_i=0\}_{i=1,\dots,K}$ form a transverse family, so that $X$ is a
manifold with corners and $X_0=\{h_i> 0\text{ for all }i\}$. Let 
$$
  C=\{C_a\}_{a\in \AA}
$$
be a transverse family of (finitely many) nice submanifolds of $X$, each
obtained by intersecting a linear subspace with $X$.

\begin{lemma}\label{lem:blow-up-analytic}
The manifold with corners $\YY:=\Bl(X,C)$ carries a canonical
analytic structure which agrees with the one induced from $F$ on its
interior $\YY_0\subset F$. 
\end{lemma}

\begin{proof}
By Lemma~\ref{lem:blow-up}, $\YY$ is a manifold with corners.
We give it an analytic structure by equipping the sphere
$S^{d-1}\subset\R^d$ with its standard analytic structure and
declaring the homeomorphism 
$\Phi:[0,\infty)\times S^{d-1}\to\wt\R^d$ in the proof of
Lemma~\ref{lem:blow-up} to be analytic. Since the second component
of $\Phi$ sending $(r,v)$ to $rv\in\R^d$ is analytic, this analytic
structure agrees with the one induced from $F$ on $\YY_0\subset F$.
Note that local analytic manifold-with-corner charts are obtained from
local analytic coordinates (e.g.~generalized polar coordinates) on the
involved spheres.  
\end{proof}

Continuing the above setup, let $Z\subset X$ be an analytic subset
such that $Z_0=Z\cap X_0\subset X_0$ is an oriented submanifold.

\begin{proposition}\label{prop:Stokes-ana}
Stokes' theorem holds for the pair $(\Bl(X,C),PT(Z))$.
\end{proposition}

\begin{proof}
As in Remark~\ref{rem:seminice}, we formally add the boundary
faces $\{h_i=0\}_{i=1,\dots,K}$ to the collection $C$ to treat them on
an equal footing. 
Since by Remark~\ref{rem:local} Stokes' theorem is local, it suffices
to prove it for a neighbourhood of every point of $\YY:=\Bl(X,C)$.
By Lemma~\ref{lem:blow-up-analytic}, for a sufficiently small such
neighbourhood there exists a homeomorphism $\Phi:U\to V$ onto an open
subset $V\subset[0,\infty)^k\times\R^{n-k}$ whose restriction
$\Phi|_{U_0}:U_0:=U\cap \YY_0\cong X_0\setminus C\to V_0:=V\cap
  (0,\infty)^k\times\R^{n-k}$ is analytic. Thus
$\AA_0:=\Phi(Z_0\setminus C)\subset V_0\subset\R^n$ is semi-analytic,
so by Theorem~\ref{thm:Stokes-Pawlucki} Stokes' theorem holds for the pair
$(\R^{\dim X},\AA)$ with $\AA:=Closure(\AA_0)=\Phi(U\cap PT(Z))$, and
therefore for the pair $(U,U\cap PT(Z))$. 
\end{proof}

\begin{lemma}\label{lem:Stokes-an}
Let $(M^d\times\WW,\ZZ)$ be a pair as in~\S\ref{ss:graphs}. Assume
that $M,\WW$ are real analytic and $\ZZ\subset M^d\times\WW$ is an
analytic subset. Then Stokes' theorem holds for the basic pair $(\wt
X_\Gamma\times \WW, PT(\Delta_{\ver}\times \ZZ))$ in~\eqref{eq:graphtopair}. 
\end{lemma}

\begin{proof}
Since by Remark~\ref{rem:local} Stokes' theorem is local, it suffices
to prove it for a neighbourhood of every point.
Consider a point $\xi\in \Delta_{\ver}\times \ZZ$. As in the proof of
Lemma~\ref{lem:flip} we pick local coordinates for $M$ near the
components of $\xi$, using the same coordinates near components which
agree. We also pick manifold-with-corner coordinates for $\WW$ near
the corresponding component of $\xi$. We choose the coordinates on $M$
and $\WW$ to be analytic. In these coordinates,
\begin{itemize}
\item $X_\Gamma$ corresponds to $\R^{a+b}$ and $\WW$ to $[0,\infty)^k\times \R^\ell$
  for some $a,b,k,\ell$; 
\item the edge diagonal $\Delta_2^\Gamma$ (the blow-up locus)
  corresponds to a transverse collection of linear subspaces of $\R^{a+b}$;
\item $\Delta_\ver$ corresponds to a linear subspace of $\R^a$;
\item $\ZZ$ corresponds to an analytic subset of
  $\R^b\times[0,\infty)^k\times\R^\ell$ whose intersection with
  the interior $\R^b\times(0,\infty)\times\R^\ell$ is an oriented
  submanifold. 
\end{itemize}  
Hence, the result follows from Proposition~\ref{prop:Stokes-ana}
applied to $X=\R^{a+b}\times[0,\infty)^k\times\R^\ell$,
$C=\Delta_2^\Gamma\times[0,\infty)^k\times \R^\ell$, and
$Z=\Delta_\ver\times\ZZ$. 
\end{proof}

\subsection{Chopping off trees}\label{ss:chopping}

In this subsection we show that Stokes' theorem passes to fibre
bundles and apply this to the operation of attaching trees to a graph. 
We begin with Stokes' theorem for products. 

\begin{lemma}\label{lem:StokesProd}
Let $(\YY^1,\XX^1)$ and $(\YY^1,\XX^1)$ be two pairs as in
Definition~\ref{def:pair}. Assume that Stokes holds for both
$(\YY^1,\XX^1)$ and $(\YY^2,\XX^2)$. Then Stokes holds for the pair 
$(\YY^1\times \YY^2,\XX^1\times \XX^2)$ with 
\begin{equation}\label{eq:prodregbdry}
\p^{q-reg}(\XX^1\times \XX^2)=
(\p^{q-reg}\XX^1)\times \XX^2_0+(-1)^{\XX^1}
\XX^1_0\times \p^{q-reg}\XX^2
\end{equation}
modulo subsets of $\{\eps_{\XX^1\times\XX^2}=0\}$ and sets of measure zero.
\end{lemma}

\begin{proof}

The integrability statement for the product follows from those for the
$(\YY^i,\XX^i)$ by Remark~\ref{rem:integrability} applied to the
product metric. 

Consider next two forms $\alpha\in \Om^*(\YY^1)$
and $\beta\in \Om^*(\YY^2)$ with 
$supp\,\alpha\cap\XX^1$ and $supp\,\beta\cap\XX^2$ compact.
For the following computation we stipulate that the integral of a differential form over a manifold 
is zero unless degree of the form equals 
dimension of the manifold. Then Fubini's theorem and Stokes' theorem
for $(\YY^1,\XX^1)$ and $(\YY^2,\XX^2)$ yields
\begin{align*}
&\int_{\XX^1\times\XX^2}
d(\alpha\times\beta)\cr
= &\int_{\XX^1\times\XX^2}
d\alpha\times\beta+(-1)^\alpha\alpha
\times d\beta\cr
= &\int_{\XX^1}d\alpha\int_{\XX^2}\beta+
(-1)^\alpha\int_{\XX^1}\alpha\int_{\XX^2}d\beta\cr
= &\int_{\p^{q-reg}\XX^1}\eps_{\XX^1}\alpha\int_{\XX^2}\beta+
(-1)^\alpha\int_{\XX^1}\alpha\int_{\p^{q-reg}\XX^2}\eps_{\XX^2}\beta\cr
= &\int_{(\p^{q-reg}\XX^1)\times\XX^2_0}\eps_{\XX^1\times\XX^2}\alpha\times\beta+
(-1)^\alpha\int_{\XX^1_0\times\p^{q-reg}\XX^2}\eps_{\XX^1\times\XX^2}\alpha\times\beta\cr
= &\int_{(\p^{q-reg}\XX^1)\times \XX^2_0+(-1)^{\XX^1}
\XX^1_0\times \p^{q-reg}\XX^2}\eps_{\XX_1\times\XX_2}\alpha\times\beta\,.
\end{align*}
Let now $\beta\in \Om^*(\YY^1\times\YY^2)$
be an arbitrary form with $supp\,\beta\cap(\XX^1\times\XX^2)$ compact.
It can be arbitrarily well $C^\infty$-approximated by 
$\sum_{j=1}^{k_n}\alpha_j^n\times\beta_j^n$ with 
$\alpha_j^n$ and $\beta_j^n$ such that 
$supp\,\alpha_j^n\cap \XX^1$ and $supp\,\beta_j^n\cap \XX^2$ are compact
for all $j$ and $n$. 
The above computation applied to $\alpha_j^n,\beta_j^n$ together with the limit as $n\to\infty$
shows equation~\eqref{eq:Stokes} with $(\p^{q-reg}\XX^1)\times \XX^2_0+(-1)^{\XX^1}\XX^1_0\times \p^{q-reg}\XX^2$ as codimension $1$ boundary of 
Now $\XX^1\times \XX^2$. Lemma~\ref{lem:Stokesimplies} implies the result.
\end{proof}

Consider now a fibre bundle of manifolds with corners
$$
  F\to \UU\stackrel{\pi}{\rightarrow}\YY
$$ 
For a subspace $A\subset\YY$ we denote $\UU_A:=\pi^{-1}(A)$, and
similarly for other bundles discussed below. Associated to $\UU$ are
the fibre bundles $\UU^{F_0}$ and $\UU^{\p_1F}$ with fibre $F_0$ (the
interior of $F$) and $\p_1F$ (the codimension $1$ boundary of $F$),
respectively.

\begin{lemma}\label{lem:StokesFibre}
In the above setup, let $(\YY,\XX)$ be a pair. 
Consider the restriction $\UU_\XX$ of $\UU$ to $\XX$. If Stokes holds for 
the pair $(\YY,\XX)$, then it also holds for the pair $(\UU,\UU_\XX)$ with
\begin{equation}\label{eq:q-reg-dec}
\p^{q-reg}\UU_\XX=\UU^{F_0}_{\p^{q-reg}\XX}\amalg \UU^{\p_1F}_{\XX_0}
\end{equation}
modulo subsets of $\{\eps_{\UU_\XX}=0\}$ and sets of measure zero.
\end{lemma}

\begin{proof}
Since Stokes' theorem is local, the result follows from local
triviality of fibre bundles and Lemma~\ref{lem:StokesProd}.
\end{proof}


{\bf Application 1. }
Let $\Gamma_1$ and $\Gamma_2$ be two graphs as in~\S\ref{ss:graphs}.
We form the disjoint union graph $\Gamma:=\Gamma_1\amalg \Gamma_2$.
Consider pairs $(M^{d_i}\times \WW_i,\ZZ_i)$, $i=1,2$, and set
$(M^d\times \WW,\ZZ):=(M^{d_1+d_2}\times\WW_1\times\times\WW_2,\ZZ_1\times\ZZ_2)$.
Then the basic pairs associated via~\eqref{eq:graphtopair} to these
pairs and the graphs $\Gamma_1$, $\Gamma_2$ and $\Gamma$ are related by
\begin{equation*}
\YY_\Gamma=\YY_{\Gamma_1}\times \YY_{\Gamma_2},\quad 
\XX_\Gamma=\XX_{\Gamma_1}\times \XX_{\Gamma_2}.
\end{equation*}
If Stokes holds for the pairs $(\YY_{\Gamma_i},\XX_{\Gamma_i})$, then
by Lemma~\ref{lem:StokesProd} it holds for $(\YY_\Gamma,\XX_\Gamma)$
and equation~\eqref{eq:prodregbdry} describes quasi-regular boundary 
of $\XX_\Gamma$. We put this in words as follows: 
\begin{itemize}
\item [(GEN1)] A connected component of the quasi-regular boundary of
  $\XX_\Gamma$ is generated by the quasi-regular boundary of either
  $\XX_{\Gamma_1}$ or $\XX_{\Gamma_2}$.
\end{itemize}

{\bf Basic pairs associated to trees. }
Recall that s {\em rooted tree} is a tree with a chosen univalent
vertex (the {\em root vertex}). In the sequel, when we say ``tree''
(or ``rooted tree'') we will always mean ``tree (respectively rooted
tree) without special vertices'' unless otherwise specified. Moreover,
in that case we will always set $\WW:=\ZZ:=\pt$. 
For trees, the associated basic pairs are particularly nice. 

\begin{lemma}\label{lem:tree-struct}
For a tree $T$, let $(\YY_T,\XX_T)$ be the
associated basic pair as in~\eqref{eq:graphtopair}. Then the natural inclusion 
$$
  \XX_T\into \YY_T
$$
is a nice embedding of manifolds with corners.
\end{lemma}

\begin{proof}
Recall the collection of double diagonals:
$$
\Delta_{2T}=\{\Delta_2^l\}_{l\in \Edge}.
$$
Since $T$ has no cycles, the collection 
$\Delta_{2T}$ is transverse to $\Delta_{\ver}$.
Lemma~\ref{lem:PTtransverse} implies that the intersections 
$$
D:=\{\Delta_{\ver}\cap \Delta_2^l\}_{l\in \Edge}
$$
form a transverse collection of submanifolds in 
$\Delta_{\ver}$ and the natural inclusion 
$$
\iota:\Delta_{\ver}\into X_T
$$
lifts to a nice embedding 
$$
\wt\iota:\Bl(\Delta_{\ver},D)\into \YY_T=\Bl(X_T,\Delta_{2T})
$$
with $\im\wt\iota=\XX_T$, giving the latter the 
desired manifold with corners structure.
\end{proof}

\begin{lemma}\label{lem:tree-fibre}
Let $T$ be a rooted tree and
$(\YY_T,\XX_T)$ the associated basic pair as 
in~\eqref{eq:graphtopair} with 
$\WW=\ZZ=\pt$. Let 
$$
  \pi_{root}:\XX_T\to M
$$ 
be the composition of the natural blow-down map and
the projection onto the $M$ factor corresponding to the root vertex. 
Then $\pi_{root}$ is a fibre bundle projection.
\end{lemma}

\begin{proof}
By Ehresmann's theorem, it suffices to show that the restriction of
$\pi_{root}$ to each boundary stratum of $\XX_T$ is a submersion onto
$M$. The restriction of the blow-down map to a boundary stratum is a
submersion onto its image 
$$
  X_J=\bigcap_{l\in J}
  (\Delta_{\ver}\cap \Delta_2^l)\setminus
  \bigcup_{l\in \Edge\setminus J} \Delta_2^l
$$
for some subset $J$ of $\Edge$. The projection onto the $M$ factor
corresponding to the root vertex restricts to $X_J$ as a submersion onto $M$. 
Thus, as a composition of two surjective submersions, the restriction
of $\pi_{root}$ to any boundary stratum is a submersion onto $M$.
\end{proof}

{\bf Application 2: Attaching trees. }
The main application of Lemma~\ref{lem:StokesFibre} can 
be described as ``attaching a tree to a compactification''. 
To describe it, let $(\YY,\XX)$ be a pair and $T$ be a rooted 
tree. Let $(\YY_T,\XX_T)$ be the basic pair 
associated by~\eqref{eq:graphtopair} to $T$. Let 
$$
\pi_{root}:\XX_T\longrightarrow M
$$
be the projection onto the factor corresponding to the root
vertex. Note that $\pi_{root}$ defines a smooth fibre bundle of
manifolds with corners (with base being a manifold without
boundary). For example, if the tree $T$ has just one edge, then   
$\XX_T=\wt M^2$ and the fibre of 
$\pi_{root}$ is $M$ blown up at a point. Let $h:\YY\longrightarrow M$
be any smooth map and consider the pullback bundle
\begin{equation}\label{eq:the map h}
\begin{tikzcd}
\UU:=h^*\XX_T\arrow[rr] \arrow[d, "\pi"] &  & \XX_T \arrow[d, "\pi_{root}"] \\
\YY \arrow[rr, "h"]               &  & M               
\end{tikzcd}
\end{equation}
If Stokes holds for the pair $(\YY,\XX)$, then by
Lemma~\ref{lem:StokesFibre} it holds for the pair 
$(\UU,\UU_\XX)=(h^*\XX_T,,h_\XX^*\XX_T)$ and
equation~\eqref{eq:q-reg-dec} describes quasi-regular boundary 
of $\UU_\XX$. Again, a connected component of the quasi-regular
boundary of $\UU_\XX$ is generated 
by the quasi-regular boundary of either $\XX$ or $\XX_T$.

We now specialize this to the following setup.
Let $\Gamma_-$ be a graph as in~\S\ref{ss:graphs} and $T$ be a rooted 
tree. 
Let $h$ be one of the leaf flags of 
$\Gamma_-$. We attach the 
tree $T$ to $\Gamma_-$ by identifying the flag 
$h$ with the root flag of $T$ to get the graph $\Gamma$.
Let $(\YY_T,\XX_T)$ be the basic pair 
associated by~\eqref{eq:graphtopair} to $T$ 
, and 
let $(\YY_{\Gamma_-},\XX_{\Gamma_-})$ be the basic pair associated to
$\Gamma_-$ with some pair $(M^d\times \WW,\ZZ)$.
We apply the above procedure of ``attaching a tree to 
a compactification'' with $(\YY,\XX):=(\YY_{\Gamma_-},\XX_{\Gamma_-})$
and the evaluation map $h:\YY_{\Gamma_-}\to M$ at the leaf flag $h$
(denoting it by the same letter by a severe abuse of notation).
Let $\Gamma_-^{-h}$ denote the graph $\Gamma_-$ with the leaf flag $h$
removed from the set of flags. 

\begin{lemma}\label{lem:tree-attach}
In the situation above we have
\begin{equation*}
(\YY_\Gamma,\XX_\Gamma)=(\YY_{\Gamma_-}
\times_{(h,\pi_{root})} \YY_T,\, h|_{\XX_{\Gamma_-}}^*\XX_T),
\end{equation*}
where the first space on the right hand side is the 
fibre product of $\YY_{\Gamma_-}$ and $\YY_T$.
\end{lemma}

\begin{proof}
First, note that the map $\pi_{root}$ extends to the projection onto the
root flag 
$$
  \pi_{root}:\YY_T\longrightarrow M
$$
that we denote by the same letter. We have natural inclusions
$$
  \iota_-:\Flag(\Gamma_-)\into \Flag(\Gamma),\qquad
  \iota_T:\Flag(T)\into \Flag(\Gamma). 
$$
Note that these maps respect edges in the following sense: two flags
in the image of $\iota_-$ form an edge if and only if their preimages
in $\Flag(\Gamma_-)$ form an edge, and the same holds for the map
$\iota_T$. Thinking of each edge as parameterizing a copy of  
$\wt M^2$, the above maps on flags induce the horizontal maps in the
following commutative diagrams (the vertical maps are the natural
blow-down maps):  
\begin{equation}
\begin{tikzcd}
\YY_\Gamma\arrow[rr, "\wt\phi_-"] \arrow[d, ] &  & 
\YY_{\Gamma_-}\arrow[d, ] \\
X_\Gamma\times \WW \arrow[rr, "\phi_-"]               &  & X_{\Gamma_-}\times \WW,               
\end{tikzcd}
\qquad
\begin{tikzcd}
\YY_\Gamma\arrow[rr, "\wt\phi_T"] \arrow[d, ] &  & 
\YY_T\arrow[d, ] \\
X_\Gamma \arrow[rr, "\phi_T"]               &  & X_T.               
\end{tikzcd}
\end{equation}
The map $\wt\phi_T$ forgets the $\WW$ factor and some $\wt M^2$
factors and tautologically identifies the other $\wt M^2$ factors. The
map $\wt\phi_-$ is the identity on the $\W$ factor, and its component
landing in the $M$ factor of $\YY_{\Gamma_-}$ corresponding to the flag
$h$ is simply the map $h$. On all the other factors it tautologically
identifies exactly the copies of $\wt M^2$ forgotten by
$\wt\phi_T$. In particular,
the map $\wt\phi_-\times\wt\phi_T$ is proper.

The image of the flag $h$ under 
$\iota_-$ is the same as the image of the root 
flag of $T$ under $\iota_T$. This leads to the 
fibre product description of the image of 
the map $\wt\phi_-\times \wt\phi_T$:
$$
\im(\wt\phi_-\times \wt\phi_T)=\YY_{\Gamma_-}
\times_{(h,\pi_{root})} \YY_T:=
\{(p,q)\in \YY_{\Gamma_-}\times \YY_T\mid h(p)=\pi_{root}(q)\}.
$$
Note that the factors forgotten by $\wt\phi_-$ are picked up by
$\wt\phi_T$ and vice versa. Therefore, the product map  
$\wt\phi_-\times \wt\phi_T$ is an injective immersion. Together with
properness this implies that it is an embedding. We identify
$\YY_\Gamma$ with its diffeomorphic image under this map. 

Observe now that the maps $\iota_-$ and $\iota_T$ respect vertices: 
two flags in the image of $\iota_-$ are adjacent to the same vertex if
and only if their preimages in $\Flag(\Gamma_-)$ are, and the same
holds for the map $\iota_T$. Therefore, the map $\phi_-\times \phi_T$
restricts to a map (denoted by the same letter)
$$
\phi_-\times \phi_T:\Delta_{vert\Gamma}\times \ZZ\longrightarrow 
(\Delta_{vert\Gamma_-}\times \ZZ)\times 
\Delta_{vert T}.
$$
Recall that both $\iota_-$ and $\iota_T$ respect edges. This gives us
the further restriction to the complement of the respective blow-up
loci: 
$$
\phi_-\times \phi_T:\Delta_{vert\Gamma}\times \ZZ
\setminus \Delta_2^\Gamma\times \WW
\longrightarrow 
\left((\Delta_{vert\Gamma_-}\times \ZZ)
\setminus (\Delta_{2\Gamma_-}\times \WW)\right) 
\times (\Delta_{vert T}\setminus \Delta_{2T}).
$$
By definition the closure of the domain of this map in $\YY_\Gamma$ is
exactly $\XX_\Gamma$. On the other hand, the image is contained in 
the fibre product $h|_{\XX_{\Gamma_-}}^*\XX_T$ of $\XX_\Gamma$ and $\XX_T$.
Identifying subsets with their diffeomorphic images under the map
$\wt\phi_-\times \wt\phi_T$, we get
$$
  \XX_\Gamma\subset h|_{\XX_{\Gamma_-}}^*\XX_T.
$$
For the converse inclusion we produce an approximating sequence (see
Definition~\ref{def:approx-seq}) for any  
$$
(p,q)\in h|_{\XX_{\Gamma_-}}^*\XX_T\subset \XX\times \XX_T.
$$
For this, let $(p_n)$ be an approximating sequence for $p$. Note that
$h(p_n)\to h(p)$ as $n\to\infty$. The fibre bundle structure for 
$\pi_{root}$ allows us to pick an approximating sequence $(q_n)$ for
$q$ with the additional property that $\pi_{root}(q_n)=h(p_n)$. Then  
$(p_m,q_m)$ is the desired approximating sequence for $(p,q)$. Therefore,
$$
\XX_\Gamma=h|_{\XX_{\Gamma_-}}^*\XX_T.
$$
\end{proof}

Recall from Lemma~\ref{lem:tree-struct} that $\XX_T\subset\YY_T$
is a nice submanifold (with corners). It follows that 
$\YY_{\Gamma_-}\times_{(h,\pi_{root})}\XX_T \subset \YY_{\Gamma_-}
\times_{(h,\pi_{root})} \YY_T$ is also a nice submanifold. 
Suppose now that Stokes holds for the pair $(\YY_{\Gamma_-},\XX_{\Gamma_-})$. 
Then by Lemma~\ref{lem:StokesFibre} it holds for the pair 
$(\YY_{\Gamma_-}\times_{(h,\pi_{root})} \XX_T,\, h|_{\XX_{\Gamma_-}}^*\XX_T)$,
and thus by Remark~\ref{rem:local}(b) also for the pair
$(\YY_\Gamma,\XX_\Gamma)=(\YY_{\Gamma_-}\times_{(h,\pi_{root})}
\YY_T,\, h|_{\XX_{\Gamma_-}}^*\XX_T)$ in Lemma~\ref{lem:tree-attach}.

Moreover, equation~\eqref{eq:q-reg-dec} describes quasi-regular boundary 
of $\XX_\Gamma$: 
\begin{itemize}
\item [(GEN2)] A connected component of the quasi-regular boundary of 
$\XX_\Gamma$ is generated by the quasi-regular boundary of either 
$\XX_{\Gamma_-}$ or $\XX_T$.
\end{itemize}
To clarify the last statement, recall that a point in $h|_{\XX_{\Gamma_-}}^*\XX_T$
is a pair $(p,q)\in \XX_{\Gamma_-}\times\XX_T$ satisfying 
$h(p)=\pi_{root}(q)$. Equation~\eqref{eq:q-reg-dec} says that $(p,q)$
belongs to the quasi-regular boundary of $h|_{\XX_{\Gamma_-}}^*\XX_T$
if and only if either $p$ belongs to the quasi-regular boundary of
$\XX_{\Gamma_-}$ and $q$ to the interior of $\XX_T$, or the
other way around.

\begin{remark}\label{rem:hidden-attach}
In the setting above, observe that
$$
  \p\YY_\Gamma = (\p_1\YY_{\Gamma_-})\times_{(h,\pi_{root})}(\YY_T)_0
  \amalg (\YY_{\Gamma_-})_0\times_{(h,\pi_{root})}\p_1\YY_T.
$$
Recall Definition~\ref{def:hidden} of the hidden part of the boundary,
and note that $\p^\hidden\XX_T=\emptyset$ because the embedding
$\XX_T\into \YY_T$ is nice. Hence, using (GEN2) we conclude the implication
\begin{equation}\label{eq:hidden-attach}
\p^\hidden\XX_{\Gamma_-}=\emptyset \implies \p^\hidden\XX_\Gamma=\emptyset.
\end{equation}
\end{remark}

{\bf Chopping off trees. }
Above we have discussed the attaching of trees. Now we will discuss
the converse operation. Let $\Gamma$ be a graph as in~\S\ref{ss:graphs}. 

\begin{definition}\label{def:subtree}
A {\em rooted subtree} $(T,r)$ is a connected subtree $T\subset\Gamma$
with a root vertex $r\in T$ such that the following holds:
\begin{enumerate}
\item $T$ contains at least one edge.
\item $T$ contains no special vertices of $\Gamma$ except possibly $r$;
\item if $T$ contains one flag of an edge $l$ of $\Gamma$, then it also
  contains the other flag of $l$;
\item if $T$ contains one flag of a vertex $v\neq r$ of $\Gamma$, then it also
  contains all other flags of $v$;
\item $(T,r)$ is maximal with properties (i)--(iii). 
\end{enumerate}
\end{definition}

\begin{remark}
Thinking of a graph as a topological space, conditions (iii) and (iv)
say that $T\subset\Gamma$ is closed and $T\setminus\{r\}\subset\Gamma$
is open. Together with the maximality condition (v), this implies
that two distinct rooted subtrees can meet at most at their root
vertices. Hence, all rooted subtrees of $\Gamma$ form a finite
collection $\{(T_i,r_i)\}$ meeting at most at their root vertices. 
\end{remark}

Let $(T,r)$ be a rooted subtree of $\Gamma$. We define the graph
$\Gamma_-$ obtained by {\em chopping off $(T,r)$ from $\Gamma$}
as the graph obtained by removing from $\Gamma$ all flags of $T$
except the root flag (which becomes a leaf of $\Gamma_-$).
Thus attaching the rooted tree $T$ to $\Gamma_-$ at this leaf gives
back $\Gamma$. 
We denote by
\begin{equation}\label{eq:Gamma-cyc}
  \Gamma_\cyc\subset\Gamma
\end{equation}
the unique subgraph obtained
from $\Gamma$ by first chopping off all rooted subtrees, and then
removing all nonspecial leaves from the resulting graph. See
Figure~\ref{fig:chopping-trees}.
Note that $\Gamma_\cyc$ contains all cycles of $\Gamma$, as well as all
special vertices with their adjacent flags. 
\begin{figure}
\begin{center}
\includegraphics[width=\textwidth]{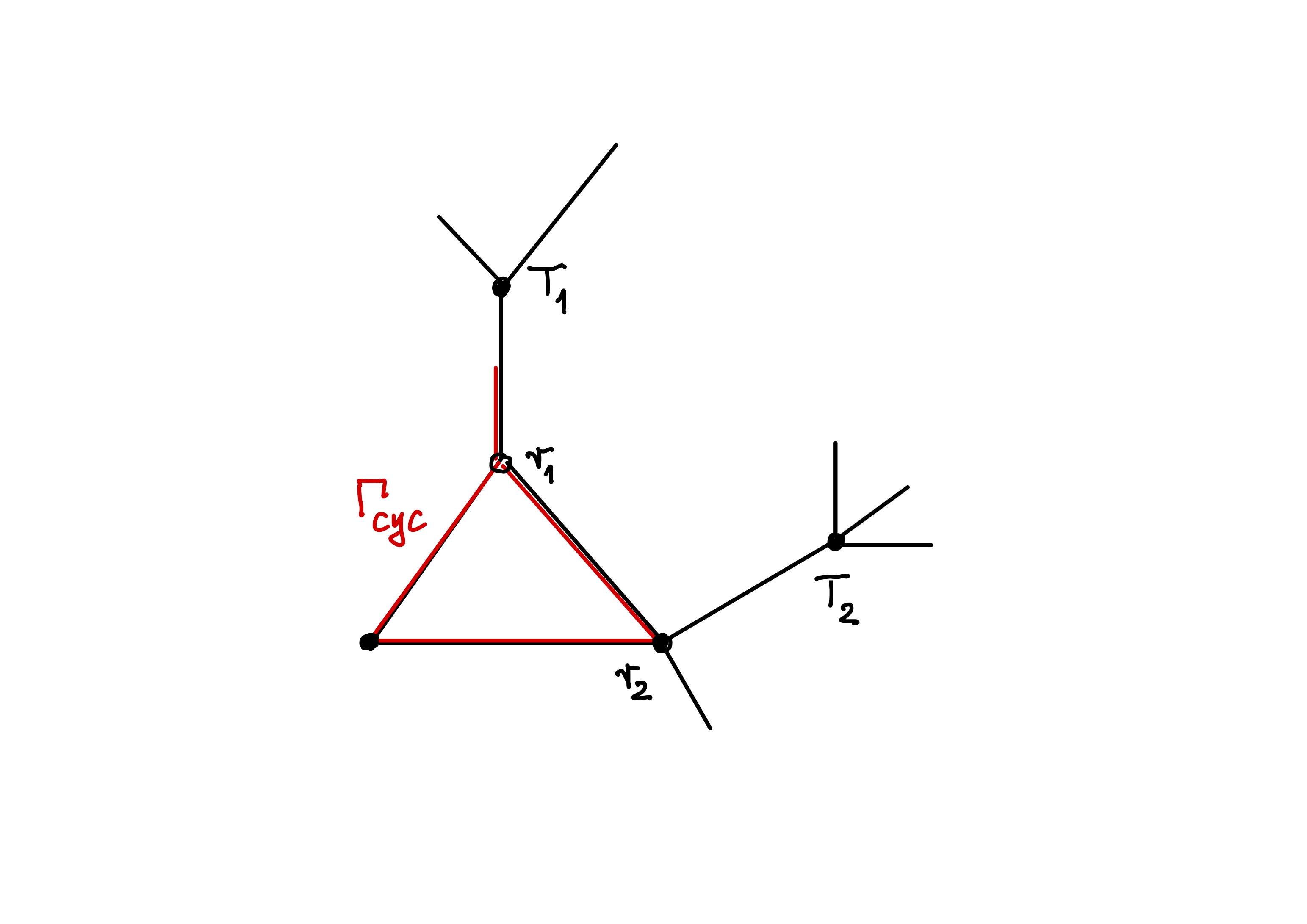}
\vspace{-2cm}
\caption{A graph $\Gamma$ and its subgraph $\Gamma_\cyc$. Here
  $(T_i,r_i)$ are the rooted trees of $\Gamma$, with $r_1$ special and
$r_2$ nonspecial.}
\label{fig:chopping-trees} 
\end{center}
\end{figure}

\begin{definition}\label{def:doublyspec}
An edge $l$ of $\Gamma$ is called {\em doubly special} if both of its
flags are special.
\end{definition}

\begin{proposition}\label{prop:reg-char}
Let $(\YY_\Gamma,\XX_\Gamma)$ be the basic pair associated
by~\eqref{eq:graphtopair} to a graph $\Gamma$ and a pair
$(M^d\times\WW,\ZZ)$.   
Assume that $\ZZ\subset M^d\times\WW$ is an analytic subset. 
Let $J$ be a subset of the set of edges of $\Gamma$ such that the
subset $\p_J^{q-reg}\XX_\Gamma\subset\p^{q-reg}\XX_\Gamma$ 
defined by~\eqref{eq:q-reg-J} is nonempty. Then the following holds.

(A) If the graph $\Gamma_J$ is disconnected, then every connected
component of $\Gamma_J$ contains a special vertex of $\Gamma$.  

(B) If $J$ contains more than one edge, then the graph $\Gamma_J$ 
has no rooted trees.

(C) Assume that $J$ consists of just one edge $l$ which is not doubly
special. Then  
$$
\p_l\XX:=\p_J^{q-reg}\XX_\Gamma
$$ 
is a primary face of the regular boundary of $\XX_\Gamma$. Moreover, 
$\p_J^{q-reg}\XX_\Gamma$ is an $S^{n-1}$-fibration over 
\begin{equation}\label{eq:whdelta2}
\wh \Delta_2^l:=(\Delta_{\ver}\times \ZZ_0)
\cap (X_{\{l\}}\times \WW_0).
\end{equation}
\end{proposition}

\begin{remark}\label{rem:trans}
The intersection in~\eqref{eq:whdelta2} is transverse in
  $X_\Gamma\times \WW_0$. To see this, 
assume that one flag of $l$ is special and the other nonspecial 
(if both flags are nonspecial the argument is similar but simpler).
To avoid unnecessarily cumbersome notation, assume that the graph
$\Gamma$ has only two vertices, one being special and univalent and
the other being nonspecial and $k$-valent. In this case 
$$
  X_\Gamma\times \WW_0\cong M^k\times M\times \WW_0.
$$
Let the variables $(x_1,\dots,x_k)$ parametrize $M^k$, the variable
$u$ parametrize the copy of $M$ corresponding to the special flag, and
$w$ parametrize $\WW_0$. Assume without loss of generality that the
double diagonal corresponding to $l$ is described by $\{u=x_1\}$. 
With this notation the main players are described as follows: 
\begin{equation}\label{eq:descr-trans}
\begin{cases}
  \Delta_\ver\times \ZZ_0=\{x_1=\dots=x_k,\ (u,w)\in\ZZ_0\},\cr
  X_{\{l\}}\times \WW_0=\{x_1=u,\ w\in \WW_0\},\cr
  \wh\Delta_2^l=\{x_1=\dots=x_k=u,\ (u,w)\in \ZZ_0\}\cong \ZZ_0.
\end{cases}
\end{equation}
For any $p=(x^*,\dots,x^*;x^*,w^*)\in \wh\Delta_2^l$ pick a chart in
$M$ centered at $x^*$ and a chart in $\WW_0$ centered at $w^*$ and
take the product chart for $X_\Gamma\times \WW_0$ centered at $p$. 
In this chart description~\eqref{eq:descr-trans}
yields the following relation between tangent spaces:
$$
T_p(\Delta_\ver\times \ZZ_0)\cap 
T_p(X_{\{l\}}\times \WW_0)=T_p(\wh\Delta_2^l).
$$
Transversality now follows from the dimension computation
\begin{align*}
&\dim(\Delta_\ver\times \ZZ_0)+\dim(X_{\{l\}}\times \WW_0)-\dim \wh \Delta_2^l\cr
=&(n+\dim\ZZ_0)+(nk+\dim\WW_0)-\dim\ZZ_0\cr
=&n(k+1)+\dim\WW_0\cr
=&\dim (X_\Gamma\times \WW_0).
\end{align*}
\end{remark}

\begin{remark}\label{rem:trans-DS}
For $l$ doubly special, the intersection~\eqref{eq:whdelta2}
is in general not transverse. It is transverse under the assumption
that $\ZZ_0$ is transverse to $D_l\times \WW_0$ in $M^d\times \WW_0$,
where $D_l\subset M^d$ is the diagonal corresponding to $l$. 
To see this, note that then $\Delta_\ver\times \ZZ_0$ is transverse to
$\Delta_\ver\times D_l\times \WW_0$ in $\Delta_\ver\times M^d\times \WW_0$.
Since $\prod_{j\in\Ver}M^{d_j}\times D_l=X_{\{l\}}$ and
$\prod_{j\in\Ver}M^{d_j}\times M^d=X_\Gamma$, composition with the
natural embedding $\iota_\ver:\Delta_\ver\into
\prod_{j\in\Ver}M^{d_j}$ from~\eqref{eq:vertexdiag} and the identity
on the other factors yields the desired transversality.
\end{remark}

\begin{remark}
If the graph $\Gamma$ contains at most one special vertex, then
assertion (A) implies connectedness of $\Gamma_J$.  
\end{remark}

\begin{proof}[Proof of Proposition~\ref{prop:reg-char}]
By Lemma~\ref{lem:Stokes-an}, analyticity of $\ZZ$ implies that Stokes
holds for all the graphs mentioned below. We borrow the notation from
the proof of Lemma~\ref{lem:flip}, replacing $\XX$ by $\XX_\Gamma$.    
Recall the neighbourhood $U$, the pair $(\pi^{-1}(U), \XX_{\Gamma U})$,
and the fact that the intersection $\p_J^{q-reg}\XX_\Gamma\cap
\pi^{-1}(U)$ is a union of connected components of the quasi-regular
boundary $\p^{q-reg}(\XX_{\Gamma U})$ of 
$\XX_{\Gamma U}$. Moreover, according to Remark~\ref{rem:transfer}
the pair $(\pi^{-1}(U),\XX_{\Gamma U})$ can be seen
as a basic pair associated as in~\eqref{eq:graphtopair} to the 
graph $\Gamma_J'$, where $\Gamma_J'$ is the result of cutting 
in $\Gamma$ all the edges in the complement of $J$. 

For (A), suppose that $\Gamma_J$ is disconnected and some
connected component $\Gamma_1$ of $\Gamma_J$ contains no special
vertices of $\Gamma$. Then the graph $\Gamma_J'$ can be represented as  
a disjoint union
$$
\Gamma_J'=\Gamma_1'\amalg\Gamma_2',
$$
where $\Gamma_1'$ contains no special vertices and both graphs have a
positive number of edges from $J$.  
Since $\Gamma_1'$ contains no special vertices, by associating
the pair $(M^d\times\WW,\ZZ)$ to $\Gamma_2'$ (and $\ZZ=\WW=\pt$ to $\Gamma_1'$)
we find ourselves in the situation of Application 1 above.
By property (GEN1) above, each component of the quasi-regular boundary
$\p_J^{q-reg}\XX_\Gamma\cap \pi^{-1}(U)$
corresponds to the quasi-regular boundary of either $\XX_{\Gamma_1'}$
or $\XX_{\Gamma_2'}$ (the other factor being interior). But this
contradicts the fact that $\p_J^{q-reg}\XX_\Gamma\cap \pi^{-1}(U)$
sits over the edge diagonals of all edges in $J$, which includes edges
from both $\Gamma_1'$ and $\Gamma_2'$.

For (B), suppose that $J$ contains more than one edge and $\Gamma_J$
is the result of attaching a rooted tree $T$ to some graph
$\Gamma_1$. Since $J$ contains more than one edge, after possibly
moving an edge from $T$ to $\Gamma_1$ we may assume that both
$\Gamma_1$ and $T$ have at least one edge. 
Then the graph $\Gamma_J'$ is the result of attaching a rooted tree
$T'$ to a graph $\Gamma_1'$ with at least one edge.
Since $T'$ contains no special vertices, by associating
the pair $(M^d\times\WW,\ZZ)$ to $\Gamma_1'$ (and $\ZZ=\WW=\pt$ to $T'$)
we find ourselves in the situation of Application 2 above.
By property (GEN2) above, each component of the quasi-regular boundary
$\p_J^{q-reg}\XX_\Gamma\cap \pi^{-1}(U)$
corresponds to the quasi-regular boundary of either $\XX_{\Gamma_1'}$
or $\XX_{T'}$ (the other factor being interior). But this
contradicts the fact that $\p_J^{q-reg}\XX_\Gamma\cap \pi^{-1}(U)$
sits over the edge diagonals of all edges in $J$, which includes edges
from both $\Gamma_1'$ and $T'$. 

For (C), assume that $J=\{l\}$ and $l$ contains exactly one special
flag. (The case where both flags of $l$ are nonspecial is similar but
simpler.) We repeat the construction and borrow the notation from case (B)
above. The difference is that now the tree $T$ consists of only the edge
$l$ and its adjacent vertices, and the graph $\Gamma_1$ consists of
only one (special) vertex.  
Thus the graph $\Gamma_J'$ is the result of attaching a rooted tree
$T'$ with exactly one edge $l$ to a graph $\Gamma_1'$ without edges.
Again, by associating
the pair $(M^d\times\WW,\ZZ)$ to $\Gamma_1'$ (and $\ZZ=\WW=\pt$ to $T'$)
we find ourselves in the situation of Application 2, so 
By property (GEN2) each component of the quasi-regular boundary
$\p_J^{q-reg}\XX_\Gamma\cap \pi^{-1}(U)$
corresponds to the quasi-regular boundary of either $\XX_{\Gamma_1'}$
or $\XX_{T'}$ (the other factor being interior). 
Since $\p_J^{q-reg}\XX_\Gamma\cap \pi^{-1}(U)$
sits over the edge diagonal of the edge $l$ of $T$, only the case of
the quasi-regular boundary of $\XX_{T'}$ and the interior of
$\XX_{\Gamma_1'}$ occurs. 
Thus, $\p_J^{q-reg}\XX_\Gamma$ corresponds to the second 
summand on the right hand side of~\eqref{eq:q-reg-dec}, which is a
smooth fibration over $\XX_0$ with fibre $\p_1F$. 
Here the base $\XX_0$ is the interior $\XX_{\Gamma_1'0}$, which is
naturally diffeomorphic to $\wh\Delta_l^2$. The space $F$
is $\R^n$ blown up at a point (corresponding to the nonspecial vertex
of $l$), so $\p_1F$ is diffeomorphic to $S^{n-1}$. 
\end{proof}

\begin{corollary}\label{cor:noDS}
In the setting of Proposition~\ref{prop:reg-char}, assume in addition
that $\ZZ\subset M^d\times\WW$ is a nice submanifold and the graph
$\Gamma$ contains no doubly special edges. Then  
$$
  \p^\hidden\XX_\Gamma=\coprod_{J\in \EE}\p_J^{q-reg}\XX_\Gamma,
$$
where $\EE$ is some (possibly empty) collection of subsets of $\Edge$
each having at least two elements. 
\qed
\end{corollary}

\begin{proof}
Consider a subset $J\subset\Edge$ such that $\p_J^{q-reg}\XX_\Gamma$
contains a hidden face. Since $\ZZ$ is a nice submanifold, its
boundary has no hidden faces, and thus $J$ cannot be empty. 
By Proposition~\ref{prop:reg-char}(C), $J$ cannot consist of one
element. 
\end{proof}

For a graph $\Gamma$, recall the definition of $\Gamma_\cyc$
from~\eqref{eq:Gamma-cyc}. 

\begin{definition}\label{def:cancel}
We say that a graph $\Gamma$ {\em admits cancellation} if for 
every subgraph $\Gamma'\subset\Gamma$
the associated graph $\Gamma'_{\cyc}$ has either no edges or at least
one $2$-valent nonspecial vertex as in Figure~\ref{fig:hidden}. 
\end{definition}

\begin{remark}\label{rem:doubly-special-cancel}
If $\Gamma$ admits cancellation then it has no doubly special edges.
(If $l$ is a doubly special edge, then the subgraph $\Gamma'$
consisting of $l$ and its adjacent vertices satisfies
$\Gamma'_\cyc=\Gamma'$ and violates Definition~\ref{def:cancel}).
\end{remark}

The following example of such a graph is central in~\cite{Cieliebak-Volkov}.

\begin{lemma}\label{lem:trivalent-cancel}
Let $\Gamma$ be a connected trivalent ribbon graph without special
vertices. Then $\Gamma$ admits cancellation if and only if it has at
least one leaf. 
\end{lemma}

\begin{proof}
Since $\Gamma$ has no special vertices, we can omit the specification
``nonspecial'' in the subsequent discussion. 

Assume first that $\Gamma$ does not admit cancellation. 
Then there exists a subgraph $\Gamma'\subset\Gamma$
with at least one edge and without $2$-valent vertices. Since each
vertex of $\Gamma'_{\cyc}$ is
at most trivalent (since $\Gamma'_{\cyc}$ is a subgraph of $\Gamma$)
and not univalent (by definition of $\Gamma'_{\cyc}$ in terms of
chopping off trees, using that it has at least one edge), this implies
that $\Gamma'_{\cyc}$ is trivalent. Since  
$\Gamma$ is trivalent and connected, the inclusion
$\Gamma'_{\cyc}\subset \Gamma$ cannot be strict and we get
$\Gamma'_{\cyc}=\Gamma$. By definition of 
$\Gamma'_{\cyc}$, this implies
that $\Gamma$ has no leaves.

Conversely, assume that $\Gamma$ has no leaves. Then it has no rooted
trees (because each rooted tree must contain a leaf or a univalent
vertex), and therefore $\Gamma_\cyc=\Gamma$. Since $\Gamma$ is
trivalent, it follows that $\Gamma_\cyc$ has at least one edge and no
$2$-valent vertices, thus (with $\Gamma'=\Gamma$) the graph $\Gamma$
does not admit cancellation. 
\end{proof}

The following example will be crucial for the discussion of coproducts
below. We define a {\em circular graph} as a graph $\Gamma$ whose ribbon surface 
$\Sigma_\Gamma$ is an annulus. It has a unique embedded cycle such
that $\Gamma$ is obtained by attaching trees to this cycle.

\begin{lemma}\label{lem:ex-circ}
Let $\Gamma$ be a circular graph with one special vertex. Then $\Gamma$
admits cancellation if and only if its cycle consists of more than one edge.
\end{lemma}

\begin{proof}
If the cycle of $\Gamma$ consists of just one edge, then this edge is
a self-loop based at a certain vertex (special or nonspecial). 
Choose the subgraph $\Gamma'\subset\Gamma$ in
Definition~\ref{def:cancel} to be the self-loop. Then $\Gamma'$ does
not contain any rooted subtrees in the sense of
Definition~\ref{def:subtree}. Thus, $\Gamma'_\cyc=\Gamma'$ consists of
just the above self-loop, violating Definition~\ref{def:cancel}.

Assume now that the cycle of $\Gamma$ contains more than one
edge. Then we have two possibilities for the subgraph $\Gamma'$.
The first possibility is that $\Gamma'$ does not contain the cycle of
$\Gamma$. Then the graph $\Gamma'_\cyc$ has no edges.
The second possibility is that $\Gamma'$ contains the cycle of
$\Gamma$. Then $\Gamma'_\cyc$ contains exactly the edges forming the
cycle of $\Gamma$. Since the cycle has at least two edges and at most
one special vertex, it must contain a $2$-valent nonspecial vertex. 
Hence, $\Gamma$ admits cancellation by Definition~\ref{def:cancel}.
\end{proof}

Now we are ready to formulate the main result of this section.

\begin{proposition}[Vanishing for hidden faces]\label{prop:cancel}
In the setting of Proposition~\ref{prop:reg-char}, assume in addition
that $\ZZ\subset M^d\times\WW$ is a nice submanifold and $\Gamma$
admits cancellation. Then for every 
$\eta\in \Om^{n-1}(\wt M^2)$ 
and compactly supported $\alpha\in \Om^*(M^s\times\WW)$ we have
$$
\int_{\p^\hidden\XX}\eta^e(\alpha)=0.
$$
\end{proposition}

\begin{proof}
By Remark~\ref{rem:doubly-special-cancel}, the graph $\Gamma$ has no
doubly special edges. Therefore, by Corollary~\ref{cor:noDS}, each
subset $J\subset\Edge$ for which $\p_J^{q-reg}\XX$ contains a hidden
face has at least $2$ elements. 
Proposition~\ref{prop:reg-char} (B) implies that $\Gamma_J$ has no
rooted subtrees, so $\Gamma_J$ is the union of $\Gamma_{J \cyc}$ and
all flags at special vertices of $\Gamma_{J \cyc}$. In particular, 
$\Gamma_{J \cyc}$ has at least two edges. 
Since $\Gamma$ admits cancellation, it follows that $\Gamma_{J \cyc}$
has a $2$-valent nonspecial vertex $v$. Then $v$ is also a
$2$-valent nonspecial vertex of $\Gamma_J$, and
Corollary~\ref{cor:cancelbasic} yields vanishing of the integral.
\end{proof}

\begin{remark}\label{rem:vanish-extend-fin}
Proposition\ref{prop:cancel} holds true for the larger class of forms
described in Remark~\ref{rem:vanish-extend}.  
\end{remark}

\begin{remark}
In the setting of Proposition~\ref{prop:reg-char}, assume in addition
that the restriction of the natural blow-down map to $\p_J^{q-reg}\XX$
is a fibre bundle projection. This is true, for example, for the  
Fulton-MacPherson compactification of the configuration space of $m$
distinct points on $M$ introduced in~\cite{Fulton-MacPherson}.  
In this case, the cancellation procedure in the proof
of~\cite[Lemma~19]{Campos-Willwacher} yields the vanishing 
result~\eqref{eq:cancellation} if all nonspecial vertices of
$\Gamma_J$ are at least $3$-valent and $\dim M\ge 3$.   
\end{remark}

\section{Stokes' theorem in the string topology setting}\label{sec:analysis}

In this section we apply the results of the previous section to the
configuration spaces that are relevant for the string topology
operations and the proof of the main theorem. 

\subsection{Configuration spaces relevant for string topology}\label{ss:general-setting}

We work in the setting of~\S\ref{sec:general graphs}. Thus
\begin{itemize}
\item $M$ is a closed oriented $n$-dimensional manifold;
\item $\Lambda=C^\infty(S^1,M)$ is its free loop space, where we
write $S^1=\R/\Z$;
\item $\Gamma$ is a graph with some special vertices and $d$ special flags;
\item $X_\Gamma$ is the associated configuration space with its subset
$\Delta_\ver\times M^d$;
\item $\wt X_\Gamma = \Bl(X_\Gamma,\Delta_2^\Gamma)$ is its blow-up along
the edge diagonal;
\item $\WW$ is a compact manifold with corners.
\end{itemize}
For $\ZZ$ we make the following more specific choice. We endow $\WW$
and $M$ with analytic structures and consider an analytic map 
$$
  \phi:\WW\longrightarrow M^d.
$$
We set 
\begin{equation}\label{eq:gen-set}
  \ZZ:=gr(\phi):=\{(\phi(w),w)\mid w\in \WW\}\subset M^d\times \WW.
\end{equation}
The basic pair~\eqref{eq:graphtopair} associated to this data is then
\begin{equation}\label{eq:graphtopair2}
  (\YY,\XX) = \Bigl(\wt X_\Gamma\times \WW,PT(\Delta_{\ver}\times gr(\phi))\Bigr).
\end{equation}
Note that $\ZZ$ is a nice analytic submanifold of $\WW$. Therefore, 
by Lemma~\ref{lem:Stokes-an}, Stokes' theorem holds 
for the pair $(\YY,\XX)$. Moreover, the vanishing results
of~\S\ref{ss:chopping} apply whenever the graph $\Gamma$ admits
cancellation in the sense of Definition~\ref{def:cancel}.

We assume that the graph $\Gamma$ does
not have self-loops at nonspecial vertices.
Then for each $l\in \Edge$ the corresponding double diagonal
$\Delta_2^l$ is transverse to $\Delta_\ver\times M^d$ in $X_\Gamma$. Set 
$$
  \ol{\Delta}^l_{2}:=\Delta_2^l\cap (\Delta_\ver\times M^d),\qquad
  \ol\Delta_2^\Gamma:=\bigcup_{l\in \Edge}\ol{\Delta}^l_{2}.
$$
Then $\{\ol{\Delta}^l_{2}\}_{l\in\Edge}$ is a (not necessarily
transverse) collection of nice submanifolds of positive codimension in  
$\Delta_\ver\times M^d$. In particular, their union
$\ol\Delta_2^\Gamma$ has measure zero in $\Delta_\ver\times M^d$.
By a slight abuse of language, we will sometimes call
$\ol\Delta_2^\Gamma$ a collection of nice submanifolds.

We impose the following condition on the map $\phi$:
\begin{itemize}
\item [(MES)] The intersection
$$
\LL_\phi:=(\Delta_\ver\times gr(\phi))
\cap (\Delta_2^\Gamma\times \WW)
$$
has measure zero in $\Delta_\ver\times gr(\phi)$.
\end{itemize}

\begin{remark}
If $\WW=\{\pt\}$ and the graph $\Gamma$ has no special vertices, then
the absence of self-loops at nonspecial vertices implies that
$\Delta_\ver$ is transverse to each member of the family
$\Delta_2^\Gamma$ and condition (MES) follows.
In general, the absence of self-loops at nonspecial vertices is not
eneough to ensure condition (MES). Consider for example the graph
$\Gamma$ with one special vertex, no nonspecial vertices and one
(doubly special) edge. Let $f:B\to \Lambda$ be a constant simplex at a
constant loop and $\phi:=e_f:B\times [0,1]\to M^2$ be its constant
evaluation map. Here $\WW=B\times [0,1]$ and $\Delta_\ver=\{\pt\}$. 
Then the image of $e_f$ is a subset of the diagonal in $M^2$, and thus
$gr(\phi)\cap (\Delta_2^\Gamma\times \WW)=gr(\phi)$.
This example shows that condition (MES) imposes certain restrictions on
the map $\phi$; see Remark~\ref{rem:AT-MES} for an example of how 
condition (MES) is typically satisfied. 
\end{remark}

Consider the following commutative diagram:
\begin{equation}\label{eq:commdiag}
\begin{gathered}
  \xymatrix{
  \Delta_\ver\times M^d\setminus \ol\Delta_2^\Gamma \ar[r]^{\iota_1}
  & \wt X_\Gamma\setminus \wt\Delta_2^\Gamma \\ 
  & \wt X_\Gamma\times \WW\setminus (\wt\Delta_2^\Gamma\times\WW) \ar[u]^{p} \\
  \Delta_\ver\times\WW\setminus (\iota^\phi)^{-1}(\LL_\phi)
  \ar[uu]^{\id\times\phi} \ar[r]^{\iota^\phi} & \Delta_\ver\times
  gr(\phi)\setminus \LL_\phi\,. \ar[u]^{\iota_2} 
}  
\end{gathered}
\end{equation}
Here $\wt\Delta_2^\Gamma$ is the preimage of $\Delta_2^\Gamma$ under
the blow-down map $\wt X_\Gamma\to X_\Gamma$,
$p$ is the natural projection forgetting the $\WW$ factor,
$\iota_1$ and $\iota_2$ are the natural inclusions away of the blow-up
locus, and $\iota^\phi$ is the canonical diffeomorphism 
\begin{equation}\label{eq:canphi}
\iota^\phi:\Delta_\ver\times \WW
\stackrel{\cong}{\longrightarrow}
\Delta_\ver\times gr(\phi) 
,\qquad
(x,q)\mapsto (x,\phi(q),q).
\end{equation}
Consider now a smooth form $\om\in\Om^*(\wt X_\Gamma)$. We denote
\begin{equation}\label{eq:restr-om}
   \om|_{\Delta_{\ver}\times M^d} := \iota_1^*\om,
\end{equation}
viewed as a measurable form on $\Delta_{\ver}\times M^d$ by extending
it by $0$ over the measure zero subset $\ol\Delta_2^\Gamma$. 

\begin{lemma}\label{lem:graph-L1}
In the setting above, the measurable forms 
$\om|_{\Delta_{\ver}\times M^d}$ on $\Delta_{\ver}\times M^d$
and $(\id\times\phi)^*(\om|_{\Delta_{\ver}\times M^d})$ on
$\Delta_{\ver}\times \WW$ are integrable. 
\end{lemma}

\begin{proof}
Let us show that $(\id\times\phi)^*(\om|_{\Delta_{\ver}\times
  M^d})=(\id\times\phi)^*\iota_1^*\om$ is integrable.
Since $\iota^\phi$ is a diffeomorphism, by the diagram above this is
equivalent to showing integrability of $\iota_2^*p^*\om$, viewed as a
measurable form on $\Delta_\ver\times gr(\phi)$ by extending it by $0$
over the measure zero subset $\LL_\phi$. Let $\gamma$ be any smooth test 
form on $\Delta_\ver\times gr(\phi)$. We extend it to
$X_\Gamma\times\WW$, still denoting it by the same letter $\gamma$. We
need to show integrability of the form $\iota_2^*p^*\om\wedge\gamma$ over  
$\Delta_\ver\times gr(\phi)$. For this, we use the basic pair
$(\YY,\XX)$ above. 
Identifying the interior $\XX_0\subset\XX$
with its diffeomorphic image under the blow-down map $\wt
X_\Gamma\times\WW\to X_\Gamma\times\WW$, we get an inclusion of  
full measure subsets (the first one being a diffeomorphism away
  from the boundary of $\WW$)
\begin{equation}\label{eq:X0-full-measure}
  \XX_0\subset \Delta_\ver\times gr(\phi)
\setminus \LL_\phi \subset \Delta_\ver\times gr(\phi).
\end{equation}
The crucial observation is now that the form $p^*\om$ is a smooth form 
on the compact manifold $\YY=\wt X_\Gamma\times\WW$, and $\gamma$ can
be considered as such via pullback under the natural blow-down
map. This makes $(\iota_2^*p^*\om\wedge\gamma)|_{\XX_0}$ 
the restriction of a smooth form on $\YY$ to $\XX_0$, and 
Stokes' theorem for $(\YY,\XX)$ from Lemma~\ref{lem:Stokes-an} implies
the desired integrability.

The argument for integrability of the form $\om|_{\Delta_{\ver}\times
  M^d}$ is analogous but simpler, using Stokes' theorem for the
basic pair associated to $\Gamma$ via equation~\eqref{eq:graphtopair}
with the data $\WW:=\ZZ:=\{\pt\}$.
\end{proof}

Consider now the pullback diagram
$$
\xymatrix{
\Delta_\ver\times \WW \ar[d] \ar[r]^{\id\times\phi} &
\Delta_\ver\times M^d \ar[d] \\
\WW \ar[r]^\phi & M^d.
}  
$$
In view of Lemma~\ref{lem:graph-L1}, we can apply
Lemma~\ref{lem:pullback} to obtain 
\begin{equation*}
\phi^*\int_{\Delta_\ver}\om=
\int_{\Delta_\ver}(\id\times\phi)^*\om,
\end{equation*}
where $\int_{\Delta_\ver}$ denotes the respective fibre integrations
and both sides of the equation are integrable forms on $\WW$. 
So we can integrate them over $\WW$ and apply Fubini's theorem to get 
\begin{equation}\label{eq:keychenan1}
\int_\WW\phi^*\int_{\Delta_\ver}\om=
\int_{\WW}\left(\int_{\Delta_\ver}(\id\times\phi)^*\om\right)=
\int_{\Delta_\ver\times \WW}(\id\times\phi)^*\om.
\end{equation}
Let us denote the pullback of $\om$ under the composition
$p\circ\iota_2$ in diagram~\eqref{eq:commdiag} still by the same letter
$\om$. Then we can rewrite the right hand side of the last equation
using~\eqref{eq:restr-om} and restricting the integral to the 
full measure subset in~\eqref{eq:X0-full-measure} to get
\begin{equation}\label{eq:angenfin}
  \int_{\WW}\left(\phi^*\int_{\Delta_\ver}\om\right) = \int_{\XX_0}\om.
\end{equation}
Equations~\eqref{eq:keychenan1} and~\eqref{eq:angenfin} will play a
crucial role in~\S\ref{sec:rel-stringtop}.
Now we will specify the setup further.
Recall from~\eqref{eq:vertexdiag} the embedding
$\iota_\ver:\Delta_{\ver} \into \prod_{j\in\Ver} M^{d_j}$. 
We begin with the 


{\bf Dream case. }
Assume the following condition:
\begin{itemize}
\item [(T)]
The map
$$
\iota_\ver\times\phi:\Delta_\ver\times\WW\longrightarrow X_\Gamma
$$
is transverse to the family $\Delta_2^\Gamma$.
\end{itemize}
Then
$$
\Bigl(\Delta_\ver\times\WW,\,C_\phi:=(\iota_\ver\times\phi)^{-1}
(\Delta_2^\Gamma)\Bigr)
$$
is a pair of a manifold with corners
and a transverse collection of nice submanifolds. Hence, we can take
the corresponding blow-up to be the desired
compactification. Moreover, $(\Delta_\ver\times gr(\phi),\LL_\phi)$
is a pair of a manifold with corners and 
a transverse collection of nice submanifolds, and the map 
$\iota^\phi$ provides an isomorphism of pairs
$$
\iota^\phi:(\Delta_\ver\times\WW,C_\phi)\stackrel{\cong}{\longrightarrow} 
(\Delta_\ver\times gr(\phi),\LL_\phi).
$$
Lemma~\ref{lem:PTtransverse} implies 
that this isomorphism lifts to the isomorphism
\begin{equation}\label{eq:intcomp}
\Bl(\Delta_\ver\times\WW,C_\phi)
\stackrel{\wt\iota^\phi}{\cong}
\Bl(\Delta_\ver\times gr(\phi),\LL_\phi)
\subset \wt X_\Gamma\times\WW,
\end{equation}
where the last inclusion is that of a submanifold with corners. Moreover,
\begin{equation}\label{eq:dream-comp}
\XX = PT(\Delta_\ver\times gr(\phi)) = \Bl(\Delta_\ver\times gr(\phi),\LL_\phi).
\end{equation}

{\bf General case. }
In general, the compactification $\XX$ can be seen as a replacement
for the intuitive compactification $\Bl(\Delta_\ver\times\WW,C_\phi)$
if the map $\iota_\ver\times\phi$ is not transverse to $\Delta_2^\Gamma$.
Since $\XX$ is by definition a subset of $\wt X_\Gamma\times\WW$, 
the map to $\wt X_\Gamma\times\WW$ that we use to pull back forms 
is just the natural inclusion of $\XX$ into $\wt
X_\Gamma\times\WW$. This inclusion serves then as a replacement for
the isomorphism $\wt\iota^\phi$ in equation~\eqref{eq:intcomp}. 

For the general case, we denote
$\phi_0:=\phi|_{\WW_0}$
and assume the following condition:
\begin{itemize}
\item [(AT)] For each edge $l$, the intersection
$$
  \wh \Delta_2^l=(\Delta_\ver\times gr(\phi_0))\cap (X_{\{l\}}\times \WW_0).
$$
in~\eqref{eq:whdelta2} is transverse in $X_\Gamma\times\WW_0$.
\end{itemize}
It is easy to see that condition (AT) for an edge $l$ is equivalent to 
$$
  \bigl(T_p\Delta_\ver \oplus T_q\phi_0(T_q\WW_0)\bigr) \oplus
  T_{(p,q)}X_{\{l\}}  = T_{(p,q)}X_\Gamma 
$$
for any $(p,q)\in (\iota_\ver\times \phi_0)^{-1}(X_{\{l\}})\subset \Delta_\ver\times \WW_0$, 
which is, in turn, equivalent to saying that the map
$$
  \iota_\ver\times \phi_0:\Delta_\ver\times \WW_0 \longrightarrow X_\Gamma
$$
is transverse to $X_{\{l\}}$. This gives us the isomorphism
\begin{equation}\label{eq:diagdtea}
  \wh\Delta_2^l\cong (\iota_\ver\times\phi_0)^{-1}X_{\{l\}}.
\end{equation}

\begin{remark}\label{rem:AT-MES}
The above conditions satisfy the implications
$$
  (T) \Longrightarrow (AT) \Longrightarrow (MES).   
$$
For the first implication, note that part of condition (T) is
transversality of the map $\id\times\phi:\Delta_\ver\times\WW\to
\Delta_\ver\times M^d$ to $\ol{\Delta}^l_{2}$, or equivalently of
$\Delta_\ver\times gr(\phi)$ to $\ol{\Delta}^l_{2}\times\WW$.
Intersection with the open subset $X_{\{l\}}\times \WW_0\subset
\Delta^l_{2}\times\WW$ gives (AT). 
For the second implication, note that transversality of the
intersection in~\eqref{eq:whdelta2} implies that $\wh \Delta_2^l$ is a
positive codimension submanifold of $\Delta_\ver\times gr(\phi_0)$, 
and as such has measure zero in $\Delta_\ver\times gr(\phi)$. Note the
following cover of the set $\LL_\phi$: 
$$
  \LL_\phi\subset \left(\Delta_\ver\times gr(\phi|_{\p\WW})\right)
  \cup\bigcup_{l\in \Edge}\wh \Delta_2^l.
$$
The first member of this union has measure zero in $\Delta_\ver\times
gr(\phi)$ because $\p\WW$ has measure zero in $\WW$. Since the second
member is a union of sets of measure zero, this shows (MES). 
\end{remark}

\begin{remark}\label{rem:descr-primary}
There are two groups of primary faces of $\XX$. The first group $\p^{\main}_1\XX$
corresponds to the codimension $1$ boundary $\p_1\WW$ of $\WW$ in the
following way. The obvious graphical embedding of $\p_1\WW$ in $M^d\times \WW$ identifies 
$\p_1\WW$ with $\p_1\ZZ$. The desired portion of $\p^{\main}\XX$ is obtained from 
$\p_1\ZZ$ of multiplication with $(\Delta_\ver\times M^d)\setminus
\Delta_2^\Gamma$. That is, we have
\begin{equation}\label{eq:descr-prim1}
  \p_1\WW\cong \p_1\ZZ,\qquad
  \p^{\main}_1\XX := ((\Delta_\ver\times M^d)\setminus \Delta_2^\Gamma)\times \p_1\ZZ.
\end{equation}
The second group consists of the faces $\p_l\XX$ corresponding to non
doubly special edges as in Proposition~\ref{prop:reg-char} (C) 
and to doubly special edges. 
In our applications, the compactifications corresponding to circular
graphs with a doubly special edge will actually be manifolds with
corners. This description will in particular give us a good
understanding of their primary faces (including those corresponding to
a doubly special edge).
\end{remark}

Now we apply the preceding discussion to the relevant Chen's integrals, 
using the notation from~\S\ref{sec:chen}.
We refer to~\S\ref{sec:graphs} for the definition of (extended)
labellings and their standardizations.
We consider two cases arising for the product and coproduct.

{\bf (a) Product. } 
Let $f_j:B_j\rightarrow \Lambda$, $j=1,2$ be two smooth maps from
compact manifolds with corners with transverse evaluations at time zero.
Let $\Gamma$ be a connected tree with two special vertices of degrees
$d_1,d_2\geq 1$ and an extended labelling. 
Set 
$$
  d:=d_1+d_2,\qquad
  \WW:=B_1\times\Delta^{d_1-1}\times B_2\times\Delta^{d_2-1},\qquad
  \phi:=ev_{f_1}\times ev_{f_2}.
$$

\begin{lemma}\label{lem:T-product}
In the above setting the map $\phi$ satisfies 
condition (T).
\end{lemma}

\begin{proof}
For a family of submanifolds we will call the intersection of all its
members {\em the total intersection of the family}.
Let us consider a connected tree $\Gamma$ with two special vertices.
It suffices to show that $\Delta_\ver^\Gamma\times gr(\phi)$ is
transverse to the total intersection of any subfamily $\CC$ of the
family $\Delta_2^\Gamma\times \WW$. 

An open neighbourhood $U\subset X_\Gamma\times \WW$ of a point in the
intersection $D$ of $\Delta_\ver\times gr(\phi)$ with the total
intersection of a subfamily $\CC$ can be described as follows.
Let $\Gamma'$ denote the result of cutting all the edges of $\Gamma$
that do not parametrize members of the subfamily $\CC$,
cf.~Remark~\ref{rem:transfer}. Then the embedding
$D\cap U\into U$ is isomorphic to the embedding of the intersection
of $\Delta_\ver^{\Gamma'}\times gr(\phi)$ with the total intersection of the family
$\Delta_2^{\Gamma'}\times \WW$ in $X_{\Gamma'}\times \WW$ for $M=\R^n$.
Renaming $\Gamma'$ back to $\Gamma$ thus reduces the question
to the case when the subfamily $\CC$ is the family
$\Delta_2^\Gamma\times \WW$ itself, at the cost of allowing the tree $\Gamma$
to be disconnected. For the sake of notational simplicity
we only consider the following two cases: $\Gamma$ is connected, or 
$\Gamma$ is a disjoint union of two connected trees with one special vertex
each.

(1) Let $\Gamma$ be connected.
We denote
\begin{equation}\label{eq:intesect-notat}
  \Delta_{tot}:=\bigcap_{l\in \Edge}\Delta_2^l\times \WW,\qquad
  D:=\Delta_{tot}\cap (\Delta_\ver\times gr(\phi)).
\end{equation}
We decompose $\Gamma$ into a connected tree $C_\Gamma$ with two roots
corresponding to the special vertices, $d_1-1$ rooted trees attached
at the first special vertex, and $d_2-1$ rooted trees attached at the
second special vertex, see Figure~\ref{fig:2S}. We denote the
variables assigned to the flags of $C_\Gamma$ by $x_A,x_Z$ (the root
flags) and $x_3,\dots,x_{\ff_c}$, where $\ff_c$ is the number of flags
of $C_\Gamma$. We denote by $x_j^i$ the variable assigned to the
$i_{\rm{th}}$ flag on the rooted tree number $j$ at the first special
vertex. Similarly, the variables $y_j^i$ are assigned to the flags of
the rooted trees at the second special vertex. 
Then $D$ is described by the following system of equations (where
$p_i\in B_i$): 
\begin{equation}\label{eq:inter-tot1}
\begin{cases}
  x_A=x_Z=x_3=\dots=x_{\ff_c}, \quad (x_A,x_Z)=(f_1(p_1,0),f_2(p_2,0));\cr
  f_1(p_1,t_1^1)=x_1^1=\dots=x_1^{\ff_1^1};\cr
  \dots\cr
  f_1(p_1,t_{d_1-1}^1)=x_{d_1-1}^1=\dots=x_{d_1-1}^{\ff_{d_1-1}^1};\cr
  f_2(p_2,t_1^2)=y_1^1=\dots=y_1^{\ff_1^2};\cr
  \dots\cr
  f_2(p_2,t_{d_2-1}^2)=y_{d_1-1}^1=\dots=y_{d_1-1}^{\ff_{d_1-1}^2}.
\end{cases}
\end{equation}
Consider the intersection
$$
  \{x_A=x_Z\}\cap\{(x_A,x_Z)=(f_1(p_1,0), f_2(p_2,0))\}\subset M^2\times \WW
$$
entering the first line of the above system.
Transversality of this intersection (which is naturally diffeomorphic
to the domain $D_{\mu(f_1,f_2)}$ of the Chas--Sullivan product of $f_1$
and $f_2$, see~\eqref{eq:def-D-mu}) is equivalent to transversality of
the time zero evaluations of $f_1$ and $f_2$. In particular, this
intersection is clean. Using this and the explicit description~\eqref{eq:inter-tot1}
of the intersection
$$
  D\cong D_{\mu(f_1,f_2)}\times \Delta^{d_1-1}\times\Delta^{d_2-1}
$$
of $\Delta_{tot}$ and $\Delta_\ver\times gr(\phi)$, we see that
the latter intersection is clean as well. We upgrade this cleanness to
transversality by computing dimensions.
Let $e$ be the number of edges of $\Gamma$ and $k$ the number of
vertices. Then
\begin{equation}
\begin{cases}
\dim \Delta_{tot}=ne+\dim\WW,\cr
\dim(\Delta_\ver\times gr(\phi))=n(k-2)+\dim\WW,\cr
\dim D=\dim \WW-n,\cr
\dim(X_\Gamma\times\WW)=2ne+\dim\WW.
\end{cases}
\end{equation}
Since $\Gamma$ is a connected tree, we have $k=e+1$. Therefore,
\begin{equation}\label{eq:transprod}
  \dim \Delta_{tot}+\dim(\Delta_\ver\times gr(\phi)) =
  \dim D+\dim(X_\Gamma\times\WW),
\end{equation}
which gives us the desired transversality.

(2) Let now $\Gamma=\Gamma_1\amalg\Gamma_2$ be the disjoint union
of two trees with one special vertex each. We retain the
notation~\eqref{eq:intesect-notat} decompose $\Gamma_1$ and $\Gamma_2$
into rooted trees glued at the special vertices, see Figure~\ref{fig:1S}.
We denote by $x_j^i$ the variable assigned to the $i^{\rm{th}}$ flag
on the rooted tree number $j$ of $\Gamma_1$,
$j=0,\dots,d_1-1$. Similarly, the variables $y_j^i$ are assigned to
the flags on the rooted trees of $\Gamma_2$.
We denote $x_A:=x_0^1$ and $y_A:=y_0^1$.
Then $D$ is described by the following system of equations:
\begin{equation}\label{eq:inter-tot2}
\begin{cases}
  f_1(p_1,0)=x_A=x_0^2=\dots=x_0^{\ff_0^1};\cr
  f_1(p_1,t_1^1)=x_1^1=\dots=x_1^{\ff_1^1};\cr
  \dots\cr
  f_1(p_1,t_{d_1-1}^1)=x_{d_1-1}^1=\dots=x_{d_1-1}^{\ff_{d_1-1}^1};\cr
  f_2(p_2,0)=y_A=y_0^2=\dots=y_0^{\ff_0^2};\cr
  f_2(p_2,t_1^2)=y_1^1=\dots=y_1^{\ff_1^2};\cr
  \dots\cr
  f_2(p_2,t_{d_2-1}^2)=y_{d_1-1}^1=\dots=y_{d_1-1}^{\ff_{d_1-1}^2}.
\end{cases}
\end{equation}
From this description we see that $D$ is naturally diffeomorphic to
$\WW$ and the intersection $D$ is clean. We upgrade this cleanness to
transversality by computing dimensions:
\begin{equation}
\begin{cases}
\dim \Delta_{tot}=ne+\dim\WW,\cr
\dim(\Delta_\ver\times gr(\phi))=n(k-2)+\dim\WW,\cr
\dim D=\dim \WW,\cr
\dim(X_\Gamma\times\WW)=2ne+\dim\WW.
\end{cases}
\end{equation}
Since $\Gamma$ is a tree with two connected components, $k=e+2$.
Therefore, we again have~\eqref{eq:transprod},
which gives the desired transversality.
\end{proof}

{\bf (b) Coproduct. }
Let $f:B\rightarrow \Lambda$ be an {\em analytic} map from a compact manifold
with corners which is nondegenerate in the sense of
Definition~\ref{def:nondeg}. Let $\Gamma$ be a circular graph with one
special vertex of degree $d$ and an extended labelling  
satisfying conditions~\ref{eq:one-boundary-vertex}
  and~\eqref{eq:nonspec-trivalent} below.
Set 
$$
  \WW:=B\times\Delta^{d-1},\qquad \phi:=ev_f.
$$
Recall from item (iv) in~\S\ref{ss:basiccomb} that we agreed to
standardize the extension of the labelling of $\Gamma$ in such a way
that the flag $A$ is the first one in the numbering of flags around
the special vertex.  

We claim that condition (AT) holds. Indeed if the edge $l$ is not doubly
special, then this follows from Remark~\ref{rem:trans}. If the edge $l$
is doubly special, let $D_l\subset M^d$ denote the corresponding diagonal.
Then condition (i) of Definition~\ref{def:nondeg} implies transversality of the
map $\phi_0$ to $D_l$. The latter is equivalent to transversality 
of $\ZZ_0=gr(\phi_0)$ to $D_l\times \WW_0$ and we conclude
by Remark~\ref{rem:trans-DS}.

\begin{remark}\label{rem:trans-upgrade}
If the special vertex lies on the cycle of $\Gamma$, then using
arguments similar to the ones in case (a) one can upgrade condition
(AT) to the following statement: the map $\iota_\ver\times \phi_0$ is
transverse to the family $\Delta_2^\Gamma$. 
\end{remark}

\begin{prop}\label{prop:mainStokes}
Consider the situation of case (a) or (b) above. 
%
Then Stokes' theorem holds for the associated compactification
$(\YY,\XX)$. Moreover, integrals over hidden faces vanish, so for any
$\om\in\Om^*(\YY)$ we have
$$
  \int_{\XX_0}d\om=\int_{\p^\main\XX}\om. 
$$
In both cases, condition (AT) (and thus by Remark~\ref{rem:AT-MES}
condition (MES)) is satisfied. In case (a), condition (T) is satisfied 
and the compactification $\XX$ is a manifold with corners given
by~\eqref{eq:dream-comp}. 
\end{prop}

\begin{proof}
Consider first case (a). 
It was shown above that condition (T) is satisfied.
So we are the dream case discussed above, and all the assertions
about case (a) from the above discussion of the dream case. 

Consider now case (b). It was shown above that Condition (AT) holds.
Assume first that the graph $\Gamma$ has a doubly special edge. Then, 
by Lemma~\ref{lem:DSd} below, the compactification $\XX$ is again a
manifold with corners (so Stokes' theorem holds) and it has no hidden
faces.

Assume now that the graph $\Gamma$ does not have a doubly special
edge. 
Then in view of conditions~\ref{eq:one-boundary-vertex} and~\eqref{eq:nonspec-trivalent} 
the graph has no
self-loops, so its cycle contains at least two edges. 
Thus Lemma~\ref{lem:ex-circ} implies that $\Gamma$ admits
cancellation, and Proposition~\ref{prop:cancel} gives us vanishing of
integrals over hidden faces.
Stokes' theorem holds by Lemma~\ref{lem:Stokes-an}.
\end{proof}

\subsection {Compactifications for circular graphs with a doubly special edge}
\label{ss:compDS}

In this subsection we prove Lemma~\ref{lem:DSd}, which was used in the
proof of Proposition~\ref{prop:mainStokes}. We work in the setup of
case (b) above, with a nondegenerate analytic map $f:B\to\Lambda$ and a circular graph $\Gamma$ with
a doubly special edge. We are interested in the associated basic pair  
\begin{equation*}
  (\YY,\XX) = \Bigl(\wt X_\Gamma\times (B\times\Delta^{d-1}),
  PT\bigl(\Delta_{\ver}\times gr(ev_f)\bigr)\Bigr). 
\end{equation*}
Recall from~\S\ref{ss:transvers} the subset $D_f\subset B\times [0,1]$
and the evaluation map 
$$
  e_f:B\times [0,1]\longrightarrow M\times M,\qquad (p,t)\mapsto (f_p(0),f_p(t)),
$$
which by Lemma~\ref{lem:coprod-trans} lifts to a smooth map between
manifolds with corners
$$
  \wt e_f:\wt B:=\Bl(B\times [0,1],D_f)\longrightarrow \wt M^2.
$$

The proof of Lemma~\ref{lem:DSd} is based on two special cases. 
The first special case is that of the graph $\Gamma_{DS}$ consisting
of just one special vertex of degree $d=2$ and one doubly special
edge. We denote the corresponding basic pair by
$$
  (\YY_{DS},\XX_{DS}) = \Bigl(\wt M^2\times (B\times [0,1]),
  PT\bigl(gr(e_f)\bigr)\Bigr). 
$$
We wish to give a more explicit description of $\XX_{DS}$.
For this, we blow up the second factor in $\YY_{DS}$ and consider the
natural blow-down map 
\begin{equation}\label{eq:whDS}
  \wh{\YY_{DS}} := \wt M^2\times \Bl(B\times [0,1],D_f)
  = \wt M^2\times \wt B
  \stackrel{\wh\pi}{\longrightarrow}\YY_{DS}.
\end{equation}
Consider the graph of $\wt e_f$, 
$$
  gr(\wt e_f) = \{(\wt e_f(q),q)\mid q\in \wt B\}\subset \wh{\YY_{DS}},
$$
with the natural graphical embedding of manifolds with corners
$$
  \wt B \stackrel{\cong}\longrightarrow gr(\wt e_f)\subset
  \wh{\YY_{DS}}, \qquad q\mapsto (\wt e_f(q),q).
$$
This way $(\wh{\YY_{DS}}, gr(\wt e_f))$ is a pair, where 
$gr(\wt e_f)$ is an embedded manifold with corners diffeomorphic
to $\wt B$.

\begin{lemma}\label{lem:DS}
The map $\wh\pi$ in~\eqref{eq:whDS} restricts to $gr(\wt e_f)$ as an
embedding of manifolds with corners whose image equals $\XX_{DS}$.
\end{lemma}

\begin{proof}
We first claim that the restriction $\wh\pi|_{gr(\wt e_f)}:gr(\wt
e_f)\to \YY_{DS}$ is an immersion. For this, let $N_q$ denote the
normal space to $T_qD_f$ at $q\in D_f$ and $S^{n-1}_q$ its 
oriented projectivization. Note that the derivative of $\wh\pi$ at $q$ 
annihilates a vector if and only if the vector is tangent to 
$S^{n-1}_q$. Therefore, it suffices to show that the restriction of 
$\wh\pi$ to $gr(\wt e_f|_{S^{n-1}_q})$ is an immersion. 
For this, observe that $\wt e_f|_{S^{n-1}_q}$ is the projectivization 
of a linear isomorphism and thus an embedding.
Therefore, the restriction $\wh\pi|_{gr(\wt e_f|_{S^{n-1}_q})}$
(which simply forgets the second component of 
$gr(\wt e_f|_{S^{n-1}_q})\subset \wt M^2\times S^{n-1}_q$) is an
embedding and the claim is proved.

Next observe that, due to nondegeneracy of $f$, the subset
$$
  C := \{(p,t)\in B\times[0,1] \mid f_p(0)=f_p(t)\}
  \subset B\times [0,1]
$$
is closed and nowhere dense. Therefore, its preimage $\wt C\subset\wt
B$ under the natural blow-down map is closed and nowhere dense. 
Let us introduce the restriction
$$
\overset{\circ}{\wt e_f}:=\wt e_f|_{\wt B\setminus\wt C}.
$$
Nowhere density of $\wt C\subset\wt B$ implies that 
$gr(\overset{\circ}{\wt e_f})$ is dense in $gr(\wt e_f)$.
Note that $\wh\pi$ restricts to $gr(\overset{\circ}{\wt e_f})$ as a
diffeomorphism onto the interior of $\XX_{DS}$. Since the map $\wh\pi$
is proper, this implies that $\wh\pi(gr(\wt e_f))=\XX_{DS}$. 
It remains to show that $\wh\pi|_{gr(\wt e_f)}$ is injective. 
Since the blow-down map only collapses the projectivized normal spaces
along the blow-up locus $D_f$, it suffices to show that for each $q\in
D_f$ the restriction $\wh\pi|_{gr(\wt e_f|_{S^{n-1}_q})}$ is injective, 
which follows from the proof of the claim above. 
\end{proof}

The second special case concerns a circular graph $\Gamma_-$
consisting of one special vertex of degree $d\geq 2$, one doubly
special edge, and $d-2$ leaves at the special vertex. 
We denote the corresponding basic pair by $(\YY_-,\XX_-)$. 

\begin{lemma}\label{lem:DS-d}
The natural inclusion $\XX_-\into\YY_-$ is an embedding of manifolds
with corners. 
\end{lemma}

\begin{proof}
For notational simplicity, assume that the special flag $Z$
(see~\S\ref{ss:specvert}) has number $2$ in the ordering of special
flags around the special vertex. Then
$$
  (\YY_-,\XX_-) = \Bigl((\wt M^2\times M^{d-2})\times (B\times \Delta^{d-1}),
  PT\bigl(gr(ev_f)\bigr)\Bigr). 
$$
Consider the inclusions
$$
B\times \Delta^{d-1}\subset B\times [0,1]^{d-1}=
B\times [0,1]\times [0,1]^{d-2}\supset D_f\times [0,1]^{d-2}.
$$
We claim that
$$
  D_f^d := (D_f\times [0,1]^{d-2})\cap (B\times \Delta^{d-1})\subset 
  B\times \Delta^{d-1}
$$
is a nice submanifold.
The statement is non obvious only at points where the first component
of $(t_1,\dots,t_{d-1})\in\Delta^{d-1}$ is $t_1=0$ or
$t_1=1$. Consider the case $t_1=0$ and recall from
Lemma~\ref{lem:coprod-trans}(b) that $D_f\subset B\times [0,1]$ is a
nice submanifold. Thus for any $(p,0)\in D_f$
we find coordinates $(x_1,\dots,x_k)$ for $B$ near $p$ such that
$D_f$ near $(p,0)$ is defined with some $m\geq n$ by 
$$
  x_i=0\text{ for }1\leq i\leq n,\quad
  x_i\geq 0\text{ for }n+1\leq i\leq m,\quad
  t_1\ge 0.
$$
We see that $D_f^d$ is defined by 
$$
  x_i=0\text{ for }1\leq i\leq n,\quad
  x_i\geq 0\text{ for }n+1\leq i\leq m,\quad
  0\le t_1\le\dots\le t_{d-1}\le 1.
$$
A similar description in the case $t_1=1$ proves the claim.
The rest of the proof is analogous to that of Lemma~\ref{lem:DS}. 
Namely, consider the lift of the map $ev_f$ to blow-ups
$$
  \wt {ev}_f:\wt B:=\Bl(B\times \Delta^{d-1},D_f^d)\longrightarrow \wt
  X_{\Gamma_-}=\wt M^2\times M^{d-2}. 
$$
The existence of this lift is straightforward from the existence of 
the lift $\wt e_f$ provided by Lemma~\ref{lem:coprod-trans}.
Consider its graph
$$
  gr (\wt {ev}_f) = \{(\wt {ev}_f(q),q)\mid q\in \wt B\}
  \subset
  \wt X_{\Gamma_-}\times \wt B =: \wh{\YY_-}.
$$
Let $\wh\pi: \wh{\YY_-}\longrightarrow \YY_-$
denote the map blowing down the second factor. As in the proof of
Lemma~\ref{lem:DS}, we get that the map $\wh\pi$ restricts to $gr(\wt
{ev}_f)$ as an embedding whose image equals $\XX_-$, which proves the
lemma.  
\end{proof}

\begin{remark}\label{rem:no-hidden-}
The boundary of $\XX_-$ does not contain any hidden faces. Indeed,
let $\p^{DS}\wt B$ denote the part of the boundary of $\wt B$ that
lies over $D_f^d$ and let
$$
\iota:\wt B\into \wh{\YY_-}
$$
denote the embedding of $\wt B$ as the graph of $\wt {ev}_f$. Since the
target of the restriction of $\wt {ev}_f$ to the complement of $\p^{DS}\wt B$
does not have boundary, the restriction of $\iota$ to the complement
of $\p^{DS}\wt B$ is nice. Since the map $\wh\pi$ is a diffeomorphism
when restricted to $\iota (\wt B\setminus \p^{DS}\wt B)$, there can be no
hidden faces away of $\p^{DS}\wt B$. Let now $\p^{DS}_0\wt B$ denote
the part of the quasi-regular boundary of $\wt B$ contained in $\p^{DS}\wt B$.
In other words, $\p^{DS}_0\wt B$ is the part of $\p \wt B$ that lies over
the interior of $D_f^d$. For $p\in \p^{DS}_0\wt B$ write
$\wh\pi\circ\iota(p)=(q_1,q_2)$ with $q_1\in \p\wt M^2\times M^{d-2}$
and $q_2\in \wt B$. The above description of $\p^{DS}_0\wt B$ implies that
$q_2$ belongs to the interior of $\wt B$. Therefore,
$(q_1,q_2)\in \p_1\YY_{\Gamma_-}$ and thus $\p^{DS}_0\wt B$ is a primary
face of $\XX_-$.
\end{remark}

Now we consider the general case of a circular graph $\Gamma$ with one
special vertex of degree $d$ and a doubly special edge. We denote the
corresponding basic pair by $(\YY,\XX)$. Let $\Gamma_-$ be the result
of chopping off all the rooted subtrees from $\Gamma$ as
in~\S\ref{ss:chopping}. 
Then the $\Gamma$ can be seen as a result of attaching $d-2$ rooted
trees $T_1,\dots,T_{d-2}$ to the graph $\Gamma_-$ along its $d-2$ leaf
flags (some of the trees might consist of just the root vertex).

\begin{lemma}\label{lem:DSd}
In the setting above, the natural inclusion $\XX\into\YY$ is an
embedding of manifolds with corners, and $\XX$ fibres over
$\XX_-$ in the category of manifolds with corners. Moreover, 
$$
  \p^\hidden\XX=\emptyset.
$$
\end{lemma}

\begin{proof}
The first sentence follows from Lemma~\ref{lem:DS-d} together with
successive application of 
Lemma~\ref{lem:tree-attach}. 
The last assertion follows from Remark~\ref{rem:no-hidden-} and
Remark~\ref{rem:hidden-attach}. 
\end{proof}


\begin{remark}\label{rem:smooth-L1}
Let us elaborate on the boundary component $\p_l\XX$ corresponding to
the doubly special edge $l=(A,Z)$. The transversality condition (AT)
for $l$ implies that $\p_l\XX$ is an $S^{n-1}$-bundle over $\wh\Delta_2^l$.
Recall from Remark~\ref{rem:trans-upgrade} that
the map $\iota_\ver\times \phi_0$ is transverse to the family
$\Delta_2^\Gamma$. Note that $\wh\Delta_2^l$ is the preimage
inder $\iota_\ver\times \phi_0$ of the set of points that belong to
$\Delta_2^l$ and no other members of the family $\Delta_2^\Gamma$,
whereas $\Delta_{\ver}\times gr(\phi_0|_{\overset{\circ} {D_f^d}})$ 
is the preimage under $\iota_\ver\times \phi_0$ of the set of points
that belong to $\Delta_2^l$ as well as possibly other members of the
family $\Delta_2^\Gamma$. 
This shows that $\wh\Delta_2^l$ is a full measure subset of 
$\Delta_{\ver}\times gr(\phi_0|_{\overset{\circ} {D_f^d}})$.
Let $\Gamma_1,\Gamma_2$ be the two trees with special vertices
obtained by collapsing the doubly special edge of $\Gamma$. 
Then $\Delta_\ver=\Delta_{\ver}^{\Gamma_1}\times
\Delta_{\ver}^{\Gamma_2}$, so we get an inclusion as a full measure subset
\begin{equation}\label{eq:smooth-L1}
  \wh \Delta_2^l\into \Delta_{\ver}^{\Gamma_1}\times
  \Delta_{\ver}^{\Gamma_2} \times gr(\phi_0|_{\overset{\circ}{D_f^d}}). 
\end{equation}
\end{remark}

\section{Ribbon graphs}\label{sec:graphs}

This section contains background material on ribbon graphs. It is
structured as follows. 
In~\S\ref{ss:ribbonbasicdef} we introduce the basic notions concerning
ribbon graphs, and 
in~\S\ref{ss:specvert} we define the six main types of graphs used in
this paper. 
In~\S\ref{ss:basiccomb} we introduce (extended) labellings for ribbon
graphs, the reordering map $\bar R_\Gamma$, and the sign exponent
$\eta_3(\Gamma)$. 
In~\S\ref{ss:automorph} we discuss automorphisms. 
In~\S\ref{ss:op-graphs} we introduce three operations on ribbon
graphs: disjoint union, cutting, and gluing.
In~\S\ref{ss:op-graphs2} we introduce three more
operations: duality, attaching a leg, and attaching a tree.

\subsection{Basic definitions}\label{ss:ribbonbasicdef}

In this paper, by a {\em ribbon graph}
(or just a {\em graph}) $\Gamma$ we mean a finite graph
with a cyclic ordering of the adjacent edges at each vertex and with
some degree $1$ vertices removed.\footnote{
In contrast to~\cite{Cieliebak-Fukaya-Latschev} we do not require
$\Gamma$ to be connected.} 
Note that if $\Gamma$ is a tree then its ribbon structure can
equivalently be given by embedding it in the plane, so in this case
``planar'' and ``ribbon'' are synonymous. 
We assume that at most one vertex is removed from each edge, and we
will refer to the edges with a vertex removed not as edges but as {\em
  leaves}, so that an {\em edge} still ends in two vertices.
We denote the sets of vertices, edges and leaves of $\Gamma$ by $\Ver(\Gamma)$,
$\Edge(\Gamma)$ and $\Leaf(\Gamma)$, and 
their cardinalities by
$$
   k=|\Ver(\Gamma)|,\qquad e=|\Edge(\Gamma)|,\qquad s=|\Leaf(\Gamma)|.
$$
We require $k\geq 1$. A ribbon graph $\Gamma$ can be thickened in a unique way 
(up to orientation preserving homeomorphism)
to a compact oriented surface $\Sigma_\Gamma$ with boundary such that
each leaf ends on the boundary $\p\Sigma_\Gamma$.
We denote
$$
  \ell = \text{number of boundary components of }\Sigma_\Gamma,\qquad 
   g = \text{genus of }\Sigma_\Gamma. 
$$
Note that the boundary components of $\Sigma_\Gamma$ induce additional
structure on the set of leaves: it gets subdivided into
subsets according to the boundary components, and each subset obtains
a cyclic order according to the boundary orientation. 

A {\em flag} in $\Gamma$ is a pair $(v,t)$ consisting of a 
vertex $v$ and an adjacent 
edge or leaf $t$. 
A flag $(v,t)$ is called {\em interior} if $t$ is an edge, and {\em
  exterior} if $t$ is a leaf. 
We denote the set of flags of $\Gamma$ by $\Flag(\Gamma)$ and its
cardinality by
$$
  f=|\Flag(\Gamma)|. 
$$
For our purposes it will be
convenient to describe a ribbon graph in terms of its flags: a leaf
corresponds to a single flag $x$, an edge corresponds to an
unordered pair $\{x,y\}$ of flags, and a vertex corresponds to a cyclically ordered
tuple $[x_1,\dots,x_d]$ of flags. An oriented edge corresponds to an
ordered pair $(x,y)$, and a vertex with an ordering of the
adjacent edges corresponds to an ordered tuple $(x_1,\dots,x_d)$.
An {\em isomorphism} $\Gamma\to\Gamma'$ between ribbon graphs is a
bijection $\Flag(\Gamma)\to \Flag(\Gamma')$ mapping leaves to leaves,
vertices to vertices, edges to edges, and preserving the cyclic orderings at the
vertices. It induces a homeomorphism $\Sigma_{\Gamma}\to\Sigma_{\Gamma'}$.

{\bf Marked ribbon graphs. }
We say that a connected ribbon graph $\Gamma$ is {\em marked} if an
edge is chosen, and {\em o-marked} if the marked edge is in addition
oriented. A (o-)marked graph will be denoted by $(\Gamma,l)$ with $l$
the (oriented) marked edge.

{\bf Special vertices. }
A graph may have some chosen vertices called {\em special
  vertices}. Flags adjacent to special vertices will be called {\em
  special flags} (they can be either interior or exterior). An edge
containing a special flag is called {\em special}, and an edge 
containing two special flags is called {\em doubly special}. 
For certain types of marked graphs (see below) with special 
vertices the marked edge can (and will) be canonically oriented.

{\bf Rooted trees. }
A {\em root} in a tree is a distinguished univalent vertex. By a 
{\em rooted tree} we will mean a connected planar tree with a
unique root.

\subsection{The six types of graphs}\label{ss:specvert}

Now we describe the six main types of graphs used in this paper. 
They are all required to satisfy the following two conditions:
\begin{equation}\label{eq:one-boundary-vertex}
\text{Each boundary component of $\Sigma_\Gamma$ has at least one leaf
  ending on it.}
\end{equation}
\begin{equation}\label{eq:nonspec-trivalent}
\text{All nonspecial vertices of $\Gamma$ are trivalent.}
\end{equation}
Special vertices are allowed to have any positive valency.

{\bf (s) Trivalent graphs. }
Graphs of this type are connected trivalent ribbon graphs without
special vertices.

{\bf (o) Trivalent trees. }
Graphs of this type are connected trivalent planar trees without special vertices. 

{\bf (i) Trees with one special vertex (see Figure~\ref{fig:1S}). }
\begin{figure}
\begin{center}
\includegraphics[angle=0,origin=c,width=\textwidth]{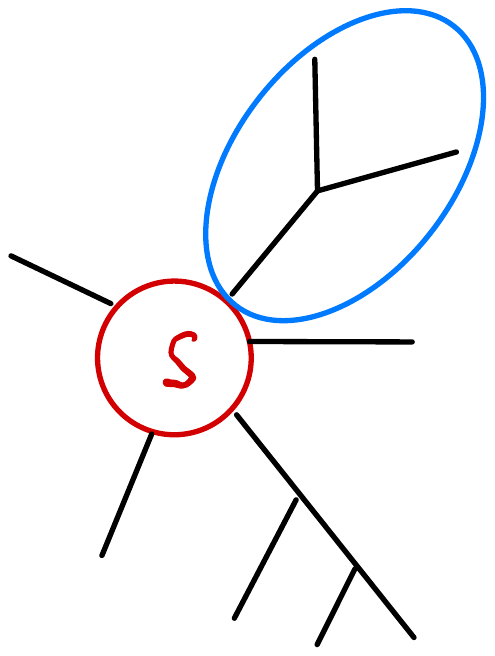}
\vspace{-3cm}
\caption{Tree with one special vertex}
\label{fig:1S} 
\end{center}
\end{figure}
Graphs $\Gamma$ of this type are connected planar trees with one
special vertex $S$ of degree $d\ge 1$. 
Then $\Gamma$ can be represented as a union of $d$ planar rooted trees 
with their roots glued together at the special vertex $S$. We call
them the {\em rooted components} of $\Gamma$. Figure~\ref{fig:1S} 
shows a rooted component encircled in blue.
Here and in the following special vertices are drawn as red circles.
Observe the inequality (with $s$ the number of leaves)
\begin{equation}\label{eq:tree-1-spec-num}
d\le s.
\end{equation}
{\bf (ii) Trees with two special vertices (see Figure~\ref{fig:2S}). }
\begin{figure}
\begin{center}
  \includegraphics[
    origin=c,width=\textwidth]{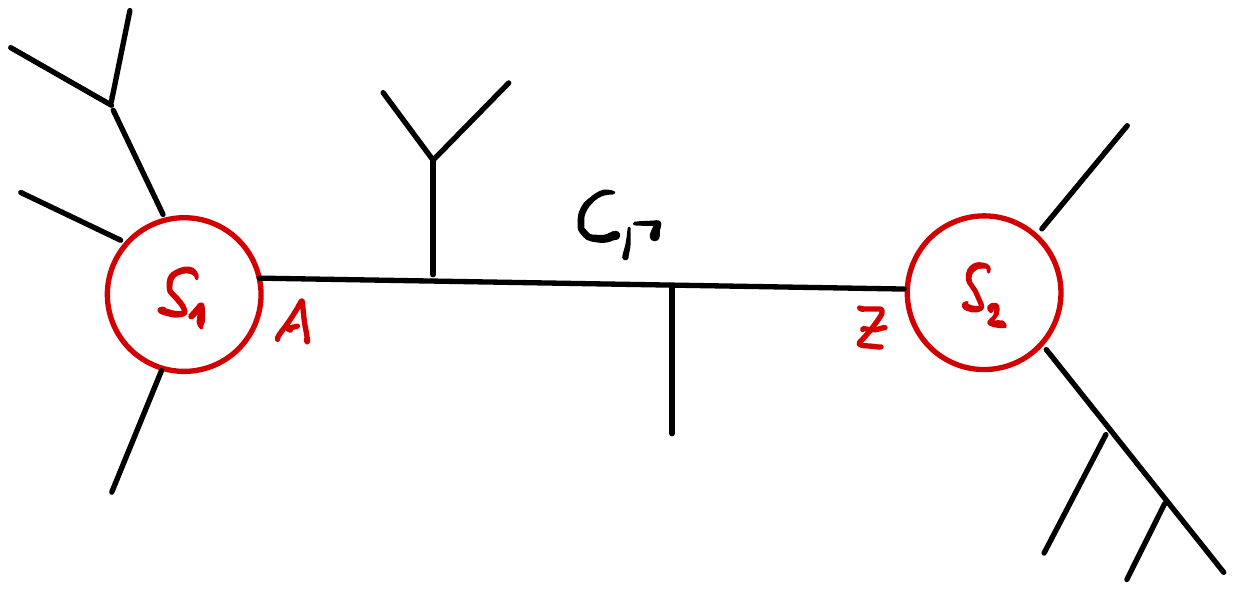}
\vspace{-3cm}
\caption{Tree with two special vertices}
\label{fig:2S} 
\end{center}
\end{figure}
Graphs $\Gamma$ of this type are connected planar 
trees with two special vertices $S_1,S_2$ of degrees $d_1,d_2\geq 1$. 
We consider the numbering of the special vertices 
$S_1,S_2$ part of the data of $\Gamma$.
Then $\Gamma$ can be represented as a union of $d_1-1$ rooted trees 
with roots at $S_1$, $d_2-1$ rooted trees with roots at $S_2$, and 
a {\em connecting tree} $C_\Gamma$ with two roots --- one corresponding to a special
flag (called $A$) at $S_1$ and the other to a special flag (called
$Z$) at $S_2$. Observe the inequality 
\begin{equation}\label{eq:tree-2-spec-num}
d_1+d_2\le s+2.
\end{equation}
Note that equality is achieved if and only if the graph has only one
edge (which must then be doubly special). 

{\bf (iii) Circular graphs. }
Graphs $\Gamma$ of this type are trivalent ribbon graphs without
special vertices whose surface $\Sigma_\Gamma$ is an annulus. Thus
$\Gamma$ has two boundary components and a unique nontrivial embedded cycle. 

{\bf (iv) Circular graphs with one special vertex
(see Figure~\ref{fig:circ}). }
\begin{figure}
\begin{center}
\includegraphics[angle=0,origin=c,width=\textwidth]{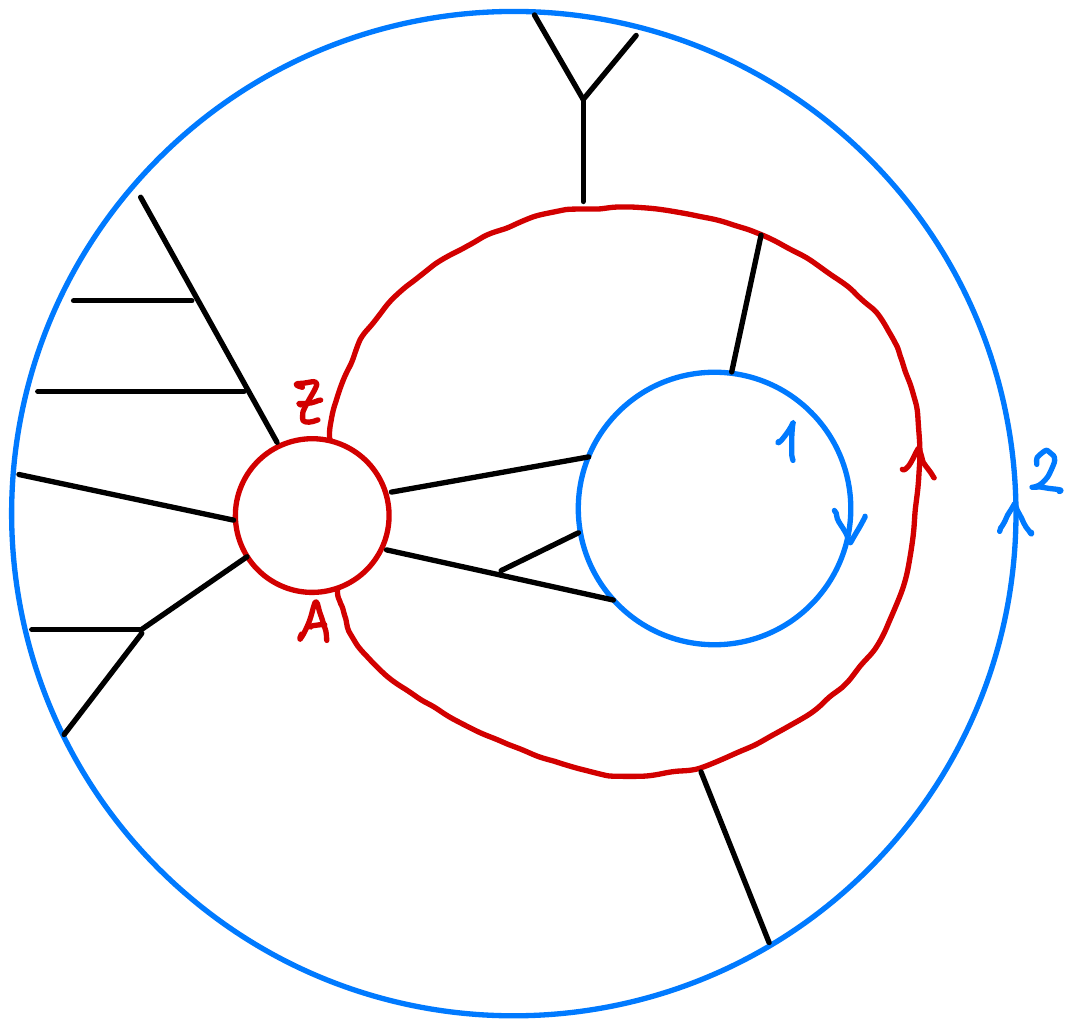}
\vspace{-1cm}
\caption{Circular graph with one special vertex}
\label{fig:circ} 
\end{center}
\end{figure}
Graphs $\Gamma$ of this type are ribbon graphs for the annulus with 
one special vertex $S$ of degree $d$. 
Assume that a numbering of the two boundary components is chosen. Then
we orient the unique nontrivial embedded cycle (drawn in red) by sliding it to the
{\em second} boundary component and picking up the boundary
orientation. Note that the special vertex $S$ may or may not lie on
the cycle. We denote the number of leaves on the $b$-th boundary
component by $s_b$, $b=1,2$. 

If the special vertex $S$ does not lie on the cycle, then there exists
a unique special flag connected to the cycle by a chain of edges not
crossing $S$. We denote this flag by $A$. Observe the inequality 
\begin{equation}\label{eq:circ-1-spec-off-cyc}
  1\le d\le s_1+s_2.
\end{equation}
If the special vertex $S$ lies on the cycle, then
we denote by $A$ the special flag emanating from $S$ in the
direction of the cycle and by $Z$ the incoming one.
Observe the inequality 
\begin{equation}\label{eq:circ-1-spec-on-cyc}
2\le d\le s_1+s_2+2.
\end{equation}
Note that equality $d=s_1+s_2+2$ is achieved if and only if the graph
has only one edge (which must then be doubly special). Note also that 
a doubly special edge must be a self-loop.

\subsection{Labellings, edge and vertex orders}\label{ss:basiccomb}

\begin{definition}\label{def:labelling}
Consider the following additional data on a ribbon graph $\Gamma$: 
\begin{enumerate}
\item a numbering of the boundary components of $\Sigma_\Gamma$ by
   $1,\dots,\ell$; 
\item a numbering of the leaves ending on the $b$-th boundary
  component by $1,\dots,s_b$ compatible with the cyclic order given by the orientation;
\item a numbering of the vertices; 
\item a numbering of the flags at each vertex compatible with the
  cyclic order; 
\item a numbering of the edges by $1,\dots,e$;
\item an orientation of each edge.
\end{enumerate}
We call the data (i) and (ii) a {\em labelling}, the additional data
(iii) -- (vi) an {\em extension of the labelling}, and all data
together an {\em extended labelling}. 
Note that a labelling is uniquely determined by an ordered set of
$\ell$ leaves, the first ones on their boundary components. 
\end{definition}

\begin{remark}\label{rem:labelling}
This notion of a labelling differs from the one
in~\cite{Cieliebak-Fukaya-Latschev} where a ``labelling'' comprises items
(i)--(iv). The reason for this is that the operations associated to
ribbon graphs in~\cite{Cieliebak-Fukaya-Latschev} require a choice of
items (i)--(iv), whereas the operations in this paper require only a choice of items (i)--(ii). 
\end{remark}

An {\em isomorphism} $\Gamma\to\Gamma'$ between labelled ribbon graphs
is an isomorphism of ribbon graphs which preserves the labellings. In
other words, the induced homeomorphism $\Sigma_{\Gamma}\to\Sigma_{\Gamma'}$ 
matches the numberings of the boundary components and leaves. 

Next, we introduce notation for certain isomorphism classes of
connected ribbon graphs and describe some standardizations on
extensions of their labellings in the presence of special vertices. 
We follow the notation of the previous subsection. In particular, all
the graphs are required to satisfy conditions~\eqref{eq:one-boundary-vertex}
and~\eqref{eq:nonspec-trivalent}. 

{\bf Rooted trees. }
Rooted planar trees will always be given the canonical labelling by
numbering leaves counterclockwise from the root.
The set of isomorphism classes of
rooted planar trees with $s$ leaves will be denoted by $RT_s$. The
corresponding set of rooted trees with all nonroot vertices being
trivalent will be called $RT_s^3$. 

{\bf (s) Trivalent graphs. }
The set of isomorphism classes of labelled connected trivalent ribbon graphs
without special vertices of genus $g\ge 0$ with $\ell\ge 1$ boundary components 
will be denoted by $\RR_{\ell,g}$. This notation is borrowed from~\cite{Cieliebak-Volkov}.

{\bf (o) Trivalent trees. }
The set of isomorphism classes of labelled connected trivalent ribbon trees
without special vertices with $s$ leaves will be denoted by $\RR_s$.

{\bf (i) Trees with one special vertex.} 
Consider a labelled planar tree $\Gamma$ with one special vertex
$S$ of degree $d$. We number its rooted components in counterclockwise
order $T_1,\dots,T_d$ where $T_1$ contains the leaf number $1$. We
denote by $A$ the special flag given by the root of $T_1$.
We require that any extension of the labelling be
standardized as follows: the special vertex $S$ is the {\em last} one
in the numbering of vertices, and the special flag $A$ is the first
one in the numbering of flags around $S$. 
We call the labelling on $\Gamma$ {\em canonical} if the induced
labelling on $T_1$ is the canonical one as a rooted tree (i.e., leaf
number $1$ in the labelling comes right after the root in the
counterclockwise order on $T_1$). 
We denote the set of isomorphism classes of labelled planar trees with
$s$ leaves and one special $d$-valent vertex by $\RR_{s;d}$, and its
subset of such trees with canonical labelling by 
$\RR_{s;d}^\can\subset \RR_{s;d}$. 

{\bf (ii) Trees with two special vertices. } 
Given a labelled tree with two special vertices, we require that any
extension of the labelling be standardized as follows: 
the special vertices $S_1,S_2$ are the {\em last} two in the numbering
of the vertices, and the special flags $A$ and $Z$ are the first ones
in the numberings of flags around $S_1$ and $S_2$. 
We denote the set of isomorphism classes of labelled trees with $s$
leaves and two special vertices of degrees $d_1,d_2$ by $\RR_{s;d_1,d_2}$.

{\bf (iv) Circular graphs with one special vertex.}
For an extended labelling on a circular graph with one special vertex,
we require that the special vertex $S$ be the {\em last} one in the
numbering of vertices, and the special flag $A$ be the first one in
the numbering of flags around $S$. We denote the set of isomorphism
classes of labelled circular graphs with one special $d$-valent vertex
and $s_1,s_2$ leaves on the boundary components by $\RR_{s_1,s_2;d}$. 

{\bf Generalized labellings. }
We will need one more variation of the concept of a
labelling. Consider a planar tree $\Gamma$ with one special vertex.
A {\em generalized labelling} on $\Gamma$ is a labelling together with
a numbering of the special flags compatible with their cyclic order. 
Thus, in contrast to the above standardization of an extended
labelling, this allows the special flag $A$ to have {\it any} number
(not necessarily $1$). 
An isomorphism of trees with generalized labellings is a tree isomorphism
that preserves the generalized labellings. The set of isomorpism classes of generalized labelled trees 
with one special $d$-valent vertex and $s$ leaves will be denoted by
$\RR_{s;d}^\gen$. Observe that the cyclic group $\Z_d$ acts freely on 
$\RR_{s;d}^\gen$ by cyclic renumbering of the special flags, and the
subset $\RR_{s;d}\subset\RR_{s;d}^\gen$ (giving the special flag $A$
number $1$) is a fundamental locus for this action.

{\bf Marked graphs. }
An isomorphism class of marked ribbon graphs of a certain type will be
denoted by adding the upper index $m$ to the corresponding notation
for the isomorphism class of graphs without marking,
e.g.~$\RR_{s_1,s_2;d}^m$. It is understood that an isomorphism of
marked graphs must map the marked edge to the marked edge. 

{\bf Edge and vertex order. }
An extended labelling on a ribbon graph $\Gamma$ gives rise to two
bijective maps (where $f=|\Flag(\Gamma)|$)
$$ 
   O_e:\{1,\dots,2e,\dots,f\}\longrightarrow \Flag(\Gamma),\qquad
   O_v:\{1,\dots,f\}\longrightarrow \Flag(\Gamma). 
$$
The first one is called the {\em edge order} on $\Flag(\Gamma)$; it is
determined by items (i), (ii), (v) and (vi) of the extended labelling,
mapping the numbers $1,\dots,2e$ to the flags corresponding to edges
according to (v) and (vi), and the remaining numbers to the leaves
according to (i) and (ii). 
The second one is called the {\em vertex order}; it is determined by
items (iii) and (iv), numbering the flags in the order (iii) of vertices
and using the ordering (iv) at each vertex.   
By composition we obtain the {\em reordering permutation}
\begin{equation}\label{eq:reod}
   \bar R_\Gamma:=O_v^{-1}\circ O_e:\{1,\dots,f\}\longrightarrow 
   \{1,\dots,f\}.
\end{equation}
The map $\bar R_\Gamma$ behaves as follows under changes of the
extended labelling of $\Gamma$: a change in (i), (ii), (v) and (vi) leads to  
precomposition of the edge order $O_e$ with some permutation 
$\eta\in S_f$, whereas a change in (iii) and (iv) leads to precomposition 
of the vertex order $O_v$ with a some permutation $\sigma^{-1}$, so
altogether $\bar R_\Gamma$ is replaced by $\sigma\circ \bar R_\Gamma\circ \eta$. 
We say that $\eta$ and $\sigma$ act on a graph $\Gamma$ with extended labelling 
and give us a graph $\sigma\Gamma\eta$.

{\bf Relabellings. }
Consider a ribbon graph $\Gamma$. Denote by $B$ the (unordered) set of
its boundary components and by $\Leaf_b$ the (unordered) set of leaves on $b\in B$.
A labelling of $\Gamma$ can be described as a collection
$(\phi,\{\psi_b\}_{b\in B})$ of bijections
$$
  \phi:\{1,\dots,\ell\}\stackrel{\cong}\longrightarrow B,\qquad 
  \psi_b:\{1,\dots,s_b\}\stackrel{\cong}\longrightarrow \Leaf_b
$$
such that $\psi_b$ induces the cyclic order on $b$. The group
$$
  S_\ell\times\prod_{b\in B}\Z_{s_b}
$$
acts freely and transitively on labellings of $\Gamma$ by
$$
  (\rho,\{\sigma_b\}_{b\in B})\cdot (\phi,\{\psi_b\}_{b\in B})
  = (\phi\circ\rho,\{\psi_b\circ\sigma_b\}_{b\in B}).
$$
Note that a labelling $(\phi,\{\psi_b\}_{b\in B})$ induces a bijection 
$$
  lb:\{1,\dots,s\}\stackrel{\cong}\longrightarrow \Leaf(\Gamma),
$$
where $s=\sum_{b\in B}s_b$ is the number of leaves of $\Gamma$, 
by sending $1,\dots,s$ to 
$$
  \psi_{\phi(1)}(1),\dots,\psi_{\phi(1)}(s_{\phi(1)}),
  \psi_{\phi(2)}(1),\dots,\psi_{\phi(2)}(s_{\phi(2)}),
  \psi_{\phi(\ell)}(1),\dots,\psi_{\phi(\ell)}(s_{\phi(\ell)}).
$$
This allows us to view labellings as such bijections $lb$ satisfying
suitable conditions. Since any two such bijections $lb,lb'$ are
related by $lb'=lb\circ\eta$ for some $\eta\in S_s$, we get a
canonical injection 
$$ 
  g:S_\ell\times\prod_{b\in B}\Z_{s_b} \into S_s.
$$
under which the above group action on labellings corresponds to the
action 
$$
  lb\mapsto lb\circ\eta,\qquad \eta\in {\rm im}(g)\subset S_s.
$$
Note, however, that the injection $g$ is {\em not} a group homomorphism.
We call 
\begin{equation}\label{eq:relab}
  S(\bs) := {\rm im}(g)\subset S_s
\end{equation}
the {\em set of relabellings} associated to the partition
$\bs=\{s_b\}$ of $s=\sum_{b\in B}s_b$. If the $s_b$ are the numbers
of leaves on the boundary components of a graph $\Gamma$ we 
say that the partition $\bs$ is {\em induced by $\Gamma$}.

Recall that an isomorphism $\phi:(\Gamma,lb)\to(\Gamma',lb')$ is a
graph isomorphism $\phi:\Gamma\to\Gamma'$ such that the following
diagram commutes:
$$
\xymatrix{ 
  \{1,\dots,s\} \ar[r]^{lb} \ar[rd]^{lb'} & \Leaf(\Gamma) \ar[d]^\phi \\
  & \Leaf(\Gamma').
}
$$
Since a relabelling $\eta\in S(\bs)$ acts on $lb$ by precomposition,
its action descends to isomorphism classes of labelled graphs. In
particular, fixing a numbering $1,\dots,\ell$ of the boundary
components, the group $\Z_{s_1}\times\cdots\times\Z_{s_\ell}$ acts by
cyclic renumbering of the leaves on each boundary component on the
set of isomorphism classes of labelled graphs with $s_1,\dots,s_\ell$
leaves on their boundary components. For example, for the sets defined
above we get
\begin{itemize}
\item actions of $\Z_s$ on $\RR_s$, $\RR_{s;d}$ and $\RR_{s;d_1,d_2}$;
\item an action of $\Z_{s_1}\times\Z_{s_2}$ on $\RR_{s_1,s_2;d}$.
\end{itemize}
We also get the involution
\begin{equation}\label{eq:renumber-bdry-basic}
\tau:\RR_{s_1,s_2;d}\stackrel{\cong}{\longrightarrow} \RR_{s_2,s_1;d}
\end{equation}
swapping the order of the two boundary components.



{\bf Special labellings of o-marked ribbon graphs. }
Consider a labelled ribbon graph $\Gamma$ with an o-marked edge $l$.
The labelling is called {\em special} if it has the following
properties (see the right hand side of Figure~\ref{fig:cut-glue}):
\begin{itemize}
\item The boundary component that runs parallel to $l$ (in the
  direction of its orientation) on its left has number $1$ in the
  numbering of boundary components. The first leaf on this boundary
  component encountered when moving in the direction of the boundary
  orientation from the portion to the left of $l$ has number $1$ in
  the numbering of its leaves.
\item If the boundary component that runs parallel to $l$ on its right
  differs from the one on its left, then it has number $2$ in the
  numbering of boundary components. In this case, the first leaf on this boundary
  component encountered when moving in the direction of the boundary
  orientation from the portion to the right of $l$ has number $1$ in
  the numbering of its leaves. 
\end{itemize}

{\bf The sign exponent $\eta_3$. }

\begin{definition}\label{def:eta3}
Let $\Gamma$ be an extended labelled ribbon graph. We recall the basic
properties of $\eta_3(\Gamma)$ defined in~\cite[Appendix
  A]{Cieliebak-Fukaya-Latschev} for connected $\Gamma$. 
The sign exponent $\eta_3(\Gamma)$ does not depend on items (ii) and
(iv) in Definition~\ref{def:labelling}, and it changes by $1$ under
the following operations: 
\begin{itemize}
\item swapping the order of two adjacent boundary components in (i);
\item swapping the order of two adjacent vertices in (iii);
\item swapping the order of two adjacent edges in (v);
\item flipping the orientation of an edge in (vi).
\end{itemize}
For a disconnected graph $\Gamma:=\Gamma_1\amalg\Gamma_2$
with connected $\Gamma_1$ and $\Gamma_2$, we define 
$\eta_3(\Gamma):=\eta_3(\Gamma_1)+\eta_3(\Gamma_2)$ for the canonical
disjoint union extended labelling of $\Gamma$ and extend the definition 
to all other extended labellings of $\Gamma$ using the four items above.
\end{definition}

\subsection{Automorphisms}\label{ss:automorph}

By definition, an {\em automorphism} of a (labelled) ribbon graph is a
self-isomorphism. The following simple result (Lemma~2.1 
in~\cite{Cieliebak-Volkov}) is crucial for the discussion of
automorphisms. 

\begin{lemma}\label{lem:no-auto}
If an automorphism of a connected ribbon graph fixes at least one flag,
then it must be trivial.
\end{lemma}


\begin{lemma}\label{lem:no-auto1}
(a) Connected trees with two special vertices have no nontrivial automorphisms. \\
(b) Any automorphism of a circular graph with one special vertex
that preserves its boundary components must be trivial.
\end{lemma}

\begin{proof}
For part (a), let $\phi$ be an automorphism of a tree $\Gamma$ with
two special vertices $S_1,S_2$. Since the numbering of the special
vertices is part of the structure of the graph, we have
$\phi(S_1)=S_1$ and $\phi(S_2)=S_2$. Since the subtree $C_\Gamma$
connecting $S_1$ and $S_2$ is uniquely defined, it must be preserved
by $\phi$. In particular the flags $A$ (adjacent to $S_1$) and $Z$
(adjacent to $S_2$) of $C_\Gamma$ are fixed by $\phi$. By
Lemma~\ref{lem:no-auto}, the automorphism $\phi$ must be trivial. 

For part (b), let $\phi$ be an automorphism of a circular graph
$\Gamma$ with one special vertex $S$ preserving the boundary components.
Then the unique embedded cycle of $\Gamma$ is preserved by $\phi$.
Assume first that the special vertex $S$ does not lie on the cycle.
Since $A$ is the unique flag of $S$ which is connected to 
the cycle via a chain of edges not crossing $S$, we have 
$\phi(A)=A$. Assume now that the special vertex $S$ lies on the
cycle. Since $\phi$ preserves the boundary components, it preserves
the orientation of the cycle. Therefore, $\phi(A)=A$ and $\phi(Z)=Z$.
In both cases, Lemma~\ref{lem:no-auto} implies that the automorphism
$\phi$ must be trivial. 
\end{proof}

Observe that trees without special vertices and trees with one special
vertex can have nontrivial automorphisms. 

Since leaves are flags, Lemma~\ref{lem:no-auto} has the following
immediate consequence. 

\begin{lemma}\label{lem:no-autoribb}
A labelled ribbon graph has no nontrivial automorphisms. 
\end{lemma}

The following statement is an immediate corollary of
Lemma~\ref{lem:no-auto1}. 

\begin{lemma}\label{lem:free}
The actions of $\Z_s$ on $\RR_{s;d_1,d_2}$ and of $\Z_{s_1}\times
\Z_{s_2}$ on $\RR_{s_1,s_2;d}$ defined in~\S\ref{ss:basiccomb} are free. 
\end{lemma}

\subsection{Operations on ribbon graphs 1}
\label{ss:op-graphs}

In this subsection we describe three operations on ribbon graphs:
disjoint union, cutting, and gluing. 

{\bf Disjoint union.}
Let $\Gamma_1$ and $\Gamma_2$ be two connected (extended) labelled graphs. Then the 
disjoint union $\Gamma_1\amalg\Gamma_2$ inherits a natural (extended)
labelling, putting $\Gamma_1$ before $\Gamma_2$ in all numberings. We
will assume this (extended) labelling on $\Gamma_1\amalg\Gamma_2$
unless otherwise specified. 

{\bf Cutting and gluing (see Figure~\ref{fig:cut-glue}). }
\begin{figure}
\begin{center}
\includegraphics[angle=0,origin=c,width=\textwidth]{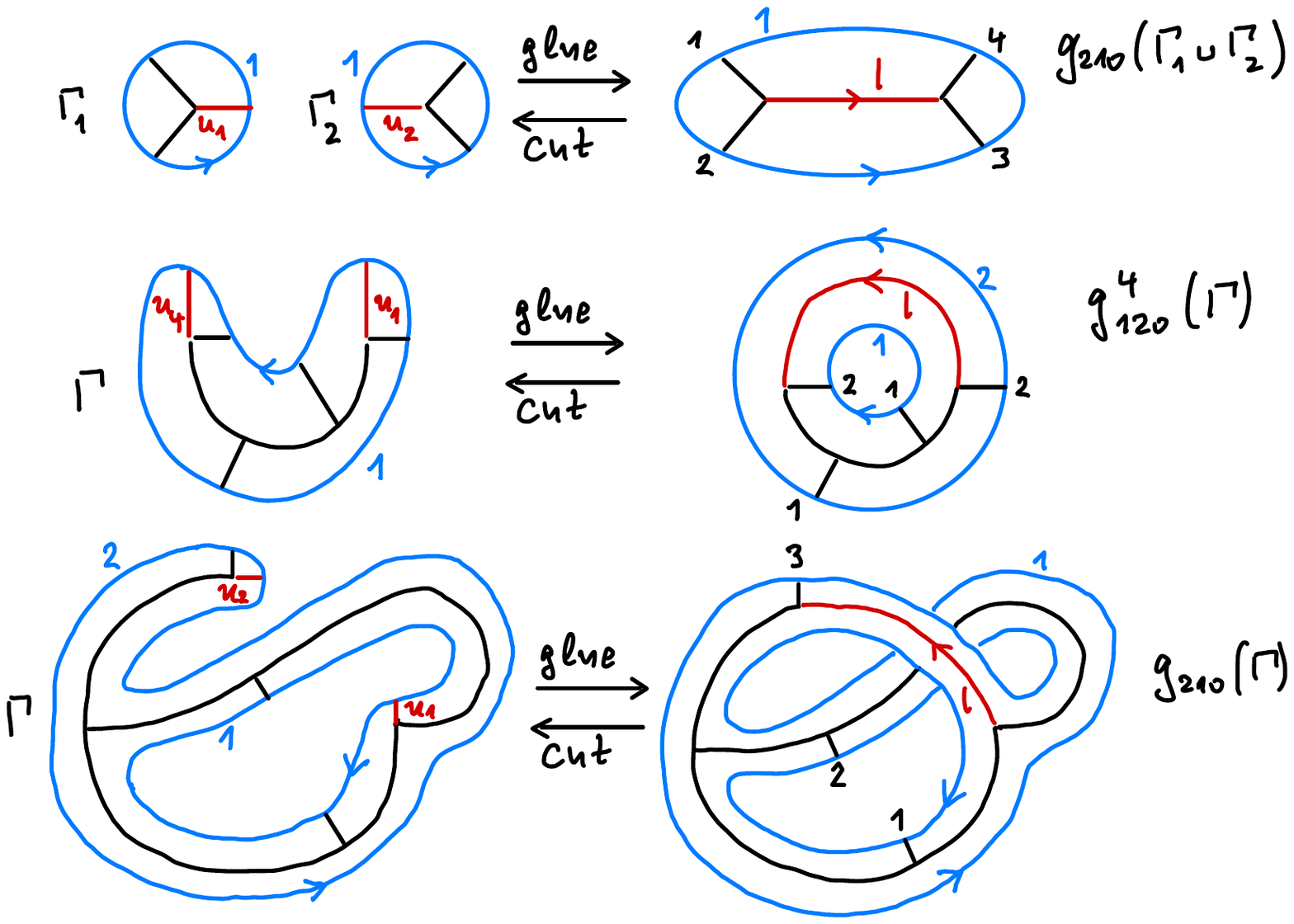}
\caption{Cutting and gluing}
\label{fig:cut-glue} 
\end{center}
\end{figure}
Given a marked graph $(\Gamma,l)$ we define a new graph $\Gamma\setminus
l$ by cutting open the marked edge $l$. Formally, this just means that
we remove the corresponding unordered pair $l=\{u,v\}$ from the set of
edges, so the interior flags $u,v$ become exterior ones. Note
that the resulting graph $\Gamma\setminus l$ is not marked, but it has
a distinguished unordered pair of leaves $u,v$. 

Conversely, given a graph $\Gamma$ with a distinguished
unordered pair of leaves $u,v$ we define a new graph
$\Gamma\cup\{u,v\}$ by gluing $u$ and $v$ to a new edge.
Formally, this just means that we add the unordered pair $l:=\{u,v\}$ to
the set of edges, so the exterior flags $u,v$ become interior
ones. Note that the resulting graph $\Gamma\cup\{u,v\}$ has no more
distinguished unordered pair of leaves, by it is marked by the
interior edge $l=\{u,v\}$.
The resulting cutting and gluing operations
$$
   \{\text{marked graphs}\} \longleftrightarrow \{\text{graphs with
     a given unordered pair of leaves}\} 
$$
are clearly inverse to each other, and they induce operations
$$
   \{\text{o-marked graphs}\} \longleftrightarrow \{\text{graphs with
     a given ordered pair of leaves}\} 
$$

Next we discuss how (extended) labellings get transferred under these operations.
Consider a graph $\Gamma$ with a given ordered pair of
leaves $u,v$. Suppose we are given a labelling for
$\Gamma$ such that $u$ comes before $v$ in the ordering of the
leaves, and denote the difference of their positions by $|v-u|$.
Assume that $u,v$ lie either on the same boundary component (in which
case we require that $u,v$ are not adjacent in the cyclic order on
that component), or on adjacent ones in the ordering of boundary components.  
The glued graph $\Gamma\cup(u,v)$ inherits a labelling by requalifying
the flags $u$ and $v$ from ``exterior'' to ``interior'', and keeping
the ordering of the remaining leaves.  

Assume now that an extension of the labelling of $\Gamma$ is given. 
Then the labelling of the glued graph $\Gamma\cup(u,v)$ inherits an extension by
putting the new oriented edge $l=(u,v)$ in first position for the edge
order.
Note that the vertex order of flags remains the same under the gluing operation.
Since we must move $u$ past $|v-u|-1$ leaves to put it next to
$v$ in the edge order, the sign exponents of the reordering maps of
the two graphs are related by 
\begin{equation}\label{eq:gluing-sign}
   \bar R_{\Gamma\cup(u,v)} \equiv \bar R_\Gamma + |v-u|-1.
\end{equation}
The preceding discussion gives rise to the following gluing operations.

\begin{definition}\label{def:g-prod}
Let $\Gamma$ be a labelled graph with at least two boundary components
having $s_1$ resp.~$s_2$ leaves on the first two boundary components. Assume that $s_1+s_2\ge 3$.
Let $u$ and $v$ be the first leaves on the first two boundary components.
We define the o-marked labelled graph
$$
  g_{210}(\Gamma):=\Gamma\cup (u,v).
$$
\end{definition}

\begin{definition}\label{def:g-coprod}
Let $\Gamma$ be a labelled graph with $s\geq 4$ leaves on the first
boundary component. Let $u_1,u_j$ be the first and the $j$-th leaf on
the first boundary components for some $j\in \{3,\dots,s-1\}$.
We define the o-marked labelled graph
$$
g_{120}^j(\Gamma):=\Gamma\cup (u_1,u_j).
$$
\end{definition}

Next we discuss some special cases in more detail.

{\bf Trees. }
Consider two labelled trees $\Gamma_1,\Gamma_2$ possibly with special
vertices as below: 
\begin{equation}\label{eq:g210-glue}
\begin{cases}
 \text{either } \Gamma_1\in \RR_{s_1;d_1} \text{ and } 
 \Gamma_2\in \RR_{s_2;d_2},\cr 
 \text{or } \Gamma_1\in \RR_{s_1} \text{ and } 
 \Gamma_2\in \RR_{s_2;d_1,d_2}.
 \end{cases}
\end{equation} 
Assume that $s_1+s_2\ge 3$. 
The resulting tree
$g_{210}(\Gamma_1\amalg\Gamma_2)$ has two special vertices, $s:=s_1+s_2-2$ leaves,
and a marked edge oriented towards the second special vertex. As
explained above, this tree is given a special labelling. 
The set of isomorphism classes of such trees 
will be denoted by $\RR_{s;d_1,d_2}^{ms}$. 
Dropping the condition on the labelling being 
special we get the set $\RR_{s;d_1,d_2}^{m}$.
Let $\RR_{s;d_1,d_2}^{\rm sep}$ denote the subset of trees in
$\RR_{s;d_1,d_2}^{m}$ for which the marked edge separates the two
special vertices (i.e., every path connecting the two special vertices
passes though the marked edge), and $\RR_{s;d_1,d_2}^{\rm nonsep}$ the
subset of trees for which the marked edge does not separate the
special vertices, so that we have a natural splitting
\begin{equation}\label{eq:Rom-splitting}
\RR_{s;d_1,d_2}^{m}=\RR_{s;d_1,d_2}^{\rm sep}\amalg \RR_{s;d_1,d_2}^{\rm nonsep}.
\end{equation}
The gluing operation gives thus two natural maps
\begin{equation}\label{eq:def-gl}
\begin{cases}
  gl_1:
  \coprod_{s_1+s_2=s+2}
  \RR_{s_1;d_1}\times 
  \RR_{s_2;d_2}\stackrel{\cong}{\longrightarrow} \RR_{s;d_1,d_2}^{\rm sep}
  \cap \RR_{s;d_1,d_2}^{ms}\subset \RR_{s;d_1,d_2}^{\rm sep}  ,\cr
  gl_2:
  \coprod_{s_1+s_2=s+2}
  \RR_{s_1}\times \RR_{s_2;d_1,d_2}\stackrel{\cong}{\longrightarrow} 
  \RR_{s;d_1,d_2}^{\rm nonsep}\cap
  \RR_{s;d_1,d_2}^{ms}\subset \RR_{s;d_1,d_2}^{\rm nonsep}.
\end{cases}
\end{equation}
The following statement is an immediate corollary of Lemma~\ref{lem:free}.

\begin{lemma}\label{lem:prodfree}
The actions of $\Z_s$ by cyclic relabelling of leaves on $\RR_{s;d_1,d_2}^{\rm sep}$
and on $\RR_{s;d_1,d_2}^{\rm nonsep}$ are free. The images of the maps $gl_1$
and $gl_2$ are fundamental domains for this action.
\end{lemma}

{\bf Circular graphs. }
For circular graphs we have three different cases.

{\bf Case 1. }
Consider $\Gamma\in\RR_{s:d}$ with $s\ge 4$.
The glued graph $g_{120}^j(\Gamma)$ is circular with one special vertex.
The numbers of leaves on the boundary components are given by
\begin{equation*}
  s_1=j-1,\qquad s_2=s-j-1.
\end{equation*}
The set of isomorphism classes of labelled graphs 
that arise this way will be denoted by 
$\RR_{s_1,s_2;d}^{cms}$. 
Note also that the marked edge
belongs to the cycle and is oriented positively in the direction of the cycle.
In the decoration ``$cms$''
the letter $c$ stands for ``marked edge belongs to the cycle''; and
the combination``$ms$'' stands for 
``marked special''. Since the orientation of the marked edge is canonical in this case,
we drop it from the notation. Note also 
that the special vertex of 
$g_{120}^j(\Gamma)$ lies on the cycle if and only if the flags 
$u_1,u_j$ belong to 
{\em different} rooted components 
of $\Gamma$.
Dropping the condition that the labelling be special we get the set of isomorphism classes of graphs 
$
\RR_{s_1,s_2;d}^{cm}
$. 
The above gluing operation thus gives us a natural map
\begin{equation}\label{eq:glueg1}
  gl_3:\coprod_{j\in \{3,\dots,s-1\}}\RR_{s:d}\stackrel{\cong}{\longrightarrow}
  \RR_{s_1,s_2;d}^{cms} \subset \RR_{s_1,s_2;d}^{cm}.
\end{equation}
The following statement is an immediate corollary of Lemma~\ref{lem:free}.

\begin{lemma}\label{lem:free1}
The action of $\Z_{s_1}\times \Z_{s_2}$ on 
$\RR_{s_1,s_2;d}^{cm}$ by cyclic relabelling of leaves is free.
The image $\RR_{s_1,s_2;d}^{cms}$ of the map $gl_3$ is a fundamental domain
for this action.
\end{lemma}

{\bf Case 2. }
Let $\Gamma_1\in\RR_{s_1,s_2}$ be a circular graph without special 
vertices, and $\Gamma_2\in\RR_{s_3;d}$ a tree with one 
special vertex. Assume that $s_1+s_3\ge 3$.
Let $\tau_{23}$ be the relabelling swapping the second and third
boundary components of the graph $\Gamma_1\amalg\Gamma_2$.
The glued graph $g_{210}((\Gamma_1\amalg\Gamma_2)\tau_{23})$
is circular with one special vertex.
Here we need the permutation $\tau_{23}$ to get the result of the gluing connected. The numbers of
leaves on the boundary components are given by
\begin{equation*}
  \wt s_1=s_1+s_3-2,\qquad \wt s_2=s_2.
\end{equation*}
The set of isomorphism classes of labelled graphs 
that arise this way will be denoted by $\RR_{\wt s_1,\wt s_2;d}^{ncb1s}$. 
Note also that both the special vertex of the glued graph and its marked edge {\em do not} belong to the cycle. Moreover, the marked edge belongs to the chain of edges 
connecting the cycle to the special vertex and is oriented ``from the cycle to the special vertex''. The 
special vertex lies between the cycle and the second boundary component. In the decoration ``$ncb1ms$''the combination $ncb$ stands for ``The marked edge does not belong to the cycle and lies between''.
The number ``$1$'' stands for the special vertex between the cycle and the second boundary component.
Dropping the condition that the labelling be special we get the set of isomorphism classes of graphs 
$\RR_{\wt s_1,\wt s_2;d}^{ncb1}$. 
The above gluing operation thus gives us a natural map
\begin{equation}\label{eq:glueg2}
  gl_4: \coprod_{\wt s_1=s_1+s_3-2}\RR_{s_1,s_2}\times 
  \RR_{s_3;d}\stackrel{\cong}{\longrightarrow}
  \RR_{\wt s_1,\wt s_2;d}^{ncb1s} \subset \RR_{\wt s_1,\wt s_2;d}^{ncb1}.
\end{equation}
The following statement is an immediate corollary of Lemma~\ref{lem:free}.

\begin{lemma}\label{lem:free2}
The action of $\Z_{\wt s_1}$ on 
$\RR_{\wt s_1,\wt s_2;d}^{ncb1}$ by cyclic relabelling of leaves is free.
The image $\RR_{\wt s_1,\wt s_2;d}^{ncb1s}$ of the map $gl_4$ is a
fundamental domain for this action. 
\end{lemma}

In analogy with the involution $\tau$
from~\eqref{eq:renumber-bdry-basic}, we have the involution (denoted
by the same letter)
\begin{equation}\label{eq:renumber-bdry}
\tau:\RR_{s_1,s_2;d}^m\stackrel{\cong}{\longrightarrow} \RR_{s_2,s_1;d}^m
\end{equation}
renumbering the two boundary components. We define
\begin{equation}\label{eq:renumber-bdry1}
\RR_{s_2,s_1;d}^{ncb2}:=\tau(\RR_{s_1,s_2;d}^{ncb1}).
\end{equation}

{\bf Case 3. }
Let $\Gamma_1\in\RR_{s_1}$ be a tree without special vertices, and 
$\Gamma_2\in\RR_{s_2,s_3;d}$ a circular graph with one special
vertex. Assume that  $s_1+s_2\ge 3$.
The glued graph $g_{210}(\Gamma_1\amalg\Gamma_2)$ is circular with one special vertex. The numbers of
leaves on the boundary components are given by
\begin{equation*}
  \wt s_1=s_1+s_2-2,\qquad
  \wt s_2=s_3.
\end{equation*}
The set of isomorphism classes of labelled graphs 
that arise this way will be denoted by 
$\RR_{\wt s_1,\wt s_2;d}^{nc1s}$.
In this case the special vertex of the glued graph
lies on the cycle if and only if the corresponding assertion is 
true about $\Gamma_2$. The marked 
edge lies between the first boundary component and the cycle
and is oriented from the boundary component to the cycle. Note that
the marked edge  {\em cannot} lie
between the cycle and the special 
vertex even if the special vertex does 
lie on the cycle. 
Dropping the condition that the labelling be special we get the set of isomorphism classes of graphs 
$\RR_{\wt s_1,\wt s_2;d}^{nc1}$. 
The above gluing operation thus gives us a natural map
\begin{equation}\label{eq:glueg4}
  gl_5: \coprod_{\wt s_1=s_1+s_2-2}\RR_{s_1}\times \RR_{s_2,s_3;d}
  \stackrel{\cong}{\longrightarrow} \RR_{\wt s_1,\wt s_2;d}^{nc1s}
  \subset \RR_{\wt s_1,\wt s_2;d}^{nc1}.
\end{equation}
The following statement is an immediate corollary of Lemma~\ref{lem:free}.

\begin{lemma}\label{lem:free4}
The action of $\Z_{\wt s_1}$ on $\RR_{\wt s_1,\wt s_2;d}^{nc1}$ 
by cyclic relabelling of leaves is free.
The image $\RR_{\wt s_1,\wt s_2;d}^{nc1s}$ of the map $gl_5$
is a fundamental domain for this action.
\end{lemma}

Recall from~\eqref{eq:renumber-bdry} the involution $\tau$ renumbering
the boundary components of the graph and define
\begin{equation}\label{eq:renumber-bdry2}
\RR_{s_2,s_1;d}^{nc2}:=\tau(\RR_{s_1,s_2;d}^{nc1}).
\end{equation}

Observe the following relation:
\begin{equation}
\label{eq:omcirc} 
\RR_{s_1,s_2;d}^{m}=
\RR_{s_1,s_2;d}^{cm}\amalg
\left(\RR_{s_1,s_2;d}^{ncb1}
\amalg 
\RR_{s_1,s_2;d}^{ncb2}\right)
\amalg
\left(\RR_{s_1,s_2;d}^{nc1}
\amalg
\RR_{s_1,s_2;d}^{nc2}\right).
\end{equation}

\subsection{Operations on ribbon graphs 2}\label{ss:op-graphs2}

In this subsection we introduce three more operations: duality,
attaching a leg, and attaching a tree.

{\bf Duality (see Figure~\ref{fig:duality}). }
\begin{figure}
\begin{center}
\includegraphics[width=\textwidth]{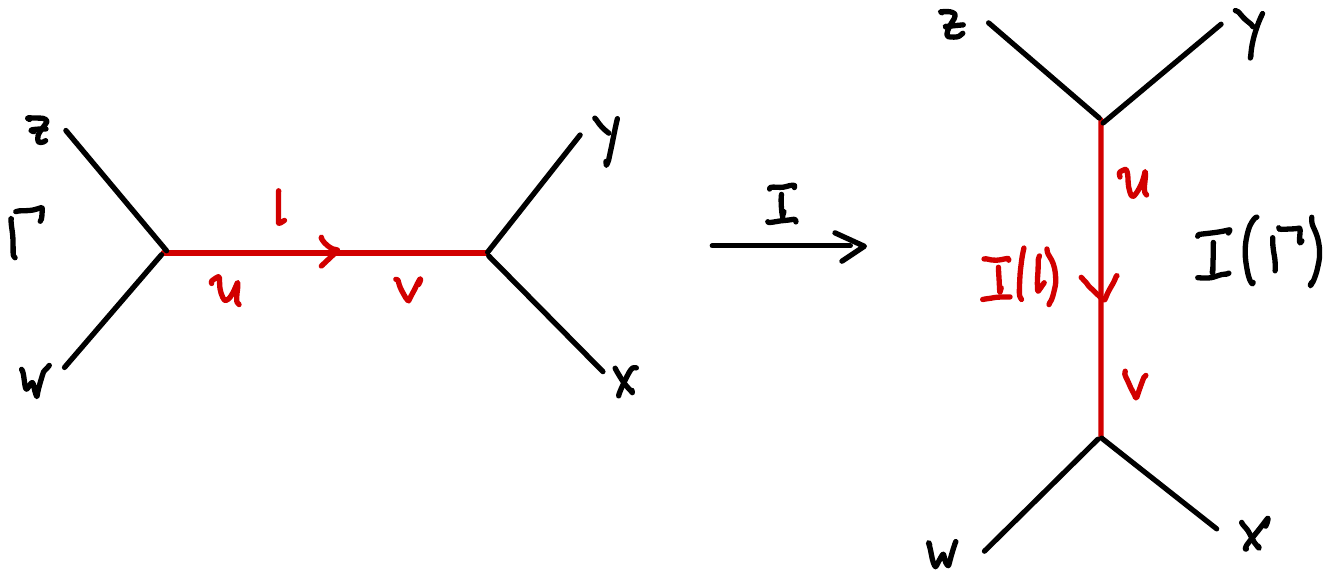}
\vspace{-3.5cm}
\caption{The duality operation $I$}
\label{fig:duality} 
\end{center}
\end{figure}
The following duality operation on ribbon graphs plays a crucial role
in~\cite{Cieliebak-Volkov}.
Let $(\Gamma,l)$ be an o-marked labelled trivalent ribbon graph with
$l=(u,v)$ the oriented marked nonspecial edge. Let $(z,w,u)$ and $(v,x,y)$ be the
two vertices connected by $l$ as shown in Figure~\ref{fig:duality}. We
define the o-marked labelled graph $(I(\Gamma),I(l))$ using the same
set of flags, but assembling them into vertices and edges slightly
differently. Namely, we let $I(l):=(u,v)$ be the oriented marked edge
of $I(\Gamma)$ and $(y,z,u)$ and $(v,w,x)$ be its adjacent vertices,
see Figure~\ref{fig:duality}. The other vertices and edges stay the same. 
Geometrically, the operation $I$ is cutting out a subtree with $4$
leaves and pasting back the dual subtree. In particular, the type of the graph 
remains the same. Since the two graphs have the same boundary
components and leaves, $I(\Gamma)$ inherits the labelling from $\Gamma$.   
 
{\bf Attaching a leg (see Figure~\ref{fig:leg}). }
\begin{figure}
\begin{center}
\includegraphics[origin=c,width=\textwidth]{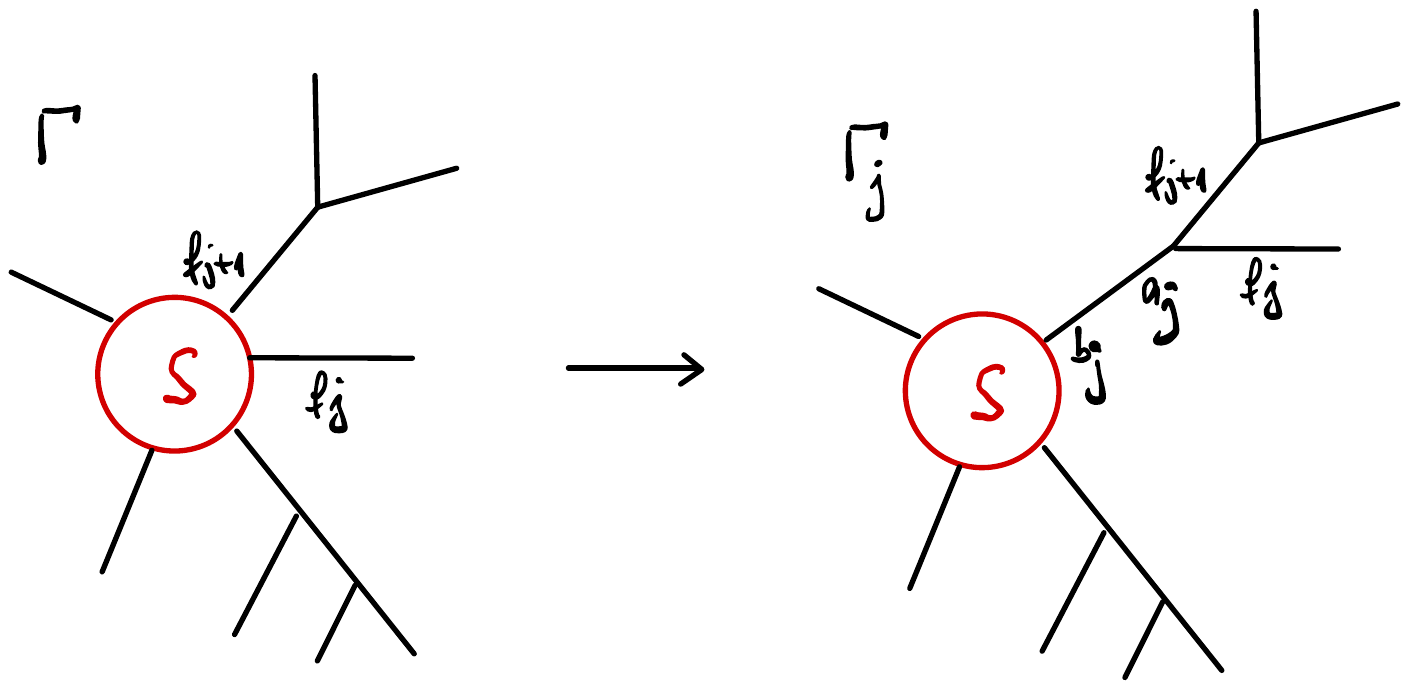}
\vspace{-3cm}
\caption{Attaching a leg}
\label{fig:leg} 
\end{center}
\end{figure}
Let $\Gamma$ be a labelled graph with special vertices and let 
$$
S=(f_1,\dots,f_d)
$$
be a special vertex of $\Gamma$. Assume that $d\ge
2$ and its flags are ordered according to our standardization.
Fix some $j\in \{1,\dots,d\}$. 
We add two more flags $a_j$ and $b_j$ to the set of flags of $\Gamma$ and two more
relations: an edge $\{a_j,b_j\}$ and a vertex $[f_j,f_{j+1},a_j]$,
where we set $f_{j+1}:=f_1$ for $j=d$. Note that the new ordered set of
special flags at $S$ is
$$
  (f_1,\dots,f_{j-1},b_j,f_{j+2},\dots,f_d),
$$
so the flags $f_j$ and $f_{j+1}$ have become nonspecial. 
If $j=d$, then the 
new ordered set of
special flags at $S$ is
$$
  (b_d,f_2,\dots,f_{d-1}).
$$
Note that this ordering is compatible with our standardization of
extensions of labellings in~\S\ref{ss:basiccomb}.
We call the resulting labelled graph $\Gamma_j$.

If the labelling of $\Gamma$ is given an extension, then we give the
labelling of $\Gamma_j$ the following extension: the new edge
$(a_j,b_j)$ is given the first position in the ordering of edges and
oriented as written, and the new vertex $(f_j,f_{j+1},a_j)$
is given the last number in the ordering 
of nonspecial vertices with its flags ordered as written.

{\bf Attaching a tree. } 
Let $\Gamma$ be a graph and $T$ a rooted tree 
(see the beginning of~\S\ref{ss:specvert} ). Let $l$ be a leaf of 
$\Gamma$ and $v$ its vertex. The result of {\em attaching $T$
to $\Gamma$ along $l$} is the quotient
$\Gamma\amalg T/\sim$, where the root vertex of $T$ is identified with
$v$ and the root flag of $T$ is identified with $l$.

\section{Integrals over configuration spaces}\label{sec:intconfig}

In this section we continue our discussion of integrals over
configuration spaces. Their analytical aspects were dealt with
in~\S\ref{sec:general graphs} and~\S\ref{sec:analysis}. Now we give
precise definitions with regards to signs and orientations, prove
their independence of the extension of a labelling, and discuss the
effect of a change of labelling. 

Throughout this section, $M$ denotes a closed oriented manifold of
dimension $n$, and we use the notation and terminology
from~\S\ref{sec:graphs}.

\subsection{Configuration spaces associated to ribbon graphs}
\label{ss:configspec}

Let $\Gamma$ be a ribbon graph, possibly with special vertices, with
an extended labelling. 
We denote by $d_j$ ($j\in\Ver$) the valencies of the nonspecial
vertices, by $d$ the total number of special flags, by $e$ the number
of edges, by $s$ the number of leaves, and by $f$ the number of flags.  
Recall from~\S\ref{ss:basiccomb} the vertex and edge orders induced by
the extended labelling and the reordering permutation $\bar R_\Gamma$. 

Recall from~\S\ref{ss:graphs} the following spaces: 
$X_\Gamma,Y_\Gamma$ with the reordering diffeomorphism
\begin{equation*}
  R_\Gamma:Y_\Gamma=\Bigl(\prod_{j\in\Ver}M^{d_j}\Bigr)\times
M^d\stackrel{\cong}{\longrightarrow} X_\Gamma=(M^2)^e\times M^s;
\end{equation*}
the slim diagonals $M\cong M_j\subset M^{d_j}$; the vertex diagonal
\begin{equation*}
  \Delta_\ver = \Delta_{\ver}^\Gamma = \prod_{j\in\Ver} M_j\subset \prod_{j\in\Ver} M^{d_j};
\end{equation*}
the diagonal $\Delta_2\subset M^2$; the double diagonal
$$
 \Delta_2^l = (M^2\times\cdots\times M^2\times \Delta_2\times
  M^2\times\cdots\times M^2)\times M^s\subset X_\Gamma
$$
corresponding to an edge $l\in \Edge$; and the (fat) {\em edge diagonal}
$$
   \Delta_2^\Gamma = \bigcup_{l\in\Edge(\Gamma)}\Delta_2^l\subset X_\Gamma.
$$
Moreover, we define
\begin{align*}
  \ol{\Delta}^l_{2} &:= R_\Gamma^{-1}(\Delta_2^l) \cap (\Delta_{\ver}\times
   M^d),\qquad l\in\Edge(\Gamma), \cr
   \ol\Delta_2^Y &:= R_\Gamma^{-1}(\Delta_2^\Gamma)\cap (\Delta_{\ver}\times
   M^d)=\bigcup_{l\in \Edge(\Gamma)}
   \ol{\Delta}^l_{2}
   \subset\Delta_{\ver}\times M^d\subset Y_\Gamma.
\end{align*}
The relevant uncompactified configuration space is 
\begin{equation}\label{eq:open-config}
   \overset\circ{\XX_\Gamma}:= (\Delta_{\ver} \times M^d)\setminus\ol\Delta_2^Y.
\end{equation}
The projection 
\begin{equation}\label{eq:pi-signs}
  \pi:\Delta_{\ver}\times M^d \longrightarrow M^d
\end{equation}
is the canonical map forgetting the factor $\Delta_{\ver}$.

\subsection{Definition of the integrals}\label{ss:Integrals}

We retain the notation from~\S\ref{ss:configspec}.
Let $\wt G\in\Om^{n-1}(\wt M^2)$ be a propagator as
in~\S\ref{sec:prop},
and $G$ its integrable pushforward to $M^2$
(which is smooth outside the diagonal).
Integrability of $G$ follows from Lemma~\ref{lem:graph-L1}
for the graph with two vertices (both nonspecial) and one edge connecting them.
Let 
$$ 
   \alpha = \alpha_1\otimes\dots\otimes\alpha_s 
$$
be a decomposable tensor of differential forms $\alpha_j\in\Om^*(M)$
associated to the leaves of $\Gamma$. We call such $\alpha$ {\em
  adapted to $\Gamma$}. It induces via cross product a form  
$$
  \cross(\alpha) := \alpha_1\times\cdots\times\alpha_s \in \Om^*(M^s).
$$
To this data we associate the following differential forms, where for the
last two we assume in addition that a marked oriented edge $l$ 
has been chosen:
\begin{itemize}
\item $G^e:=G\times\dots\times G$ on $(M^2)^e$;
\item $G^e(\alpha):=G^e\times\cross(\alpha)$ on $X_\Gamma$; 
\item $G^e_l:=G\times\dots\times dG\times\dots\times G$ on
  $(M^2)^e$ (with $dG$ at the position $n(l)$ of $l$);
\item $G^e_l(\alpha):=G^e_l\times\cross(\alpha)$ on $X_\Gamma$.
\end{itemize}
We define
\begin{equation}\label{eq:defI}
   I_\Gamma(\alpha) := (-1)^{\bar R_\Gamma
     +(n-1)\eta_3(\Gamma)}\int_{\Delta_\ver}R_\Gamma^*G^e(\alpha) 
\end{equation}   
\begin{equation}\label{eq:defIl}   
   I_{\Gamma,l}(\alpha) := (-1)^{\bar R_\Gamma
     +(n-1)(\eta_3(\Gamma)+n(l)-1)}\int_{\Delta_\ver}R_\Gamma^*G^e_l(\alpha). 
\end{equation}
Here the integral sign denotes the pushforward along the map 
$\pi$ in~\eqref{eq:pi-signs} in the sense of~\S\ref{sec:fibre} (which
equals the usual Lebesgue integral in the case $d=0$ without special vertices). 
The number $n(l)$ is the position of the marked edge $l$ in the
ordering of edges, and $\eta_3(\Gamma)$ is the sign exponent from
Definition~\ref{def:eta3}. 

{\bf Existence of the integrals. } 
We asssume the following condition:
\begin{equation}\label{eq:no-selfloops}
\text {The graph $\Gamma$ does not have self-loops at nonspecial vertices.} 
\end{equation}
Then, by Lemma~\ref{lem:graph-L1}, the integrands above define
integrable forms on $\Delta_\ver\times M^d$ (which are smooth on its full
measure open subspace $\overset\circ{\XX_\Gamma}$ defined
in~\eqref{eq:open-config}).
According to Lemma~\ref{lem:push-forward-int}, the pushforwards
in~\eqref{eq:defI} and~\eqref{eq:defIl} are therefore well-defined as
integrable forms on $M^d$.

The subsequent discussion follows closely the one in~\S\ref{sec:general
  graphs} and~\S\ref{sec:analysis}, this time keeping track of the
reordering diffeomorphism $R_\Gamma$. For now we restrict to the case
$\WW=\{\pt\}$. 
The reordering diffeomorphism $R_\Gamma:Y_\Gamma\longrightarrow X_\Gamma$
lifts to a diffeomorphism
\begin{equation}\label{eq:defYorient}
  \wt R_\Gamma:\wt Y_\Gamma=\Bl(Y_\Gamma,R_\Gamma^{-1}(\Delta_2^\Gamma))
\longrightarrow \wt X_\Gamma=\Bl(X_\Gamma,\Delta_2^\Gamma).
\end{equation}
Recall that the integrable form $G^e(\alpha)$ on $X_\Gamma$ defined
above lifts to a smooth form $\wt G^e(\alpha)$ on $\wt X_\Gamma$,
whose pullback $\wt R_\Gamma^*\wt G^e(\alpha)\in\Om^*(\wt Y_\Gamma)$
we integrate over the compactified configuration space $\XX_\Gamma$.

\begin{remark}\label{rem:deg-reg}
(A) 
Assume that $\Gamma$ is a tree with one special vertex. Then by
Lemma~\ref{lem:PTtransverse} the compactification
$$
  \XX_\Gamma = \Bl(\Delta_\ver\times M^d,\ol\Delta_2^Y)
$$ 
is a manifold with corners.
Moreover, by the obvious extension of 
Lemma~\ref{lem:tree-fibre} from the case of a univalent special vertex
to a $d$-valent one, the projection $\pi$ from~\eqref{eq:pi-signs}
extends to a fibration of compact manifolds with corners that we denote by 
the same letter
$$
  \pi:\XX_\Gamma\longrightarrow M^d
$$
and we get
$$
  \int_{\Delta_\ver} R_\Gamma^*G^e(\alpha) = \pi_*\wt R_\Gamma^*\wt G^e(\alpha).
$$
By Remark~\ref{rem:fibre-smooth} the right hand side is smooth, so 
$I_\Gamma(\alpha)$ is smooth in this case.
Moreover, $I_\Gamma(\alpha)$ is the cross product of a decomposable
tensor of differential forms on $M$, which has homogeneous degree if
the tensor $\alpha$ does.

(B) Assume now that $\Gamma$ is a tree with two special vertices,
possibly with a marked edge $l$. Then $I_\Gamma(\alpha),I_{\Gamma,l}(\alpha)\in
\Om_{int}^*(M^{d_1}\times M^{d_2})$ are integrable forms in the sense
of Definition~\ref{def:L1}. 
(In fact, they are smooth outside the diagonal $D\subset M^d$ where the
variables associated to the flags $A$ and $Z$ in Figure~\ref{fig:2S}
are equal; we will not need this fact, so we omit the proof.)
For future use we split the forms according to their bidegrees as
\begin{equation}\label{eq:bi-deg}
  I_\Gamma(\alpha) = \sum_{p,q}I_\Gamma^{p,q}(\alpha), \qquad 
  I_{\Gamma,l}(\alpha) = \sum_{p,q}I_{\Gamma,l}^{p,q}(\alpha). 
\end{equation}
\end{remark}

{\bf Types of graphs. }
For the subsequent discussion of signs we assume in addition:
\begin{equation}\label{eq:odd-valent}
\text{All nonspecial vertices of $\Gamma$ are odd-valent.} 
\end{equation}
For the proofs of our main results in~\S\ref{sec:rel-stringtop} we
will only need graphs $\Gamma$ of the following types (see~\S\ref{ss:specvert}):
\begin{itemize}
\item (o) trivalent trees;
\item (i) trees with one special vertex (see Figure~\ref{fig:1S});
\item (ii) trees with two special vertices (see Figure~\ref{fig:2S});
\item (iii) circular graphs without special vertices;
\item (iv) circular graphs with one special vertex (see Figure~\ref{fig:circ});
\end{itemize}
as well as disconnected graphs of the types
$(o)+(ii)$, $(i)+(i)$, $(o)+(iv)$, $(i)+(iii)$.
For the definition of the Maurer-Cartan element in~\S\ref{ss:MCan} we will need 
\begin{itemize}
\item (s) connected trivalent ribbon graphs without special vertices
  of genus $g\ge 0$ with $\ell\ge 1$ boundary components.
\end{itemize}
Recall from~\S\ref{ss:specvert} that for the types of graphs above, all
nonspecial vertices are trivalent and each boundary component has at
least one leaf ending on it. Note that this implies
conditions~\eqref{eq:no-selfloops} and~\eqref{eq:odd-valent}. 


\begin{remark}
Recall that graphs of types (ii) and (iv) above have two distinguished
special flags $A,Z$. In these cases let $D\subset M^d$ be the subset
in which the variables corresponding to $A$ and $Z$ are equal; in all
other cases set $D=\emptyset$. Then  
the integrable forms on $M^d$ defined by~\eqref{eq:defI}
and~\eqref{eq:defIl} are actually smooth away from $D$. We will not
use this in the sequel, but is useful to keep in mind.
\end{remark}

\subsection{Change of extension}\label{ss:intconfigext}

In this subsection we show that $I_\Gamma$ and $I_{\Gamma,l}$ do not
depend on the extension of the labelling, 
and we describe the effect of rotating the numbering
at a special vertex.

Recall the standardizations on the extension of a labelling of 
described after Definition~\ref{def:labelling},
which fix in all cases the positions of the special vertices in
the numbering of vertices and the numbering of flags around each
special vertex.
Recall from~\S\ref{ss:basiccomb} that a change in the order and
orientation of edges is described by precomposition of the edge order
with some permutation $\tau$ of $\{1,\dots,f\}$. 
A change in the order of the nonspecial vertices and numbering of 
the flags around each nonspecial vertex is described by 
precomposition of the vertex order with some permutation $\sigma^{-1}$.
We write the action on the graph as $\Gamma\mapsto \sigma\Gamma\tau$.
A permutation $\rho$ of $\{1,\dots,f\}$ induces a diffeomorphism
$$
   M^\rho:M^f\to M^f,\qquad (x_1,\dots,x_f)\mapsto (x_{\rho(1)},\dots x_{\rho(f)}).
$$
It follows directly from these definitions and~\eqref{eq:reod} that 
\begin{equation}\label{eq:R-trans}
  \bar R_{\sigma\Gamma\tau}=\sigma\circ \bar R_\Gamma\circ\tau \quad\text{and}\quad
  R_{\sigma\Gamma\tau}=M^\tau\circ R_\Gamma\circ M^\sigma.
\end{equation} 
Moreover, by~\cite[equation~(99)]{Cieliebak-Volkov} we have the
invariance properties 
\begin{equation}
\begin{aligned}\label{eq:formstrans}
  (M^\tau)^*G^e(\alpha) &= (-1)^\fs G^e(\alpha),\cr
   (M^\tau)^*G^e_l(\alpha) &= (-1)^{\fs+(n-1)(n'(l)-n(l))}G^e_l(\alpha)
\end{aligned}
\end{equation}
where $n'(l)$ is the position of the marked edge $l$ in the
  numbering of edges of $\Gamma\tau$, and using the sign exponent
  $\eta_3$ from Definition~\ref{def:eta3} we set 
$$
  \fs:=\tau+(n-1)(\eta_3(\Gamma\tau)-\eta_3(\Gamma)).
$$

\begin{lemma}\label{lem:ext-indep-basic}
 In the above setting we have
\begin{equation}\label{eq:ext-indep-downst}
 (-1)^{\bar R_{\sigma\Gamma\tau}+(n-1)\eta_3(\sigma\Gamma\tau)}R_{\sigma\Gamma\tau}^*G^e(\alpha)=
 (-1)^{\bar R_\Gamma+(n-1)\eta_3(\Gamma)+n\sigma}
 (M^\sigma)^*R_\Gamma^*G^e(\alpha).
\end{equation}
\end{lemma}

\begin{proof}
We compute
\begin{equation}\label{eq:integrand-compute}
\begin{aligned}
&(-1)^{\bar R_{\sigma\Gamma\tau}+(n-1)\eta_3(\sigma\Gamma\tau)}R_{\sigma\Gamma\tau}^*G^e(\alpha)\cr
&\stackrel{\eqref{eq:R-trans}}{=}(-1)^{\bar R_{\sigma\Gamma\tau}+(n-1)\eta_3(\sigma\Gamma\tau)}
(M^\sigma)^*R_\Gamma^*(M^\tau)^*G^e(\alpha)\cr
&\stackrel{\eqref{eq:formstrans}}{=}(-1)^{\bar R_{\sigma\Gamma\tau}+(n-1)\eta_3(\sigma\Gamma\tau)+\fs}
(M^\sigma)^*R_\Gamma^*G^e(\alpha).
\end{aligned}
\end{equation}
Note that the permutation $\sigma$ changing the vertex order consists
of cyclic relabellings at nonspecial vertices, which are even
permutations because of assumption~\eqref{eq:odd-valent}, and a
permutation $\sigma_\ver$ of the vertices.
The sign exponent $\eta_3$ associated to the graph $\sigma\Gamma\tau$
differs from that associated to the graph $\Gamma\tau$ by the sign
exponent $\sigma_\ver=\sigma$ and we get
$$
  \eta_3(\sigma\Gamma\tau)-\eta_3(\Gamma\tau)=\sigma.
$$
Using this and the first equation in~\eqref{eq:R-trans}, we compute the sign exponent
\begin{align*}
&\bar R_{\sigma\Gamma\tau}+(n-1)\eta_3(\sigma\Gamma\tau)+\fs\cr
&=(\bar R_\Gamma+\sigma+\tau)+(n-1)(\eta_3(\Gamma\tau)+\sigma)+[\tau+(n-1)(\eta_3(\Gamma\tau)-\eta_3(\Gamma))]\cr
&=\bar R_\Gamma+(n-1)\eta_3(\Gamma)+n\sigma.
\end{align*}
This together with~\eqref{eq:integrand-compute} gives us equation~\eqref{eq:ext-indep-downst}.
\end{proof}

The following lemma was established in~\cite[Lemma~7.5]{Cieliebak-Volkov}
in the case without special vertices. 

\begin{lemma}\label{lem:extindep}
In the above setting we have
$$
  I_{\sigma\Gamma\tau}(\alpha)=I_{\Gamma}(\alpha) \quad\text{and}\quad
  I_{\sigma\Gamma\tau,l}(\alpha)=I_{\Gamma,l}(\alpha).
$$
\end{lemma}

\begin{proof}
We apply the pushforward $\int_{\Delta_\ver}$ to both sides of
equation~\eqref{eq:ext-indep-downst} to get
\begin{align*}
I_{\sigma\Gamma\tau}(\alpha)
&\stackrel{\eqref{eq:ext-indep-downst}}{=}
(-1)^{\bar R_\Gamma+(n-1)\eta_3(\Gamma)+n\sigma}
\int_{\Delta_\ver}(M^\sigma)^*R_\Gamma^*G^e(\alpha)
\cr
&=
(-1)^{\bar R_\Gamma+(n-1)\eta_3(\Gamma)}
\int_{\Delta_\ver}R_\Gamma^*G^e(\alpha)\cr
&=I_\Gamma(\alpha).
\end{align*}
Here the last equality is the definition of the operation
$I_\Gamma$. The second equality follows by invariance of integration under
$M^\sigma|_{\Delta_\ver}$ (Corollary~\ref{cor:pullback-aut})
with the sign exponent $n\sigma_\ver=n\sigma$, where $\sigma_\ver$
is the permutation of vertices induced by $\sigma$ as in the proof
of Lemma~\ref{lem:ext-indep-basic}.
This proves the first assertion; the proof of the second one is analogous.
\end{proof}

In view of Lemma~\ref{lem:extindep}, we can unambiguously define
$I_\Gamma$ and $I_{\Gamma,l}$ for any {\em labelled} graph $\Gamma$
(with a marked edge in the second case) satisfying
conditions~\eqref{eq:no-selfloops} and~\eqref{eq:odd-valent}, 
by applying definitions~\eqref{eq:defI} and~\eqref{eq:defIl} for any
extension of the labelling. 

Consider now a tree $\Gamma$ with one special $d$-valent vertex and
generalized labelling. Recall from~\S\ref{ss:basiccomb} that the
latter is a numbering of the leaves (a labelling) and a numbering of
the special flags, both compatible with the given cyclic orders. 
We can extend this to an ``extended labelling'' without the
standardization for its special flags and define $I_\Gamma$ and
$I_{\Gamma,l}$ by~\eqref{eq:defI} and~\eqref{eq:defIl}. By the
preceding discussion these definitions do not depend on the
``extension'', but they depend on the generalized labelling as follows. 
Recall from equation~\eqref{eq:def-sigma-an} the analytic action $\sigma_{an}$ of
$\sigma\in\Z^d$ on $\Om^*(M^d)$. 
The following lemma was established in~\cite[Lemma~7.9]{Cieliebak-Volkov}
in the case without special vertices. 

\begin{lemma}\label{lem:rotation-spec}
Let $\Gamma$ be a tree with one special $d$-valent vertex and a
generalized labelling. For $\sigma\in\Z_d$, consider the graph 
$\sigma\Gamma$ with special flags cyclicly renumbered according to $\sigma$.
Then 
\begin{equation}\label{eq:rotation-spec}
  I_{\sigma\Gamma}(\alpha) = \sigma_{an}^{-1}(I_\Gamma(\alpha)).
\end{equation}
\end{lemma}

\begin{proof}
Equation~\eqref{eq:R-trans} with $\tau=\id$ yields
$$
  R_{\sigma\Gamma}= R_\Gamma\circ M^\sigma.
$$  
Since the numberings of vertices and edges and the orientations of
edges are the same for $\sigma\Gamma$ and $\Gamma$, we have  
$$
\eta_3(\sigma\Gamma)=\eta_3(\Gamma).
$$
Using this, we compute
\begin{align*}
I_{\sigma\Gamma}(\alpha)
&\stackrel{(1)}{=}(-1)^{\bar R_{\sigma\Gamma}+
(n-1)\eta_3(\sigma\Gamma)}
\int_{\Delta_\ver}(M^\sigma)^*
R_\Gamma^*G^e(\alpha)\cr
&\stackrel{(2)}{=}(-1)^{\bar R_\Gamma+\sigma+(n-1)\eta_3(\Gamma)}
(M^\sigma)^*\int_{\Delta_\ver}
R_\Gamma^*G^e(\alpha)\cr
&\stackrel{(3)}{=}(-1)^\sigma (M^\sigma)^* 
I_\Gamma(\alpha))\cr
&\stackrel{(4)}{=}\sigma_{an}^{-1}(I_\Gamma(\alpha)).
\end{align*}
Here equality~(1) follows from definition~\eqref{eq:defI} of $I_\Gamma$;
equality~(2) from the first equation in~\eqref{eq:R-trans} and
invariance of integration under $M^\sigma$
(Corollary~\ref{cor:pullback-aut}, noting that $M^\sigma$ is the
identity between the fibres of~\eqref{eq:pi-signs});
equality~(3) again from the definition of $I_\Gamma$;
and equality~(4) from Lemma~\ref{lem:acta} below with the trivial 
partition of $d$, $\eta=\sigma$, $\eta_b=0$, and $\beta=I_\Gamma(\alpha)$.
\end{proof}

Recall from~\S\ref{ss:basiccomb} that the cyclic group $\Z_d$ acts
freely by cyclic renumbering of the special flags on the set
$\RR_{s;d}^\gen$ of isomorphism classes of generalized labelled trees
with one special $d$-valent vertex and $s$ leaves. 
Recall the operation $N_{an}$ on $\Om^*(M^d)$  
from~\eqref{eq:def-Nalg}. Now Lemma~\ref{lem:rotation-spec}
yields the following result.

\begin{lemma}\label{lem:rotation-spec-cor}
Let $\Gamma$ be a tree with one special $d$-valent vertex and $s$
leaves and a generalized labelling. Let $\Z_d\Gamma$ denote the orbit
of $\Gamma$ under the free $\Z_d$-action on the set $\RR_{s;d}^\gen$. 
Then  
\begin{equation}\label{eq:rotation-spec-cor}
  \sum_{\wh\Gamma\in \Z_d\Gamma}I_{\wh\Gamma}
  = \sum_{\sigma\in \Z_d}I_{\sigma\Gamma}
  = \sum_{\sigma\in \Z_d}\sigma_{an}^{-1}\circ I_{\Gamma}
  = N_{an}\circ I_\Gamma.
\end{equation}
\end{lemma}

\subsection{Change of labelling}\label{ss:intconfiglab}

Now we examine what happens if we change the labelling. 
Let $\Gamma$ be an
extended labelled graph, possibly with a marked
edge $l$. Let $\eta\in S(\bs)\subset
S_s$ be a relabelling associated to a partition $\bs$ of $s$ as
in~\S\ref{ss:basiccomb}. By a slight abuse of 
language, we denote by $\eta$ also the induced permutation of 
$\{1,\dots,f\}$ acting as the identity on interior flags.
As usual, $\Gamma\eta$ denotes the graph $\Gamma$ with the new
labelling. Equation~\eqref{eq:R-trans} specializes to
\begin{equation}\label{eq:R-trans-2}
  \bar R_{\Gamma\eta} = \bar R_\Gamma\circ\eta\quad\text{and}\quad
  R_{\Gamma\eta} = M^\eta\circ R_\Gamma. 
\end{equation}
Recall from~\S\ref{ss:gradedvect} the three actions of $\eta$ on
$\alpha\in(\Om^*(M))^{\otimes s}$: the naive action $\eta(\alpha)$,
the analytic action $\eta_{an}(\alpha)$, and the algebraic action 
$\eta_{alg}(\alpha)$. 
Recall also that $\eta_b$ denotes the permutation of boundary
components corresponding to $\eta$. 


\begin{lemma}[{\cite[\S7]{Cieliebak-Volkov}}]\label{lem:acta}
For $\eta\in S(\bs)$ and $\beta\in\Om^*(M)^{\otimes s}$ we have
$$
   (M^\eta)^*\cross(\beta) = (-1)^{\eta+(n-1)\eta_b} \cross\bigl(\eta_{an}^{-1}(\beta)\bigr).
$$
\end{lemma}

Since $M^\eta$ acts only on the variables corresponding to $\alpha$,
we compute:
\begin{align*}
   R_{\Gamma\eta}^*G^e(\alpha) 
   &\stackrel{\eqref{eq:R-trans-2}}{=} R_\Gamma^*(M^\eta)^*G^e(\alpha) 
   = R_\Gamma^*\bigl(G^e\times(M^\eta)^*\cross(\alpha)\bigr) \cr
   &\stackrel{\rm Lemma~\ref{lem:acta}}{=} (-1)^{\eta+(n-1)\eta_b}R_\Gamma^*G^e(\eta_{an}^{-1}\alpha). 
\end{align*}
Combining this with the sign used in the definitions of
$I_{\Gamma}(\alpha)$ and $I_{\Gamma,l}(\alpha)$, we get
\begin{equation}\label{eq:good-an-equiv}
\begin{aligned}
   &(-1)^{\bar R_{\Gamma\eta}+(n-1)\eta_3(\Gamma\eta)}R_{\Gamma\eta}^*G^e(\alpha)\cr
   &= (-1)^{\bar R_{\Gamma\eta}+\eta+(n-1)\eta_b+(n-1)\eta_3(\Gamma\eta)}R_\Gamma^*G^e(\eta_{an}^{-1}\alpha)\cr
   &=(-1)^{\bar R_{\Gamma}+(n-1)\eta_3(\Gamma)}
   R_\Gamma^*G^e(\eta_{an}^{-1}\alpha).
\end{aligned}
\end{equation}
Here the first equality follows from the equality just above, and the
second one follows from~\eqref{eq:R-trans-2} and
$\eta_3(\Gamma\eta)=\eta_3(\Gamma)+\eta_b$ (see Definition~\ref{def:eta3}).
The same holds with $G_l^e(\alpha)$ in place of $G^e(\alpha)$.
In view of definitions~\eqref{eq:defI} and~\eqref{eq:defIl}, this implies
\begin{equation}
  I_{\Gamma\eta}(\alpha) = I_\Gamma(\eta_{an}^{-1}(\alpha))\quad\text{and}\quad
  I_{\Gamma\eta,l}(\alpha) = I_{\Gamma,l}(\eta_{an}^{-1}(\alpha)).
\end{equation}

\begin{definition}\label{def:good-exp}
Let $\Gamma$ be a 
labelled graph, possibly with special vertices and a marked edge $l$. 
If $\Gamma$ has two special vertices, we fix in addition a bidegree $(p,q)$.
Let $\alpha\in (\Om^*(M))^s$ be a decomposable tensor of homogeneous degree. 
A sign exponent depending on $(\Gamma,\alpha,p,q)$ is called {\em
  nonessential} if it depends only on the total degrees of $\alpha$
and $I_\Gamma(\alpha)$, the bidegree $(p,q)$, and the numerics of $\Gamma$,
but {\em not} on the individual tensor factors of $\alpha$ and
$I_\Gamma(\alpha)$ or the labelling of $\Gamma$.
\end{definition}

\begin{definition}\label{def:good-alg}
Let $\Gamma$ be a labelled graph, $N$ any manifold, and
$\Om_{int}^*(N)$ the space of integrable forms in Definition~\ref{def:L1}.
Let $\bs$ be the partition induced by $\Gamma$, and $S(\bs)$ the set
of relabellings defined by equation~\eqref{eq:relab}. 
A linear map
$$
  J_\Gamma:(\Om^*(M))^{\otimes s}\longrightarrow \Om_{int}^*(N)
$$
is called a {\em good operation} if 
it satisfies the following equivariance property
for any $\eta\in S(\bs)$:
$$
J_{\Gamma\eta}(\alpha)=J_\Gamma(\eta_{alg}^{-1}\alpha).
$$
\end{definition}

Recall the maps $P,P_b$ relating the algebraic and analytic actions in
the commuting diagram~\eqref{eq:alg-ana-action2}. The following lemma
relates Definitions~\ref{def:good-exp} and~\ref{def:good-alg}.

\begin{lemma}\label{lem:alg-ana-action2}
In the setting above, let $\fs$ be a sign exponent of the form
\begin{equation}\label{eq:exp-main}
  \fs(\alpha) = P(\alpha)+(n-1)P_b(\alpha) + \fs_{noness}(\alpha),
\end{equation}
where $\fs_{noness}$ is a nonessential sign exponent. Then the operation
$$
  J_\Gamma:=(-1)^{\fs}I_\Gamma
$$
is a good operation.
\end{lemma}

\begin{proof}
Since by diagram~\eqref{eq:alg-ana-action2} the sign exponent
$P(\alpha)+(n-1)P_b(\alpha)$ intertwines the algebraic and analytic
actions, this follows from equation~\eqref{eq:good-an-equiv} and the
definition of a nonessential sign exponent. 
\end{proof}

\begin{remark}\label{rem:def-good}
(A) For further examples of good operations relevant for our purposes
see~\S\ref{ss:MCan} and~\S\ref{sec:ops} below.

(B) For a fixed target, the space of good operations has a natural
linear structure. 
Postcomposition of a good operation with a pullback or fibre
integration or (if the target is $\Om_{int}^*(N_1\times N_2)$) projection
onto a bidegree is again a good operation.
Therefore, by Lemma~\ref{lem:alg-ana-action2}, a map
$$
  \alpha\mapsto\sum_{p,q}(-1)^{\fs_{\Gamma,p,q}(\alpha)}I_\Gamma(\alpha)^{p,q}
$$
is a good operation if each $\fs_{\Gamma,p,q}$ has the form~\eqref{eq:exp-main}.

(C) If $\Gamma$ has a marked edge $l$, then everything carries over to
$I_{\Gamma,l}$ in place of $I_\Gamma$. 
\end{remark}

\subsection{The Maurer-Cartan element}\label{ss:MCan}

Consider now the de Rham algebra $(\Om^*(M),d,(\cdot,\cdot))$
with the intersection pairing~\eqref{eq:defineintpair}.
Recall that this pairing is nondegenerate but not perfect.
Let $\HH\subset\Om^*(M)$ be the harmonic subspace associated to the
propagator $\wt G$ above.
The pairing $(\cdot,\cdot)$ restricts to $\HH$ as a perfect pairing,
so we get a cyclic cochain complex $(\HH, d=0, (\cdot,\cdot))$.
Proposition~\ref{prop:canondIBL} associates to this cyclic complex a
canonical dIBL-algebra (with trivial differential)  
\begin{equation}\label{eq:deRhamcanon}
   \dIBL(\HH) = \Bigl((B^{{\text{\rm
         cyc}}*}\HH)[2-n],\fp_{1,1,0}=0,\,\fp_{1,2,0},\,\fp_{2,1,0}\Bigr). 
\end{equation}
Recall from~\S\ref{ss:basiccomb} the set $\RR_{\ell,g}$ of isomorphism
classes of labelled trivalent ribbon graphs with $\ell\geq 1$ boundary
components and genus $g\geq 0$. We define
\begin{equation}\label{eq:mlg}
   \m_{\ell,g} := \frac{1}{\ell!}\sum_{\Gamma\in \RR_{\ell,g}}\m_{\Gamma}\in 
   (B\HH[3-n]^{\otimes\ell})^*,
\end{equation}
where
\begin{equation}\label{eq:mGamma}
   \m_{\Gamma }(\alpha):=(-1)^{s_\Gamma(\alpha)}
   I_\Gamma(\alpha) 
\end{equation}
for a decomposable $\alpha=\alpha_1\otimes\dots\otimes\alpha_s\in
\HH^{\otimes s}$ of homogeneous degree and 
\begin{equation}\label{eq:sGamma}
  s_\Gamma(\alpha):= n\ell+s(s+1)/2+
  P(\alpha)+(n-1)\bigl((\ell+1)(s+1) + P_b(\alpha)\bigr).
\end{equation}
It is proved in~\cite[\S8]{Cieliebak-Volkov} that 
$\{\m_{\ell,g}\}$ is indeed a Maurer-Cartan element for $\dIBL(\HH)$. 

\subsection{Boundary strata}\label{ss:bdryloci}

In this subsection we show that the signed integrals of $\wt
R_\Gamma^*\wt G^e(\alpha)$ over regular boundary loci of $\XX_\Gamma$
do not depend on the extension of the labelling of $\Gamma$.
We continue in the setting of the previous subsections, where now in
addition we consider a pair $(M^d\times\WW,\ZZ)$ as
in~\S\ref{sec:general graphs} and its associated basic pair from
Definition~\ref{def:basic-pair}, 
\begin{equation*}
  (\YY_\Gamma,\XX_\Gamma) = \bigl(\wt Y_\Gamma\times \WW,PT(\Delta_{\ver}\times \ZZ)\bigr).
\end{equation*}
We denote the pullback of $\wt R_\Gamma^*\wt G^e(\alpha)$ to
$\YY_\Gamma$ under the projection $\wt Y_\Gamma\times \WW\to \wt
Y_\Gamma$ by the same expression.

We need some preparation. Let $\sigma$ and $\tau$ be permutations of
the set  $\{1,\dots,f\}$ as in~\S\ref{ss:intconfigext}, where $\tau$
changes the edge order and $\sigma$ the vertex order. 

(A) Assume first that $\tau=\id$. Equation~\eqref{eq:R-trans} with
$\tau=\id$ yields $R_{\sigma\Gamma}=R_\Gamma\circ M^\sigma$. Therefore,
$M^\sigma:Y_{\sigma\Gamma}\to Y_\Gamma$
restricts to a diffeomorphism between the blow-up loci
$R_\Gamma^{-1}(\Delta_2^{\sigma\Gamma})\to R_\Gamma^{-1}(\Delta_2^\Gamma)$
and thus lifts to a diffeomorphism between the blow-ups
$\wt M^\sigma:\wt Y_{\sigma\Gamma}\longrightarrow \wt Y_\Gamma$.
We take the product of this map with the identity on $\WW$ and
denote the result by the same letter to get the diffeomorphism
$$
  \wt M^\sigma:\YY_{\sigma\Gamma}\longrightarrow \YY_\Gamma
$$
(note that $\Delta_2^{\sigma\Gamma}=\Delta_2^\Gamma$).
On the other hand, $M^\sigma$ gives rise to a diffeomorphism
$$
  \Delta_\ver^{\sigma\Gamma}\times \ZZ\longrightarrow
  \Delta_\ver^\Gamma\times \ZZ.
$$
Therefore, by Remark~\ref{rem:diffeo-pairs}, $\wt M^\sigma$ restricts
to a diffeomorphism of quasi-regular submanifolds with boundary
$$
  \wh\XX_{\sigma\Gamma}\longrightarrow \wh\XX_\Gamma
$$
defined as in equation~\eqref{eq:Xhat}. 
This yields an identification of the respective primary boundaries.
In particular, the primary faces that correspond
an edge $l$ of $\Gamma$ are related by
\begin{equation}\label{eq:key-diffeo}
  \wt M^\sigma:\p_l\XX_{\sigma\Gamma}\stackrel{\cong}\longrightarrow\p_l\XX_\Gamma. 
\end{equation}

(B) In the setting of~\S\ref{ss:intconfigext},
equation~\eqref{eq:ext-indep-downst} of Lemma~\ref{lem:ext-indep-basic}
lifts to
\begin{equation}\label{eq:ext-indep-upst}
 (-1)^{\bar R_{\sigma\Gamma\tau}+(n-1)\eta_3(\sigma\Gamma\tau)}
 \wt R_{\sigma\Gamma\tau}^*\wt G^e(\alpha)=
 (-1)^{\bar R_\Gamma+(n-1)\eta_3(\Gamma)+n\sigma}
 (\wt M^\sigma)^*\wt R_\Gamma^*\wt G^e(\alpha).
\end{equation}
Indeed, both sides of~\eqref{eq:ext-indep-upst}
agree on the interior of $\wt Y_{\sigma\Gamma\tau}$ due
to~\eqref{eq:ext-indep-downst} and we conclude by taking the closure.

The goal of this subsection is the following result about integration
over boundary strata. 

\begin{lemma}\label{lem:bdryindep}
Let $\Gamma$ be a labelled graph as in~\S\ref{ss:Integrals}.
For an edge $l$ of $\Gamma$ and an extension of the labelling of $\Gamma$
consider the expression
$$
  S(\Gamma,l):=(-1)^{\bar R_\Gamma+(n-1)\eta_3(\Gamma)}
  \int_{\p_l\XX_\Gamma}\wt R_\Gamma^*\wt G^e(\alpha),
$$
where $\p_l\XX_\Gamma$ denotes the boundary stratum corresponding to
the edge $l$. Then $S(\Gamma,l)$ does not depend on the chosen
extension of the labelling of $\Gamma$.
\end{lemma}

\begin{proof}
Let $\Gamma'$ denote the graph $\Gamma$ with a potentially different
extension of the labelling. Assume first that $\Gamma'$ differs from
$\Gamma$ by the ordering and orientation of edges. Then 
$\Gamma'=\Gamma\tau$ for some permutation $\tau$ responsible for
the change in the ordering and orientation of edges.
Equation~\eqref{eq:ext-indep-upst} with $\sigma=\id$ implies
that the integrands involved in $S(\Gamma',l)$ and $S(\Gamma,l)$ coincide.
Note that the two families $R_{\Gamma'}^{-1}(\Delta_2^{\Gamma'})$
and $R_\Gamma^{-1}(\Delta_2^\Gamma)$ differ only by renumbering their members.
Therefore, the respective blow-ups $\YY_{\Gamma'}=\wt
Y_{\Gamma'}\times \WW$ and $\YY_\Gamma=\wt Y_\Gamma\times \WW$
(see equation~\eqref{eq:defYorient}) coincide. Since
$\Delta_\ver^\Gamma=\Delta_\ver^{\Gamma'}$, we get that
$\XX_\Gamma=\XX_{\Gamma'}$ and thus $\p_l\XX_\Gamma=\p_l\XX_{\Gamma'}$,
so the spaces over which we integrate also coincide and the integrals agree.

Assume now that $\Gamma'$ differs from $\Gamma$ by the ordering of
nonspecial vertices and cyclic permutations of the flags around each vertex.
The $\Gamma'=\sigma\Gamma$ for some permutation $\sigma$ responsible for
this change of extension. We compute
\begin{align*}
S(\Gamma',l)
&\stackrel{(1)}{=}(-1)^{\bar R_{\sigma\Gamma}+(n-1)\eta_3(\sigma\Gamma)}
\int_{\p_l\XX_{\sigma\Gamma}}\wt R_{\sigma\Gamma}^*\wt G^e(\alpha)\cr
&\stackrel{(2)}{=} (-1)^{\bar R_\Gamma+(n-1)\eta_3(\Gamma)+n\sigma}
\int_{\p_l\XX_{\sigma\Gamma}}(\wt M^\sigma)^*\wt R_\Gamma^*\wt G^e(\alpha)\cr
&\stackrel{(3)}{=} (-1)^{\bar R_\Gamma+(n-1)\eta_3(\Gamma)}
\int_{\p_l\XX_\Gamma}\wt R_\Gamma^*\wt G^e(\alpha)\cr
&\stackrel{(4)}{=} S(\Gamma,l).
\end{align*}
Here equality~(1) is the definition of $S(\Gamma',l)$, and equality~(2)
follows from equation~\eqref{eq:ext-indep-upst} with $\tau=\id$.
Equality~(3) follows from invariance of integration 
(Corollary~\ref{cor:pullback-aut}) under the map $\wt M^\sigma$
in~\eqref{eq:key-diffeo};
here $M^\sigma:\Delta_\ver\to\Delta_\ver$  
changes the orientation by $(-1)^{n\sigma_\ver}=(-1)^{n\sigma}$, where
$\sigma_\ver$ is the permutation of vertices induced by $\sigma$ 
as in the proof of Lemma~\ref{lem:ext-indep-basic}, so the same change
of orientation appears on $\XX_\Gamma=PT(\Delta_\ver\times\ZZ)$ and
thus on its boundary strata. 
Equality~(4) is the definition of $S(\Gamma,l)$.
\end{proof}

\begin{remark}\label{rem:edgechoice}
Lemma~\ref{lem:bdryindep} allows us to choose the extension of the
labelling of $\Gamma$ as we please when integrating over the boundary
of $\XX_\Gamma$. This freedom will be used e.g.~for the cancellation
of boundary strata in~\S\ref{ss:chainprod}.
If a graph appears as the result of attaching a leg to some other
graph, then the edge created as the result of attaching comes first in
the ordering of edges and is oriented towards the corresponding 
special vertex. Moreover, the new vertex created as the result of
attaching comes last in the ordering of vertices,
see~\S\ref{ss:op-graphs2}.  
If a connected tree with two special vertices has a doubly special edge, 
then the edge comes first in the ordering of edges and 
is oriented from the first special vertex to the second one. If a
circular graph contains a doubly special edge, then it comes first in
the ordering of edges and is oriented by sliding it to the second
boundary component of the graph and picking up the boundary
orientation, see~\S\ref{ss:ribbonbasicdef}.
\end{remark}

\section{The operations $\G$, $\G^2$ and $\F^2$}\label{sec:ops}

In this section we define chain homotopies $\G^2$ and $\F^2$ that will
be used in the next section to prove compatibility for the product and
the coproduct. Moreover, we redefine the chain map $\G$ from
equation~\eqref{eq:bG} in a more convenient form and establish some
basic properties of the maps $\G$, $\G^2$ and $\F^2$. 
We retain the notation from the previous section. Thus $M$ is a
closed oriented $n$-dimensional manifold, $\HH\subset\Om^*(M)$ a
harmonic subspace, and $\wt G$ an associated propagator. We abbreviate
$$
  \Om := \Om^*(M).
$$
Recall from~\S\ref{ss:basiccomb} the sets
$\RR_{s;d}$, $\RR_{s;d_1,d_2}$ and $\RR_{s_1,s_2;d}$.
By a slight abuse of language, we will often identify a graph $\Gamma$
with its isomorphism class in one of these sets. 

\subsection{The chain map $\G$}\label{ss:defG}



We will define $\G$ (as well as $\G^2$ and $\F^2$) on decomposable
tensors of homogeneous degree and extend it by linearity.
Fix positive integers $1\le d\le s$ and consider 
$\alpha=\alpha_1\otimes\dots\otimes\alpha_s\in \HH^{\otimes s}$ of
homogeneous degree. For $\Gamma\in \RR_{s;d}$ (a tree with one special
vertex of degree $d$) we set 
\begin{equation}\label{eq:def-G-Gamma}
  \G_\Gamma^{alg}(\alpha):=(-1)^{s_\Gamma^{alg}(\alpha)}I_\Gamma(\alpha),\qquad
  \G_\Gamma(\alpha):=(-1)^{s_\Gamma(\alpha)}I_\Gamma(\alpha),
\end{equation}
where $I_\Gamma(\alpha)$ is defined by~\eqref{eq:defI}
and the sign exponents are defined below. 
Recall from Remark~\ref{rem:deg-reg} that
$I_\Gamma(\alpha)\in \Om^{\otimes d}$.
Using this, we define
\begin{equation}\label{eq:def-G-sd}
  \G_{s;d}^{alg}(\alpha):=\sum_{\Gamma\in \RR_{s;d}}\G_\Gamma^{alg}(\alpha)
,\qquad
\G_{s;d}(\alpha):=\sum_{\Gamma\in \RR_{s;d}}\G_\Gamma(\alpha),
\end{equation}
\begin{equation}\label{eq:G-redef}
  \G:=\bigoplus_{1\le d\le s}\G_{s;d}: \bigoplus_{1\le
    s}\HH^{\otimes s}\to \bigoplus_{1\le d}\Om^{\otimes d}.
\end{equation}
(The maps $\G_\Gamma^{alg}$ and $\G_{s;d}^{alg}$ will only be
used in the proof of Proposition~\ref{prop:G-consist} below.) 
Recall the sign operators $P$ from~\eqref{eq:defP} and $Q$
from~\eqref{eq:defQ}, with their sign exponents denoted by the same letters.
Using these, we define the sign exponents
$$
  s_\Gamma^{alg}(\alpha) := \frac{s(s+1)}{2}+\frac{d(d+1)}{2}+P(\alpha)+P(I_\Gamma(\alpha))
$$
and
\begin{equation}\label{eq:def-G-an-alg}
\begin{aligned}
  s_\Gamma(\alpha) &:= s_\Gamma^{alg}(\alpha) + P(I_\Gamma(\alpha))+Q(I_\Gamma(\alpha))+n+1\cr
  &= P(\alpha)+\frac{s(s+1)}{2}+I_\Gamma(\alpha)(d+1)+n+1.
\end{aligned}
\end{equation}
Note that $P_b=0$ (see equation~\eqref{eq:defPb}) because there is only one
boundary component.
Thus, the sign exponents $s_\Gamma(\alpha)$ and $s_\Gamma^{alg}(\alpha)$ have the
form~\eqref{eq:exp-main}, and by Lemma~\ref{lem:alg-ana-action2} the
operations $\G_\Gamma$ and $\G_\Gamma^{alg}$ are good
(see Definition~\ref{def:good-alg}). 



\begin{proposition}\label{prop:G-consist}
The definition of $\G$ in~\eqref{eq:G-redef}
agrees with the one in~\eqref{eq:bG}. 
\end{proposition}

\begin{proof}
We proceed in two steps.

{\bf Step 1. }
Recall from~\S\ref{ss:homottrans} (with $\AA=\Om$ and $\BB=\HH$) the 
$A_\infty$-homotopy transfer map defined for $k\geq 2$ by
\begin{equation*}
  \fg_k = \sum_{T\in RT_k^3}\fg_T:\HH^{\otimes k}\to\Om.
\end{equation*}
Here $RT_k^3$ denotes the set of isomorphism classes of trivalent
rooted trees with $k$ leaves, and the operation $\fg_T$ is defined by
applying the product $\fm_2$ at each vertex and the propagator along
each edge. 
Recall from~\S\ref{ss:homottrans} that we agreed to number the leaves
of a rooted tree counterclockwise from the root, which is exactly its
canonical labelling from~\S\ref{ss:basiccomb} when we consider it as a
tree with one special vertex. Thus, declaring the root vertex be the
special vertex yields a canonical bijection
\begin{equation*}
  RT_k^3\stackrel{\cong}{\longrightarrow}\RR_{k,1}^{\can}. 
\end{equation*}
With this identification, we obtain for each $T\in RT_k^3$ the equality
\begin{equation}\label{eq:consist1}
  \g_T = \G_T^{alg}:\HH^{\otimes k}\to\Om.
\end{equation}
Up to signs, this follows directly by writing the operation on the
right hand side as an iterated integral via Fubini's theorem; for the
sign computation see~\cite{Volkov-thesis}.

{\bf Step 2. }
According to~\S\ref{ss:basiccomb} we canonically associate to a
labelled planar tree $\Gamma\in\RR_{s;d}$ its rooted components
$T_1,\dots,T_d$ in counterclockwise order, where $T_1$ contains the
leaf number $1$. 
The labelling of $\Gamma$ is called canonical if the leaf number
$1$ in the labelling comes right after the root in the counterclockwise
order on $T_1$.  
Conversely, $d$ trivalent rooted trees $T_1,\dots,T_d$ glue along
their roots to a canonically labelled tree $T_1\cdots T_d$ with a
special vertex of degree $d$ whose leaf number $1$ is the first leaf
of $T_1$ in the counterclockwise order. This gives a canonical bijection
$$
  \RR_{s;d}^{\can} \cong
  \coprod_{s_1+\dots+s_d=s}RT_{s_1}^3\times\cdots\times RT_{s_d}^3.
$$
Recall from~\S\ref{ss:gradedvect} the relabelling (right) action of
the cyclic group $\Z_s$ on $\RR_{s;d}$ and the generator $\tau_{2\to
  1}$ of $\Z_s$. Then the last displayed equation can be upgraded to 
\begin{equation}\label{eq:canon-conn}
  \RR_{s;d} \cong \coprod_{s_1+\dots+s_d=s}\coprod_{r=0,\dots,s_1-1}
  (RT_{s_1}^3\times\cdots\times RT_{s_d}^3)\tau_{2\to 1}^r,
\end{equation}
where the effect of $\tau_{2\to 1}^r$ is to assign the leaf number
$r+1$ in the canonical labelling the new number $1$. Note that, since
$r+1\le s_1$, the new leaf number $1$ is still in the first rooted
component. For $\Gamma\in \RR_{s,d}^\can$ with its decomposition
$\Gamma=T_1\dots T_d$ into rooted components we have 
\begin{equation}\label{eq:rooted-dec-H}
  \G_\Gamma^{alg}=\G_{T_1}^{alg}\otimes\dots\otimes
  \G_{T_d}^{alg}:\HH^{\otimes s}\to\Om^{\otimes d}.
\end{equation}
This is again clear up to signs, and for the signs we refer
to~\cite{Volkov-thesis}. 
Together with~\eqref{eq:consist1} it gives us
\begin{equation}\label{eq:rooted-dec}
  \G_\Gamma^{alg}=\g_{T_1}^\Om\otimes\dots\otimes
  \g_{T_d}^\Om:\HH^{\otimes s}\to\Om^{\otimes d}. 
\end{equation}
Now we compute
\begin{align*}
  \G_{s;d}^{alg}
  &\stackrel{(1)}{=} \sum_{\Gamma\in \RR_{s;d}}\G_\Gamma^{alg} \cr
  &\stackrel{(2)}{=} \sum_{s_1+\cdots+s_l=s}
  \sum_{\substack{T_j\in RT_{s_j}^3\\j=1,\dots,d}}\sum_{r=0}^{s_1-1}
  \G_{(T_1\cdots T_d)\tau_{2\to 1}^r}^{alg} \cr
  &\stackrel{(3)}{=} \sum_{s_1+\cdots+s_l=s}
  \sum_{\substack{T_j\in RT_{s_j}^3\\j=1,\dots,d}}
  \G_{T_1\cdots T_d}^{alg}\circ
  (1+t_\Om+\cdots+t_\Om^{s_1-1}) \cr
  &\stackrel{(4)}{=} \sum_{s_1+\cdots+s_l=s}
  \sum_{\substack{T_j\in RT_{s_j}^3\\j=1,\dots,d}}
  (\g_{\Gamma_1}\otimes\cdots\otimes\g_{\Gamma_d})
  \circ (1+t_\Om+\cdots+t_\Om^{s_1-1}) \cr
  &\stackrel{(5)}{=} \sum_{s_1+\cdots+s_d=s}
  (\g_{s_1}\otimes\cdots\otimes\g_{s_d})
  \circ (1+t_\Om+\cdots+t_\Om^{s_1-1}) \cr
  &\stackrel{(6)}{=} \fG_{s,d}.
\end{align*}
Here equality~(1) is the definition~\eqref{eq:def-G-sd} of $\G_{s;d}^{alg}$; 
equality~(2) follows from~\eqref{eq:canon-conn};
equality~(3) holds since $\G_\Gamma^{alg}$
is a good operation, where $t_\HH$ is the cyclic operation defined
in~\S\ref{ss:hochcycainf} with $\AA=\HH$; 
equality~(4) follows from~\eqref{eq:rooted-dec}; 
equality~(5) follows from~\eqref{eq:def-htrmap};
and equality (6) follows from the definition~\eqref{eq:fFlevelwise} of 
$\fG_{s;d}$ (with $\ff$ replaced by $\fg$).

Now recall from~\eqref{eq:bG} that $\G=(-1)^{n+1}QP\fG$ is related to
$\fG$ by the sign operators $P,Q$ switching between the algebraic and
analytic permutation actions. Comparison of the sign
exponents in~\eqref{eq:defP} for $P$ and in~\eqref{eq:defQ} for $Q$
with the sign exponent in~\eqref{eq:def-G-an-alg} concludes the proof.
\end{proof}

\subsection{The chain homotopy $\G^2$ for the product}
\label{ss:defG^2}
Fix positive integers $d_1,\,d_2,s$ subject to  $2\le d_1+d_2\le
s+2$ and consider $\alpha=\alpha_1\otimes\dots\otimes\alpha_s\in 
\HH^{\otimes s}$ of homogeneous degree. For $\Gamma\in R_{s;d_1,d_2}$
(a tree with two special vertices of degrees $d_1,d_2$) we set
\begin{equation}\label{eq:G2-Gamma}
  \G_\Gamma^2(\alpha):=\sum_{p+q=\deg I_\Gamma(\alpha)}(-1)^{s_{\Gamma,p,q}^2(\alpha)}I_\Gamma(\alpha)^{p,q},
\end{equation}
where $I_\Gamma(\alpha)$ is defined by~\eqref{eq:defI}
and the sign exponent $s_{\Gamma,p,q}^2(\alpha)$ is defined below.
For the composition into bidegrees $(p,q)$ see
Remark~\ref{rem:deg-reg} and Remark~\ref{rem:def-good}.
Using this, we define
\begin{equation*}
  \G_{s;d_1,d_2}^2(\alpha):=\sum_{\Gamma\in R_{s;d_1,d_2}}\G_\Gamma^2(\alpha),
\end{equation*}
\begin{equation}\label{eq:G2-def}
   \G^2:=\bigoplus_{
   \substack{2\le d_1+d_2\le s+2\\1\le s}}
   \G_{s;d_1,d_2}^2:
   \bigoplus_{1 \le s} \HH^{\otimes s}\to \bigoplus
   _{1\le d_1,d_2}
   \Om_{int}^*(M^{d_1}\times M^{d_2}),
\end{equation}
where $\Om_{int}^*$ denotes the space of integrable forms as in Definition~\ref{def:L1}. 
Finally, for an o-marked tree $(\Gamma,l)\in \RR_{s;d_1,d_2}^{om}$ we set
\begin{equation*}\label{eq:differential}
  \bH_{\Gamma,l}(\alpha) := \sum_{p+q=\deg I_{\Gamma,l}(\alpha)}(-1)^{\wt
    s^2_{\Gamma,p,q}(\alpha)}I_{\Gamma,l}(\alpha)^{p,q}, 
\end{equation*}
where $I_{\Gamma,l}(\alpha)$ is defined by~\eqref{eq:defIl} and the
sign exponent $\wt s^2_{\Gamma,p,q}(\alpha)$ is defined below.


\begin{remark}\label{rem:signs}
The sign exponents are defined as follows. For a tree $\Gamma$ with two
special vertices of degrees $d_1,d_2$ and a decomposable tensor 
$\alpha=\alpha_1\otimes\dots\otimes\alpha_s\in \HH^{\otimes s}$
of homogeneous degree we define
\begin{align*}
  s^2_{\Gamma,p,q}(\alpha)
  &:= s(s+1)/2+P(\alpha)+d_1d_2+pd_2+p(d_1+1)+q(d_2+1) \cr
  & \ \ \ \ \ +(n-1)(d_2+p+1)+1.
\end{align*}
If $\Gamma$ has in addition a marked edge we define
\begin{align*}
  \wt s^2_{\Gamma,p,q}(\alpha)
  &:= (s+1)(s+2)/2+P(\alpha)+d_1d_2+pd_2+p(d_1+1)+q(d_2+1)\cr
  & \ \ \ \ \ +(n-1)(s-d_1+p)+1
\end{align*}
(which does not depend on the position of the marked edge).
Note that both $s^2_{\Gamma,p,q}(\alpha)$ and $\wt
s^2_{\Gamma,p,q}(\alpha)$ have the form~\eqref{eq:exp-main},
so it follows from Lemma~\ref{lem:alg-ana-action2} that
$\G_\Gamma^2$ and $\bH_{\Gamma,l}$ are good operations
(see Definition~\ref{def:good-alg}).
\end{remark}

\begin{remark}\label{rem:L1-degrees}
In~\S\ref{ss:chainprod} the maps $\G_\Gamma^2$ and
$\bH_{\Gamma,l}^2$ will always be used in conjunction with
postcomposition with $I^{\otimes 2}$, see~\S\ref{sec:chen} for the
definition of $I$.
Let $d_1$ and $d_2$ denote the numbers of flags at the two special
vertices of $\Gamma$. The value of $I^2\circ \G_\Gamma^2(\alpha)$ on a
pair of smooth simplices $f_j:B_j\longrightarrow \Lambda$, $j=1,2$, is
the integral
$$
  \int_{B_1\times\Delta^{d_1-1}\times B_2\times\Delta^{d_2-1}}
  (ev_{f_1}\times ev_{f_2})^*G_\Gamma^2(\alpha).
$$
Decomposing $G_\Gamma^2(\alpha)$ into its bidegree parts, this integral
is a sum of $\deg G_\Gamma^2(\alpha)+1$ terms. Since the domain
of integration is a product and we pull back by a product
map, all summands vanish for degree reasons except 
the one with $G_\Gamma^2(\alpha)^{p,q}$ for
$(p,q)=(\dim B_1+(d_1-1),\,\dim B_2+(d_2-1))$.
Therefore, whenever we work with $I^2\circ \G_\Gamma^2(\alpha)$, we
can ignore all the summands in $\G_\Gamma^2(\alpha)$ and
$\bH_{\Gamma,l}(\alpha)$ except the one with bidegree
\begin{equation}\label{eq:pq}
  (p,q)=(\gamma_1,\gamma_2):=(\dim B_1+(d_1-1),\,\dim B_2+(d_2-1)).
\end{equation}
With this value of $(p,q)$ understood, we simplify the notation to
$$
  \wt s^2_\Gamma(\alpha):=\wt s^2_{\Gamma,p,q}(\alpha),
  \qquad s^2_\Gamma(\alpha):=s^2_{\Gamma,p,q}(\alpha).
$$
\end{remark}

\subsection{The chain homotopy $\F^2$ for the coproduct}
\label{ss:defG^2coprod}

Fix positive integers $s_1,\,s_2,d$ subject to $1\le d\le s_1+s_2+2$
and consider
 $\alpha = \alpha_1\otimes\dots\otimes\alpha_{s_1+s_2}\in
  \HH^{\otimes s_1}\otimes\HH^{\otimes s_2}$
of homogeneous degree. For $\Gamma\in R_{s_1,s_2;d}$ (a circular graph
with one special vertex of degree $d$ and $s_i$ leaves on the $i$-th
boundary component) we set
\begin{equation*}
  \F_\Gamma^2(\alpha):=(-1)^{r^2_\Gamma(\alpha)}I_\Gamma(\alpha),
\end{equation*}
where $I_\Gamma(\alpha)$ is defined by~\eqref{eq:defI}
and $r^2_\Gamma(\alpha)$ is a sign exponent. 
Using this, we define
\begin{equation}\label{eq:def-F-Gamma}
\F_{s_1,s_2;d}^2(\alpha):=
\sum_{\Gamma\in R_{s_1,s_2;d}}
\F_\Gamma^2(\alpha),
\end{equation}
\begin{equation}\label{eq:def-F-sd}
   \F^2:=\bigoplus_{
   \substack{1\le d\le s_1+s_2+2\\1\le s_1,s_2}}
   \F_{s_1,s_2;d}^2:
   \bigoplus_{1 \le s_1,s_2} \HH^{\otimes s_1}
   \otimes \HH^{\otimes s_2}
   \to \bigoplus
   _{1\le d}\Om_{int}^*(M^d).
\end{equation}
Finally, for an o-marked graph $(\Gamma,l)\in \RR_{s_1,s_2;d}^{m}$ we set
\begin{equation*}\label{eq:differential2}
  \bH_{\Gamma,l}(\alpha):=(-1)^{\wt r^2_\Gamma(\alpha)}I_{\Gamma,l}(\alpha),
\end{equation*}
where $I_{\Gamma,l}(\alpha)$ is defined by~\eqref{eq:defIl} and
$\wt r^2_\Gamma(\alpha)$ is a sign exponent. 
We will not define the sign exponents $r^2_\Gamma(\alpha)$ and $\wt
r^2_\Gamma(\alpha)$. They have again the form~\eqref{eq:exp-main}, so
that $\F_\Gamma^2(\alpha)$ and $\bH_{\Gamma,l}$ are good operations by
Lemma~\ref{lem:alg-ana-action2}.

\subsection{Behaviour under relabelling}\label{subsec:well-def}

In this subsection we collect some consequences of the fact that 
$\G_\Gamma^2$, $\F_\Gamma^2$ and $\bH_\Gamma$ are good operations. We
will use the properties of the action of the set of relabellings on
labelled trees with two special vertices and on labelled circular graphs,
see Lemma~\ref{lem:free}.
Recall the generator $\tau_{2\to 1}$ of 
$\Z_s$ and the symmetrization operator
$N_{alg}$ from~\eqref{eq:def-Nalg}.
Recall from~\eqref{eq:renumber-bdry-basic} the involution $\tau$
swapping the numbering of the two boundary components of a
circular graph.
Since the transposition $\tau$ maps the set $\RR_{s_1,s_2;d}$
bijectively onto the set $\RR_{s_2,s_1;d}$ and the operation 
$\F_\Gamma^2$ is good, we get the invariance property
\begin{equation}\label{eq:F2inv}
  \F_{s_1,s_2;d}^2=\F_{s_1,s_2;d}^2\circ\tau.
\end{equation}

\begin{lemma}\label{lem:GNl1l2}
The maps $\G^2$ and $\F^2$ factor through the symmetrization operator
$N_\HH=N_{alg}$ (respectively $N_\HH\otimes N_\HH$) and are,
therefore, well-defined as maps from the cyclic to the noncyclic
complex: 
\begin{equation}\label{eq:GNl1l2}
  \G^2:\bigoplus_{1\le s}B_s^{\text{\rm cyc}}\HH\to \bigoplus_{1\le
    d_1,d_2} \Om_{int}^*(M^{d_1}\times M^{d_2}),
\end{equation}
\begin{equation}\label{eq:GN1N2l}
  \F^2:\bigoplus_{1\le s_1,s_2}B_{s_1}^{\text{\rm cyc}}\HH \otimes
  B_{s_2}^{\text{\rm cyc}}\HH \to\bigoplus_{2\le d}\Om_{int}^*(M^d).
\end{equation}
\end{lemma}

\begin{proof}
Fix positive integers $d_1,d_2,s$ subject to $1\le d\le s_1+s_2+2$ and  
recall from~\S\ref{ss:defG^2coprod} that the component
$\F_{s_1,s_2;d}^2$ of $\F^2$ is defined as a sum of good operations
over the set $\RR_{s_1,s_2;d}$. According to Lemma~\ref{lem:free}, the group  
$\Z_{s_1}\otimes \Z_{s_2}$ acts freely on this set. This together with 
equation~\eqref{eq:def-Nalg} allows us to argue as follows. We pick a
fundamental locus for this action and replace the sum over 
$\RR_{s_1,s_2;d}$ by the sum over the fundamental locus precomposed
with $N_\HH\otimes N_\HH$. This finishes the argument for $\F^2$. The
one for $\G^2$ is completely analogous. 
\end{proof} 

\begin{lemma}\label{lem:trade}
The following equations hold true:
\begin{equation}\label{eq:tradesep}
\sum_{s_1+s_2=s+2\atop d_1\leq s_1,d_2\leq  s_2}
  \sum_{\Gamma_1\in \RR_{s_1;d_1}\atop \Gamma_2\in \RR_{s_2;d_2}}
  \bH_{g_{210}(\Gamma_1,\Gamma_2)}\circ N_\HH
  =
  \sum_{d_1+d_2\leq s+2}\sum_{\Gamma\in \RR_{s;d_1,d_2}^{\rm sep}}
  \bH_{\Gamma},
\end{equation}
\begin{equation}\label{eq:tradenonsep}
\sum_{s_1+s_2=s+2\atop d_1+d_2\leq
    s_2}\sum_{\Gamma_1\in \RR_{s_1}\atop \Gamma_2\in \RR_{s_2;d_1,d_2}}
  \bH_{g_{210}(\Gamma_1\amalg\Gamma_2)}\circ N_\HH
  =
  \sum_{d_1+d_2\leq s+2}\sum_{\Gamma\in
    \RR_{s;d_1,d_2}^{\rm nonsep}}\bH_{\Gamma},
\end{equation}
\begin{equation}\label{eq:tradecoprod1}
\sum_{d\leq s_1+s_2+2\atop
  \Gamma\in\RR_{s_1+s_2+2;d}}
  \sum_{3\le j\le s_1+s_2+1}
  \bH_{g_{120}^j
  (\Gamma)}\circ 
  N^{\otimes 2}_\HH
  = 
  \sum_{d\leq s_1+s_2+2\atop
  \Gamma\in\RR_{s_1,s_2;d}^{cm}}
\bH_{\Gamma},
\end{equation}  
\begin{equation}\label{eq:tradecoprod2}
\sum_{d\le r_2
\atop s_1+r_2\ge 3}
\sum_{\Gamma_1\in \RR_{s_1,r_1}\atop 
\Gamma_2\in \RR_{r_2;d}}
\bH_{g_{210}(
(\Gamma_1\amalg \Gamma_2)\tau_{23})}
\circ (N_\HH\otimes \id) 
=
\sum_{d\le s_1+s_2+1}\sum_{\Gamma\in\RR_{s_1,s_2;d}^{ncb1}} 
\bH_{\Gamma},
\end{equation}
\begin{equation}\label{eq:tradecoprod3}
 2\sum_{r_1+r_2\ge 3
  \atop d\leq
    r_2+s_2+2}\sum_{\Gamma_1\in \RR_{r_1}\atop \Gamma_2\in \RR_{r_2,s_2;d}}
\bH_{g_{210}(\Gamma_1\amalg \Gamma_2)}
\circ (N_\HH\otimes \id)
=
2\sum_{d\leq s_1+s_2+1}
\sum_{\Gamma\in
    \RR_{s_1,s_2;d}^{nc1}} 
\bH_{\Gamma}
\end{equation}
Here $\tau_{23}$ is the relabelling swapping the second and third
boundary components of the graph $\Gamma_1\amalg\Gamma_2$.
\end{lemma}

\begin{proof}
The proof follows exactly the same pattern as the one of
Lemma~\ref{lem:GNl1l2}. Therefore, we will only indicate the relevant
free actions and fundamental loci.  

For equation~\eqref{eq:tradesep}
recall that the image of the gluing map $gl_1$ (see~\eqref{eq:def-gl})
is a fundamental locus of the free $\Z_s$-action on $\RR_{s;d_1,d_2}^{\rm sep}$ (Lemma~\ref{lem:prodfree}).

For equation~\eqref{eq:tradenonsep} recall that the image of the map $gl_2$
(see~\eqref{eq:def-gl})
is a fundamental locus for the free
$\Z_s$-action on $\RR_{s;d_1,d_2}^{\rm nonsep}$ (Lemma~\ref{lem:prodfree}).

For equation~\eqref{eq:tradecoprod1}
recall that the image of $gl_3$
(see~\eqref{eq:glueg1})
is a fundamental locus of the free 
$\Z_{s_1}\times \Z_{s_2}$-action
on $\RR_{s_1,s_2;d}^{cm}$
(Lemma~\ref{lem:free1}).

For equation~\eqref{eq:tradecoprod2}
recall that the image of $gl_4$ 
(see~\eqref{eq:glueg2}) is a fundamental locus of the free 
$\Z_{s_1}$-action on $\RR_{s_1,s_2;d}^{ncb1}$
(Lemma~\ref{lem:free2})

For equation~\eqref{eq:tradecoprod3}
recall that the image of $gl_5$
(see~\eqref{eq:glueg4})
is a fundamental locus of the free
$\Z_{s_1}$-action on $\RR_{s_1,s_2;d}^{nc1}$
(Lemma~\ref{lem:free4})
\end{proof}

\begin{lemma}\label{lem:fulldiffeven}
Let $\Gamma$ be a labelled tree with two special vertices
(resp.~a circular graph with one special vertex). 
Set $\star:=\wt s^2_{\Gamma,p,q}(\alpha)$
(resp.~$\star:=\wt r^2_\Gamma(\alpha)$). 
Let $\alpha\in \HH^{\otimes s}$ (resp.~$\alpha\in \HH^{\otimes
  s_1}\otimes \HH^{\otimes s_2}$) be a decomposable tensor of
homogeneous degree.
Then for each bidegree $(p,q)$ we have
\begin{equation}\label{eq:fulldiffeven}
  \sum_{l\in \Edge(\Gamma)}\bH_{\Gamma,l}(\alpha)^{p,q} =
  (-1)^{\fs_\p} \left(\int_{\Delta_\ver} dR_\Gamma^*G^e(\alpha)\right)^{p,q}
\end{equation}
with the sign exponent
\begin{equation}\label{eq:s-is-good}
  \fs_\p:=\star+\bar R_\Gamma+(n-1)\eta_3(\Gamma).
\end{equation}
\end{lemma}

\begin{proof}
Pick an extension of the labelling for $\Gamma$. Recall that $n(l)$
denotes the number of the marked edge with respect to the resulting
edge order. We compute
\begin{align*}
  &\sum_{l\in \Edge(\Gamma)}\bH_{\Gamma,l}(\alpha)^{p,q} \cr
  &\stackrel{(1)}{=} (-1)^{\star}\sum_{l\in \Edge(\Gamma)}
  (-1)^{\bar R_\Gamma+(n-1)\eta_3(\Gamma)+(n-1)(n(l)-1)}
  \left(\int_{\Delta_\ver}R_\Gamma^*G^e_l(\alpha)\right)^{p,q} \cr
  &\stackrel{(2)}{=} (-1)^{\star+\bar R_\Gamma+(n-1)\eta_3(\Gamma)}\sum_{l\in \Edge(\Gamma)}
  (-1)^{(n-1)(n(l)-1)}\left(\int_{\Delta_\ver}R_\Gamma^*G^e_l(\alpha)\right)^{p,q}\cr
  &\stackrel{(3)}{=} (-1)^{\star+\bar R_\Gamma+(n-1)\eta_3(\Gamma)}
  \left(\int_{\Delta_\ver}dR_\Gamma^*G^e(\alpha)\right)^{p,q}.
\end{align*}
Here for (1) we use the definition of $\bH_{\Gamma,l}$ and observe
that the sign exponent $\star$ is the same for all $l\in
\Edge(\Gamma)$ due to Remark~\ref{rem:signs}.
For (2) we pull out the sign that does not depend on $l$, and for (3)
we assemble the full differential. 
\end{proof}

\subsection{Product formulas}\label{subsec:prod-form}

The product formulas in this section run in parallel to the gluing
operations in~\S\ref{ss:op-graphs}. On the right hand side of all the
formulas below we have an operation associated to a marked graph,
where we drop the marked edge from the notation to save space. 

\begin{lemma}\label{lem:H-gluing}
For $\Gamma_1\in \RR_{s_1;d_1}$ and $\Gamma_2\in \RR_{s_2;d_2}$ we have
$$
  (\G_{\Gamma_1}\otimes\G_{\Gamma_2})\circ c_{120} = \bH_{g_{210}(\Gamma_1\amalg\Gamma_2)}.
$$
For $\Gamma_1\in \RR_{s_1}$ and $\Gamma_2\in \RR_{s_2;d_1,d_2}$ we have
$$
  (\m_{\Gamma_1}\otimes\G^2_{\Gamma_2})\circ c_{120} = \bH_{g_{210}(\Gamma_1\amalg\Gamma_2)}.
$$
\end{lemma}

\begin{proof}
For the first equation, let $\Gamma_1\in \RR_{s_1;d_1}$, $\Gamma_2\in
\RR_{s_2;d_2}$, and denote $\wt\Gamma:=g_{210}(\Gamma_1\amalg\Gamma_2)$.
Let $e_j$ denote the number of edges of the graph $\Gamma_j$,
$j=1,2$. Observe that $s:=s_1+s_2-2$ is the number of leaves of
$\wt\Gamma$. Let $\alpha\in \HH^{\otimes s}$ be a decomposable tensor
adapted to $\wt\Gamma$. We write $\alpha=\alpha_1\otimes\alpha_2$
with $\alpha_1\in \HH^{s_1-1}$ and $\alpha_2\in \HH^{s_2-1}$, so that
the form $e_a\alpha_1$ is adapted to $\Gamma_1$,  
$e^a\alpha_1$ is adapted to $\Gamma_2$, and $e_a\alpha_1\otimes
e^a\alpha_2$ is one of the summands in $c_{120}(\alpha)$.
Here $(e_a)$ is a basis of $\HH$ and $(e^a)$ its dual basis.
In the following computation we assume the Einstein summation
convention, ignore signs, and identify $X_\Gamma$ with $Y_\Gamma$ via 
$R_\Gamma$ for all graphs $\Gamma$ involved.
Since the number of vertices of the graph $\Gamma_1\amalg\Gamma_2$
equals the number of vertices of $\wt\Gamma$, we get 
\begin{equation}\label{eq:triplediageq}
  \Delta_\ver^{\Gamma_1}\times \Delta_\ver^{\Gamma_2} =
  \Delta_\ver^{\wt\Gamma}.
\end{equation}
We compute with suitable sign exponents which we will not spell out:
\begin{align*}
  (\G_{\Gamma_1}\otimes\G_{\Gamma_2})\circ c_{120}(\alpha)
  &\stackrel{(1)}{=} (-1)^\star\int_{\Delta_\ver^{\Gamma_1}}e_a\alpha_1G^{e_1} 
  \int_{\Delta_\ver^{\Gamma_2}}e^a\alpha_2G^{e_2} \cr
  &\stackrel{(2)}{=} (-1)^{\star\star}\int_{\Delta_\ver^{\Gamma_1}\times 
  \Delta_\ver^{\Gamma_2}}e_a\alpha_1e^a\alpha_2G^{e_1}G^{e_2} \cr
  &\stackrel{(3)}{=} (-1)^{\star\star\star}\int_{\Delta_\ver^{\wt\Gamma}}dG\times
  G^e\alpha \cr
  &\stackrel{(4)}{=} \bH_{\wt\Gamma}(\alpha).
\end{align*}
Here equality~(1) follows from the definition of the operations
$\G_\Gamma$ and $c_{120}$ (see~\eqref{eq:defc120}), and
equality~(2) follows from Fubini's theorem for fibre integration. 
For equality~(3) we use equation~\eqref{eq:triplediageq}, bring $e^a$
to the left past $\alpha_1$, and use that $dG=e_a\times e^a$.
Equality~(4) follows from the definition of the operation $\bH$.

The second equation is proved analogously,
where one has to compute separately for each bidegree part of
$(\G_{\Gamma_1}\otimes\G^2_{\Gamma_2})\circ c_{120}(\alpha)$
because the definition of $\G^2_{\Gamma}$ in~\eqref{eq:G2-Gamma} has
different signs in front of different bidegree parts.
\end{proof}

The proofs of the following lemmas are similar to that of
Lemma~\ref{lem:H-gluing} and therefore omitted.

\begin{lemma}\label{lem:H-gluing1}
Let $\Gamma\in \RR_{s;d}$ with $s\ge 4$, and $j\in\{3,\dots,s-1\}$. Then 
$$
  \G_{\Gamma}\circ c_{210} = \frac{1}{2} \sum_{3\le j\le s_1+s_2+1}\bH_{g_{120}^j(\Gamma)}.
$$
\end{lemma}

\begin{lemma}\label{lem:H-gluing2}
Let $\Gamma_1\in\RR_{s_1,s_2}$ and $\Gamma_2\in\RR_{s_3;d}$ with $s_1+s_3\ge 3$. Then 
$$
  \tau_{23}[\m_{\Gamma_1}\otimes\G_{\Gamma_2}]
  \circ (c_{120}\otimes\id) = \bH_{g_{210}(
  (\Gamma_1\amalg \Gamma_2)\tau_{23})}.
$$
Here on the right hand side $\tau_{23}$ is the relabelling swapping
the orders of the second and third boundary components, and on the
left hand side its algebraic action on $(B^{cyc*}\HH)^{\otimes 3}$ 
is denoted by the same letter.
\end{lemma}

\begin{lemma}\label{lem:H-gluing3}
Let $\Gamma_1\in\RR_{s_1}$ and $\Gamma_2\in\RR_{s_2,s_3;d}$ with $s_1+s_2\ge 3$. Then 
$$
  (\m_{\Gamma_1}\otimes \G_{\Gamma_2}^2)\circ (c_{120}\otimes \id)=\bH_{g_{210}(\Gamma_1\amalg\Gamma_2)}. 
$$
\end{lemma}

\section{Proof of the main theorem}\label{sec:rel-stringtop}

  
In this section we prove Theorem~\ref{thm:intro} and
Corollary~\ref{cor:intro} from the Introduction.
In~\S\ref{ss:mainrun} we reduce Theorem~\ref{thm:intro} (restated as
Theorem~\ref{thm:prodcoprod}) to two chain level statements, and 
in~\ref{ss:simpconn} we derive Corollary~\ref{cor:intro} (restated as
Corollary~\ref{cor:prodcoprod}).  
In~\S\ref{ss:chainprod} and~\S\ref{ss:chaincoprod} we prove the two
chain level statements.

Throughout this section we use the setup from~\S\ref{sec:stringtop}:
$M$ is a closed oriented connected $n$-dimensional 
manifold, $\Lambda:=C^\infty(S^1,M)$ its loop space, and
$\Lambda_0\subset \Lambda$ the subspace of constant loops. We fix
a basepoint $q_0\in M\cong\Lambda_0$.

%

\subsection{The main theorem}\label{ss:mainrun}

For $M$ as above consider its de Rham cyclic DGA
$$
  \Om = \Om^*(M).
$$
Let $\HH\subset\Om$ be a harmonic subspace as in Section~\ref{ss:coch}.
Recall from~\S\ref{ss:gradedvect} its bar complex 
$$
  B\HH = \bigoplus_{k=1}^{\infty}\HH[1]^{\otimes k},
$$
cyclic bar complex $B^{\text{\rm cyc}}\HH$, dual cyclic bar complex
$B^{\text{\rm cyc}*}\HH$, and the cyclization operator 
$$
  N_\HH=N_{alg}:B^\cyc\HH\to B\HH.
$$
By~\eqref{eq:deRhamcanon} the degree shifted dual cyclic bar complex
carries a canonical dIBL-structure 
\begin{equation*}
   \dIBL(\HH) = \Bigl((B^{{\text{\rm
         cyc}}*}\HH)[2-n],\fp_{1,1,0}=0,\,\fp_{1,2,0},\,\fp_{2,1,0}\Bigr) 
\end{equation*}
with the operations defined by~\eqref{eq:dIBLcanon} as 
\begin{equation*}
\fp_{2,1,0}:=(c_{120}\circ N_{alg})^*,\qquad
\fp_{1,2,0}:=(c_{210}\circ N_{alg}^{\otimes 2})^*.
\end{equation*}
Here $c_{120}$ and $c_{210}$ are the coproduct and product on $B\HH$
defined in~\eqref{eq:defc120} and~\eqref{eq:defc210}, respectively.

Let $\wt G$ be a propagator as in Proposition~\ref{prop:existsprop}.
It induces via equation~\eqref{eq:mlg} a Maurer-Cartan element  
$$
  \m = \{\m_{\ell,g}\}
$$ 
on $\dIBL(\HH)$. We will use the following twisted
operations from~\eqref{eq:twistedoper}:  
\begin{equation*}
  \fp_{1,1,0}^\fm = \fp_{2,1,0}(\fm_{1,0},\cdot),\qquad
  \fp_{2,1,0}^\fm = \fp_{2,1,0},\qquad
  \fp_{1,2,0}^\fm = \fp_{1,2,0} + \wh\fp_{2,1,0}^{conn}(\fm_{2,0},\cdot).
\end{equation*}
By homotopy transfer (Proposition~\ref{prop:KS}), the propagator $\wt
G$ induces an $A_\infty$-structure on $\HH$. We view $B\HH$ as a chain
complex with the induced Hochschild differential.
Recall from~\S\ref{ss:dga1} the Hochschild complex $C(\Om)$ of $\Om$
viewed as a DGA (which differs from $B\Om$ for $\Om$ viewed as an
$A_\infty$-algebra by signs and a total degree shift of $1$), the
Connes cyclic complex $C^\lambda(\Om)$, its dual $C_\lambda^*(\Om)$,   
and the inclusion of the reduced subcomplex
$$
  \iota: \ol C_\lambda^*(\Om) \into C_\lambda^*(\Om).
$$ 
Homotopy transfer combined with suitable sign operators in~\eqref{eq:bG}
yields degree $0$ chain maps
$$
  \G:B\HH \stackrel{\sim}\longrightarrow C(\Om)[1],\qquad
  \G_\lambda:B^\cyc\HH \stackrel{\sim}\longrightarrow C^\lambda(\Om),
$$
where according to equation~\ref{eq:G-iso} the latter one induces
a degree $0$ isomorphism
\begin{equation}\label{eq:G-iso2}
  \G_\lambda^*:HC_\lambda^*(\Om)\stackrel{\cong}\longrightarrow
  H(B^{\cyc *}\HH,\fp_{1,1,0}^\m).
\end{equation}
Recall from Theorem~\ref{thm:cyc} Chen's iterated integral map
$$
   \bar J_\lambda: C_*(\Lambda_{S^1},q_0)\to \ol C^*_\lambda(\Om),
$$
and from~\S\ref{ss:stringtop} the equivariant string topology operations
$\mu^{S^1}$, $\lambda^{S^1}$ on reduced $S^1$-equivariant loop space
homology $H_*^{S^1}(\Lambda,q_0)$.   
We denote the maps $\fp_{2,1,0}$, $\fp_{1,2,0}^\m$ by the same
letters on chain level and on the level of homology.
Consider the induced maps on homology
\begin{equation*}
  H_*^{S^1}(\Lambda,q_0)\stackrel{\bar J_{\lambda *}}{\longrightarrow}
  \ol{HC}_\lambda^*(\Om)\stackrel{\iota_*}{\longrightarrow}
  HC_\lambda^*(\Om)\stackrel{\G_\lambda^*}{\longrightarrow}
  H(B^{\text{\rm cyc}*}\HH,\fp_{1,1,0}^{\m}).
\end{equation*} 
Here on the right hand side we use $B^{\text{\rm cyc}*}\HH$
without the degree shift by $2-n$ because with this convention the map
$\G_\lambda^*$ has degree zero. Since the maps $\iota_*$ and $J_{\lambda *}$
also have degree zero, so does the composed map. 

The following theorem corresponds to Theorem~\ref{thm:intro} from the
Introduction. 

\begin{theorem}\label{thm:prodcoprod}
In the setup above, the composed degree $0$ map on homology
\begin{equation}\label{eq:F}
  \F := \G_\lambda^*\circ\iota_*\circ \bar J_{\lambda *}:H_*^{S^1}(\Lambda,q_0)\longrightarrow
H(B^{\text{\rm cyc}*}\HH,\fp_{1,1,0}^{\m})
\end{equation}
intertwines the string bracket $\mu^{S^1}$ with $\fp_{2,1,0}$,
\begin{equation}\label{eq:prodhom}
  \F\circ\mu^{S^1} = \fp_{2,1,0}\circ \F^{\otimes 2}: H_*^{S^1}(\Lambda,q_0)^{\otimes 2}
  \longrightarrow H(B^{\text{\rm cyc}*}\HH,\fp_{1,1,0}^{\m}),
\end{equation}
and the string cobracket $\lambda^{S^1}$ with $2\,\fp_{1,2,0}^\m$,
\begin{equation}\label{eq:coprodhom}
  2\,\fp_{1,2,0}^{\m}\circ\F = \F^{\otimes 2}\circ\lambda^{S^1}: H_*^{S^1}(\Lambda,q_0)
  \longrightarrow H(B^{\text{\rm cyc}*}\HH,\fp_{1,1,0}^{\m})^{\otimes 2}.
\end{equation}
\end{theorem}

The main steps in the proof are the following two propositions. The
first one gives a chain level statement which will imply
equation~\eqref{eq:prodhom}. The second one  
gives a chain level statement for equation~\eqref{eq:coprodhom}.
Recall from~\S\ref{ss:defchainprod} and~\S\ref{ss:transvers} the
chain-level definitions of the loop product $\mu$ and the
Goresky--Hingston coproduct $\ol\lambda$.    
Recall from~\S\ref{sec:chen} the Chen maps 
\begin{equation*}
  I:C_*(\Om)\to C^*(\Lambda),\qquad 
  I_\lambda:C_*^\lambda(\Om)\to C^*(\Lambda),\qquad
   I_\lambda^2:C_*^\lambda(\Om)^{\otimes 2}\to C^*(\Lambda\times\Lambda)
\end{equation*}
and their adjoint maps $J,J_\lambda,J_\lambda^2$. 
Recall from~\S\ref{sec:ainf} the operations $b_\HH,t_\HH,N_\HH$ on the
Hochschild complex $B\HH$ of the $A_\infty$-algebra $\HH$ (with its
$A_\infty$ structure induced from $\Om$ via
Proposition~\ref{prop:KS}). We denote the differentials induced by
the singular differential $\p$ and the Hochschild differential $b_\HH$
on the tensor products $C_*(\Lambda)^{\otimes 2}$
resp.~$(B\HH/\im(1-t_\HH))^{\otimes 2}$ (as derivations) by the same symbols. 

\begin{proposition}\label{prop:chainprod}
There exists a chain homotopy $I^{\otimes 2}\circ\G^2$ of degree
$n-1$ such that for any $\alpha\in B\HH$ and all smooth simplices
$f_j:B_j\longrightarrow \Lambda$, $j=1,2$ with transverse time zero evaluations
the following relation holds: 
\begin{equation}\label{eq:homotopy2}
\begin{aligned}
&\left<I^{\otimes 2}\circ\G^{\otimes 2}\circ c_{120}\circ N_\HH(\alpha),f_1\otimes f_2\right>+
\left<I^{\otimes 2}\circ\G^2\circ b_\HH(\alpha),f_1\otimes f_2\right> \cr
= &\left<I_\lambda\circ\G_\lambda(\alpha),\mu(f_1\otimes f_2)\right>
+\left<I^{\otimes 2}\circ\G^2(\alpha),\p(f_1\otimes f_2)\right>.
\end{aligned}
\end{equation}
\end{proposition}

\begin{proposition}\label{prop:chaincoprod}
There exists a chain homotopy $I\circ\F^2$ of degree
$n-2$ such that for any $\alpha_1,\alpha_2\in B\HH$ 
and each nondegenerate smooth simplex $f:B\longrightarrow \Lambda$
(see Definition~\ref{def:nondeg}) the following relation holds:
\begin{equation}\label{eq:homotopy2coprod}
\begin{aligned}
&\left<I\circ\G\circ 2c_{210}\circ N_\HH^{\otimes 2}(\alpha_1\otimes
  \alpha_2),f\right> \cr
&+\wh\fp_{210}^{conn}(2\m_{2,0}, (\G^*\circ J)(f))(\alpha_1\otimes \alpha_2)
+\left<I\circ\F^2\circ b_\HH(\alpha_1\otimes \alpha_2),f\right>\cr
= &-\left<I_\lambda^{ 2}\circ \G_\lambda^{\otimes 2}(\alpha_1\otimes
\alpha_2),\ol\lambda(f)\right>
-\left<I\circ\F^2(\alpha_1\otimes \alpha_2),\p f\right>.
\end{aligned}
\end{equation}
\end{proposition}

Equations~\eqref{eq:homotopy2} and~\eqref{eq:homotopy2coprod} are
illustrated in the following two diagrams, where the outer rectangles
commute and the dashed diagonal arrows denote (part of) the chain
homotopies:
\begin{equation*}
\xymatrix{
  B_*^{\cyc}\HH \ar[d]_{c_{120}\circ N_\HH} \ar@{-->}[rd]^{\G^2}
  \ar[r]^{\G_\lambda}
  & C_*^\lambda(\Om) \ar[r]^{I_\lambda} & C^*(\Lambda) \ar[d]^{\mu^\vee} \\
  B_*\HH^{\otimes 2} \ar[r]^{\G^{\otimes 2}} & C_*(\Om)^{\otimes 2} 
  \ar[r]^{I^{\otimes 2}} & C^*(\Lambda)^{\otimes 2} \\
}
\end{equation*}
\begin{equation*}
\xymatrix{
  B_*^{\cyc}\HH^{\otimes 2} \ar[d]_{2\,c_{210}\circ N_\HH^{\otimes 2}} \ar@{-->}[rd]^{\F^2}
  \ar[r]^{\G_\lambda^{\otimes 2}} & C_*^\lambda(\Om)^{\otimes 2} 
  \ar[r]^{I_\lambda^2} & C^*(\Lambda\times\Lambda) \ar[d]^{\ol\lambda^\vee} \\ 
  B_*\HH \ar[r]^\G & C_*(\Om) \ar[r]^I & C^*(\Lambda) \\
}
\end{equation*}

\begin{remark}
In the preceding two diagrams all maps have degree $0$ except the
following ones: $\mu^\vee$ has degree $n$, $\ol\lambda^\vee,\G^2$ have
degree $n-1$, $c_{120},c_{210},\F^2$ have degree $n-2$, and $\G$ has degree $1$.
Note that $\G$ picks up degree $1$ because its target space $C_*(\Om)$
does not have the degree shift by $1$. The reason for this seemingly
strange convention is to have the subsequent map $I$ of degree $0$.
\end{remark}

Note that equation~\eqref{eq:homotopy2} mixes
equivariant operations ($I_\lambda$, $c_{120}\circ N_\HH$) with
nonequivariant ones ($I$, $\mu$), and the signs are not the expected
ones for a chain homotopy equation. Likewise for equation~\eqref{eq:homotopy2coprod}. 
The reason is that $\G_\lambda$ maps a cyclic complex to another cyclic complex,
whereas $\G^2$ and $\F^2$ map a cyclic complex to a noncylic one.
These apparent discrepancies fall into place in the following proof.


\begin{proof}[Proof of Theorem~\ref{thm:prodcoprod} assuming
    Propositions~\eqref{prop:chainprod} and~\eqref{prop:chaincoprod}]
In this proof we will not distinguish between
$J_\lambda:C_*(\Lambda)\to C_\lambda^*(\Om)$ and its composition with
the inclusion $C_\lambda^*(\Om)\into C^*(\Om)$ (see
Remark~\ref{rem:cyc-noncyc}), and we will not distinguish between
$\G_\lambda$ and its composition with $\iota$. 

We first prove equation~\eqref{eq:prodhom}. Consider $\alpha\in B\HH$
and two smooth simplices $f_j:B_j\longrightarrow \Lambda$, $j=1,2$
with transverse evaluations at $t=0$. Let 
$$
  h := \left<I^{\otimes 2}\circ\G^2(\alpha),\p(f_1\otimes f_2)\right>
  -\left<I^{\otimes 2}\circ\G^2\circ b_\HH(\alpha),f_1\otimes f_2\right>
$$
be the terms in~\eqref{eq:homotopy2} which will die in homology.
Since $I$ and $J$ as well as $I_\lambda$ and $J_\lambda$ are adjoint
pairs (see~\S\ref{sec:chen}), equation~\eqref{eq:homotopy2} implies
$$
  \left<c_{120}\circ N_\HH(\alpha),(\G^*)^{\otimes 2}\circ J^{\otimes 2}(f_1\otimes f_2)\right>
  = \left<\alpha,\G_\lambda^*\circ J_\lambda\circ \mu(f_1\otimes f_2)\right>+h.
$$
The same holds true for finite sums of smooth simplices in
place of $f_1$ and $f_2$ if elements of the first sum are transverse
to elements of the second sum. Assume now that $b_\HH\alpha=0$ and
consider $c_1,c_2\in H_*^{S^1}(\Lambda)$.
Recall the mark and erase maps $\MM,\EE$ from equation~\eqref{eq:Gysin}.
We represent $\MM c_1,\MM c_2$ 
by finite sums $f_1,f_2$ as above. Then, using
$\fp_{2,1,0} = (c_{120}\circ N_\HH)^*$
(see equation~\eqref{eq:dIBLcanon}), the last displayed equation descends to homology as
$$
  \left<[\alpha],\fp_{2,1,0}\circ(\G^*)^{\otimes 2}\circ J_*^{\otimes 2}(\MM c_1\otimes \MM c_2)\right>=
  \left<[\alpha],\G_\lambda^*\circ J_{\lambda*}\circ\mu(\MM c_1\otimes \MM c_2)\right>.
$$
Since $[\alpha]$ was arbitrary, using~\eqref{eq:cycaux} and
$\mu^{S^1}=\EE\circ\mu \circ(\MM\otimes\MM)$ from~\S\ref{ss:stringtop} we get
$$
\fp_{2,1,0}\circ(\G^*)^{\otimes 2}\circ \bar J_{\lambda*}^{\otimes 2}(c_1\otimes c_2)=
\G_\lambda^*\circ \bar J_{\lambda *}\circ\mu^{S^1}(c_1\otimes c_1).
$$ 

The proof of equation~\eqref{eq:coprodhom} is similar. Consider a
nondegenerate smooth simplex $f:B\longrightarrow \Lambda$
and $\alpha_1,\alpha_2\in B\HH$. Set 
$$
h:=-\left<I\circ\F^2\circ b_\HH(\alpha_1\otimes \alpha_2),f\right>
-\left<I\circ\F^2(\alpha_1\otimes \alpha_2),\p f\right>,
$$
so that equation~\eqref{eq:homotopy2coprod} implies
\begin{align*}
  & \left<2c_{210}\circ N_\HH^{\otimes 2}(\alpha_1\otimes \alpha_2), \G^*\circ J(f)\right>+ 
\wh\fp_{210}^{conn}(2\m_{2,0}, (\G^*\circ J)(f))(\alpha_1\otimes \alpha_2) \cr
  &= -\left<\alpha_1\otimes \alpha_2,(\G_\lambda^*)^{\otimes 2}\circ J_\lambda^2\circ
\ol\lambda(f)\right>+h.
\end{align*}
Again, the same holds true for a finite sum of nondegenerate
smooth simplices in place of $f$. Assume now that 
$b_\HH\alpha_1=b_\HH\alpha_2=0$ and consider $c\in H_*^{S^1}(\Lambda)$. 
We represent $\MM c$ by a finite sum $f$ as above. We combine the
terms on the left hand side using $\fp_{1,2,0} = (c_{210}\circ
N_\HH^{\otimes 2})^*$ from equation~\eqref{eq:dIBLcanon}
and the definition of the twisted coproduct $\fp_{1,2,0}^\m$, 
so that the last displayed equation descends to homology as
$$
  -\left<[\alpha_1]\otimes[\alpha_2], 2\fp_{1,2,0}^\m\circ\G^*\circ
  J_*(\MM c)\right> = 
  \left<[\alpha_1]\otimes[\alpha_2],(\G_\lambda^*)^{\otimes 2}\circ 
  J_{\lambda *}^{\otimes 2}\circ \ol\lambda(\MM c)\right>.
$$
Now on the left hand side we use $\bar
J_{\lambda *}=-J_*\MM$ from~\eqref{eq:cycaux}. For the right hand side
we compute 
$$
  J_{\lambda *}^{\otimes 2}\circ ol\lambda\circ\MM
  = \bar J_{\lambda *}^{\otimes 2}\circ\EE^{\otimes 2}\circ \ol\lambda\circ\MM
  = \bar J_{\lambda *}^{\otimes 2}\circ \ol\lambda^{S^1}
  = \bar J_{\lambda *}^{\otimes 2}\circ\lambda^{S^1}.
$$
Here the first equality holds by~\eqref{eq:cycaux}, the second one
by~\eqref{eq:GH-relation}, and the third one because the Chen map  
$\bar J_{\lambda *}=-J_*\MM$ vanishes on constant loops
by~\eqref{eq:mark0}.
Since $[\alpha_1]\otimes [\alpha_2]$ was arbitrary, inserting these
expressions we get 
$$
  2\fp_{1,2,0}^\m\circ \G_\lambda^*\circ \bar J_{\lambda*}(c) =
  (\G_\lambda^*)^{\otimes 2}\circ \bar J_{\lambda *}^{\otimes 2}\circ\lambda^{S^1}(c).
$$
\end{proof}

\subsection{The simply connected case}\label{ss:simpconn}

Recall from~\ref{ss:stringtop} that the string bracket and cobracket
descend to operations $\mu^{S^1}$ and $\lambda^{S^1}$ on
$H_*(\Lambda,q_0)$ defining an involutive Lie bialgebra structure.
In view of Theorem~\ref{thm:prodcoprod}, these should correspond to
the operations induced by $\fp_{2,1,0}$ and $2\,\fp_{1,2,0}^\m$ on the
{\em reduced homology} of $(B^{\text{\rm cyc}*}\HH,\fp_{1,1,0}^\m)$. 
Unfortunately, the differential $\fp_{1,1,0}^\m$ does {\em not} descend to
the reduced subspace of $B^{\text{\rm cyc}*}\HH$ (see
Remark~\ref{rem:reduced} below). 
On the other hand, the differential on $C_\lambda^*(\Om)$ {\em does} descend
to its reduced subcomplex $\ol C_\lambda^*(\Om)$. We use this in
connection with the isomorphism $\G_\lambda^*$ from~\eqref{eq:G-iso2} to define
the reduced homology
$$  
  \ol{H}(B^{\text{\rm cyc}*}\HH,\fp_{1,1,0}^\m) := \ol{HC}_\lambda^*(\Om).
$$
Now we specialize to the case that $M$ is simply connected. 
Then $\bar J_{\lambda*}$ is an isomorphism by Theorem~\ref{thm:cyc}
and we obtain the following corollary, which corresponds to
Corollary~\ref{cor:intro} from the Introduction.  

\begin{cor}\label{cor:prodcoprod}
In Theorem~\ref{thm:prodcoprod}, assume in addition that
$M$ is simply connected. Then the operations $\fp_{2,1,0}$ and
$2\,\fp_{1,2,0}^\m$ descend to operations $\ol\fp_{2,1,0}$ and $2\,\ol\fp_{1,2,0}^\m$
on $\ol{H}(B^{\text{\rm cyc}*}\HH,\fp_{1,1,0}^\m)$ which correspond to
$\mu^{S^1}$ and $\lambda^{S^1}$ under the isomorphism
\begin{equation*}
  \bar J_{\lambda *}:H_*^{S^1}(\Lambda,q_0)\stackrel{\cong}\longrightarrow
  \ol{H}(B^{\text{\rm cyc}*}\HH,\fp_{1,1,0}^{\m}).
\end{equation*}
\end{cor}

\begin{proof}
For the coproduct, consider the following diagram (where we have
dropped the differential $\fp_{1,1,0}^\m$):
$$
\xymatrix{
  H_*^{S^1}(\Lambda,q_0) \ar[d]^{\bar J_{\lambda*}}_\cong 
  \ar[rrr]^{\lambda^{S^1}}
  & & & H_*^{S^1}(\Lambda,q_0)^{\otimes 2} \ar[d]^{\bar J_{\lambda*}^{\otimes 2}}_\cong \\
  \ol{H}(B^{\text{\rm cyc}*}\HH) \ar[r]^{\G_\lambda^*\circ\iota_*} \ar@/_2pc/[rrr]^{2\,\ol\fp_{1,2,0}^\m}
  & H(B^{\text{\rm cyc}*}\HH) \ar[r]^{2\,\fp_{1,2,0}^\m}
  & H(B^{\text{\rm cyc}*}\HH)^{\otimes 2}
  & \ol{H}(B^{\text{\rm cyc}*}\HH)^{\otimes
    2}. \ar[l]_{(\G_\lambda^*\circ\iota_*)^{\otimes 2}}
}
$$
Here the square commutes by equation~\eqref{eq:coprodhom}, and we have
defined  
$$
  \ol\fp_{1,2,0}^\m := \frac12 \bar J_{\lambda*}^{\otimes 2}\circ
  \lambda^{S^1} \circ (\bar J_{\lambda*})^{-1}: \ol{H}(B^{\text{\rm
      cyc}*}\HH) \to \ol{H}(B^{\text{\rm cyc}*}\HH)^{\otimes 2}. 
$$
This definition is made so that $\bar J_{\lambda *}$ intertwines
$\lambda^{S^1}$ and $\ol\fp_{1,2,0}^\m$. Moreover, it follows that
the lower square in the diagram also commutes, which means that
$\ol\fp_{1,2,0}^\m$ is descended from $\fp_{1,2,0}^\m$. 
The proof for the product is analogous.
\end{proof}

\begin{remark}\label{rem:reduced}
Recall that the differential $\fp_{1,1,0}^\m$ on $B^{\text{\rm cyc}*}\HH$ 
is defined in terms of rooted trees decorated with the wedge product
at the vertices, the homotopy operator $P$ at the edges (including the
root edge), the inclusion $\iota:\HH\into\Om$ at the leaves, and the
projection $\pi:\Om\to\HH$ at the root vertex. This preserves the 
reduced subspace $\ol{B^{\text{\rm cyc}*}\HH}$ of operations vanishing on
words containing a $1$ if $P$ has the properties $P\circ\pi=0$ and
$P\circ P=0$ (corresponding to a ``special propagator'' in the
terminology of~\cite{Cieliebak-Hajek-Volkov}). 
Moreover, by~\cite[Corollary~6.20]{Cieliebak-Hajek-Volkov}, the
twisted operation $\fp_{1,2,0}^\m$ then descends to $\ol{B^{\text{\rm cyc}*}\HH}$
where it coincides with the untwisted operation $\fp_{1,2,0}$. 
Unfortunately, while $P\circ\pi=0$ can always be achieved, we do not
know how to achieve $P\circ P=0$ in our analytic setting. Therefore,
we cannot assume that $\fp_{1,1,0}^\m$ preserves $\ol{B^{\text{\rm cyc}*}\HH}$
and apply the preceding argument.
It is interesting to compare this with
Theorem~1.2 and Theorem~1.3 of~\cite{Naef-Willwacher}, which provide 
the noncyclic version of Corollary~\ref{cor:prodcoprod} above and have
no twisting in the formula for the coproduct.
The explanation for this appears to be that the condition $P\circ P=0$
can be arranged in the algebraic approach of~\cite{Naef-Willwacher}.
\end{remark}

The rest of this section is devoted to the proof of
Propositions~\ref{prop:chainprod} and~\ref{prop:chaincoprod}.
The proofs will be somewhat sketchy with respect to signs and
orientations.
Details on signs and orientations
can be found in~\cite{Volkov-thesis}.

\subsection{Proof of Proposition~\ref{prop:chainprod}}\label{ss:chainprod}

Recall  
the gluing operation $g_{210}$ on labelled graphs from
Definition~\ref{def:g-prod} applied in the first case of
equation~\eqref{eq:g210-glue}. 
The first step is to collect the terms on the left hand side of
equation~\eqref{eq:homotopy2} to a full differential under the integral. 
We fix a positive integer $s$, two smooth simplices $f_i:B_i\to \Lambda$,
$i=1,2$ with transverse time zero evaluations, and a decomposable  
$\alpha\in\HH^{\otimes s}$ of homogeneous degree.  
To simplify notation, 
for an integrable form $\beta$ on $M^{d_1}\times M^{d_2}$ we abbreviate
(using the notation from~\S\ref{ss:chenHoch})
$$
  \phi_{f_1f_2}(\beta)
  :=\left<I^{\otimes 2}(\beta),f_1\otimes f_2\right>
  =\int_{B_1\times \Delta^{d_1-1}\times B_2\times
    \Delta^{d_2-1}}(ev_{f_1}\times ev_{f_2})^*\beta. 
$$
Let $left_1$ denote the first term on the left hand side
of~\eqref{eq:homotopy2} and $left_2$ the second one, without the form
$\alpha$ inserted. We compute
\begin{align*}
  left_1
  &\stackrel{(1)}{=} \phi_{f_1f_2}\circ\G^{\otimes 2}\circ c_{120}\circ
  N_\HH \cr
  &\stackrel{(2)}{=} \sum_{s_1+s_2=s+2\atop d_1\leq s_1,d_2\leq  s_2}
  \sum_{\Gamma_1\in \RR_{s_1;d_1}\atop \Gamma_2\in \RR_{s_2;d_2}}
  \phi_{f_1f_2}\circ(\G_{\Gamma_1}\otimes\G_{\Gamma_2})\circ
  c_{120}\circ N_\HH \cr
  &\stackrel{(3)}{=} \sum_{s_1+s_2=s+2\atop d_1\leq s_1,d_2\leq  s_2}
  \sum_{\Gamma_1\in \RR_{s_1;d_1}\atop \Gamma_2\in \RR_{s_2;d_2}}
  \phi_{f_1f_2}\circ \bH_{g_{210}(\Gamma_1,\Gamma_2)}\circ N_\HH \cr
  &\stackrel{(4)}{=} \sum_{d_1+d_2\leq s+2}\sum_{\Gamma\in \RR_{s;d_1,d_2}^{\rm sep}}
  \phi_{f_1f_2}\circ\bH_{\Gamma}.
\end{align*}
Here equality (1) is simply the definition of $left_1$ and $\phi_{f_1f_2}$;
equality (2) is writing out the definition of $\G$
in~\eqref{eq:G-redef} (which agrees with the one in~\eqref{eq:bG} by 
Proposition~\ref{prop:G-consist});
equality (3) is the first assertion in Lemma~\ref{lem:H-gluing};
and equality (4) is equation~\eqref{eq:tradesep}.
Similarly, we compute
\begin{align*}
  left_2
  &\stackrel{(1)}{=} \phi_{f_1f_2}\circ\G^2 \circ b_\HH \cr
  &\stackrel{(2)}{=} \fp_{2,1,0}\bigl(\m_{1,0}\otimes(\phi_{f_1f_2}\circ
  \G^2)\bigr) \cr
  &\stackrel{(3)}{=} \fp_{2,1,0}\bigl(\phi_{f_1f_2}(\m_{1,0}\otimes \G^2)\bigr) \cr
  &\stackrel{(4)}{=} \sum_{s_1+s_2=s+2\atop d_1+d_2\leq
    s_2}\sum_{\Gamma_1\in \RR_{s_1}\atop \Gamma_2\in \RR_{s_2;d_1,d_2}}
  \fp_{2,1,0}\bigl(\phi_{f_1f_2}(\m_{\Gamma_1}\otimes\G^2_{\Gamma_2})\bigr) \cr
  &\stackrel{(5)}{=} \sum_{s_1+s_2=s+2\atop d_1+d_2\leq
    s_2}\sum_{\Gamma_1\in \RR_{s_1}\atop \Gamma_2\in \RR_{s_2;d_1,d_2}}
  \phi_{f_1f_2}\circ(\m_{\Gamma_1}\otimes\G^2_{\Gamma_2})\circ c_{120}\circ
  N_\HH \cr
  &\stackrel{(6)}{=} \sum_{s_1+s_2=s+2\atop d_1+d_2\leq
    s_2}\sum_{\Gamma_1\in \RR_{s_1}\atop \Gamma_2\in \RR_{s_2;d_1,d_2}}
  \phi_{f_1f_2}\circ\bH_{g_{210}(\Gamma_1\amalg\Gamma_2)}\circ N_\HH \cr
  &\stackrel{(7)}{=} \sum_{d_1+d_2\leq s+1}\sum_{\Gamma\in
    \RR_{s;d_1,d_2}^{\rm nonsep}} \phi_{f_1f_2}\circ\bH_{\Gamma}. 
\end{align*}
Here equality (1) is simply the definition of $left_2$ and $\phi_{f_1f_2}$;
for (2) we recall $\m_{1,0}$ from~\S\ref{ss:MCan} and
use equation~\eqref{eq:p110m};
for (3) we interchange ``$\phi_{f_1f_2}$'' and 
``$\m_{1,0}\otimes$'', which holds in view of
\begin{align*}
  \m_{1,0}\otimes(\phi_{f_1f_2}\circ\G^2)(\beta_1\otimes \beta_2)
  &= \m_{1,0}(\beta_1)(\phi_{f_1f_2}\circ\G^2)(\beta_2) \cr
  = \phi_{f_1f_2}\bigl(\m_{1,0}(\beta_1)\G^2(\beta_2)\bigr) 
  &= \phi_{f_1f_2}\bigl((\m_{1,0}\otimes\G^2)(\beta_1\otimes\beta_2)\bigr),
\end{align*}
where we use the convention from~\cite{Cieliebak-Fukaya-Latschev} for 
pulling operations past elements without signs.
For (4) we write out the definition of $\m_{1,0}$ using~\eqref{eq:mlg}
with $(\ell,g)=(1,0)$
and the definition~\eqref{eq:G2-def} of $\G^2$ in terms of graphs;
for (5) we use $\fp_{2,1,0} = (c_{120}\circ N_\HH)^*$ from~\eqref{eq:dIBLcanon};
for (6) we use the second assertion in Lemma~\ref{lem:H-gluing};
and for (7) we use equation~\eqref{eq:tradenonsep}.

Combining the two computations using~\eqref{eq:Rom-splitting} and
inserting $\alpha\in B\HH$ we obtain 
\begin{equation*}
(left_1+left_2)(\alpha)
=\sum_{d_1+d_2\le s+2}\sum_{\Gamma\in R_{s;d_1,d_2}}
\sum_{l\in \Edge(\Gamma)}\phi_{f_1f_2}\circ \bH_{\Gamma,l}(\alpha).            
\end{equation*}
Recall the sign exponent
$$
  \fs_\p=\bar R_\Gamma+(n-1)\eta_3(\Gamma)+
  \wt s^2_\Gamma(\alpha)
$$
from equation~\eqref{eq:s-is-good}.
Here in accordance with Remark~\ref{rem:L1-degrees} we write $\wt
s^2_\Gamma(\alpha)$ in place of $\wt s^2_{\Gamma,p,q}(\alpha)$
because the bidegree $(p,q)$ is determined by~\eqref{eq:pq}.  
Then
\begin{equation}\label{eq:main-integral}
\begin{aligned}
  (left_1+left_2)(\alpha)
  &\stackrel{(1)}{=}\sum_{d_1+d_2\le s+2}\sum_{\Gamma\in
  R_{s;d_1,d_2}}(-1)^{\fs_\p}\phi_{f_1f_2}\int_{\Delta^\Gamma_\ver}dR_\Gamma^*G^e(\alpha) \cr
  &\stackrel{(2)}{=}\sum_{d_1+d_2\leq s+2}\sum_{\Gamma\in \RR_{s;d_1,d_2}}(-1)^{\fs_\p}
  \int_{(\XX_\Gamma)_0}d\wt R_\Gamma^*\wt G^e(\alpha) \cr 
  &\stackrel{(3)}{=}\sum_{d_1+d_2\leq s+2}\sum_{\Gamma\in \RR_{s;d_1,d_2}}(-1)^{\fs_\p}
  \int_{\p^\main\XX_\Gamma}\wt R_\Gamma^*\wt G^e(\alpha). 
\end{aligned}
\end{equation}
Here for equality (1) we use Lemma~\ref{lem:fulldiffeven}; 
for (2) we use equation~\eqref{eq:angenfin}; 
and for (3) we use Proposition~\ref{prop:mainStokes}.
The space $\XX=\XX_\Gamma$ appearing in the last two displayed
equations is defined in~\S\ref{ss:general-setting}. It is associated
to a tree $\Gamma$ with an extended labelling and two special vertices
of degrees $d_1,d_2$ and the evaluation map 
\begin{equation}\label{eq:phi-ev}
   \phi = \ev_{f_1}\times\ev_{f_2}: \WW = B_1\times
   \Delta^{d_1-1}\times B_2\times \Delta^{d_2-1} \to M^{d_1}\times M^{d_2}
\end{equation}
as the proper transform of the vertex diagonal
$\Delta_\ver=\Delta^\Gamma_\ver$ times the graph of $\phi$,
$$
   \XX_\Gamma = PT(\Delta_\ver\times gr(\phi))\subset \wt X_\Gamma\times\WW.
$$
By Proposition~\ref{prop:mainStokes}, $\XX_\Gamma$ is a manifold with corners.
According to Remark~\ref{rem:descr-primary}, its primary boundary components
fall into two groups: those corresponding to the boundary of $\WW$,
and those corresponding to the edges of $\Gamma$.
We decompose the first group further as follows, ignoring sets of
measure zero. Let $\pi_{B_i}$ and $\pi_{\Delta^{d_i-1}}$ denote the
projection maps from $\XX_\Gamma$ to $B_i$ and $\Delta^{d_i-1}$, respectively. Set
$$
  \p_{B_i}\XX_\Gamma := \pi_{B_i}^{-1}(\p B_i),\qquad
  \p^{\Delta^{d_i-1}}_j\XX_\Gamma := \pi_{\Delta^{d_i-1}}^{-1}(\p_j\Delta^{d_i-1}),
$$
where $\p_j\Delta^{d_i-1}$ is the $j$-th boundary component of
$\Delta^{d_i-1}$, see~\eqref{eq:bdrycyc}. 
We now discuss the contributions of the various boundary components to
the last integral in~\eqref{eq:main-integral}. 

1. The boundary components $\p_l\XX_\Gamma$ corresponding to nonspecial edges
$l$ cancel by duality, see Figure~\ref{fig:duality}. This is analogous
to the discussion in~\S 7.2 of~\cite{Cieliebak-Volkov}.  

2. The boundary components $\p_{B_1}\XX_\Gamma$ and
$\p_{B_2}\XX_\Gamma$ give rise to the
second term on the right hand side of~\eqref{eq:homotopy2}. 

3. For $i=1,2$ and $j\in\{1,\dots,d_i\}$, the boundary component
$\p^{\Delta^{d_i-1}}_j\XX_\Gamma$ cancels with $\p_l\XX_{\Gamma_j^i}$.
Here the tree $\Gamma_j^i$ is obtained from $\Gamma$ by attaching a
leg at the flag number $j$ at the $i$-th special vertex
(see~\S\ref{ss:op-graphs2}), and $\p_l\XX_{\Gamma_j^i}$ is the
boundary component corresponding to the new special (but not doubly
special) edge $l$ created by attaching the leg.

To see this, we use the freedom from Lemma~\ref{lem:bdryindep} to
choose as extension of the labelling on the tree $\Gamma_j^i$
the canonical one from~\S\ref{ss:op-graphs2}. 
Consider $i=1$ and $j\in \{1,\dots,d_1\}$ (the case $i=2$ is
analogous). To simplify notation we set $\Gamma_j:=\Gamma_j^1$.
Figure~\ref{fig:cancel} shows the situation at the first special vertex.
\begin{figure}
\begin{center}
\includegraphics[angle=0,origin=c,width=\textwidth]{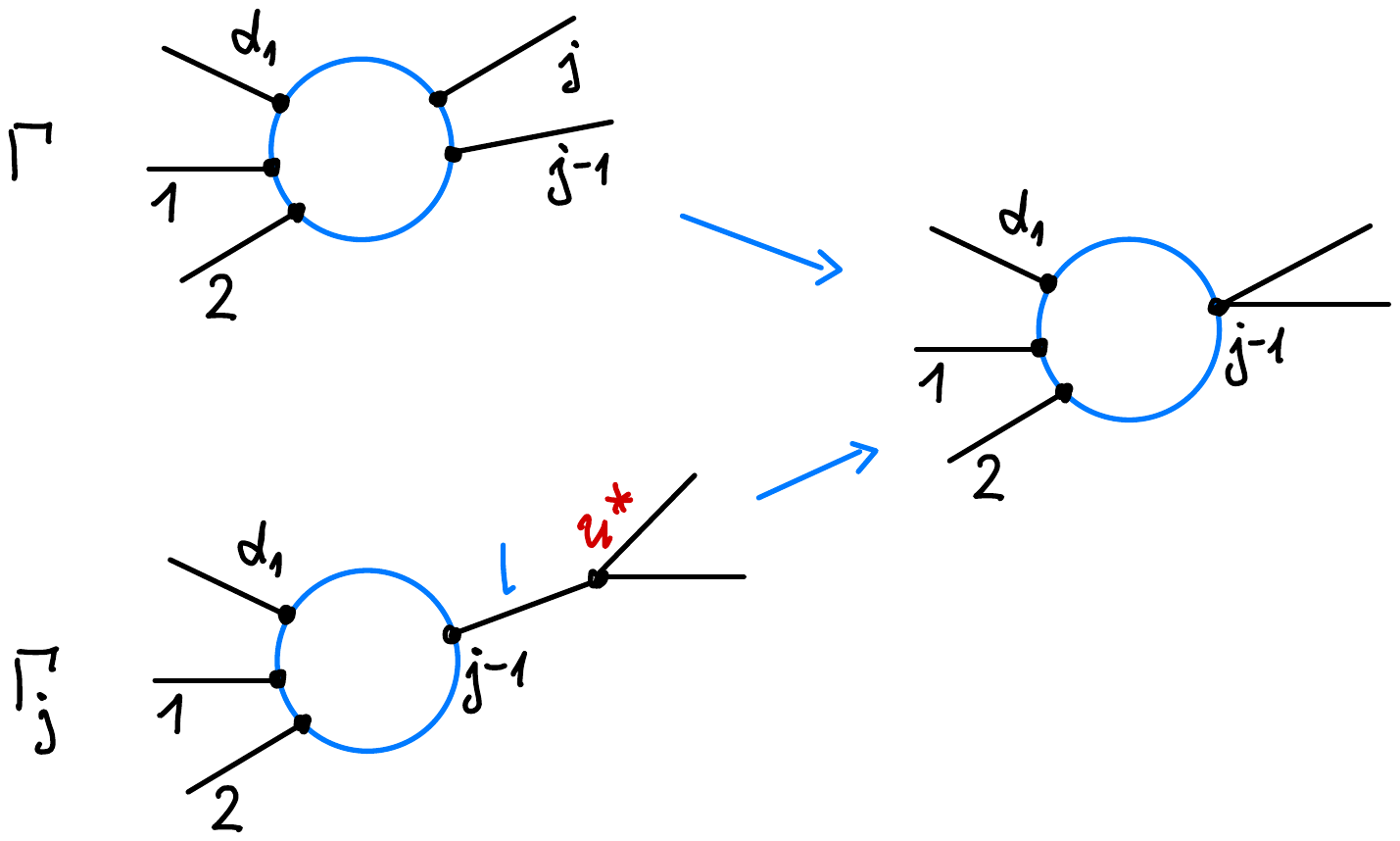}
\vspace{-2cm}
\caption{Cancellation of boundary strata}
\label{fig:cancel} 
\end{center}
\end{figure}
Recall that 
$$
  \p_j\Delta^{d_1-1}=\{t_{j-1}=t_j\},
$$
where we set $t_0=0$ and $t_{d_1+1}=1$. 
Let $V$ denote the number of nonspecial vertices of $\Gamma$ and
$(v_1,\dots,v_V)$ the corresponding variables in $M$. The variables
corresponding to the nonspecial vertices of $\Gamma_j$
are denoted by $(u_1,\dots,u_V,u^*)$, where $u^*$ is the variable
corresponding to the new vertex created by attaching the leg. Set
\begin{equation*}
  V_{\Gamma_j} := \{((u_1,\dots,u_V,u^*),p,(\tau_1,\dots,\tau_{d_1-2}))
  \in \Delta^{\Gamma_j}_\ver\times B_1
  \times\Delta^{d_1-2}\}
\end{equation*}
and 
\begin{equation*}
  V_\Gamma^j := \{((v_1,\dots,v_V),p,(t_1,\dots,t_{d_1-1})
  \in \Delta^\Gamma_\ver\times B_1\times\p_j\Delta^{d_1-1}\}.
\end{equation*}
Note that in $V_\Gamma^j$ we have $t_{j-1}=t_j$. Set 
$$
  W_{\Gamma_j}:=V_{\Gamma_j}\times B_2\times \Delta^{d_2-1}\qquad\text{and}\qquad 
  W_\Gamma^j:=V_\Gamma^j\times B_2\times \Delta^{d_2-1}.
$$
The embedding $\iota^\phi$ from~\eqref{eq:canphi} gives 
us the full measure inclusion
\begin{equation}\label{eq:WGammaj}
  \p^{\Delta^{d_1-1}}_j\XX_\Gamma\into W_\Gamma^j\,.
\end{equation}
On the other hand, recall from Remark~\ref{rem:descr-primary}
and~\eqref{eq:diagdtea} that $\p_j\XX_{\Gamma_j}$ is an 
$S^{n-1}$-fibration over $(\iota_\ver\times\phi_0)^{-1}X_{\{l\}}$. Here
$\phi_0:=\phi|_{\WW_0}$ and $X_{\{l\}}\subset X_{\Gamma_j}$ is the subset
where precisely the variables at the ends of the new edge $l$ agree. 
In other words, we have a fibration
\begin{equation}\label{eq:fibration}
  S^{n-1}\to \p_l\XX_{\Gamma_j}\to 
  (\iota_\ver\times\phi_0)^{-1}X_{\{l\}}.
\end{equation}
The embedding $\iota^\phi$ from~\eqref{eq:canphi} gives 
us the full measure inclusion
\begin{equation}\label{eq:WGammajj}
(\iota_\ver\times\phi_0)^{-1}X_{\{l\}}\into W_{\Gamma_j}.
\end{equation}
Observe now that there is a diffeomorphism 
$$
  \Psi:V_{\Gamma_j}\stackrel{\cong}{\longrightarrow} V_\Gamma^j
$$
defined by
\begin{align*}
  (t_1,\dots,t_{d_1-1}) &:= (\tau_1,\dots,\tau_{j-1},\tau_{j-1},\dots,\tau_{d_1-2}),\cr
  (v_1,\dots,v_V) &:= (u_1,\dots,u_V).
\end{align*}
The diffeomorphism $\Psi$ induces a diffeomorphism 
\begin{equation*}
  \Psi\times\id:W_{\Gamma_j}\stackrel{\cong}{\longrightarrow} W_\Gamma^j.
\end{equation*}
Since the graph $\Gamma_j$ has the edges of $\Gamma$ plus the new edge
$l$, the integrands differ by one Green kernel factor, 
$$
  \wt R_{\Gamma_j}^*\wt G^{e+1}(\alpha) = \pm \wt G\wedge \wt R_\Gamma^*\wt G^e(\alpha).  
$$
Therefore, we get
\begin{equation*}
  \int_{\p_l\XX_{\Gamma_j}}\wt R_{\Gamma_j}^*
  \wt G^{e+1}(\alpha)
  = \pm \int_{W_{\Gamma_j}}R_\Gamma^*G^e(\alpha) 
  = \pm \int_{W_\Gamma^j}R_\Gamma^*G^e(\alpha) 
  = \pm \int_{\p^{\Delta^{d_1-1}}_j\XX_\Gamma}\wt R_\Gamma^*\wt G^e(\alpha).
\end{equation*}
Here the first equality follows by integration over the fibre of the
fibration~\eqref{eq:fibration} using~\eqref{eq:fibre-int-G}
and the inclusion~\eqref{eq:WGammajj}, the
second one by invariance of integration under the diffeomorphism
$\Psi\times\id$, and the third one from  
the inclusion~\eqref{eq:WGammaj}.  
This proves the asserted cancellation modulo signs. 
We refer to~\cite{Volkov-thesis} for the straightforward but tedious
computation that these two terms come with opposite signs. 


4. The boundary components from doubly special edges give rise to the
Chas--Sullivan term --- the first term on the right hand side
of~\eqref{eq:homotopy2}.

To see this, we write the target $M^{d_1}\times M^{d_2}$
of the evaluation map $\phi$ from~\eqref{eq:phi-ev}
as $M^2\times M^{d_1-1}\times M^{d_2-1}$,
where the second $M$ factor has been moved 
to the second place from the place number
$d_1+1$. Then the evaluation map can be written as 
$$
  \phi = \ev_{f_1}^0\times \ev_{f_2}^0 \times \wh {\ev_{f_1}}\times \wh {\ev_{f_2}},
$$
where $\ev_f^0$ denotes the evaluation at time $0$ and $\wh{\ev_f}$
the evaluation at the other times as in~\eqref{ev-hat0}.
Recall from~\S\ref{ss:defchainprod} that the domain of
the Chas--Sullivan loop product $\mu(f_1\otimes f_2)$ is defined as the
fibre product 
$$
  B_{12} = D_{\mu(f_1,f_2)} := (\ev_{f_1}^0\times \ev_{f_2}^0)^{-1}(\Delta_2) =
  \{(p_1,p_2)\in B_1\times B_2\mid f_{1,p_1}(0)=f_{2,p_2}(0)\}.
$$
So the restriction of $\phi$ to $B_{12}\times \Delta^{d_1-1}\times
\Delta^{d_2-1}$ writes out as 
\begin{equation}\label{eq:liprod}
  (p,t^1,t^2)\mapsto \Bigl((\ev_{f_1}^0\times \ev_{f_2}^0)(p),
  \wh{ev_{\mu(f_1\otimes f_2)}}\bigl(p,\frac{1}{2}(t^1t^2)\bigr)\Bigr),
\end{equation}
where $\frac{1}{2}(t^1t^2)$ is obtained by
adding $1$ to all entries of $t^2$ and putting them to the right
of $t^1$, and then dividing all the entries by $2$.

Consider now a labelled tree $\Gamma$ with one special vertex and
generalized labelling as in~\S\ref{ss:basiccomb}, i.e., a
numbering $1,\dots,s$ of its leaves and a numbering 
$1,\dots,d$ of its
special flags, where $1\leq d\leq s$. 
For $0\leq k\leq d$ let $\Gamma_k$ be the labelled tree obtained
from $\Gamma$ by splitting its special vertex into two special
vertices $1,2$ connected by an edge such that the special flags
$1,\dots,k$ end up on special vertex $1$. See Figure~\ref{fig:cs}.
\begin{figure}
\begin{center}
\includegraphics[angle=0,origin=c,width=\textwidth]{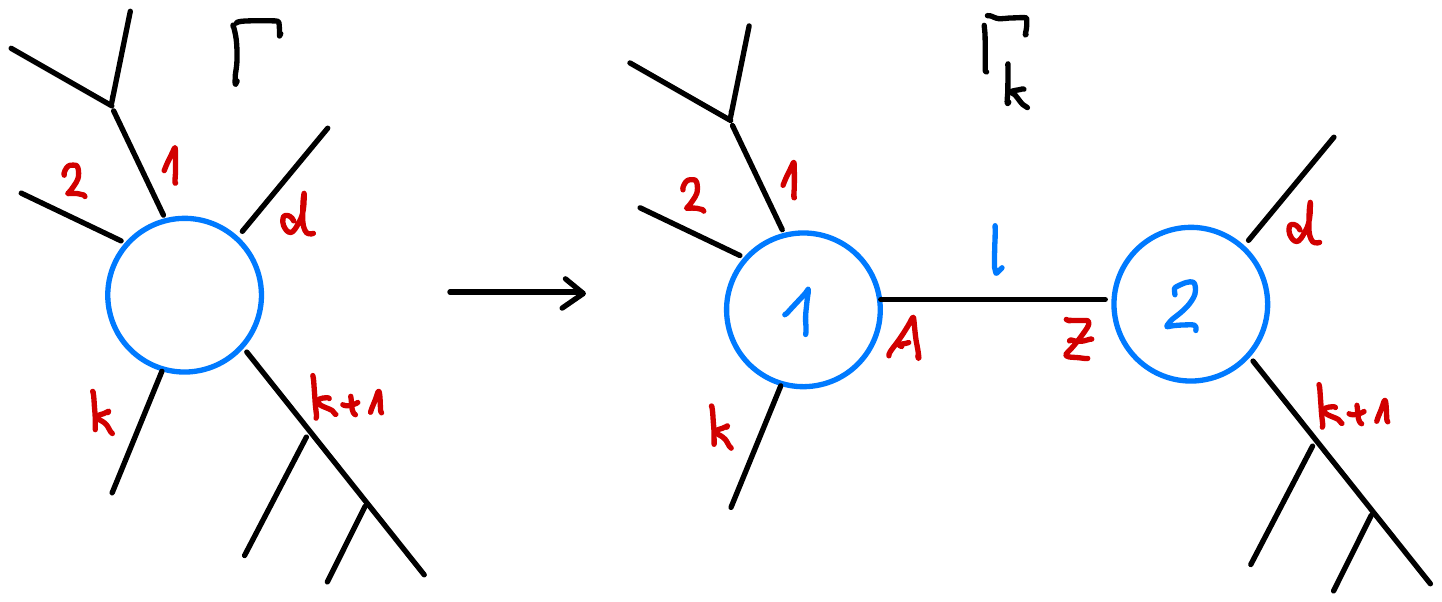}
\vspace{-3cm}
\caption{The Chas--Sullivan term}
\label{fig:cs} 
\end{center}
\end{figure}
We denote the doubly special edge of $\Gamma_k$ by $l$ and its flags
on special vertices $1$ and $2$ by $A$ and $Z$, respectively, so that
$l=(A,Z)$.

\begin{remark}\label{rem:del-doubly-spec}
The edge and vertex orders of $\Gamma$ induce ones for $\Gamma_k$ by
making its doubly special edge the first one in the edge order. 
This allows us to compare the reordering maps $\bar R_\Gamma$ and 
$\bar R_{\Gamma_k}$. Let $f$ denote the number of flags of $\Gamma$,
so that $\Gamma_k$ has $f+2$ flags. The 
last $f$ flags of $\Gamma_k$
are canonically identified with the flags of $\Gamma$ and the new
flags $A$ and $Z$ get numbers $1$ and $2$, respectively. Denoting
by $\id_{\{1,2\}}$ the identity map on the set $\{1,2\}$, we
have the relation 
\begin{equation*}
\bar R_{\Gamma_k}=r\circ (\id_{\{1,2\}}\times \bar R_\Gamma),
\end{equation*}
where $r$ is the bijection responsible for moving the flags $A$ and
$Z$ to their positions in the vertex order of 
$\Gamma_k$. This
involves moving $(A,Z)$ as a whole and then moving $Z$ to the right
past the first $k$ special flags of $\Gamma$. Hence, the sign
exponents are related by
\begin{equation*}
  \bar R_{\Gamma_k}=\bar R_\Gamma+k.
\end{equation*}
A straightforward but tedious analysis of the relation between the
orientation complexes for $\Gamma$ and $\Gamma_k$
(see~\cite{Volkov-thesis}) reveals that
\begin{equation*}
\eta_3(\Gamma_k)=\eta_3(\Gamma).
\end{equation*}
\end{remark}

Let
$$
  \Delta^m_k := \{t\in \Delta^m\mid 0\leq t_1\leq\dots\leq t_k\leq 1/2
  \leq t_{k+1}\leq\dots\leq t_m\leq 1\}\subset \Delta^m.
$$
Note that multiplication with $1/2$ maps $\Delta^k\times\Delta^{m-k}$
diffeomorphically onto $\Delta^m_k$.
We identify $\Delta^k\times\Delta^{m-k}$ with $\Delta^m_k$
under this diffeomorphism and observe the splitting
\begin{equation}\label{eq:splitsimpl}
  \Delta^m = \bigcup_{k=0}^m\Delta^m_k,
\end{equation}
where the subsimplices intersect only along faces.

Consider the compactified configuration space $\XX_{\Gamma_k}$
associated to $\Gamma_k$ and its primary boundary component
$\p_l\XX_{\Gamma_k}$ corresponding to the doubly special edge $l$.  
Our task is to reinterpret the integral
$$
  K_{\Gamma_k}(\alpha) := (-1)^{\fs_\p} \int_{\p_l\XX_{\Gamma_k}}R_{\Gamma_k}^*\wt G^{e+1}(\alpha),
$$
with the sign exponent $\fs_\p$ from 
equation~\eqref{eq:s-is-good} and $e$ the number of edges of $\Gamma$. 
Recall from Remark~\ref{rem:descr-primary}
and~\eqref{eq:diagdtea} that $\p_j\XX_{\Gamma_j}$ is an 
$S^{n-1}$-fibration over $(\iota_\ver\times\phi_0)^{-1}X_{\{l\}}$, that is
\begin{equation*}
  S^{n-1}\to \p_l\XX_{\Gamma_j}\to 
  (\iota_\ver\times\phi_0)^{-1}X_{\{l\}}.
\end{equation*}
The embedding $\iota^\phi$ from~\eqref{eq:canphi} gives 
us the full measure inclusion
\begin{equation*}
(\iota_\ver\times\phi_0)^{-1}X_{\{l\}}\into 
\Delta_\ver\times B_{12}\times\Delta^{d}_k.
\end{equation*}
We use fibre integration to kill the $\wt G$ factor in the integrand
corresponding to the doubly special edge to get
$$
  K_{\Gamma_k}(\alpha)
  = (-1)^{\star\star}\int_{\Delta_\ver\times B_{12}\times\Delta^{d}_k}
  (\id\times\wh{ev_{\mu(f_1\otimes f_2)}})
  ^*R_{\Gamma}^*
  G^e(\alpha)
$$
for some sign exponent $\star\star$. 
The pullback in the
integrand follows from formula~\eqref{eq:liprod}. 
We apply equation~\eqref{eq:keychenan1} to the right hand side of the last displayed equation to get
\begin{equation}\label{eq:Kalpha}
\begin{aligned}
K_{\Gamma_k}(\alpha)=&(-1)^{\star\star}\int_{B_{12}\times\Delta^{d}_k}
\wh{ev_{\mu(f_1\otimes f_2)}}^*
\int_{\Delta_\ver}R_{\Gamma}^*
  G^e(\alpha)
  \cr
=&(-1)^{\star\star\star}\int_{B_{12}\times
\Delta^{d}_k}
\wh{ev_{\mu(f_1\otimes f_2)}}^*
I_\Gamma(\alpha),
\end{aligned}
\end{equation}
where for the second equality we have used definition~\eqref{eq:defI} 
of the operation $I_\Gamma$.
Let now $\RR_\Gamma=\{\Gamma_k\}_{k=0}^{d}$ denote the set of all
labelled trees (up to isomorphism) with two special vertices arising
by splitting the special vertex of $\Gamma$ as above. Let us denote by 
$\RR_{s;d_1,d_2}^{DS}$ the subset of $\RR_{s;d_1,d_2}$ corresponding
to graphs with a doubly special edge.
The set of isomorpism classes of generalized labelled trees 
with one special $d$-valent vertex and $s$ leaves will be denoted by
$\RR_{s;d}^\gen$. These sets are related by 
\begin{equation}\label{eq:DS}
\coprod_{d_1+d_2=d+2}\RR_{s;d_1,d_2}^{DS}=
\coprod_{\Gamma\in \RR_{s;{d}}^\gen}\RR_\Gamma.
\end{equation}
In order to see the desired contribution to the first term on the 
right hand side of equation~\eqref{eq:homotopy2}
we compute 
\begin{align*}
&\sum_{d_1+d_2=d+2}\sum_{\wh\Gamma\in\RR_{s;d_1,d_2}^{DS}}K_{\wh\Gamma}(\alpha)\cr
&\stackrel{(1)}{=}\sum_{\Gamma\in\RR_{s;{d}}^\gen} \sum_{k=0}^dK_{\Gamma_k}(\alpha)\cr
&\stackrel{(2)}{=}\sum_{\Gamma\in\RR_{s;{d}}^\gen} (-1)^{\star\star\star}\int_{B_{12}\times
\Delta^{d}}
  \wh{ev_{\mu(f_1\otimes f_2)}}^*I_\Gamma(\alpha)\cr
&\stackrel{(3)}{=}\sum_{\Gamma\in\RR_{s;
{d}}} (-1)^{\star\star\star}\int_{B_{12}\times\Delta^{d}}
\wh{ev_{\mu(f_1\otimes f_2)}}^*(N_{an}\circ I_\Gamma(\alpha))\cr
&\stackrel{(4)}{=}\sum_{\Gamma\in\RR_{s;{d}}} (-1)^{\star\star\star\star}\int_{B_{12}\times
\Delta^{d}}
\wh{ev_{\mu(f_1\otimes f_2)}}^*(N_{an}\circ \G_\Gamma(\alpha))\cr
&\stackrel{(5)}{=}\left<I_\lambda\circ
\G_{s;{d}}(\alpha),\mu(f_1\otimes f_2)\right>.
\end{align*}
Here for equality~(1) we use~\eqref{eq:DS}, and
for equality~(2) we use~\eqref{eq:splitsimpl} and~\eqref{eq:Kalpha}.
For equality~(3) we recall that $\RR_{s;d}$ is a fundamental
locus for the free $\Z_{d}$ action on $\RR_{s;d}^\gen$ by
cyclicly relabelling the special flags, and we trade each orbit for the
symmetrization operator $N_{an}$ using~\eqref{eq:rotation-spec-cor}. 
Equality~(4) follows (up to sign) from
definition~\eqref{eq:def-G-Gamma} of $\G_\Gamma$, 
and equality~(5) follows (up to sign) from the second equality
in~\eqref{eq:cycchenpair} and definition~\eqref{eq:def-G-sd}
of $\G_{s;d}$.

Since the last displayed expression is the first term on the right
hand side of~\eqref{eq:homotopy2}, this concludes the proof of 
Proposition~\ref{prop:chainprod}.

\subsection{Proof of Proposition~\ref{prop:chaincoprod}}
\label{ss:chaincoprod}

We follow the same strategy as for Proposition~\ref{prop:chainprod}.
The first step is to collect the terms on the left hand side of
equation~\eqref{eq:homotopy2coprod} to a full differential under the integral. 
We fix a nondegenerate smooth simplex $f:B\to \Lambda$ and decomposable  
$\alpha_i\in\HH^{\otimes s_i}$ of homogeneous degree, $i=1,2$.  
We denote the three terms in the left hand side
of~\eqref{eq:homotopy2coprod} without the form
$\alpha_1\otimes\alpha_2$ inserted by $left_1$, $left_2$ and $left_3$.
For an integrable form $\beta$ on $M^d$ we abbreviate (using the notation from~\S\ref{ss:chenHoch})
$$
  \phi_f(\beta)
  :=\left<I(\beta),f\right>
  =\int_{B\times \Delta^{d-1}}
  ev_f^*\beta. 
$$
Recall
the operation $g_{120}^j$ on labelled graphs from
Definition~\ref{def:g-coprod} applied in the situation of
equation~\eqref{eq:glueg1}.
We compute
\begin{align*}
left_1
&\stackrel{(1)}{=}\phi_f\circ\G\circ 2c_{210}\circ N^{\otimes 2}_\HH\cr
&\stackrel{(2)}{=}\sum_{d\leq s_1+s_2+2}\sum_{\Gamma\in\RR_{s_1+s_2+2;d}}
\phi_f\circ\G_\Gamma\circ 2c_{210}\circ N^{\otimes 2}_\HH\cr
&\stackrel{(3)}{=}\sum_{d\leq s_1+s_2+2}\sum_{\Gamma\in\RR_{s_1+s_2+2;d}}
\sum_{3\le j\le s_1+s_2+1}\phi_f\circ\bH_{g_{120}^j(\Gamma)}\circ N^{\otimes 2}_\HH\cr
&\stackrel{(4)}{=}\sum_{d\leq s_1+s_2+2}\sum_{\Gamma\in\RR_{s_1,s_2;d}^{cm}}
\phi_f\circ\bH_{\Gamma}.
\end{align*}
Here equality~(1) is simply the definition of $left_1$ and 
$\phi_f$;
equality~(2) is writing out the definition
~\eqref{eq:def-G-sd} and~\eqref{eq:G-redef} of $\G$;
equality~(3) is Lemma~\ref{lem:H-gluing1};
and equality~(4) is equation~\eqref{eq:tradecoprod1}.

For $left_2$ we introduce the following notation. The map $\tau_{23}$
is the relabelling swapping the order of the second  
and the third boundary components of a graph with three boundary components; its 
(algebraic) action on $(B^{cyc*}\HH)^{\otimes 3}$ 
swapping the last two factors is denoted by the same letter.
The map $\tau$ from~\eqref{eq:renumber-bdry} 
is the relabelling swapping the order of the boundary components of a
graph with two boundary components; its (algebraic) action on on  
$B^{cyc*}\HH\otimes B^{cyc*}\HH$ swapping the two factors is denoted
by the same letter.
Recall the operation $g_{210}$ on labelled
graphs from Definition~\ref{def:g-prod} applied in the situation of 
equation~\eqref{eq:glueg2}.
Recall $\m_{2,0}$ from~\S\ref{ss:MCan} and abbreviate
$\psi:=(\G^*\circ J)(f)$.  
Then
\begin{align*}
&left_2\cr
&\stackrel{(1)}{=}\wh\fp_{2,1,0}^\conn(2\m_{2,0}\otimes\psi)\cr
&\stackrel{(2)}{=}
2p_{210}^{12}
(\tau_{23}(2\m_{2,0}\otimes\psi))\cr
&\stackrel{(3)}{=} 
2\tau_{23}(2\m_{2,0}\otimes 
(\G^*\circ J)(f))\circ (c_{120}\otimes\id)
\circ (N_\HH\otimes \id)\cr
&\stackrel{(4)}{=}
2\sum_{d\le r_2\atop s_1+r_2\ge 3}\sum_{\Gamma_1\in \RR_{s_1,r_1}}\sum_{\Gamma_2\in \RR_{r_2;d}}
\tau_{23}(\m_{\Gamma_1}\otimes 
(\phi_f\circ\G_{\Gamma_2}))\circ (c_{120}\otimes\id)
\circ (N_\HH\otimes \id)\cr
&\stackrel{(5)}{=}
2\sum_{d\le r_2\atop s_1+r_2\ge 3}\sum_{\Gamma_1\in \RR_{s_1,r_1}}\sum_{\Gamma_2\in \RR_{r_2;d}}
\phi_f\circ
\tau_{23}[\m_{\Gamma_1}\otimes
\G_{\Gamma_2}]
\circ (c_{120}\otimes\id)
\circ (N_\HH\otimes \id)\cr
&\stackrel{(6)}{=}
2\sum_{d\le r_2
\atop s_1+r_2\ge 3}\sum_{\Gamma_1\in \RR_{s_1,r_1}}\sum_{\Gamma_2\in \RR_{r_2;d}}
\phi_f\circ 
\bH_{g_{210}((\Gamma_1\amalg \Gamma_2)\tau_{23})}
\circ (N_\HH\otimes \id)\cr 
&\stackrel{(7)}{=}
2\sum_{d\le s_1+s_2}\sum_{\Gamma\in\RR_{s_1,s_2;d}^{ncb1}}
\phi_f\circ 
\bH_{\Gamma}\stackrel{(8)}{=}
2\sum_{d\le s_1+s_2}\left(\frac{1}{2}(\id+\tau)\right)\sum_{\Gamma\in\RR_{s_1,s_2;d}^{ncb1}}
\phi_f\circ 
\bH_{\Gamma}\cr
&\stackrel{(9)}{=}\sum_{d\le s_1+s_2}
\sum_{\Gamma\in\RR_{s_1,s_2;d}^{ncb1}}
\phi_f\circ 
(\bH_{\Gamma}+\bH_{\tau\Gamma})\cr
&\stackrel{(10)}{=}\sum_{d\le s_1+s_2}\sum_{\Gamma\in\RR_{s_1,s_2;d}^{ncb1}\amalg \RR_{s_1,s_2;d}^{ncb2}}
\phi_f\circ 
\bH_{\Gamma}.
\end{align*}
Here equality (1) follows from the definition of $left_2$ and $\psi$;
equality (2) from~\eqref{eq:hat-unwrap}; and 
equality~(3) from~\eqref{eq:p21012} and~\eqref{eq:dIBLcanon} and the
definition of $\psi$. 
For~(4) we write out the definition of $\m_{2,0}$ using~\eqref{eq:mlg}
and~\eqref{eq:mGamma} with $(\ell,g)=(2,0)$
and the definition~\eqref{eq:def-G-Gamma},~\eqref{eq:def-G-sd}
and~\eqref{eq:G-redef} of $\G$ in terms of graphs and make use of
$\phi_f$. 
Equality~(5) follows from 
\begin{align*}
[\m_{\Gamma_1}\otimes(\phi_f\circ \G_{\Gamma_2})]
(\beta_1\otimes\beta_2\otimes\beta_3)
&= \m_{\Gamma_1}(\beta_1\otimes\beta_2)
(\phi_f\circ \G_{\Gamma_2})(\beta_3)\cr
&= \phi_f[\m_{\Gamma_1}(\beta_1\otimes\beta_2)\G_{\Gamma_2}(\beta_3)]\cr
&= \phi_f\circ[\m_{\Gamma_1}\otimes
\G_{\Gamma_2}](\beta_1\otimes\beta_2\otimes\beta_3)
\end{align*}
for any $\beta_j\in B^{cyc}\HH$, $j=1,2,3$. 
Equality~(6) is Lemma~\ref{lem:H-gluing2}; 
equality~(7) is equation~\eqref{eq:tradecoprod2};
equality~(8) follows from the fact that 
the expression in big round brackets is the identity on 
$\wh E_2B^{cyc*}\HH[2-n]$;
equality~(9) holds since $\bH_\Gamma$ is a good operation (see Definition~\ref{def:good-alg});
and equality~(10) follows from equation~\eqref{eq:renumber-bdry1}.

For $left_3$
we again use the map $\tau$ from~\eqref{eq:renumber-bdry}.
Abbreviating $\psi:=(\F^{2*}\circ J)(f)$, we rewrite
\begin{align}\label{eq:left3}
  left_3
  = \psi\circ (b_\HH\otimes\id+\id\otimes b_\HH)
  = 2\psi\circ (b_\HH\otimes\id)
  = 2p_{210}^{12}(\m_{1,0}\otimes\psi).
\end{align}
Here the first equality is the definition of $left_3$;
the second one follows from the $\tau$-invariance of $\psi$ (which
holds because $\F^2$ is $\tau$-invariant, see~\eqref{eq:F2inv}); 
and third one follows from~\eqref{eq:p110m} and the
definition~\eqref{eq:p21012} of $p_{210}^{12}$.
Recall now the operation $g_{210}$ on labelled graphs
from Definition~\ref{def:g-prod} applied in the situation of 
equation~\eqref{eq:glueg4}.
We compute
\begin{align*}
&left_3\cr
&\stackrel{(1)}{=}
2p_{210}^{12}
(\m_{1,0}\otimes\psi)\cr
&\stackrel{(2)}{=} 
2(\m_{1,0}\otimes 
(\F^{2*}\circ J)(f))\circ (c_{120}\otimes\id)
\circ (N_\HH\otimes \id)\cr
&\stackrel{(3)}{=}
2\sum_{r_1+r_2\ge 3
\atop d\leq
    r_2+s_2+2}\sum_{\Gamma_1\in \RR_{r_1}\atop \Gamma_2\in \RR_{r_2,s_2;d}}
(\m_{\Gamma_1}\otimes 
(\phi_f\circ\F_{\Gamma_2}^2))\circ (c_{120}\otimes\id)
\circ (N_\HH\otimes \id)\cr
&\stackrel{(4)}{=}
2\sum_{r_1+r_2\ge 3
  \atop d\leq
    r_2+s_2+2}\sum_{\Gamma_1\in \RR_{r_1}\atop \Gamma_2\in \RR_{r_2,s_2;d}}
\phi_f\circ
[\m_{\Gamma_1}\otimes
\F_{\Gamma_2}^2]
\circ (c_{120}\otimes\id)
\circ (N_\HH\otimes \id)\cr
&\stackrel{(5)}{=}
2\sum_{r_1+r_2\ge 3
  \atop d\leq
    r_2+s_2+2}\sum_{\Gamma_1\in \RR_{r_1}\atop \Gamma_2\in \RR_{r_2,s_2;d}}
\phi_f\circ 
\bH_{g_{210}(\Gamma_1\amalg \Gamma_2)}
\circ (N_\HH\otimes \id)\cr
&\stackrel{(6)}{=}
2\sum_{d\leq s_1+s_2+1}
\sum_{\Gamma\in
    \RR_{s_1,s_2;d}^{nc1}}\phi_f\circ 
\bH_{\Gamma}\stackrel{(7)}{=}
2\sum_{d\leq s_1+s_2+1}\left(\frac{1}{2}(\id+\tau)\right)
\sum_{\Gamma\in
    \RR_{s_1,s_2;d}^{nc1}}\phi_f\circ 
\bH_{\Gamma}\cr
&\stackrel{(8)}{=}\sum_{d\leq s_1+s_2+1}\sum_{\Gamma\in
    \RR_{s_1,s_2;d}^{nc1}}
\phi_f\circ 
(\bH_{\Gamma}+\bH_{\Gamma\tau})\cr
&\stackrel{(9)}{=}\sum_{d\leq s_1+s_2+1}\sum_{\Gamma\in
    \RR_{s_1,s_2;d}^{nc1}\amalg
    \RR_{s_1,s_2;d}^{nc2}}\phi_f\circ \bH_{\Gamma}.
\end{align*}
Here equality~(1) follows from~\eqref{eq:left3}. 
For equality~(2) recall that $p_{210}^{12}$
is just $p_{120}=(c_{120}\circ N_{\HH})^*$ 
applied to the first two factors.
For equality~(3) we write out the definition of 
$\m_{1,0}$ using~\eqref{eq:mlg}
and~\eqref{eq:mGamma} with $(\ell,g)=(1,0)$
and $\F^2$ using~\eqref{eq:def-F-Gamma} 
and~\eqref{eq:def-F-sd} in terms of graphs and make use of $\phi_f$.
Equality~(4) follows from a computation analogous to the one for
$left_2$ above.
Equality~(5) is Lemma~\ref{lem:H-gluing3};
equality~(6) is equation~\eqref{eq:tradecoprod3};
equality~(7) follows from the fact that 
the expression in big round brackets is the identity on 
$\wh E_2B^{cyc*}\HH[2-n]$;
equality~(8) holds since $\bH_\Gamma$ is a good operation (see Definition~\ref{def:good-alg});
and equality~(9) follows from equation~\eqref{eq:renumber-bdry2}.

In view of equation~\eqref{eq:omcirc}, the three terms on the left
hand side with $\alpha_1\otimes\alpha_2$ inserted combine to
\begin{align}\label{eq:fdiffcirc}
 (left_1+left_2+left_3)(\alpha_1\otimes\alpha_2)=
&\sum_{d\leq s_1+s_2+1}\sum_{\Gamma\in
    \RR_{s_1,s_2;d}^{m}} 
    \phi_f\circ\bH_{\Gamma}(\alpha_1\otimes\alpha_2)\cr
=&\sum_{d\leq s_1+s_2+1}\sum_{\Gamma\in
    \RR_{s_1,s_2;d}}
    \sum_{l\in\Edge}
    \phi_f\circ\bH_{\Gamma,l}(\alpha_1\otimes\alpha_2).
\end{align}
We recall the sign exponent
$$
\fs_\p=\bar R_\Gamma+(n-1)\eta_3(\Gamma)+
\wt r^2_\Gamma(\alpha)
$$
from equation~\eqref{eq:s-is-good} and continue as follows:
\begin{align*}
&(left_1+left_2+left_3)(\alpha_1\otimes\alpha_2)\cr
&\stackrel{(1)}{=}\sum_{d\leq s_1+s_2+1}\sum_{\Gamma\in \RR_{s_1,s_2;d}}(-1)^{\fs_\p}
\phi_f\int_{\Delta^\Gamma_\ver}dR_\Gamma^*G^e(\alpha_1\otimes\alpha_2)\cr
&\stackrel{(2)}{=}
  \sum_{d\leq s_1+s_2+1}\sum_{\Gamma\in \RR_{s_1,s_2;d}}(-1)^{\fs_\p}
  \int_{(\XX_\Gamma)_0}d\wt R_\Gamma^*\wt G^e(\alpha_1\otimes\alpha_2)\cr 
  &\stackrel{(3)}{=}
  \sum_{d\leq s_1+s_2+1}\sum_{\Gamma\in \RR_{s_1,s_2;d}}(-1)^{\fs_\p}
  \int_{\p^\main\XX_\Gamma}\wt R_\Gamma^*\wt G^e(\alpha_1\otimes\alpha_2). 
\end{align*}
Here for equality (1) we use 
Lemma~\ref{lem:fulldiffeven};
for (2) we use equation~\eqref{eq:angenfin}; 
and for (3) we use Proposition~\ref{prop:mainStokes}.

Here the space $\XX=\XX_\Gamma$ appearing in the last two displayed
equations is defined in~\S\ref{ss:general-setting}. It is associated
to a circular graph $\Gamma$ with an extended labelling and one special vertex
of degree $d$ and the evaluation map 
\begin{equation}\label{eq:phi-ev2}
   \phi = \ev_f: \WW = B\times\Delta^{d-1} \to M^d
\end{equation}
as the proper transform of the vertex diagonal
$\Delta_\ver=\Delta^\Gamma_\ver$ times the graph of $\phi$,
$$
   \XX_\Gamma = PT(\Delta_\ver\times gr(\phi))\subset \wt
   X_\Gamma\times\WW = \YY_\Gamma.
$$
According to Remark~\ref{rem:descr-primary}, its primary boundary components
fall into two groups: those corresponding to the boundary of $\WW$,
and those corresponding to the edges of $\Gamma$.
We decompose the first group further as follows, ignoring sets of
measure zero. Let $\pi_{B}$ and $\pi_{\Delta^{d-1}}$ denote the
projection maps from $\XX_\Gamma$ to $B$ and $\Delta^{d-1}$, respectively. Set
$$
  \p_{B}\XX_\Gamma := \pi_{B}^{-1}(\p B),\qquad
  \p^{\Delta^{d-1}}_j\XX_\Gamma := \pi_{\Delta^{d-1}}^{-1}(\p_j\Delta^{d-1}),
$$
where $\p_j\Delta^{d-1}$ is the $j$-th boundary component of
$\Delta^{d-1}$, see~\eqref{eq:bdrycyc}. 
We now discuss the contributions of the various boundary components to
the last displayed integral. 

1. The boundary components $\p_l\XX_\Gamma$ corresponding to nonspecial edges
$l$ cancel by duality, see Figure~\ref{fig:duality}. This is analogous
to the discussion in~\S 7.2 of~\cite{Cieliebak-Volkov}.  

2. The boundary components $\p_{B}\XX_\Gamma$ give rise to the
second term on the right hand side of~\eqref{eq:homotopy2}. 

3. For $j\in\{1,\dots,d_i\}$, the boundary component
$\p^{\Delta^{d-1}}_j\XX_\Gamma$ cancels with $\p_l\XX_{\Gamma_j}$.
Here the tree $\Gamma_j$ is obtained from $\Gamma$ by attaching a
leg at the flag number $j$ at the special vertex
(see~\S\ref{ss:op-graphs2}), and $\p_l\XX_{\Gamma_j}$ is the
boundary component corresponding to the new special (but not doubly
special) edge $l$ created by attaching the leg.
This is analogous to the corresponding cancellation in~\S\ref{ss:chainprod}.

4. The boundary components from doubly special edges give rise to the
Goresky--Hingston term --- the first term on the right hand side
of~\eqref{eq:homotopy2coprod}. 

To see this, consider a labelled circular graph $\Gamma$ with
$s_i$ leaves on the $i$-th boundary component, one special
vertex of degree $d$, and a doubly special edge $l=(A,Z)$. See
Figure~\ref{fig:gh}. 
\begin{figure}
\begin{center}
\includegraphics[angle=0,origin=c,width=\textwidth]{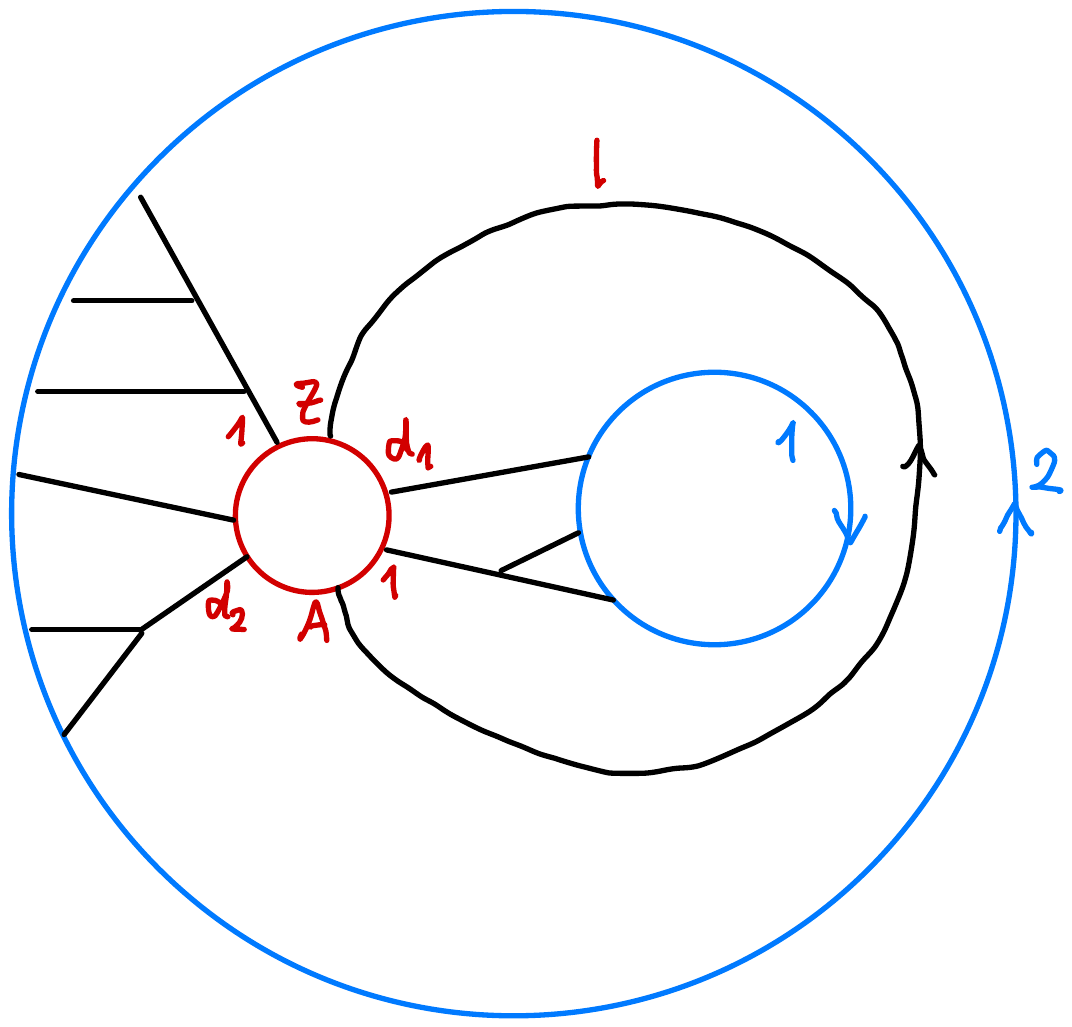}
\vspace{-1cm}
\caption{The Goresky--Hingston term}
\label{fig:gh} 
\end{center}
\end{figure}
Recall that according to our convention from~\S\ref{ss:specvert} the
doubly special edge $l$ inherits its orientation from the second
boundary component. 
We number the special flags between $A$ and $Z$ in the cyclic order
(connected to the first boundary component) by $1,\dots,d_1$, and the
special flags between $Z$ and $A$ by $1,\dots,d_2$. Note that $1\leq
d_i\leq s_i$ and 
$$
  d=d_1+d_2+2 \geq 4.
$$
The following operation on $\Gamma$ will be central for the subsequent
discussion. We collapse the doubly special edge of $\Gamma$ to a
point to obtain two trees $\Gamma_1$ and $\Gamma_2$. Here the tree
$\Gamma_j$, $j=1,2$ has $s_j$ leaves and one special vertex with 
$d_j$ special flags. Both trees naturally come equipped with a generalized labelling.
On the level of ribbon surfaces, this operation corresponds to collapsing the 
central circle of an annulus to a point to get two discs glued at
their centers. This way the trees $\Gamma_1$ and $\Gamma_2$ come out
glued at their special vertices, but this does not affect 
the operations $I_{\Gamma_j}$ associated to them.
Let $\RR_{s_1,s_2;d}^{DS}$ denote the set of isomorphism classes of circular graphs  
with $s_i$ leaves on the $i$-th boundary component, one special
vertex of degree $d$, and a doubly special edge. The collapsing operation above 
yields the bijection
\begin{equation}\label{eq:collaps-graph}
  \RR_{s_1,s_2;d}^{DS}\stackrel{\cong}{\longrightarrow}
\coprod_{d_1+d_2=d-2\atop d_1,d_2\ge 0}
\RR_{s_1;d_1}^{\gen}\times \RR_{s_2;d_2}^{\gen}.
\end{equation}
We abbreviate $\alpha:=\alpha_1\otimes\alpha_2$. Our task is to reinterpret the integral 
$$
  K_\Gamma(\alpha):=(-1)^{\fs_\p}\int_{\p_l\XX_\Gamma}
  \wt R_\Gamma^*\wt G^e(\alpha),
$$
where $\p_l\XX_\Gamma$ denotes the boundary component corresponding to
the doubly special edge $l$
and the sign exponent $\fs_\p$ is defined in~\eqref{eq:s-is-good}.
Recall the subset $D_f\subset B\times[0,1]$ defined
in~\eqref{eq:defDf}, which is a nice submanifold by Lemma~\ref{lem:coprod-trans}(b).
Let 
$$
  \ol t:=(t_1,\dots,t_{d_1},t,\wh t_1,\dots,\wh t_{d_2})\in \Delta^{d-1}
$$
denote the variables on the corresponding simplex $\Delta^{d-1}$. Here
evaluation at $0$ and $t$ corresponds to the flags $A$ and $Z$
(see~\S\ref{ss:specvert}) forming the doubly special edge. Therefore, we set
$$
  D_f^d:=\{(p,\ol t)\in B\times \Delta^{d-1}\mid  (p,t)\in D_f\}.
$$
This differs slightly from the simplified case of
Lemma~\ref{lem:DS-d}, where the special flag $Z$
had number $2$ in the ordering of special flags around the special
vertex. The rest of the notation carries over.

In the following discussion $*,**,***,****$ denote suitable sign
exponents which we will not spell out. 
We use Remark~\ref{rem:smooth-L1} to integrate over the $S^{n-1}$
fibre and kill the $\wt G$ factor in $\wt G^e(\alpha)$, rewriting 
$K_\Gamma(\alpha)$ as an integral over $\wh \Delta_2^l$, and then use
equation~\eqref{eq:smooth-L1} to conclude
$$
  K_\Gamma(\alpha) = (-1)^\star\int_{\Delta_\ver^{\Gamma_1}\times
    \Delta_\ver^{\Gamma_2}  \times gr(\phi_0|_{\overset{\circ}
      {D_f^d}})}(R_{\Gamma_1}\times R_{\Gamma_2})^*\left(
  (G^{e_1}\times G^{e_2})(\alpha_1\otimes\alpha_2)\right).
$$
We apply invariance of integration under the product of the identity map on 
$
\Delta_\ver^{\Gamma_1}\times \Delta_\ver^{\Gamma_2}
$
with the graphical embedding of the source of 
$\phi_0$ into its graph. Observe that the form
$
(G^{e_1}\times G^{e_2})(\alpha_1\otimes \alpha_2)
$
does not depend on the source of $\phi_0$, so the pullback under the
graphical embedding of $\phi_0$ is the same as the pullback under
$\phi_0$ itself. This gives us 
$$
  K_\Gamma(\alpha) =
  (-1)^{\star\star}\int_{\Delta_\ver^{\Gamma_1}\times
    \Delta_\ver^{\Gamma_2}\times \overset{\circ}{D_f^d}} 
  (\id\times\phi_0)^*(R_{\Gamma_1}\times R_{\Gamma_2})^*\left(
  (G^{e_1}\times G^{e_2})(\alpha_1\otimes\alpha_2)\right).
$$
We apply equation~\eqref{eq:keychenan1} to the right hand side of this
equation to get 
$$
  K_\Gamma(\alpha) = (-1)^{\star\star}\int_{\overset{\circ}{D_f^d}}
  \phi_0^*\int_{\Delta_\ver^{\Gamma_1}\times
    \Delta_\ver^{\Gamma_2}}(R_{\Gamma_1}\times R_{\Gamma_2})^*\left( 
  (G^{e_1}\times G^{e_2})(\alpha_1\otimes\alpha_2)\right).
$$
By Fubini's theorem for fibre integration and the 
definition~\eqref{eq:defI} of the operation $I_\Gamma$ this becomes
\begin{equation}\label{eq:Iprefin}
K_\Gamma(\alpha)=
(-1)^{\star\star\star}\int_{\overset{\circ} {D_f^d}}\phi_0^*(I_{\Gamma_1}(\alpha_1)\times 
I_{\Gamma_2}(\alpha_2)).
\end{equation}
The restriction $ev_f|_{D_f^d}$ writes out as
\begin{equation}\label{eq:evf}
\begin{aligned}
&ev_f|_{D_f^d}(p,t_1,\dots,t_{d_1},t,\wh t_1,\dots,\wh t_{d_2})\cr
&=\bigl(f_p(0),f_p(t_1),\dots,f_p(t_{d_1}),f_p(t),
f_p(\wh t_1),\dots,f_p(\wh t_{d_2})\bigr).
\end{aligned}
\end{equation}
Note that $f_p(0)=f_p(t)$ since we have restricted to $D_f^d$. 
In what follows a circle over a manifold with corners will denote its
interior. Motivated by the definition of the loop coproduct 
in~\S\ref{ss:transvers}, we introduce the reparametrization diffeomorphism 
\begin{equation*}\label{eq:reparam}
\psi:\overset{\circ}{D_f}\times 
\overset{\circ}{\Delta^{d_1}}\times 
\overset{\circ}{\Delta^{d_2}}
\longrightarrow \overset{\circ}{D_f^d},\qquad
  \bigl((p,\tau),(s_1,\dots,s_{d_1}),(\wh s_1,\dots,\wh s_{d_2})\bigr)
  \mapsto (p,\ol t)
\end{equation*}
defined by 
$$
  t:=\tau,\qquad t_j:=ts_j,\,j=1,\dots,d_1,\qquad 
\wh t_j:=t+\wh s_j(1-t),\, j=1,\dots,d_2. 
$$
Since we have restricted to the interiors, the map $\psi$ is a
diffeomorphism. 
Recall from~\eqref{eq:chain-coproduct} 
the chain level coproduct
$$
   \ol\lambda f=(\lambda^1f,\lambda^2f):D_f^d\to \Lambda\times\Lambda.
$$
Recall that $\phi=ev_f$. Denote by 
$$
  \ev^0f:\overset{\circ}{D_f}\to M,\qquad
  \wh{ev_{\lambda^1f}}:\overset{\circ}{D_f}\times\overset{\circ}{\Delta^{d_1}}\to M^{d_1},\qquad
  \wh{ev_{\lambda^2f}}:\overset{\circ}{D_f}\times\overset{\circ}{\Delta^{d_2}}\to M^{d_2}
$$
the evaluation maps at time $0$, at times $s_1,\dots,s_{d_1}$, and at
times $\wh s_1,\dots,\wh s_{d_2}$, respectively. 
Then by definition of $\ol\lambda f$ and $\psi$ we have the relation
\begin{equation}\label{eq:commpsi}
  ev_f\circ\psi = \ev^0_f\times \wh{ev_{\lambda^1f}}\times \ev^0_f\times 
  \wh{ev_{\lambda^2f}}.
\end{equation}
This allows us to finish the manipulation of $K_{\Gamma}(\alpha)$ to get
\begin{equation}\label{eq:KGamma-fin}
  K_{\Gamma}(\alpha) = (-1)^{\star\star\star\star}\int_{\overset{\circ}{D_f}\times 
  \overset{\circ}{\Delta^{d_1}}\times \overset{\circ}{\Delta^{d_2}}}
  \wh{ev_{\ol\lambda f}}^*(I_{\Gamma_1}(\alpha_1)\times 
I_{\Gamma_2}(\alpha_2))
\end{equation}
To see this, we apply invariance of integration under $\psi$
to $K_\Gamma(\alpha)$ from~\eqref{eq:Iprefin}, use 
relation~\eqref{eq:commpsi}, and note that the form  
$I_{\Gamma_1}(\alpha_1)\times I_{\Gamma_2}(\alpha_2))$ 
does not depend on the factors
occupied by the maps $\ev^0_f$ in~\eqref{eq:commpsi}.

Finally, we sum equation~\eqref{eq:KGamma-fin} over all
$\Gamma\in\RR_{s_1,s_2;d}^{DS}$ to obtain 
\begin{align*}
&\sum_{\Gamma\in \RR_{s_1,s_2;d}^{DS}}
K_\Gamma(\alpha_1\otimes\alpha_2)\cr
&\stackrel{(1)}{=}
\sum_{d_1+d_2=d-2\atop
d_1,d_2\ge 0}\sum_{\Gamma_j\in \RR_{s_j;d_j}^\gen}
(-1)^{\star\star\star\star}
\int_{\overset{\circ}{D_f}\times 
\overset{\circ}{\Delta^{d_1}}\times 
\overset{\circ}{\Delta^{d_2}}}
\wh{ev_{\ol\lambda f}}^*(I_{\Gamma_1}(\alpha_1)\times 
I_{\Gamma_2}(\alpha_2))
\cr
&\stackrel{(2)}{=}
\sum_{d_1+d_2=d-2\atop
d_1,d_2\ge 0}\sum_{\Gamma_j\in \RR_{s_j;d_j}}
(-1)^{\star\star\star\star}
\int_{\overset{\circ}{D_f}\times 
\overset{\circ}{\Delta^{d_1}}\times 
\overset{\circ}{\Delta^{d_2}}}
\wh{ev_{\ol\lambda f}}^*(N_{an}\circ I_{\Gamma_1}(\alpha_1)
\times N_{an}\circ I_{\Gamma_2}(\alpha_2))\cr
&\stackrel{(3)}{=}
\sum_{d_1+d_2=d-2\atop
d_1,d_2\ge 0}\sum_{\Gamma_j\in \RR_{s_j;d_j}}
I_\lambda^2(
\G_{\Gamma_1}(\alpha_1)\otimes
\G_{\Gamma_2}(\alpha_2))(\ol\lambda f)\cr
&\stackrel{(4)}{=}
\sum_{d_1+d_2=d-2\atop
d_1,d_2\ge 0}
I_\lambda^2(
\G_{s_1;d_1}(\alpha_1)\otimes
\G_{s_2;d_2}(\alpha_2))(\ol\lambda f)
\end{align*}
Here for equality~(1) we use equations~\eqref{eq:collaps-graph}
and~\eqref{eq:KGamma-fin}. 
For equality~(2) we recall that $\RR_{s_j;d_j}$, $j=1,2$ is a
fundamental locus for the free $\Z_{d_j}$ action on $\RR_{s_j;d_j}^\gen$
by cyclicly relabelling special flags, and we trade  
each orbit for the symmetrization operator $N_{an}$
using~\eqref{eq:rotation-spec-cor}. 
Equality~(3) follows from the definition~\eqref{eq:prodpair} of the
Chen map $I_\lambda^2$ and the definition~\eqref{eq:def-G-Gamma} of 
$\G_\Gamma$ with $\Gamma\in\{\Gamma_1,\Gamma_2\}$;
and equality~(4) follows from the definition~\eqref{eq:def-G-sd} 
of $\G_{s;d}$ with $(s,d)\in \{(s_1,d_1),(s_2,d_2)\}$.

Since the last displayed expression is the first term on the right
hand side of~\eqref{eq:homotopy2coprod}, this concludes the proof of 
Proposition~\ref{prop:chaincoprod}.

\appendix

\section{Proof of Proposition~\ref{prop:const2}}\label{app:const2}

The proof of Proposition~\ref{prop:const2} follows by combining Jones' 
article~\cite{Jones} and Goodwillie's theorem~\cite{Goodwillie}. 
It is most elegantly carried out in the formalism of mixed complexes
due to Kassel~\cite{Kassel87}, which we will first recall following the
notation in~\cite{Cieliebak-Volkov-cyc}.

\begin{definition}\label{def:mixed}
A {\em mixed complex}\footnote{
The degrees are opposite to those in~\cite{Cieliebak-Volkov-cyc} where
$|\delta|=+1$, $|D|=-1$, and $|u|=+2$.}
$(C,\delta,D)$ is a $\Z$-graded $\R$-vector space
$$
  C=\bigoplus_{k\in\Z}C_k
$$ 
with two linear maps $\delta,D:C\to C$ of degrees $|\delta|=-1$ and $|D|=+1$ satisfying 
$$
   \delta^2=0,\qquad D^2=0,\qquad \delta D+D\delta=0. 
$$
\end{definition}

Let $u$ be a formal variable of degree $|u|=-2$. To a mixed complex 
$(C,\delta,D)$
we associate the chain complex 
$$
   C[[u,u^{-1}] := \bigoplus_{k\in\Z}C_k[[u,u^{-1}],\qquad
   \delta_u:=\delta+uD,
$$
where $C_k[[u,u^{-1}]$ denotes the space of Laurent series
$\sum_{i\geq i_0}c_iu^i$ with $c_i\in C_{k+2i}$.
This complex has the subcomplex $C[[u]]$ of degreewise power series
in $u$ and the quotient complex 
\begin{equation}\label{eq:cyc-versions}
  C[u^{-1}] := C[[u,u^{-1}]/uC[[u]].
\end{equation}
The corresponding homologies are denoted by $HC_*^{[[u,u^{-1}]]}$
etc. They fit into the tautological exact sequence
(see~\cite[Proposition 2.5]{Cieliebak-Volkov-cyc})
\begin{equation}\label{eq:tautseq}
\dots\longrightarrow HC_{*+2}^{[[u]]}
\stackrel{\cdot u}{\longrightarrow} 
HC_*^{[[u,u^{-1}]}
\stackrel{p_*}{\longrightarrow}
HC_*^{[u^{-1}]}
\stackrel{D_{0*}}{\longrightarrow}
HC_{*+1}^{[[u]]}\longrightarrow\dots
\end{equation}
where $p$ is the map forgetting positive powers of $u$, and $D_0$ is
the map $D$ applied to the constant term in $u$. An easy staircase
argument shows that if $D$ vanishes on the $\delta$-homology of $C$,
then the connecting morphism $D_{0*}$ in the tautological exact sequence
vanishes.

This formalism (due to H.~Cartan) provides an alternative description of 
equivariant homology.
Namely, an $S^1$-space $Y$ gives rise to a mixed complex $(C_*(Y),d,\Delta)$,
where $(C_*(Y),d)$ is the singular chain complex and $\Delta$ is the
BV operator (called $Q$ in~\cite{Cieliebak-Volkov-cyc}).
We denote the homology groups corresponding to this mixed complex
by $H_*^{[[u,u^{-1}]}Y$ etc. 
Now we have the following canonical isomorphism (see~\cite{Jones})
\begin{equation}\label{eq:Cartan}
H_*^{S^1}Y\cong HC_*^{[u^{-1}]}Y.
\end{equation}
An analogous statement holds for relative homology. In the proof below
we will identify the two sides of~\eqref{eq:Cartan}. 
Our main example will be the loop space $\Lambda$ with the obvious 
$S^1$-action be reparametrizing loops.

Now we are ready to prove Proposition~\ref{prop:const2}, which we
restate for the reader's convenience.

\begin{proposition}
Let $X$ be a simply connected topological space, $\Lambda=C^0(S^1,X)$
its loop space, and $\Lambda_0\subset\Lambda$ the subspace of constant
loops. Pick a basepoint $q_0\in X$ and consider the inclusion of pairs
$\iota_0:(\Lambda,q_0)\rightarrow (\Lambda,\Lambda_0)$. Then the induced map 
\begin{equation}
  \iota_{0*}^{S^1}:H_*^{S^1}(\Lambda,q_0)\rightarrow H_*^{S^1}(\Lambda,\Lambda_0)
\end{equation}
on relative equivariant homology is injective.
\end{proposition}

\begin{proof}
On the subspace of constant loops $\Lambda_0$ the $S^1$ action is
trivial, so the BV operator vanishes on homology and we get 
\begin{equation}\label{eq:Qvanish}
  H_*^{[u^{-1}]}\Lambda_{0*}\stackrel{\Delta_{0*}=0}{\longrightarrow}
  H_{*+1}^{[[u,u^{-1}]}\Lambda_0.
\end{equation}
Consider the following diagram:
\begin{equation*}
\begin{tikzcd}
{H_*^{[[u,u^{-1}]}\Lambda} \arrow[rr, "p_*"]                &  & {H_*^{[u^{-1}]}\Lambda} \arrow[r, "\Delta_{0*}"]        & {} \\
{H_*^{[[u,u^{-1}]}(q_0)} \arrow[d, "p_*"] \arrow[u, "\cong"] &  &                                                 &    \\
{H_*^{[u^{-1}]}(q_0)} \arrow[rr] \arrow[rruu, "\alpha"]      &  &
{H_*^{[u^{-1}]}\Lambda_0} \arrow[uu, "\iota_*"]  &   
\end{tikzcd}
\end{equation*}
Here the first vertical arrow induced by the inclusion of a point is
an isomorphism by Goodwillie's theorem~\cite[Corollary V.3.3]{Goodwillie}.
The upper row is part of the tautological exact
sequence~\eqref{eq:tautseq} for the mixed complex associated to the
$S^1$-space $\Lambda$, the map $\iota$ is the natural inclusion of 
constant loops, and the map $\alpha$ is induced by the inclusion of
the point $q_0$.   
We have the following sequence of inclusions and equalities:
$$
\im\iota_*\stackrel{(1)}{\subset}
\ker \Delta_{0*}\stackrel{(2)}{=}\im\alpha
\stackrel{(3)}{\subset}\im\iota_*.
$$
Here inclusion~(1) follows from~\eqref{eq:Qvanish}; equality~(2) follows 
from $\im p_*=\im\alpha$ and exactness of the upper row of the
diagram; and inclusion~(3) is clear. We conclude
\begin{equation}\label{eq:iotaalpha}
 \im\iota_*=\im\alpha.
\end{equation}
Consider now the diagram
\begin{equation*}
\begin{tikzcd}
{H_*^{[u^{-1}]}(q_0)} \arrow[r, "\alpha", hook] \arrow[d] & {H_*^{[u^{-1}]}\Lambda} \arrow[r, "\beta", tail] \arrow[d, "="] & {H_*^{[u^{-1}]}(\Lambda,q_0)} \arrow[r, "0"] \arrow[d, "\iota_{0*}^{S^1}"] & {H_{*-1}^{[u^{-1}]}(q_0)} \\
{H_*^{[u^{-1}]}\Lambda_0} \arrow[r, "\iota_*"]           & {H_*^{[u^{-1}]}\Lambda} \arrow[r, "\kappa"]                     & {H_*^{[u^{-1}]}(\Lambda,\Lambda_0).}                                  &                         
\end{tikzcd}
\end{equation*}
Here the horizontal rows are the exact sequences for the pairs
$(\Lambda,q_0)$ and $(\Lambda,\Lambda_0)$, and the vertical 
arrows are induced by the obvious inclusions. Since the inclusion of
the point $q_0$ in $\Lambda$ admits an equivariant left inverse (the
constant map $\Lambda\to q_0$), we get that $\alpha$ 
(induced by the inclusion of $q_0$) admits a left inverse as
well. Thus $\alpha$ is injective, the rightmost map in the
corresponding exact sequence is zero, and $\beta$ is
surjective. Therefore, the above diagram induces the diagram
\begin{equation*}
\begin{tikzcd}
{H_*^{[u^{-1}]}\Lambda/\im\alpha} \arrow[d, "="] \arrow[r,"\bar\beta", "\cong"'] & {H_*^{[u^{-1}]}(\Lambda,q_0)} \arrow[d, "\iota_{0*}^{S^1}"] \\
{H_*^{[u^{-1}]}\Lambda/\im\iota_*} \arrow[r, hook, "\bar\kappa"]                       & {H_*^{[u^{-1}]}(\Lambda,\Lambda_0)}                  
\end{tikzcd}
\end{equation*}
Here the map $\bar\beta$ induced by $\beta$ is an isomorphism because
$\beta$ is surjective, and the left vertical arrow is the identity by
equation~\eqref{eq:iotaalpha}. Together with injecivity of the map
$\bar\kappa$ induced by $\kappa$, this implies that the map
$\iota_{0*}^{S^1}$ is injective. 
\end{proof}


\bibliographystyle{abbrv}
\bibliography{./000_chen}

\end{document}